\documentclass[a4paper]{amsart}

\usepackage{amsmath,amssymb,amsthm}
\usepackage[utf8]{inputenc}
\usepackage[T1]{fontenc}
\usepackage{scalerel}
\usepackage{multicol}
\usepackage{enumerate}
\usepackage{comment} 
\usepackage{microtype}
\usepackage[dvipsnames,svgnames,x11names,hyperref]{xcolor}
\usepackage{stmaryrd} 
\usepackage{wrapfig}
\definecolor{darkred}{RGB}{139,0,0}
\definecolor{darkblue}{RGB}{0,0,139}
\definecolor{darkgreen}{RGB}{0,100,0}
\usepackage[linktocpage]{hyperref}
\hypersetup{
  colorlinks   = true, 
  urlcolor     = darkred, 
  linkcolor    = darkred, 
  citecolor   = darkgreen}
\usepackage{tikz}
\usepackage{tikz-cd}
\usepackage{bbm}
\usepackage{verbatim}
\usepackage{mathtools}

\newtheorem{MainThm}{Theorem}

\newtheorem{thm}{Theorem}[section]
\newtheorem{cor}[thm]{Corollary}
\newtheorem{lem}[thm]{Lemma}
\newtheorem{prop}[thm]{Proposition}
\newtheorem*{claim}{Claim}
\newtheorem{addendum}[thm]{Addendum}

\theoremstyle{definition}
\newtheorem{defn}[thm]{Definition}
\theoremstyle{remark}
\newtheorem{rem}[thm]{Remark}
\newtheorem{example}[thm]{Example}
\newtheorem{warning}[thm]{Warning}
\numberwithin{equation}{section}

\newcommand{\bF}{\mathbb{F}}

\newcommand{\bN}{\mathbb{N}}

\newcommand{\bQ}{\mathbb{Q}}
\newcommand{\bR}{\mathbb{R}}

\newcommand{\bZ}{\mathbb{Z}}

\newcommand{\gA}{\bold{A}}
\newcommand{\gB}{\bold{B}}
\newcommand{\gC}{\bold{C}}
\newcommand{\gD}{\bold{D}}
\newcommand{\gE}{\bold{E}}
\newcommand{\gF}{\bold{F}}

\newcommand{\gI}{\bold{I}}
\newcommand{\gJ}{\bold{J}}
\newcommand{\gK}{\bold{K}}
\newcommand{\gL}{\bold{L}}
\newcommand{\gM}{\bold{M}}
\newcommand{\gN}{\bold{N}}

\newcommand{\gP}{\bold{P}}

\newcommand{\gR}{\bold{R}}
\newcommand{\gS}{\bold{S}}
\newcommand{\gT}{\bold{T}}

\newcommand{\gX}{\bold{X}}

\newcommand{\cO}{\mathcal{O}}

\newcommand\lra{\longrightarrow}

\newcommand\colim{\operatorname*{colim}}

\newcommand\rk{\mathrm{rk}}

\newcommand{\bunit}{\mathbbm{1}}
\newcommand{\bk}{\mathbbm{k}}

\newcommand\Ext{\mathrm{Ext}}
\newcommand\Tor{\mathrm{Tor}}
\newcommand\Cotor{\mathrm{Cotor}}
\newcommand\St{\mathrm{St}}
\newcommand\GL{\mathrm{GL}}

\newcommand\fil{\mathrm{fil}}
\newcommand\gr{\mathrm{gr}}
\newcommand\End{\mathrm{End}}
\newcommand\Hom{\mathrm{Hom}}
\newcommand\Map{\mathrm{Map}}
\newcommand\map{\mathrm{map}}

\newcommand\Ker{\mathrm{Ker}}

\renewcommand\Bar{\mathrm{Bar}}
\newcommand\Cobar{\mathrm{Cobar}}
\newcommand\Prim{\mathrm{Prim}}

\title{A chromatic approach to homological stability}
\author{Oscar Randal-Williams}
\email{o.randal-williams@dpmms.cam.ac.uk}
\address{Centre for Mathematical Sciences\\
Wilberforce Road\\
Cambridge CB3 0WB\\
UK}

\begin{document}

\begin{abstract}
We propose a way to organise the subject of ``higher-order homological stability'', in the context of a graded $E_2$-algebra $\mathbf{R}$, along the same lines that the chromatic perspective organises stable homotopy theory. 

From this point of view proving a (higher-order) homological stability theorem corresponds to producing Smith--Toda complexes in the category of $\mathbf{R}$-modules: using this perspective we prove that whenever $\mathbf{R}$ is defined over a field of positive characteristic and satisfies some standard properties, there is a sequence of higher-order homological stability theorems whose slopes tend to 1.

We propose that in a higher-order stable range the ``stable homology'' should be interpreted as certain Bousfield localisations in the category of $\mathbf{R}$-modules, leading to a chromatic tower and monochromatic layers. Given the existence of suitable Smith--Toda complexes we establish several properties of these localisations, in particular explaining how higher-order stabilisation maps yield periodic families in the monochromatic layers.

We explain how to associate to such an $\mathbf{R}$ a Hopf algebra which completely governs the kinds of higher-order stability maps that it enjoys, in the sense that the cohomology of this Hopf algebra has precisely the same stability patterns as $\mathbf{R}$. When $\mathbf{R}$ comes from a sequence of groups, this Hopf algebra has a concrete description as the coinvariants of the $E_1$-Steinberg modules.
\end{abstract}

\maketitle

\setcounter{tocdepth}{1} 
\tableofcontents

\newpage

\section{Introduction}

\subsection{Premise}\label{sec:Premise}
Recent thinking on the phenomenon of homological stability for a sequence
$$X_0 \lra X_1 \lra X_2 \lra \cdots$$
of spaces (or groups, or ....) has been to arrange their chains into an $E_2$-algebra
$$\gR \simeq \bigoplus_{n \geq 0} C_*(X_n;\bk)$$
in the category of chain complexes of $\bk$-modules with an additional $\bN$-grading, and to recognise the original stabilisation maps as multiplication by an element $\sigma \in C_0(X_1;\bk)$ with respect to this $E_2$-structure. From this point of view formulating the question of homological stability uses only a small part of the $E_2$-algebra structure on $\gR$, but in answering this question one can make use of the full $E_2$-algebra structure. This point of view, initiated in \cite{KupersMiller} and developed in \cite{e2cellsI}, has led to improvements in homological stability ranges for many examples \cite{e2cellsII, e2cellsIII, e2cellsIV, RWDedekind, KMPImproved, JansenMiller, Jansen2, BMSPartialBases}.

More interestingly, it has suggested that homological stability is perhaps only the beginning of the story. In \cite{e2cellsII} Galatius, Kupers, and the author applied this point of view to the mapping class groups $\Gamma_{g,1}$ of genus $g$ surfaces with one boundary, and showed that not only does one have $H_d(\Gamma_{g,1}, \Gamma_{g-1,1};\bZ)=0$ for $d < \tfrac{2}{3}g$, corresponding to homological stability for these groups, but there are \emph{secondary} stabilisation maps
$$\varphi_* : H_{d-2}(\Gamma_{g-3,1}, \Gamma_{g-4,1};\bZ) \lra H_d(\Gamma_{g,1}, \Gamma_{g-1,1};\bZ)$$
which are epimorphisms for $d < \tfrac{3}{4}g$ and isomorphisms for $d < \tfrac{3}{4}g - 1$. This kind of secondary homological stability has since been discovered in further examples and contexts \cite{MillerWilson, Himes}. As these maps change the genus $g$ by 3, rather than by 1, it would perhaps be better to call the phenomenon {periodicity} rather than stability. 

In this paper we will make the case that such higher-order homological stability phenomena should be organised along the same lines that the chromatic perspective organises stable homotopy theory. This will provide answers to questions such as:
\begin{enumerate}[(i)]
\item When should ``higher-order homological stability'' be expected to hold for $\gR$?

\item What should be meant by e.g.\ ``homology of $\gR$ in the secondary stable range'', and what properties does it have?

\item How does one tell what kind of higher-order stabilities a given $\gR$ should have?

\end{enumerate}

\subsection{Context and notation}

In Section \ref{sec:Foundations} we will give a detailed account of the technical framework in which we shall work, but for now we make the following brief remarks. Fix a commutative ring $\bk$ and write $\mathsf{D}(\bk)$ for the derived $\infty$-category of $\bk$ (i.e.\ the category of $H\bk$-module spectra, or the localisation of the category of chain complexes $\mathsf{Ch}(\bk)$ at the quasiisomorphisms). This is a stable $\infty$-category, and has a symmetric monoidality given by (derived) tensor product $- \otimes_\bk -$. The $\bZ$-graded objects in this category $\mathsf{D}(\bk)^\bZ$ inherits a symmetric monoidality by Day convolution, and our basic datum will be
$$\gR \in \mathsf{Alg}_{E_2}(\mathsf{D}(\bk)^\bZ),$$
an $E_2$-algebra in $\mathsf{D}(\bk)^\bZ$. Most of our discussion will take place in the $\infty$-category $\gR\text{-}\mathsf{mod}$ of left $\gR$-module objects in $\mathsf{D}(\bk)^\bZ$, which has an $E_1$-monoidal structure $- \otimes_\gR -$ because $\gR$ is an $E_2$-algebra.

Objects $X \in \mathsf{D}(\bk)$ have homology(=homotopy) groups, which we write as $\pi_d(X)$. They are represented by the object $S^d \in \mathsf{D}(\bk)$, the $d$-th suspension of the monoidal unit $\bk$. Objects $X \in \mathsf{D}(\bk)^\bZ$ therefore have bigraded homotopy groups $\pi_{n,d}(X) = \pi_d(X(n))$.  They are represented by the object $S^{n,d} \in \mathsf{D}(\bk)^\bZ$ which evaluates to $S^d$ at the integer $n$ and to 0 otherwise. We will say that an object $X \in \mathsf{D}(\bk)^\bZ$ \emph{has a slope $\lambda$ vanishing line} if there exists a $\kappa$ such that $\pi_{n,d}(X)=0$ for $d < \lambda n + \kappa$.

\subsection{Kinds of stability}\label{sec:KindsOfStab}

We wish to contemplate the phenomenon of homological stability in sufficient breadth that it includes all the following kinds of examples.

\subsubsection{Ordinary homological stability}\label{sec:ordinarystab}

The most common situation in the study of homological stability is to have an element $\sigma \in \pi_{1,0}(\gR)$, form the mapping cone $\gR/\sigma$ of the left $\gR$-module endomorphism
$$S^{1,0} \otimes \gR \cong \gR \otimes S^{1,0} \overset{\gR \otimes \sigma}\lra \gR \otimes \gR \overset{- \cdot -}\lra \gR,$$
and be able to show that there is a vanishing line of the form $\pi_{n,d}(\gR/\sigma)=0$ for $d < \lambda n + \kappa$. When $\gR$ arises from taking the $\bk$-chains of an $E_2$-algebra $\coprod_{n \geq 0} X_n$ in $\bN$-graded spaces, then
$$\pi_{n,d}(\gR/\sigma) = H_d(X_n, X_{n-1} ; \bk)$$
so a vanishing range for these groups is equivalent to a range in which the maps $H_d(X_{n-1};\bk) \to H_d(X_n;\bk)$ are isomorphisms or epimorphisms.

This kind of homological stability theorem, in most cases with $\pi_{n,d}(\gR/\sigma)=0$ for $d < \tfrac{1}{2}n$, is known for an enormous number of examples, far too many to do justice to here. In the special case where $\gR = \bigoplus_{n \geq 0} C_*(G_n;\bk)$ comes from a family of discrete groups $G_n$, examples include: symmetric groups; braid groups; Coxeter groups; general linear groups or unitary groups over quite general rings; mapping class groups of orientable surfaces, or of non-orientable surfaces, or of certain 3-manifolds; automorphism groups of free groups, or of free nilpotent groups, or of RAAGs. There are also many examples not associated to discrete groups, including: moduli spaces of high-dimensional manifolds; Artin monoids; diagram algebras such as the Iwahori--Hecke or Temperley--Lieb algebra; reductive Borel--Serre spaces.

\subsubsection{Homological multi-stability}\label{sec:multistab}

The following simple example is quite enlightening. Consider the $E_2$-algebra $\coprod_{n \geq 0} X_n$ where $X_n = \mathrm{Conf}_n^\text{rb}(I^2)$ is the space of configurations of $n$ points in the interior of $I^2 = [0,1]^2$, each of which is coloured either red or blue. This is the free $\bN$-graded $E_2$-algebra on two generators both of grading 1. Taking $\bk$-chains gives $\mathbf{RB} \simeq \gE_2(S^{1,0} r \oplus S^{1,0} b)$. This has $\pi_{*,0}(\mathbf{RB}) = \bk[r,b]$ a polynomial ring on two generators, so $\mathbf{RB}$ cannot have ordinary homological stability in the sense of the previous section, as the ranks of these groups grow.

However $\mathbf{RB}$ does enjoy the following property: the square
\begin{equation}\label{eq:RedBlueStab}
\begin{tikzcd}
\mathbf{RB} \otimes S^{2,0} \rar{- \cdot r} \dar{- \cdot b}& \mathbf{RB} \otimes S^{1,0} \dar{- \cdot b}\\
\mathbf{RB} \otimes S^{1,0} \rar{- \cdot r} & \mathbf{RB},
\end{tikzcd}
\end{equation}
which is commutative up to a preferred homotopy, is (co)cartesian in a range of degrees. This is immediate from the known homology of the free $E_2$-algebra $\gE_2(S^{1,0} r \oplus S^{1,0} b)$, given by the work of F.\ Cohen \cite{CLM}. In fact, the total homotopy cofibre $\mathbf{RB}/(r,b)$ of this square has $\pi_{n,d}(\mathbf{RB}/(r,b))=0$ for $d < \tfrac{1}{2}n$. 

Viewing this square as being \emph{cocartesian} in a range of degrees, it gives a way to determine $\mathbf{RB}(n)$, in a range of degrees, from $\mathbf{RB}(n-1)$ and $\mathbf{RB}(n-2)$, namely as the pushout of the diagram obtained by removing the bottom-right corner. Similarly, viewing it as being \emph{cartesian} in a range of degrees it gives a way to determine $\mathbf{RB}(n)$, in a range of degrees, from $\mathbf{RB}(n+1)$ and $\mathbf{RB}(n+2)$, namely as the corresponding pullback. Either interpretation lets us reasonably refer to the high (co)cartesianness of this square as a kind of homological stability theorem for $\mathbf{RB}$. (In ordinary homology stability we try to approximate $\gR(n)$ using $\{\gR(m)\}_{m<n}$ or $\{\gR(m)\}_{m>n}$, and the recipe for doing so is $\gR(n-1)$ or $\gR(n+1)$ respectively: in the situation at hand the recipe is just slightly more complicated.) The slight downside is that at the level of homotopy groups \eqref{eq:RedBlueStab} only gives a Mayer--Vietoris-style sequence
$$ \pi_{n-2,d}(\mathbf{RB}) \overset{(r_*, b_*)}\lra \pi_{n-1,d}(\mathbf{RB}) \oplus \pi_{n-1,d}(\mathbf{RB}) \overset{(\begin{smallmatrix}b_*\\-r_*\end{smallmatrix})}\lra \pi_{n,d}(\mathbf{RB}) \overset{\partial}\lra \pi_{n-2,d-1}(\mathbf{RB})$$
valid in a stable range. This is not exactly a recipe for determining $\pi_{n,*}(\mathbf{RB})$ from $\{\pi_{m,*}(\mathbf{RB})\}_{m<n}$, but will have to do. (In fact in this particular example, and when $\bk$ a field, the connecting maps $\partial$ are zero. But this cannot be expected in general, see Example \ref{ex:MVConnectingMap}.)

\subsubsection{Secondary homological stability}\label{sec:SecondaryStab}

We describe again the ``secondary homological stability'' for mapping class groups of surfaces, in the terms we have now introduced. There is a $\mathbf{MCG} \in \mathsf{Alg}_{E_2}(\mathsf{D}(\bZ)^\bZ)$ such that $\mathbf{MCG}(g) \simeq C_*(\Gamma_{g,1};\bZ)$, the $\bZ$-chains of the mapping class group of a genus $g$ surface with one boundary. It has ordinary homological stability using the generator $\sigma$ of $H_0(\Gamma_{1,1};\bZ) = \pi_{1,0}(\mathbf{MCG})$, in the sense that $\pi_{g,d}(\mathbf{MCG}/\sigma)=0$ for $d < \tfrac{2}{3}g$. Furthermore,  Galatius, Kupers, and the author \cite{e2cellsII} have constructed a left $\mathbf{MCG}$-module map
$$\varphi : \mathbf{MCG}/\sigma \otimes S^{3,2} \lra \mathbf{MCG}/\sigma$$
and shown that its cofibre $\mathbf{MCG}/(\sigma, \varphi)$ satisfies $\pi_{g,d}(\mathbf{MCG}/(\sigma, \varphi))=0$ for $d < \tfrac{3}{4}g$. The map $\varphi$ does \emph{not} come from an endomorphism of $\mathbf{MCG}$ reduced modulo $\sigma$: it is essential that $\sigma$ be killed before $\varphi$ can be defined.

\subsection{A general higher-order homological stability theorem}

The examples of the previous section lead us to set our goal as follows: a higher order homological stability theorem for $\gR$ means producing an $\gR$-module $\gR/(\alpha_1, \ldots, \alpha_r)$ as the iterated cofibre of a sequence of endomorphisms, and proving that it has a vanishing line for its homotopy groups.

We will usually assume that $\bk$ is a field and $\gR$ is an $E_2$-algebra in $\mathsf{D}(\bk)^\bZ$ satisfying:

\begin{enumerate}
\item[(C)] It is \emph{connected} in the sense that
$$\pi_{n,d}(\gR)=0 \text{ for $d<0$ or $n<0$, and } 1 : \bk \to \pi_{0,0}(\gR) \text{ is an isomorphism.}$$

\item[(SCE)] By (C) there is a unique augmentation $\epsilon : \gR \to \bk$, making $\bk$ into a $\gR$-module and we will further assume the \emph{standard connectivity estimate}
$$\pi_{n,d}(\bk \otimes_\gR \bk)=0 \text{ for }d < n.$$

\item[(F)] It is \emph{finite type} in the sense that
$$\pi_{n,n}(\bk \otimes_\gR \bk) \text{ is finite-dimensional for each $n$}.$$

\end{enumerate}
In practice we often verify (F) by instead verifying:
\begin{enumerate}
\item[(F$^\prime$)] $\pi_{n,d}(\gR)$ is finite-dimensional for each $n$ and $d$.
\end{enumerate}
That (F$^\prime$) implies (F) is by running the bar spectral sequence. If (F) is only known to hold for $n < N$ then all our results will go through in a range of degrees, but for simplicity we will not keep track of this.

For examples of $E_2$-algebras $\gR \simeq \bigoplus_{n \geq 0} C_*(G_n;\bk)$ coming from a sequence of groups, Axiom (C) always holds and Axiom (F$^\prime$) is simply that each $H_d(G_n;\bk)$ is finitely-generated. Axiom (SCE) has an interpretation in terms of the high-connectivity of the ``splitting complex'' \cite[Section 17.2]{e2cellsI} associated to the groups $G_n$, and is known to hold in many examples. We will discuss some examples in Section \ref{sec:examples}, but the main focus in this paper is on the consequences of satisfying these axioms.

\begin{MainThm}\label{MainThm:A}
Let $\bk$ be a field of positive characteristic and $\gR \in \mathsf{Alg}_{E_2}(\mathsf{D}(\bk)^\bZ)$ satisfy (C), (SCE), and (F). Let $\lambda < 1$ be given. Then there is a sequence of $\gR$-module endomorphisms
\begin{equation}\label{eq:STcxDefn}
\alpha_i : \gR/(\alpha_1, \ldots, \alpha_{i-1}) \otimes S^{n_i, d_i} \lra \gR/(\alpha_1, \ldots, \alpha_{i-1}), \text{ for } 1 \leq i \leq r,
\end{equation}
with cofibres $\gR/(\alpha_1, \ldots, \alpha_{i})$, such that:
\begin{enumerate}[(i)]
\item\label{it:MainThm:A:1} The sequence of slopes $\tfrac{d_i}{n_i}$ is non-decreasing, they are all strictly less than $\lambda$, and they are all rational numbers of the form $\tfrac{k}{k+1}$.

\item\label{it:MainThm:A:2} $\gR/(\alpha_1, \ldots, \alpha_r)$ has a vanishing line of slope $\lambda$, and in fact has a vanishing line of slope $\min\{\tfrac{k}{k+1} \, | \, k \in \bN \text{ with } \lambda \leq \tfrac{k}{k+1}\}$.

\item\label{it:MainThm:A:NonNilp} Each endomorphism $\alpha_i$ is non-nilpotent.

\item\label{it:MainThm:A:Simultaneous} Furthermore, assuming that $\gR$ is $E_3$, the $\alpha$'s $\alpha_s, \ldots, \alpha_t$ having a fixed slope $\lambda' < \lambda$ may all be realised as endomorphisms of $\gR/(\alpha_1, \ldots, \alpha_{s-1})$, and may be arranged into a coherently homotopy commutative cube. 
\end{enumerate}
\end{MainThm}

\begin{rem}\mbox{}
\begin{enumerate}[(i)]
\item With only the stated axioms one cannot expect a similar result for slopes $\lambda\geq 1$, see Section \ref{sec:NotBeyondSlope1}.

\item It was a surprise to the author that the analogue of this theorem is false for $\bk=\bQ$, at least if we insist on property (\ref{it:MainThm:A:NonNilp}): see Section \ref{sec:CharZeroCounterexample} for a counterexample, which we learnt from Robert Burklund.

\item We will prove more precise, and more technical, versions of this theorem in Sections \ref{sec:HigherStabMaps} and \ref{sec:HigherStabMapsRevisited}. Those sections develop a collection of tools to construct (non-nilpotent) endomorphisms of $\gR$-modules, of more general scope than Theorem \ref{MainThm:A}. 

\item The additional assumption in part (\ref{it:MainThm:A:Simultaneous}) that $\gR$ be an $E_3$-algebra could be relaxed if one knew that certain structural results about connected graded commutative Hopf algebras over a field of positive characteristic also held in the non-commutative setting. We discuss this in Remark \ref{rem:RevistedMethodForE2Alg}.

\end{enumerate}
\end{rem}

By analogy with chromatic homotopy theory we call an $\gR$-module $\gR/(\alpha_1, \ldots, \alpha_{i-1})$ obtained as the iterated cofibre of a sequence of inductively-defined endomorphisms a \emph{Smith--Toda complex}. This name does not fit perfectly but the alternative ``generalised Moore complex'' suggests that there is a homology theory in which they realise a certain cyclic module, which we doubt is true.

Such complexes, and their non-nilpotent endomorphisms $\alpha_i$, lead to infinite families of non-trivial elements in $\pi_{*,*}(\gR)$ in exactly the same way (and with the same kind of indeterminacies and difficulties) that ordinary Smith--Toda complexes lead to the $\alpha, \beta, \gamma, \ldots$ families in the stable homotopy groups of spheres. See \cite[Sections 2.4, 2.5]{RavenelNP} for an overview, which the reader will easily translate into our setting. We will not dwell on it here, as the discussion in the following section accounts for such patterns more systematically.

A significant distinction between our setting and ordinary chromatic homotopy theory is that our $\alpha_i$ have no intrinsic meaning. They are not e.g.\ characterised by their effect on a suitable homology theory (though our discussion in Section \ref{sec:HigherStabMapsRevisited} tries to come close to this), and this makes them quite different to the $v_n$'s. This presents a difficulty in proving Theorem \ref{MainThm:A}, which we will solve by working in a suitable deformation of the category of $\gR$-modules in which analogues of the $\alpha_i$ can be characterised and then running something close to the proof of the Periodicity Theorem of Hopkins--Smith \cite{HopkinsSmith}. But even here there is a basic distinction to the ordinary situation, in that our $\alpha_i$ can be coarsely ordered by their slopes $\tfrac{d_i}{n_i}$ but there can be several of the same slope and there is no preferred ordering (or even basis) of the collection of $\alpha$'s of a given slope.

\subsection{Homology in a higher-order stable range}

For an $\gR$-module $\gM$ we propose the definition of its ``stable homology up to slope $\lambda$'' to be the Bousfield localisation $L_\lambda^f(\gM)$ of $\gM$ away from the class $\mathcal{A}_\lambda^f$ of all finite $\gR$-modules with a slope $\lambda$ vanishing line. This is a localisation at an essentially small class of compact objects, so can be constructed ``telescopically'' and is therefore smashing: $L_\lambda^f(\gM) \simeq L_\lambda^f(\gR) \otimes_\gR \gM$. We will give an introduction to these notions in Section \ref{sec:localisation}.

The existence of a Smith--Toda complex $\gR/(\alpha_1, \ldots, \alpha_r)$  as guaranteed by Theorem \ref{MainThm:A} allows us to establish some basic properties of these approximations.

\begin{MainThm}\label{MainThm:B}
Let $\gR/(\alpha_1, \ldots, \alpha_r)$ be a Smith--Toda complex with a slope $\lambda$ vanishing line, and such that the sequence of slopes $\tfrac{d_i}{n_i}$ is non-decreasing and are all $< \lambda$. Write $\lambda' := \tfrac{d_r}{n_r}$.
\begin{enumerate}[(i)]
\item The homotopy fibre $C_\lambda^f(\gR)$ of the localisation map $\gR \to L_\lambda^f(\gR)$ has a vanishing line of slope $\lambda$, and also has a vanishing line of slope $\lambda'$.

\item If $\lambda' < \bar{\lambda} \leq \lambda$ then any finite $\gR$-module with a slope $\bar{\lambda}$ vanishing line in fact has a slope $\lambda$ vanishing line. Equivalently $L^f_{\bar{\lambda}} = L^f_\lambda$.
\end{enumerate}
\end{MainThm}

\begin{rem}
We prove more precise, and more technical, versions of these statements in Theorems \ref{thm:HigherStabRanges} and \ref{thm:STQuantisation}.
\end{rem}

\begin{example}
In the case of ordinary homological stability this recovers what we expect. If $\gR$ has a $\sigma \in \pi_{1,0}(\gR)$ such that $\gR/\sigma$ has a slope $\lambda$ vanishing line, then $L_\lambda^f(\gR) \simeq \sigma^{-1}\gR$ and the fibre of the map $\gR \to \sigma^{-1}\gR$ clearly has a slope $\lambda$ vanishing line. This is the ordinary meaning of ``stable homology''. When $\gR$ arises from taking the $\bk$-chains of an $E_2$-algebra in spaces, the homology of this localisation is accessible via the Group-completion Theorem \cite{McDuffSegal}, and indeed that has been the most successful method for calculating stable homology.
\end{example}

\begin{example}\label{ex:RedBlueStabHom}
A more interesting example is the $E_2$-algebra $\mathbf{RB}$ from Section \ref{sec:multistab}, of chains on the configuration spaces of red-or-blue points in $I^2$, where $\mathbf{RB}/(r,b)$ has a slope $\tfrac{1}{2}$ vanishing line. In this case we will explain in Example \ref{ex:RBLocalissation} how $L^f_{1/2}(\mathbf{RB})$ may be described by the homotopy cartesian square
\begin{equation*}
\begin{tikzcd}
L^f_{1/2}(\mathbf{RB}) \rar \dar & r^{-1}\mathbf{RB} \dar\\
b^{-1}\mathbf{RB} \rar & r^{-1} b^{-1} \mathbf{RB}.
\end{tikzcd}
\end{equation*}
For $\bk$ a field one may calculate $\pi_{*,*}(\mathbf{RB})$ and hence $\pi_{*,*}(L^f_{1/2}(\mathbf{RB}))$, using F.~Cohen's description of the homology of free $E_2$-algebras. We do so in Example \ref{ex:RBLocalissation}. Perhaps surprisingly, it is no longer connective: for example $\partial(r^{-1} b^{-1}) \neq 0 \in \pi_{-2,-1}(L^f_{1/2}(\mathbf{RB}))$.
\end{example}

\subsubsection{Quantisation of vanishing lines}

When $\bk$ is a field of positive characteristic and (C), (SCE), and (F) hold, so that Theorem \ref{MainThm:A} applies, it follows from Theorem \ref{MainThm:B} that a finite $\gR$-module $\gM$ with a vanishing line of slope $\lambda <1$ in fact has a vanishing line of slope $\min_k\{\tfrac{k}{k+1} \, | \, \lambda \leq \tfrac{k}{k+1}\}$. We call this phenomenon \emph{quantisation of vanishing lines}. Applied to $\gR/\sigma$ it explains why all known ordinary homological stability theorems have slopes of the form $\tfrac{k}{k+1}$ (when they are $<1$).

\subsubsection{Monochromatic layers and Adams periodicity}

A consequence of the quantisation of vanishing lines is that the only distinct localisation functors for slopes $<1$ are the $L^f_{k/k+1}$, which assemble into a tower
$$\gM \lra \cdots \lra L^f_{3/4}(\gM) \lra L^f_{2/3}(\gM) \lra L^f_{1/2}(\gM) \lra 0,$$
letting us define the monochromatic objects
$$M^f_{k/k+1}(\gM) := \mathrm{fib}(L^f_{k/k+1}(\gM) \to L^f_{k-1/k}(\gM)).$$
These objects have a slope $\tfrac{k-1}{k}$ vanishing line, and enjoy periodicities of slope exactly $\tfrac{k-1}{k}$, governed by the endomorphisms $\alpha_s, \ldots, \alpha_t$ of this slope: we call this \emph{Adams periodicity}. The most general formulation is quite technical, and applying it involves producing Smith--Toda complexes $\gR/(\alpha_1, \ldots, \alpha_{s-1}, \alpha_s^{p^{M_s}}, \ldots, \alpha_t^{p^{M_t}})$ and in particular making sense of this expression. We will explain this in Section \ref{sec:AdamsPeriodicity}, but for now give the following representative example.

\begin{example}
Suppose an $E_2$-algebra $\gR$ has a $\sigma \in \pi_{1,0}(\gR)$ such that $\pi_{n,d}(\gR/\sigma)=0$ for $d < \tfrac{2}{3}n$, and a $\varphi : S^{3,2} \otimes \gR/\sigma \to \gR/\sigma$ such that $\pi_{n,d}(\gR/(\sigma, \varphi))=0$ for $d < \tfrac{3}{4}n$. (This is the case for $\gR = \mathbf{MCG}$ for example.) For some $N \in \bN$ suppose furthermore that $\gR/\sigma^N$ has an endomorphism $\varphi_r : S^{3r,2r} \otimes \gR/\sigma^N \to \gR/\sigma^N$ such that $\gR/(\sigma^N, \varphi_r)$ has a slope $\tfrac{3}{4}$ vanishing line ($\varphi_r$ will usually be related to the $r$ fold iterate of $\varphi$, though it is an endomorphism of a different object so work is needed to make this precise). Then Adams periodicity, which is Theorem \ref{thm:AdamsPeriodicityGeneral}, provides isomorphisms
$$(\varphi_r)_* : \pi_{n-3r,d-2r}(M^f_{3/4}(\gR)) \overset{\sim}\lra \pi_{n,d}(M^f_{3/4}(\gR))$$
defined for $d < \tfrac{2n+2N-6}{3}$.

By the methods used to prove Theorem \ref{MainThm:A}, when $\bk$ is a field of positive characterstic and $\gR$ satisfies axioms (C), (SCE), and (F) then for each $N$ one can indeed find an endomorphism $\varphi_r$ of $\gR/\sigma^N$ for some $r=r(N)$ such that $\gR/(\sigma^N, \varphi_r)$ has a slope $\tfrac{3}{4}$ vanishing line. The above implies that each homotopy class $x \in \pi_{n,d}(M^f_{3/4}(\gR))$ lies in some ``$\varphi$-periodic family'', but the period of this family depends on $d - \tfrac{2}{3}n$: the further $x$ is from the vanishing line of $\pi_{*,*}(M^f_{3/4}(\gR))$, the longer the period.
\end{example}

\subsection{Determining the kinds of stability}

Theorem \ref{MainThm:A} guarantees that higher-order stabilisation maps $\alpha_i$ exist, but does not say, for a given $\gR$, what they are. Concretely: what slopes arise as the slope of an $\alpha_i$, or more generally how many $\alpha_i$'s there are of a given slope? For example:
\begin{enumerate}[(i)]
\item In the example $\mathbf{RB}$ we must form $\mathbf{RB}/(r,b)$ to get a vanishing line of slope $\tfrac{1}{2}$, and it not possible to achieve this with $\mathbf{RB}/\sigma$ for any single $\sigma$ of slope 0.
\item In the example $\mathbf{MCG}$ the object $\mathbf{MCG}/\sigma$ has a vanishing line of slope $\tfrac{2}{3}$, and we did not need to kill any endomorphism of slope $\tfrac{1}{2}$ to achieve this.
\end{enumerate}
What part of the structure of $\mathbf{RB}$ versus $\mathbf{MCG}$ accounts for these differences?

The object
$$\bk \otimes_\gR \bk \simeq \Bar(\gR)$$
inherits an $E_1$-coalgebra structure, being the Koszul dual of the $E_1$-algebra $\gR$, but as $\gR$ is in fact $E_2 = E_1 \otimes E_1$ there is also a residual $E_1$-algebra structure, making $\bk \otimes_\gR \bk$ into an $E_1$-bialgebra. Axiom (SCE) says that the bigraded homotopy groups of this object below the diagonal vanish, and we collect its diagonal homotopy groups together to form
$$\Delta_\gR := \bigoplus_{n \geq 0} \pi_{n,n}(\bk \otimes_\gR \bk).$$
This has the structure of a graded connected bialgebra, and so a Hopf algebra: we call it the \emph{stability Hopf algebra} associated to $\gR$. Perhaps the central message of this paper is the slogan:
\begin{center}
$\Delta_\gR$ precisely controls the kinds of stability that $\gR$ has.
\end{center}
We will make this slogan precise in Sections \ref{sec:StabHopfAlg} and \ref{sec:HigherStabMapsRevisited}, but the rough idea is as follows.

Taking cobar constructions gives an $E_2$-algebra map
$$\gR \simeq \Cobar(\bk \otimes_\gR \bk) \lra \Cobar(\Delta_\gR) =: \mathbf{r},$$
and base-change along this map converts a Smith--Toda complex for $\gR$ to one for $\mathbf{r}$. We will provide a collection of tools which allow one to reverse this process: from the existence of a Smith--Toda complex for $\mathbf{r}$ one may find a Smith--Toda complex of a similar form for $\gR$. The most important of these tools is that $\gR \to \mathbf{r}$ satisfies a limited kind of \emph{Nilpotence Theorem}: if $\psi$ is a slope $\lambda < 1$ endomorphism of an $\gR$-module $\gM$ which has a slope $\lambda$ vanishing line, then $\psi$ is nilpotent if and only if $\mathbf{r} \otimes_\gR \psi$ is. 

The $E_2$-algebra $\mathbf{r} = \Cobar(\Delta_\gR)$ has homotopy groups
$$\pi_{n,d}(\mathbf{r}) = \Cotor^{n-d}_{\Delta_\gR}(\bk, \bk)_n,$$
i.e.\ the cohomology of the Hopf algebra $\Delta_\gR$. Conversely, for any connected graded Hopf algebra $A$ of finite type over a field $\bk$, $\mathbf{a} := \Cobar(A)$ is an object of $\mathsf{Alg}_{E_2}(\mathsf{D}(\bk)^\bZ)$ satisfying our axioms (C), (SCE), and (F), so the methods of this paper apply to it. In this sense any improvement beyond the slogan must take place purely in the realm of cohomology of connected graded Hopf algebras.

When $\gR \simeq \bigoplus_{n \geq 0} C_*(G_n;\bk)$ arises from a sequence of groups, axiom (SCE) is usually established by showing that certain $G_n$-complexes $T^{E_1}(n)$ associated to these groups are wedges of $n$-spheres \cite[Definition 17.6]{e2cellsI}. In this case there are $G_n$-modules $\mathrm{St}^{E_1}(n) := \widetilde{H}_n(T^{E_1}(n);\bk)$, the \emph{$E_1$-Steinberg modules}, and there is a concrete description of the stability Hopf algebra as:
$$\Delta_\gR = \bigoplus_{n \geq 0} H_0(G_n ; \mathrm{St}^{E_1}(n)).$$
For such examples one should try to understand $\Delta_\gR$ by analysing these coinvariants. There are very few $E_2$-algebras $\gR$ occurring ``in nature'' (e.g.\ those discussed in Section \ref{sec:KindsOfStab}) for which we have complete knowledge of $\Delta_\gR$, and when we do it is usually because this Hopf algebra is an exterior algebra $\Lambda_\bk[\bar{\sigma}]$ on one generator (which implies that $\gR/\sigma$ has a slope 1 vanishing line). But there are many examples where we have partial information about $\Delta_\gR$. In Section \ref{sec:examples} we will discuss some of these from the point of view taken here.

\vspace{1ex}

\noindent\textbf{Acknowledgements}. 
Arriving at the formulation presented here has been a long process, with many extensive false turns: work of Palmieri \cite{Palmieri} and of Krause \cite{KrauseThesis} was valuable for convincing me that I was eventually on the right track. A conversation with Robert Burklund was extremely useful, and I particularly acknowledge the influence of his way of thinking on this work. Discussions with S{\o}ren Galatius, Jeremy Hahn, Maxime Ramzi, Dhruv Ranganathan, Jan Steinebrunner, and Kelly Wang have been clarifying, even though much of what was discussed has not made it in to this paper. I am grateful to Alexis Aumonier and Greg Arone for pointing out a mistake in an earlier version. I was partially supported by the ERC under the European Union's Horizon 2020 research and innovation programme (grant agreement No.\ 756444), and by the Danish National Research Foundation through the Copenhagen Centre for Geometry and Topology (DNRF151).

\section{Foundations}\label{sec:Foundations}
The ideas we are trying to explain can no doubt be implemented in many possible frameworks, in a way that should make no essential difference to the underlying ideas. For concreteness we will work in an $\infty$-categorical setting, relying on Lurie's \cite{HA} for technical support. It will sometimes be convenient to import results from \cite{e2cellsI}, which is written in model-categorical terms: we comment on this in Section \ref{sec:ImportingFromE2cellsI} below. For readers familiar with that source we try to use compatible notation, with the significant exception that here an $E_k$-algebra $\gR$ is by default unital, and we occasionally pass to non-unital $E_k$-algebras by taking augmentation ideals, whereas in \cite{e2cellsI} the symbol $\gR$ typically denoted a non-unital $E_k$-algebra, which was occasionally unitalised.

\subsection{$\bk$-modules}

We write $\mathsf{D}(\bk)$ for the derived $\infty$-category of the commutative ring $\bk$, which we choose to consider as the category of (left) $H\bk$-module spectra \cite[Remark 7.1.1.16]{HA}: as is typical, we often just write $\bk$ for $H\bk$. This is a presentable \cite[Corollary 4.2.3.7]{HA} stable $\infty$-category \cite[Corollary 7.1.1.5]{HA}, and we write $\map_\bk(X,Y) := \map_{\mathsf{D}(\bk)}(X,Y)$ for the mapping spectra in $\mathsf{D}(\bk)$. Objects $X$ of this category have homotopy groups
$$\pi_d(X) := [S^d, X]_{\mathsf{D}(\bk)},$$
where $S^d := \Sigma^d \bk$, and these have the structure of modules over $\bk = [\bk, \bk]_{\mathsf{D}(\bk)}$.

As $\bk$ is a commutative algebra object in spectra, the symmetric monoidal structure on spectra \cite[\S 4.8.2]{HA} induces a symmetric monoidal structure $\otimes_\bk$ on $\mathsf{D}(\bk)$, whose values may be computed by the two-sided bar construction \cite[Theorem 4.5.2.1]{HA}. This monoidal structure preserves colimits in each variable \cite[Corollary 4.4.2.15]{HA}. The homotopy groups of a tensor product may be approached by a (strongly convergent) K{\"u}nneth spectral sequence \cite[Proposition 7.2.1.19]{HA}:
\begin{equation}\label{eq:KunnethSS} 
E^2_{p,q} = \bigoplus_{q'+q''=q}\Tor_p^\bk(\pi_{q'}(X), \pi_{q''}(Y)) \Longrightarrow \pi_{p+q}(X \otimes_\bk Y).
\end{equation}
For us $\bk$ will usually be a field, or a PID, in which case this simplifies as usual.

The right adjoint to $- \otimes_\bk X$ defines internal mapping objects $\underline{\map}_\bk(X,Y)$. These are $\bk$-module spectra, with underlying spectra 
$$\map_\bk(\bk, \underline{\map}_\bk(X,Y)) \simeq \map_\bk({\bk \otimes_\bk X}, Y) \simeq \map_\bk(X,Y).$$
To not be overwhelmed by notation, we will write $\map_\bk(X,Y)$ for these internal mapping $\bk$-module spectra.

\subsection{$\bZ$-graded objects}

Our underlying category will usually be the presentable stable $\infty$-category
$$\mathsf{D}(\bk)^\bZ = \mathsf{Fun}(\bZ, \mathsf{D}(\bk))$$
of $\bZ$-graded objects in $\mathsf{D}(\bk)$. For an object $X$ of this category we write $X(n) \in \mathsf{D}(\bk)$ for its value at $n \in \bZ$. Such objects have bigraded homotopy groups 
$$\pi_{n,d}(X) := \pi_d(X(n)) = [S^{n,d}, X]_{\mathsf{D}(\bk)^\bZ},$$
where $S^{n,d}$ is the $\bZ$-graded object which is $S^d$ in grading $n$ and zero otherwise.

This category is equipped with the Day convolution symmetric monoidal structure \cite[\S 2.2.6]{HA} which we continue to write as $\otimes_\bk$: explicitly
$$(X \otimes_\bk Y)(n) \simeq \bigoplus_{n'+n'' = n} X(n') \otimes_\bk Y(n'').$$
From this formula it is clear that this monoidal structure also preserves colimits in each variable. The K{\"u}nneth spectral sequence \eqref{eq:KunnethSS} yields a similar spectral sequence for computing the bigraded homotopy groups $\pi_{*,*}(X \otimes_\bk Y)$.

The right adjoint to $- \otimes_\bk X$ defines internal mapping objects $\underline{\map}_\bk(X,Y)$, having values $\underline{\map}_\bk(X,Y)(n) \simeq \map_\bk(S^{n,0} \otimes_\bk X,Y)$.

\subsection{$E_k$-algebras}

Let $(\mathsf{C}, \otimes, \bunit)$ be a symmetric monoidal $\infty$-category which is presentable, has all small limits, and such that the tensor product commutes with colimits in each variable. We write $\mathsf{Alg}_{E_k}(\mathsf{C})$ for the $\infty$-category of (unital) $E_k$-algebras \cite[\S 5.1]{HA} in $\mathsf{C}$, which is again presentable using \cite[Corollary 3.2.3.5]{HA}. Similarly, we write $\mathsf{Alg}_{E_k}^\mathrm{aug}(\mathsf{C})$ for the $\infty$-category of $E_k$-algebras $\gR$ equipped with an augmentation $\epsilon : \gR \to \bunit$. We usually write $\gI := \mathrm{fib}(\epsilon : \gR \to \bunit)$, a non-unital $E_k$-algebra, calling it the augmentation ideal of $\gR$. 

Dunn--Lurie additivity \cite[\S 5.1.2]{HA} allows us to identify 
$$\mathsf{Alg}_{E_{k+\ell}}(\mathsf{C}) \simeq \mathsf{Alg}_{E_{k}}(\mathsf{Alg}_{E_\ell}(\mathsf{C})),$$
similarly with augmentations.

If $\mathsf{C}$ is stable (so that $\gR \simeq \bunit \oplus \gI$) then taking augmentation ideals identifies $\mathsf{Alg}_{E_k}^\mathrm{aug}(\mathsf{C})$ with the $\infty$-category $\mathsf{Alg}_{E_k^\text{nu}}(\mathsf{C})$ of non-unital $E_k$-algebras. The various forgetful functors
$$\mathsf{Alg}_{E_k}(\mathsf{C}) \to \mathsf{C}_{\bunit /} \to \mathsf{C} \quad \quad \mathsf{Alg}_{E_k}^\mathrm{aug}(\mathsf{C}) \to \mathsf{C}_{\bunit / / \bunit } \to \mathsf{C} \quad\quad \mathsf{Alg}_{E_k^\text{nu}}(\mathsf{C}) \to \mathsf{C},$$
for which we generically write $U$, all have left adjoint ``free $E_k$-algebra'' functors for which we write $\gE_k^{(\text{nu})}(-)$, taking a little care with units and augmentations. We will often use F.~Cohen's description of $\pi_{*,*}(\gE_k(X))$ for $X \in \mathsf{D}(\bk)^\bZ$ when $\bk$ is a field, in terms of the induced product, Browder bracket, and Dyer--Lashof operations. Our reference for this will be \cite[Section 16]{e2cellsI}.

If $\mathsf{C}$ is pointed then any object of $\mathsf{C}$ may be considered as a non-unital $E_k$-algebra, defining a ``trivial algebra'' functor $Z^{E_k^\text{nu}} : \mathsf{C} \to \mathsf{Alg}_{E_k^\text{nu}}(\mathsf{C})$. (This is inconvenient to implement in the way that operads are handled in \cite{HA}. Treating operads as monoids in symmetric sequences in $\mathsf{C}$---which has been compared with Lurie's definition in \cite{HaugsengSSeq, HaugsengSSeqAlg} or in \cite[Section 5.2.5]{ShiThesis} with our presentability hypothesis---it is simple: as $\mathsf{C}$ is pointed there is a unique map $\epsilon: E_k^\text{nu} \to \bunit_{\mathsf{SSeq}(\mathsf{C})}$ of operads in $\mathsf{C}$, restricting along which gives the required functor $Z^{E_k^\text{nu}} = \epsilon^* : \mathsf{C} = \mathsf{Alg}_{\bunit_{\mathsf{SSeq}(\mathsf{C})}}(\mathsf{C}) \to \mathsf{Alg}_{E_k^\text{nu}}(\mathsf{C})$.) The trivial algebra functor has a left adjoint
$$Q^{E_k^\text{nu}} : \mathsf{Alg}_{E_k^\text{nu}}(\mathsf{C}) \lra  \mathsf{C},$$
the (derived) indecomposables. As $U^{E_k^\text{nu}} Z^{E_k^\text{nu}} = \mathrm{Id}_\mathsf{C}$, on left adjoints we obtain $Q^{E_k^\text{nu}} \gE_k^\text{nu}(X) \simeq X$.

\subsection{Modules over an $E_k$-algebra}\label{sec:Modules}

For $\gR \in \mathsf{Alg}_{E_k}(\mathsf{C})$ with $k \geq 1$, there is an $\infty$-category $\gR\text{-}\mathsf{mod}$ of left $\gR$-modules, obtained by forgetting down from $E_k$-algebras to $E_1$-algebras, using that the latter are equivalent to associative algebras \cite[Example 5.1.0.7]{HA}, and invoking \cite[Definition 4.2.1.13]{HA}. By the discussion in \cite[\S 4.2.3]{HA}, $\gR\text{-}\mathsf{mod}$ is presentable, has all small limits, and the forgetful map $\gR\text{-}\mathsf{mod} \to \mathsf{C}$ preserves limits and colimits. In particular, if $\mathsf{C}$ is stable then so is $\gR\text{-}\mathsf{mod}$. A morphism $\phi : \gR \to \gS$ of unital $E_1$-algebras induces a functor
$$\mathrm{Res}^\gS_\gR :  \gS\text{-}\mathsf{mod} \lra \gR\text{-}\mathsf{mod},$$
which preserves limits and colimits. It has a left adjoint $\mathrm{Ind}^\gS_\gR$, extension of scalars, which may be computed by the relative tensor product $\gS \otimes_\gR -$ \cite[Proposition 4.6.2.17]{HA}. A commonly-used example of this this construction is to apply it to the unit map $\bk \to \gR$ to obtain the free $\gR$-module functor
$$\gR \otimes_\bk - : \mathsf{D}(\bk)^\bZ \lra \gR\text{-}\mathsf{mod}.$$

Using Dunn--Lurie additivity and \cite[Corollary 4.8.5.20]{HA}, if $k \geq 2$ then the category $\gR\text{-}\mathsf{mod}$ is equipped with an $E_{k-1}$-monoidal structure $\otimes_\gR$, and if $\phi : \gR \to \gS$ is a map of $E_k$-algebras then extension of scalars $\mathrm{Ind}^\gS_\gR \simeq \gS \otimes_\gR -  : \gR\text{-}\mathsf{mod} \to \gS\text{-}\mathsf{mod}$ obtains the structure of an $E_{k-1}$-monoidal functor. The underlying bifunctor
$$- \otimes_\gR - : \gR\text{-}\mathsf{mod} \times \gR\text{-}\mathsf{mod} \lra \gR\text{-}\mathsf{mod}$$
may be described in terms of the relative tensor product as $A_\gR \otimes_{\gR \otimes_\bk \gR}(- \otimes -)$, for the $\gR\text{-}\gR\otimes_\bk\gR$-bimodule $A_\gR$ obtained by restricting the $\gR\text{-}\gR$-bimodule $\gR$ along the pair of algebra maps $(Id, \mu) : (\gR, \gR \otimes_\bk \gR) \to (\gR, \gR)$, where $\mu$ is the $E_{k-1}$-algebra multiplication on $\gR$ in the category of associative algebras given by Dunn--Lurie additivity. (The bimodule $A_\gR$ was constructed explicitly in \cite[\S 12.2.2]{e2cellsI} and called the ``adapter''.) 


We will usually have a fixed $E_2$-algebra $\gR$ in $\mathsf{D}(\bk)^\bZ$, and work in the \mbox{($E_1$-)}monoidal stable $\infty$-category $\gR\text{-}\mathsf{mod}$. As this category is only monoidal, some points must be considered a little more carefully than usual. One is that it is convenient to have a supply of central objects, i.e.\ objects for which left and right tensoring are equivalent. The most common central objects are free $\gR$-modules, via
\begin{align*}
(\gR \otimes_\bk X) \otimes_\gR - &\simeq \mathrm{Ind}_{\gR \otimes_\bk \gR}^\gR[ (\gR \otimes_\bk X) \otimes_\bk -] \simeq \mathrm{Ind}_{\gR \otimes_\bk \gR}^\gR[ \mathrm{Ind}_\bk^\gR (X) \otimes_\bk \mathrm{Ind}_\gR^\gR(-)]\\
&\simeq \mathrm{Ind}^\gR_{\bk \otimes_\bk \gR} [X \otimes_\bk -] 
\end{align*}
and the same showing $- \otimes_\gR (\gR \otimes_\bk X) \simeq \mathrm{Ind}^\gR_{\gR \otimes_\bk \bk} [- \otimes_\bk X]$; the symmetry of $\otimes_\bk$ shows these are the same as $\gR$-modules. Another kind of central object is as follows.

\begin{lem}\label{lem:E2AlgebraIsModuleCentral}
If $\phi : \gR \to \gS$ is a map of $E_k$-algebras with $k \geq 2$, then there is a natural equivalence $\gS \otimes_\gR - \simeq - \otimes_\gR \gS : \gR\text{-}\mathsf{mod} \to \gR\text{-}\mathsf{mod}$.
\end{lem}
\begin{proof}
If $k \geq 3$ then $\otimes_\gR$ is $E_2$-monoidal which gives such an isomorphism. If $k=2$ then let $A_\gR$ be the $\gR\text{-}\gR\otimes\gR$-bimodule discussed above, and consider
$$A_\gR \otimes_{\gR \otimes \gR} (\gS \otimes \gR)$$
as an $\gR\text{-} \gR$-bimodule (i.e.\ neglect the right $\gS$-module structure). We claim this is equivalent to $\gS$ as a $\gR\text{-}\gR$-bimodule, and the same for $A_\gR \otimes_{\gR \otimes \gR} (\gR \otimes \gS)$, which together proves the lemma.

Functoriality of the construction of the bimodules $A$ gives a map of $\gR\text{-}\gR\otimes\gR$-bimodules $A_\gR \to \mathrm{Res}^{\gS\text{-}\gS \otimes\gS}_{\gR\text{-} \gR \otimes \gR} A_\gS \simeq \mathrm{Res}^{\gR\text{-}\gS \otimes\gR}_{\gR\text{-} \gR \otimes \gR}\mathrm{Res}^{\gS \text{-}\gS }_{\gR\text{-}\gS \otimes\gR} \gS$, which is adjoint to a map
$$A_\gR \otimes_{\gR \otimes \gR} (\gS \otimes \gR) \simeq \mathrm{Ind}^{\gR\text{-}\gS \otimes\gR}_{\gR\text{-} \gR \otimes \gR} A_\gR \lra \mathrm{Res}^{\gS \text{-}\gS}_{\gR\text{-}\gS \otimes\gR} \gS.$$
This is an equivalence: when we neglect the module structures the target is $\gS$, the source is $\gR \otimes_{\gR} \gS \simeq \gS$, and the map is easily checked to be the identity. Restricting the right module structure along $\gR \simeq \bunit \otimes \gR \to \gS \otimes \gR$ it remains an equivalence, and the target becomes $\gS$ as an $\gR\text{-} \gR$-bimodule.
\end{proof}

A further subtlety of the ($E_1$-)monoidal setting concerns internal mapping objects. The functor $- \otimes_\gR \gM$ has a right adjoint $\underline{\map}_\gR^r(\gM, -)$, and similarly the functor $\gM \otimes_\gR -$ has a right adjoint $\underline{\map}_\gR^l(\gM, -)$. These two internal mapping objects will generally be different as $\gR$-modules, but neglecting their $\gR$-module structures they are equivalent as objects of $\mathsf{D}(\bk)^\bZ$, via
$$\map_\gR(\gR \otimes_\bk S^{n,0}, \underline{\map}_\gR^r(\gM, -)) \simeq \map_\gR((\gR \otimes_\bk S^{n,0}) \otimes_\gR \gM, -) \simeq \map_\gR(S^{n,0} \otimes_\bk \gM, -)$$
 $$\map_\gR(\gR \otimes_\bk S^{n,0}, \underline{\map}_\gR^l(\gM, -)) \simeq \map_\gR(\gM \otimes_\gR(\gR \otimes_\bk S^{n,0}), -) \simeq \map_\gR(S^{n,0} \otimes_\bk \gM, -).$$
We therefore write $\underline{\map}_\gR(\gM, \gN) \in \mathsf{D}(\bk)^\bZ$ for the underlying object of either of these mapping $\gR$-modules, which is equivalently the object corepresenting $X \mapsto \map_\gR((\gR \otimes_\bk X) \otimes_\gR \gM, \gN)$.
 
A special case of this is the \emph{endomorphism object} $\End_\gR(\gM)$, an $E_1$-algebra in $\gR\text{-}\mathsf{mod}$ universal among $E_1$-algebras acting on $\gM$ on the left \cite[Section 4.7.1]{HA}. The above discussion identifies its underlying $\gR$-module with $\underline{\map}^r_\gR(\gM,\gM)$.
 
\subsection{Referring to ``Cellular $E_k$-algebras''}\label{sec:ImportingFromE2cellsI}

Occasionally we cite results from \cite{e2cellsI}, which is developed with model category theoretic foundations instead of the $\infty$-categorical ones used here. We briefly explain how those results may be safely imported into the setting used here. For the underlying model category $\mathsf{S}$ we take $H\bk\text{-}\mathsf{mod}$, modules over the Eilenberg--MacLane spectrum for $\bk$, equipped with the projectively transferred model structure from symmetric spectra equipped with the absolute projective stable model structure. The discussion in Sections 7.2.4 and 7.3.3 of \cite{e2cellsI} shows that this satisfies the axioms of \cite[Section 7.1]{e2cellsI}. We then take $\mathsf{C} = (H\bk\text{-}\mathsf{mod})^\bN$, with the projective model structure.

To compare with the setting here, we can use Haugseng's rectification theorem \cite[Theorem 1.1]{HaugsengSSeqAlg} for algebras over an operad. The symmetric monoidal $\infty$-category $\mathsf{Sp}$ of spectra can be presented as the $\infty$-categorical localisation at the weak equivalences of the 1-category of symmetric spectra \cite[Proposition 4.1.7.4]{HA}, so by Haugseng's theorem the $\infty$-categorical localisation of the 1-category $H\bk\text{-}\mathsf{mod}$ is equivalent to $H\bk$-modules in the $\infty$-category $\mathsf{Sp}$, and hence to $\mathsf{D}(\bk)$ by \cite[Remark 7.1.1.16]{HA}. Similarly, the $\infty$-categorical localisation at the weak equivalences of the 1-categories $\mathsf{Alg}_{E_2}((H\bk\text{-}\mathsf{mod})^\bN)$ used in \cite{e2cellsI} are the $\infty$-categories $\mathsf{Alg}_{E_2}(\mathsf{D}(\bk)^\bN)$ of $\bN$-graded $E_2$-algebras, which can then be included into  $\mathsf{Alg}_{E_2}(\mathsf{D}(\bk)^\bZ)$ (by axiom (C) all the $\bZ$-graded $E_2$-algebras we consider come from $\bN$-graded ones). Similarly, an associative algebra object $A$ in $(H\bk\text{-}\mathsf{mod})^\bN$ (such as the strictifications $\overline{\gR}$ constructed in \cite[Section 12.2]{e2cellsI}) presents an associative (equivalently $E_1$-) algebra $A'$ in $\mathsf{D}(\bk)^\bN$, and the localisation at the weak equivalences of the 1-category of $A$-modules is equivalent to the $\infty$-category of $A'$-modules in $\mathsf{D}(\bk)^\bN$, i.e.\ the non-negatively graded $A'$-modules in $\mathsf{D}(\bk)^\bZ$.

By this mechanism the results of \cite{e2cellsI} may be freely imported: we will not comment on it further.

\subsection{Bar, Cobar, and Koszul duality}\label{sec:BarCobar}

We follow \cite[Section 5.2]{HA}. The category of $E_1$-coalgebras in a symmetric monoidal category $\mathsf{C}$ is $\mathsf{coAlg}_{E_1}(\mathsf{C}) := \mathsf{Alg}_{E_1}(\mathsf{C}^{op})^{op}$, and similarly for (co)augmented coalgebras. As long as e.g.\ $\mathsf{C}$ has all limits and colimits, there are adjoint functors \cite[Theorem 5.2.2.17, Remark 5.2.2.19]{HA}
\begin{equation*}
\begin{tikzcd}
\Bar : {\mathsf{Alg}_{E_1}^\mathrm{aug}(\mathsf{C})} \arrow[rr, ""{name=G1}, shift left=.8ex] 
& &  {\mathsf{coAlg}_{E_1}^\mathrm{aug}(\mathsf{C})} : \Cobar. \arrow[ll, ""{name=F1}, shift left=.8ex] 
\arrow[phantom, from=F1, to=G1,  "\scalebox{0.7}{$\dashv$}" rotate=-90 ]
\end{tikzcd}
\end{equation*}
(We change notation from \cite{HA} slightly by writing $\Bar$ and $\Cobar$ for the functors taking values in (co)algebras, rather than their underlying functors to $\mathsf{C}$.) For an augmented algebra $\gA \to \bunit$ the underlying object of $\Bar(\gA)$ may be calculated as the relative tensor product $\bunit \otimes_\gA \bunit$, and hence as the colimit of a simplicial object of the form $[p] \mapsto \gA^{\otimes p}$. Dually, for a coaugmented coalgebra $\bunit \to \gC$ the underlying object of $\Cobar(\gC)$ may be calculated as the relative cotensor product $\bunit \square_\gC \bunit$ (i.e. the relative tensor product in the opposite category), and hence as the limit of a cosimplicial object of the form $[p] \mapsto \gC^{\otimes p}$. For these, see \cite[Remark 5.2.2.8]{HA}.

By Dunn--Lurie additivity, taking $E_{k-1}$-algebras on both sides gives an adjunction
\begin{equation*}
\begin{tikzcd}
\Bar : {\mathsf{Alg}_{E_k}^\mathrm{aug}(\mathsf{C})} \arrow[rr, ""{name=G1}, shift left=.8ex] 
& &  {\mathsf{Alg}_{E_{k-1}}(\mathsf{coAlg}_{E_1}^\mathrm{aug}(\mathsf{C}))} : \Cobar. \arrow[ll, ""{name=F1}, shift left=.8ex] 
\arrow[phantom, from=F1, to=G1,  "\scalebox{0.7}{$\dashv$}" rotate=-90 ]
\end{tikzcd}
\end{equation*}
When $k=2$ we will usually refer to objects on the right-hand side as (augmented) $E_1$-bialgebras. This also allows us to iterate $\Bar$, giving an adjunction
\begin{equation*}
\begin{tikzcd}
\Bar^{(k)} : {\mathsf{Alg}_{E_k}^\mathrm{aug}(\mathsf{C})} \arrow[rr, ""{name=G1}, shift left=.8ex] 
& &  {\mathsf{coAlg}_{E_k}^\mathrm{aug}(\mathsf{C})} : \Cobar^{(k)}. \arrow[ll, ""{name=F1}, shift left=.8ex] 
\arrow[phantom, from=F1, to=G1,  "\scalebox{0.7}{$\dashv$}" rotate=-90 ]
\end{tikzcd}
\end{equation*}
We will write
$$B^{E_k} : \mathsf{Alg}_{E_k}^\mathrm{aug}(\mathsf{C}) \overset{\Bar^{(k)}}\lra \mathsf{coAlg}_{E_k}^\mathrm{aug}(\mathsf{C}) \lra \mathsf{C}$$
for the underlying object of the iterated bar construction; it is given by the $k$-fold iterated relative bar construction. For an augmented $E_k$-algebra $\epsilon : \gR \to \bunit$ with augmentation ideal $\gI$, there is a functorial equivalence
\begin{equation}\label{eq:BarIsIndec}
B^{E_k}(\gR) \simeq \bunit \oplus \Sigma^k Q^{E_k^\text{nu}}(\gI).
\end{equation}
This is \cite[Theorem 2.41]{FrancisCotangent}, in different notation. A model-categorical proof was given in \cite[Theorem 13.7]{e2cellsI}.

The $\Bar$ and $\Cobar$ functors tend to yield an equivalence of categories when one restricts to suitably finitary objects on each side: see \cite{Heuts} for a complete analysis. For us the following simple setting will suffice. Say that an object $X \in \mathsf{D}(\bk)^\bZ$ is \emph{connected} if $\pi_{n,d}(X)=0$ for $n < 0$, or for $n=0$ and $d \leq 0$, and let $\mathsf{Alg}_{E_1}^\mathrm{aug, cn}(\mathsf{D}(\bk)^\bZ)$ denote the full subcategory on those augmented algebras whose augmentation ideal is connected, and similarly $\mathsf{coAlg}_{E_1}^\mathrm{aug, scn}(\mathsf{D}(\bk)^\bZ)$ denote the the full subcategory on those coaugmented coalgebras whose \emph{desuspended} coaugmentation coideal is connected (i.e.\ the coaugmentation coideal is ``simply connected'').

\begin{prop}\label{prop:BarCobarInverse}
The functors
\begin{equation*}
\begin{tikzcd}
\Bar : {\mathsf{Alg}_{E_1}^\mathrm{aug, cn}(\mathsf{D}(\bk)^\bZ)} \arrow[rr, ""{name=G1}, shift left=.8ex] 
& &  {\mathsf{coAlg}_{E_1}^\mathrm{aug, scn}(\mathsf{D}(\bk)^\bZ)} : \Cobar \arrow[ll, ""{name=F1}, shift left=.8ex] 
\arrow[phantom, from=F1, to=G1,  "\scalebox{0.7}{$\dashv$}" rotate=-90 ]
\end{tikzcd}
\end{equation*}
are adjoint equivalences. 
\end{prop}
\begin{proof}
If $\gA \to \bk$ is an augmented $E_1$-algebra with augmentation ideal $\gI$, then the skeletal filtration of the coaugmentation coideal of $\bk \otimes_\gA\bk$ has associated graded $\bigoplus_{i \geq 1} S^{0,i} \otimes \gI^{\otimes i}$, whose desuspension is connected if $\gI$ is connected. Similarly, if $\bk \to \gC$ is a coaugmented $E_1$-coalgebra with coaugmentation coideal $\gJ$, then the coskeletal filtration of augmentation ideal of $\bk \square_\gC \bk$ has associated graded $\prod_{i \geq 1} S^{0,-i} \otimes \gJ^{\otimes i}$, which is connected if $S^{0,-1} \otimes \gJ$ is connected. Thus $\Bar$ and $\Cobar$ indeed restrict to functors between the indicated categories.

To show that they are inverses, we must show that the unit $\gA \to \Cobar(\Bar(\gA))$ and counit $\Bar(\Cobar(\gC)) \to \gC$ are equivalences. This may be done by repeating the argument of Krause \cite[Proposition B.4]{KrauseAppendix}, though we must say something as our connectedness hypothesis is weaker than his.

If $\gC \in \mathsf{coAlg}_{E_1}^\mathrm{aug, scn}(\mathsf{D}(\bk)^\bZ)$ has coaugmentation coideal $\gJ$ then the fibre of the map $\mathrm{Tot}^k(\Cobar(\gC)) \to \mathrm{Tot}^{k-1}(\Cobar(\gC))$ is $S^{0,-k} \otimes \gJ^{\otimes k} = (S^{0,-1} \otimes \gJ)^{\otimes k}$ so has trivial $\pi_{n,d}$ if either $n<0$, $d \leq 0$ and $n < k$, or $0 \leq d \leq k$ and $d+n < k$. It follows that $\Cobar(\gC) \to \mathrm{Tot}^{k}(\Cobar(\gC))$ is an isomorphism on $\pi_{n,d}$ for all large enough $k$. From this \cite[Lemma B.6]{KrauseAppendix} follows.

The role of \cite[Lemma B.5]{KrauseAppendix} can be played by applying \cite[Corollary 11.6]{e2cellsI} as follows. If $f: X \to Y$ is a map between connected (in the sense we are using here) objects such that $\pi_{n,d}(Y, X)=0$ for $n < N$ and for $n=N$ and $d<D$, then \cite[Corollary 11.6]{e2cellsI} implies that $f^{\otimes k} : X^{\otimes k} \to Y^{\otimes k}$ has $\pi_{n,d}(Y^{\otimes k}, X^{\otimes k})=0$ for $n<N$ and for $n=N$ and $d<N+k-1$.

For \cite[Lemma B.7]{KrauseAppendix}, if a map $\gC \to \gD$ in $\mathsf{coAlg}_{E_1}^\mathrm{aug, scn}(\mathsf{D}(\bk)^\bZ)$ is given on coaugmentation coideals by $f : \gJ \to \gK$, and $\pi_{n,d}(\gD, \gC) = \pi_{n,d}(\gK, \gJ)=0$ for $n<N$ or for $n=N$ and $d< D$, then it follows from the above that the maps
\begin{equation*}
\begin{tikzcd}[row sep = 2ex]
 &  \pi_{N,D}(\gD, \gC) = \pi_{N,D}(\gK, \gJ) \dar\\
\pi_{N,D-1}(\Cobar(\gD), \Cobar(\gC)) \rar & \pi_{N,D-1}(\mathrm{Tot}^1(\Cobar(\gD)), \mathrm{Tot}^1(\Cobar(\gC))) 
\end{tikzcd}
\end{equation*}
are isomorphisms. It follows that $\Cobar$ detects equivalences; the case of $\Bar$ is similar.
\end{proof}

In particular, if $\gR$ satisfies axiom (C) then $\gR \to \Cobar(\Bar(\gR))$ is an equivalence.

\subsection{Filtered objects and spectral sequences} 

Write $\bZ_{\leq}$ for the poset of integers with its usual ordering, and let
$$\mathsf{fil}(\mathsf{D}(\bk)^\bZ) := \mathsf{Fun}(\bZ_{\leq}, \mathsf{D}(\bk)^\bZ)$$
be the category of filtered objects in $\mathsf{D}(\bk)^\bZ$. This is again stable, as limits and colimits are computed pointwise, and Day convolution endows it with a symmetric monoidal structure. We write objects of this category as $\fil_* X$, and the value of such a functor at $f \in \mathbb{Z}_\leq$ as $\fil_f X$. There are functors
\begin{equation}\label{eq:GrAndColim}
(\mathsf{D}(\bk)^\bZ)^\bZ \overset{\gr}\longleftarrow \mathsf{fil}(\mathsf{D}(\bk)^\bZ) \overset{\colim}\lra \mathsf{D}(\bk)^\bZ
\end{equation}
given by colimit and associated graded: the latter has components $\gr_f(\fil_* X) \simeq  \fil_f X / \fil_{f-1} X$. Both functors are (strong) symmetric monoidal left adjoints.

We write $X := \colim \fil_* X \in \mathsf{D}(\bk)^\bZ$, and think of $\fil_* X$ as a filtration of the object $X$. We say it is \emph{complete} if $\lim \fil_* X \simeq 0$. In this case there is a conditionally convergent (in the sense of Boardman \cite{Boardman}) spectral sequence
\begin{equation}\label{eq:FiltrationSS}
E^1_{n,d,f} = \pi_{n,d}(\gr_f(\fil_* X)) \Longrightarrow \pi_{n,d}(X)
\end{equation}
with differentials $d^r : E^r_{n,d,f} \to E^r_{n,d-1,f-r}$. When we wish to emphasise the filtered object from which the spectral sequence is derived, we write $E^r_{n,d,f}(\fil_* X)$. We apologise for the indexing, which does not match anybody else's.

\subsubsection{Filtered homotopy groups and $\tau$}

We find it convenient to work with filtered objects using some of the language which has been developed in the subject of synthetic spectra, but is really just about filtered objects. The following discussion is intended to be self-contained, but \cite[Section 11, Appendix A]{BHS} will be useful.

Evaluation at $f \in \bZ_\leq$ defines a functor $f^* : \mathsf{fil}(\mathsf{D}(\bk)^\bZ) \to \mathsf{D}(\bk)^\bZ$ which has a left adjoint $f_* : \mathsf{D}(\bk)^\bZ \to \mathsf{fil}(\mathsf{D}(\bk)^\bZ)$. The filtered object $f_*Y$ is 0 in filtrations $< f$ and constantly $Y$ in filtrations $\geq f$. We in particular have the filtered spheres
$$S^{n,d,f} := f_* S^{n,d},$$
using which we define the (trigraded) \emph{filtered homotopy groups}
$$\pi_{n,d,f}(\fil_* X) := [S^{n,d,f}, \fil_*X]_{\mathsf{fil}(\mathsf{D}(\bk)^\bZ)} = \pi_{n,d}(\fil_f X).$$

The identity map $S^{0,0} \to 1^*0_* S^{0,0}$ is adjoint to a map
$$\tau : S^{0,0,1} \lra S^{0,0,0}.$$
We write $C\tau$ for its cofibre. Tensoring with this map gives maps $\tau : S^{0,0,1} \otimes \fil_* X \to \fil_* X$ and hence induced maps $\tau : \pi_{n,d,f}(\fil_* X) \to \pi_{n,d,f+1}(\fil_* X)$,
which we use to consider the filtered homotopy groups $\pi_{n,d,*}(\fil_* X)$ as a graded $\bk[\tau]$-module.

There is an alternative perspective on filtered objects which can be useful. The lax symmetric monoidal functor $\prod_{f \in \bZ} f^* : \mathsf{fil}(\mathsf{D}(\bk)^\bZ) \to (\mathsf{D}(\bk)^\bZ)^\bZ$ has left adjoint $\bigoplus_{f \in \bZ} f_*$ and this adjunction is monadic (by \cite[Theorem 4.7.0.3]{HA}, as the right adjoint detects equivalences and preserves all colimits, because these are both pointwise), identifying $\mathsf{fil}(\mathsf{D}(\bk)^\bZ)$ with modules in $(\mathsf{D}(\bk)^\bZ)^\bZ$ over the commutative algebra $(\prod_{f \in \bZ} f^*)(S^{0,0,0})$. This algebra can be identified as $\bk[\tau]$, the free commutative algebra in the 1-category $((\bk\text{-}\mathsf{mod})^\bZ)^\bZ$, of bigraded $\bk$-modules, on a single generator of bigrading $(0,1)$. In this way $\mathsf{fil}(\mathsf{D}(\bk)^\bZ)$ is nothing but the derived category of bigraded $\bk[\tau]$-modules. From this point of view the functor $\gr$ is identified with base-change along $\bk[\tau] \to \bk = \bk[\tau]/\tau \simeq C\tau$, i.e.\ $(C\tau \otimes \fil_*X)(f) \simeq \gr_f(\fil_*X)$, and in particular $C\tau$ obtains the structure of a commutative algebra in $\mathsf{fil}(\mathsf{D}(\bk)^\bZ)$. Similarly, the functor $\colim$ is identified with base-change along $\bk[\tau] \to \bk[\tau^{\pm 1}]$ (the category of $\bk[\tau^{\pm 1}]$-modules in $(\mathsf{D}(\bk)^\bZ)^\bZ$ is identified with $\mathsf{D}(\bk)^\bZ$), so we allow ourselves to denote it by $\tau^{-1}$.


There are two canonical objects in $\mathsf{fil}(\mathsf{fil}(\mathsf{D}(\bk)^\bZ))$ associated to $\tau$. The first is
\begin{equation}\tag{$\beta$}\label{eq:BocksteinFilt}
\cdots \lra S^{0,0,2} \overset{\tau}\lra S^{0,0,1} \overset{\tau}\lra S^{0,0,0} \overset{Id}\lra S^{0,0,0} \overset{Id}\lra S^{0,0,0} \lra \cdots
\end{equation}
where the leftmost $S^{0,0,0}$ is the value at $0 \in \mathbb{Z}_{\leq}$. Tensoring with $\fil_* X$ it gives a filtered object in $\mathsf{fil}(\mathsf{D}(\bk)^\bZ)$ whose colimit is $\fil_*X$, and whose limit has components
$$\left(\lim_{p \leq 0} S^{0,0,-p} \otimes \fil_* X\right)(f) \simeq \lim_{p \leq 0} \fil_{f+p} X,$$
so is trivial if and only if $\fil_* X$ is complete. Its associated graded in grading $p \leq 0$ is $S^{0,0,-p} \otimes C\tau \otimes \fil_* X$. Assuming $\fil_* X$ is complete this gives a conditionally convergent spectral sequence
\begin{equation}\label{eq:tauBockstein}
\prescript{\beta}{}{E}^1_{n,d,f,p} = \pi_{n,d,f + p}(C\tau \otimes \fil_* X) \Longrightarrow \pi_{n,d,f}(\fil_* X),
\end{equation}
with differentials $\prescript{\beta}{}{d}^r : \prescript{\beta}{}{E}^r_{n,d,f,p} \to \prescript{\beta}{}{E}^r_{n,d-1,f,p-r}$. We can consider this as a spectral sequence of $\bk[\tau]$-modules starting with $\pi_{*,*,*}(C\tau \otimes \fil_* X)\otimes_\bk \bk[\tau]$, where the first factor has grading $(*,*,*,0)$ and $\tau$ has grading $(0,0,1,-1)$, and converging to the filtered homotopy groups $\pi_{*,*,*}(\fil_* X)$. It is called the \emph{$\tau$-Bockstein spectral sequence}.

The second is
\begin{equation}\tag{$I\beta$}\label{eq:InvertedBocksteinFilt}
\cdots \lra S^{0,0,2} \overset{\tau}\lra S^{0,0,1} \overset{\tau}\lra S^{0,0,0} \overset{\tau}\lra S^{0,0,-1} \overset{\tau}\lra S^{0,0,-2} \lra \cdots
\end{equation}
where $S^{0,0,0}$ is the value at $0 \in \mathbb{Z}_\leq$. Tensoring with $\fil_* X$ it gives a filtered object in $\mathsf{fil}(\mathsf{D}(\bk)^\bZ)$ whose colimit is called $\tau^{-1}\fil_*X$, and whose limit is again trivial if and only if $\fil_* X$ is complete. In this case it gives a conditionally convergent spectral sequence
\begin{equation}\label{eq:tauInvertedtauBockstein}
\prescript{I\beta}{}{E}^1_{n,d,*,*} = \pi_{n,d,*}(C\tau \otimes \fil_* X)\otimes_\bk \bk[\tau^{\pm 1}] \Longrightarrow \pi_{n,d,*}(\tau^{-1} \fil_* X),
\end{equation}
with differentials $\prescript{I\beta}{}{d}^r : \prescript{I\beta}{}{E}^r_{n,d,f,p} \to \prescript{I\beta}{}{E}^r_{n,d-1,f,p-r}$. It is called the \emph{inverted $\tau$-Bockstein spectral sequence}.

The relation between the structure of the filtered homotopy groups $\pi_{*,*,*}(\fil_*X)$ and these spectral sequences has been studied in depth by Burklund--Hahn--Senger \cite[Section 11, Appendix A]{BHS}, in the case of filtered objects arising from synthetic spectra. But much of their analysis goes through in the general setting of filtered objects. The three spectral sequences are related as follows. On one hand \eqref{eq:tauInvertedtauBockstein} is obtained by inverting $\tau$ in the spectral sequence \eqref{eq:tauBockstein}. On the other hand $\pi_{n,d,*}(\tau^{-1} \fil_* X) = \tau^{-1} \pi_{n,d,*}( \fil_* X) = \pi_{n,d}(X) \otimes_\bk \bk[\tau^{\pm 1}]$ and \eqref{eq:tauInvertedtauBockstein} is obtained from \eqref{eq:FiltrationSS} by tensoring with $\bk[\tau^{\pm 1}]$. By specialising to $\tau=1$ it follows that \eqref{eq:tauInvertedtauBockstein} is interchangeable with \eqref{eq:FiltrationSS}. Finally, whenever there is a differential $d^r([x])=[y]$ in the spectral sequence \eqref{eq:FiltrationSS}, there is a corresponding differential $\prescript{\beta}{}{d}^r([x \otimes 1]) = [y \otimes \tau^r]$ in the spectral sequence \eqref{eq:tauBockstein}, and furthermore all differentials in the latter spectral sequence arise in this way: this may be seen by adapting \cite[Theorem A.8 (3)]{BHS}.

Finally, let us comment on multiplicative aspects of these spectral sequences. Both \eqref{eq:BocksteinFilt} and \eqref{eq:InvertedBocksteinFilt} can be given the structure of commutative monoids in $\mathsf{fil}(\mathsf{fil}(\mathsf{D}(\bk)^\bZ))$, compatibly with the natural morphism $\beta \to I\beta$ in this category. This is most readily seen by observing that they both arise in an evident way from filtered objects in the 1-category of $\bk[\tau]$-modules in $((\bk\text{-}\mathsf{mod})^\bZ)^\bZ$, where the claim is quickly verified. It follows that
$$\beta \otimes- : \mathsf{fil}(\mathsf{D}(\bk)^\bZ) \lra \beta\text{-modules in }\mathsf{fil}(\mathsf{fil}(\mathsf{D}(\bk)^\bZ))$$
is strong symmetric monoidal, and similarly for $I\beta$, and hence that the functors
$$\beta \otimes-, I\beta \otimes - : \mathsf{fil}(\mathsf{D}(\bk)^\bZ) \lra \mathsf{fil}(\mathsf{fil}(\mathsf{D}(\bk)^\bZ))$$
are both lax symmetric monoidal. From this it follows that a pairing $\fil_* X \otimes \fil_* Y \to \fil_*Z$ of filtered objects induces a pairing 
$$E^r_{*,*,*}(\fil_* X) \otimes E^r_{*,*,*}(\fil_* Y) \lra E^r_{*,*,*}(\fil_* Z)$$
of spectral sequences, for each of \eqref{eq:tauBockstein} and \eqref{eq:tauInvertedtauBockstein} (and hence also \eqref{eq:FiltrationSS}). In particular, if $\fil_* X$ has a multiplicative structure then they are all multiplicative spectral sequences. Furthermore, if $\fil_* X$ has the structure of an $E_k$-algebra in $\mathsf{fil}(\mathsf{D}(\bk)^\bZ)$ then these spectral sequences have a Browder bracket, Dyer--Lashof operations, and so on: see \cite[Section 16.6]{e2cellsI} for how these operations interact with differentials.

\subsubsection{The canonical multiplicative filtration}\label{sec:CanMultFilt}

If $\gR$ is an $E_k$-algebra in $\mathsf{D}(\bk)^\bZ$ equipped with an augmentation $\epsilon : \gR \to \bk$, then it has a canonical filtration $\fil_* \gR$ to be thought of as the filtration by powers of the augmentation ideal 
$$\gI_\gR := \mathrm{fib}(\epsilon : \gR \to \bk).$$
It has been discussed in various contexts by several authors under several names: \cite{HarperHess}, \cite{KuhnPereira},  \cite[Section 5.4.2]{e2cellsI}.
Following the latter it may be constructed as follows. We first consider non-unital $E_k$-algebras in $\mathsf{D}(\bk)^\bZ$, and consider the functor
\begin{equation}\label{eq:EvAtMinus1Alg}
(-1)^* : \mathsf{Alg}_{E_k^\text{nu}}(\mathsf{fil}(\mathsf{D}(\bk)^\bZ)) \lra \mathsf{Alg}_{E_k^\text{nu}}(\mathsf{D}(\bk)^\bZ)
\end{equation}
given by evaluating at filtration $-1$. (This induces a map on non-unital $E_k$-algebras, as $(-1)^* : \mathsf{fil}(\mathsf{D}(\bk)^\bZ) \to \mathsf{D}(\bk)^\bZ$ is lax symmetric monoidal, because $(-1) + (-1) \leq -1 \in \bZ_{\leq}$.) This has a left adjoint
$$(-1)_*^{E_k} : \mathsf{Alg}_{E_k^\text{nu}}(\mathsf{D}(\bk)^\bZ) \lra \mathsf{Alg}_{E_k^\text{nu}}(\mathsf{fil}(\mathsf{D}(\bk)^\bZ)).$$
This is constructed and discussed 1-categorically in \cite[Section 5.4.1]{e2cellsI}; $\infty$-categorically it may be obtained by the adjoint functor theorem \cite[Corollary 5.5.2.9]{HTT}, using that \eqref{eq:EvAtMinus1Alg} is a functor between presentable categories which preserves limits and filtered colimits (because such (co)limits are formed underlying in $E_k$-algebras). The canonical multiplicative filtration of a non-unital $E_k$-algebra $\gI$ is 
$$\fil_*^{E_k} \gI := (-1)^{E_k}_*(\gI).$$
Its associated graded is calculated in terms of the indecomposables by
$$\gr(\fil_*^{E_k} \gI) \simeq \gE_k^\text{nu}((-1)_*(Q^{E_k^\text{nu}}(\gI))),$$
which follows from the corresponding identity between right adjoints of all these functors (see \cite[Section 5.4.3]{e2cellsI} for a 1-categorical treatment).

We extend this to the setting of an augmented $E_k$-algebra $\epsilon : \gR \to \bk$ by
$$\fil_*^{E_k} \gR := S^{0,0,0} \oplus \fil_*^{E_k} \gI_\gR,$$
having associated graded $\gr(\fil_*^{E_k}\gR) = S^{0,0,0}\oplus \gr(\fil_*^{E_k}\gI_\gR) \simeq \gE_k((-1)_*(Q^{E_k^\text{nu}}(\gI_\gR)))$. As long as $\gR$ satisfies axiom (C) the canonical multiplicative filtration $\fil_*^{E_k} \gR$ is complete, by the Claim in the proof of \cite[Theorem 10.20]{e2cellsI}. In fact, $\pi_{n,d}(\fil_{-a}^{E_k} \gR)=0$ for $d \leq a-n$.

\begin{prop}\label{prop:BarOnCanMultFilt}
There is a natural isomorphism
$$B^{E_\ell}(\fil_*^{E_k} -) \simeq \fil_*^{E_{k-\ell}} B^{E_\ell}(-) :  \mathsf{Alg}_{E_{k}}^\mathrm{aug}(\mathsf{D}(\bk)^\bZ) \lra \mathsf{Alg}_{E_{k-\ell}}^\mathrm{aug}(\mathsf{fil}(\mathsf{D}(\bk)^\bZ)).$$
\end{prop}
\begin{proof}
We establish the corresponding identity among the right adjoints to all these functors. Firstly, we claim that the right adjoint $\Omega^{E_\ell}$ to
$$B^{E_\ell} : \mathsf{Alg}_{E_{\ell}}^\mathrm{aug}(\mathsf{D}) \lra \mathsf{D}$$
is natural in $\mathsf{D}$ a symmetric monoidal category and lax monoidal functors $\phi : \mathsf{D} \to \mathsf{D}'$. To see this, we must show that the Beck--Chevalley morphism $\phi \circ \Omega^{E_\ell} \Rightarrow \Omega^{E_\ell} \circ \phi$ is an equivalence. We may test this after applying the forgetful functor $U^{E_\ell} : \mathsf{Alg}_{E_{\ell}}^\mathrm{aug}(\mathsf{D}') \to \mathsf{D}'$, as this is conservative. But $U^{E_\ell}$ and $U^{E_\ell}(\Omega^{E_\ell}(-)) \simeq \bunit \oplus \Sigma^{-\ell} (-)$ commute with $\phi$, so it is indeed an equivalence.

We apply this to the lax symmetric monoidal functor 
$$\phi = S^{0,0} \oplus (-1)^*(-) : \mathsf{Alg}_{E_{k-\ell}}^\mathrm{aug}(\mathsf{fil}(\mathsf{D}(\bk)^\bZ)) \lra \mathsf{Alg}_{E_{k-\ell}}^\mathrm{aug}(\mathsf{D}(\bk)^\bZ)$$
to obtain the required commutative square
\begin{equation*}
\begin{tikzcd}[/tikz/baseline=(tikz@f@1-2-1.base)]
\mathsf{Alg}_{E_{k-\ell}}^\mathrm{aug}(\mathsf{fil}(\mathsf{D}(\bk)^\bZ)) \rar{\Omega^{E_\ell}} \dar{S^{0,0} \oplus (-1)^*(-)} & \mathsf{Alg}_{E_{\ell}}^\mathrm{aug}(\mathsf{Alg}_{E_{k-\ell}}^\mathrm{aug}(\mathsf{fil}(\mathsf{D}(\bk)^\bZ))) \dar{S^{0,0} \oplus (-1)^*(-)}\\
\mathsf{Alg}_{E_{k-\ell}}^\mathrm{aug}(\mathsf{D}(\bk)^\bZ) \rar{\Omega^{E_\ell}}   & \mathsf{Alg}_{E_{\ell}}^\mathrm{aug}(\mathsf{Alg}_{E_{k-\ell}}^\mathrm{aug}(\mathsf{D}(\bk)^\bZ)).
\end{tikzcd}\qedhere
\end{equation*}
\end{proof}

We will consider $E_k$-algebras with their canonical multiplicative filtration unless otherwise emphasised, so we simplify notation to
$$\fil_* \gR := \fil_*^{E_k} \gR.$$

\subsection{Some $t$-structures}\label{sec:TStructures}
The stable $\infty$-category $\mathsf{D}(\bk)$ has the standard Postnikov $t$-structure \cite[Definition 1.3.5.16, Proposition 1.3.5.21]{HA}, whose connective (resp.\ coconnective) objects are those whose homotopy groups are supported in positive (resp.\ negative) degrees. Its heart is the abelian 1-category of $\bk$-modules. It is compatible with filtered colimits (i.e.\ the coconnective objects are closed under filtered colimits) and with the symmetric monoidal structure (i.e.\ the unit is connective and the connective objects are closed under tensor products). Any object $X$ obtains a Postnikov (or Whitehead) filtration $\tau_{\geq -*} X$
$$\cdots \lra \tau_{\geq 1} X \lra\tau_{\geq 0} X \lra \tau_{\geq -1} X \lra \tau_{\geq -2} X \lra \cdots$$
which is complete and whose colimit is $X$. Its associated graded has $\gr_f(\tau_{\geq -*} X) \simeq \Sigma^{-f} \pi_f(X)$, where we consider $\pi_f(X)$ as an object of (the heart of) $\mathsf{D}(\bk)$.

We obtain a Postnikov $t$-structure on $\mathsf{D}(\bk)^\bZ$ by doing the above in each grading. Its heart is the abelian 1-category of $\bZ$-graded $\bk$-modules. It is again compatible with filtered colimits and with the (Day convolution) symmetric monoidal structure.

If $\gR \in \mathsf{Alg}_{E_1}(\mathsf{D}(\bk)^\bZ)$ is connective (i.e.\ $\pi_{n,d}(\gR)=0$ for $d<0$) then $\gR\text{-}\mathsf{mod}$ has a $t$-structure whose connective (resp.\ coconnective) objects are those which are connective (resp.\ coconnective) in $\mathsf{D}(\bk)^\bZ$ (adapt \cite[Proposition 7.1.1.13]{HA} to the graded setting). Its heart is the abelian 1-category of $\bZ$-graded $\pi_{*,0}(\gR)$-modules. In particular the Postnikov filtration $\tau_{\geq -*} \gM \in \mathsf{fil}(\mathsf{D}(\bk)^\bZ)$ of an $\gR$-module $\gM$ promotes to an object of $\mathsf{fil}(\gR\text{-}\mathsf{mod})$. Its associated graded has $\gr_f(\tau_{\geq -*} \gM) \simeq \Sigma^{-f} \pi_{*,f}(\gM)$, where $\pi_{*,f}(\gM)$ is an $\gR$-module via $\gR \to \tau_{\leq 0} \gR = \pi_{*,0}(\gR)$.

We will also have cause to consider a sheared $t$-structure on $\mathsf{D}(\bk)^\bZ$, the \emph{diagonal $t$-structure}, in which an object $X$ is connective (resp.\ coconnective) if each $\Sigma^{-n}X(n)$ is connective (resp.\ coconnective) in the Postnikov $t$-structure on $\mathsf{D}(\bk)$. It is elementary to see that this is indeed a $t$-structure, that its heart is again the abelian 1-category of $\bZ$-graded $\bk$-modules, and that it is compatible with the (Day convolution) symmetric monoidal structure. We write $\tau^\text{diag}$ for its truncation functors. 

If $(\mathsf{C}_{\geq 0}, \mathsf{C}_{\leq 0})$ is a $t$-structure on a symmetric monoidal category $(\mathsf{C}, \otimes, \bunit)$ which is compatible with with a symmetric monoidal structure, then by \cite[Example 2.2.1.10]{HA} the heart $\mathsf{C}^\heartsuit = (\mathsf{C}_{\geq 0})_{\leq 0}$ has a symmetric monoidal structure such that the truncation functor $\tau_{\leq 0} : \mathsf{C}_{\geq 0} \to \mathsf{C}^\heartsuit$ is strong symmetric monoidal (so the symmetric monoidal structure on the heart is calculated as $\tau_{\leq 0}(- \otimes -)$). We already used this implicitly above when discussing the map of $E_1$-algebras $\gR \to \tau_{\leq 0} \gR = \pi_{*,0}(\gR)$. In Section \ref{sec:StabHopfAlg} we will further use it as follows: if $A$ is connective with respect to the diagonal $t$-structure on $\mathsf{D}(\bk)^\bZ$, and has the structure of an $E_1$-bialgebra, then 
 $$\tau_{\leq 0}^\text{diag} A \simeq \bigoplus_{n \in \bZ} \pi_{n,n}(A)$$
has the structure of a bialgebra object in $\bZ$-graded $\bk$-modules.

\begin{warning}\label{warn:flatness}
In this setting the inclusion $\iota: \mathsf{C}^\heartsuit \to \mathsf{C}_{\geq 0}$ is right adjoint to the strong symmetric monoidal functor $\tau_{\leq 0}$ and so is lax symmetric monoidal. In particular, it preserves $E_1$-algebra objects, but it need not preserve $E_1$-coalgebra objects! This is because for a coalgebra $C \in \mathsf{C}^\heartsuit$ to give a coalgebra in $\mathsf{C}$ one needs a map $\iota(C) \to \iota(C) \otimes \iota(C)$, but what one has is a map $\iota(C) \to \iota(C \otimes C) = \tau_{\leq 0} (\iota(C) \otimes \iota(C))$. However $\iota$ becomes strong symmetric monoidal if we restrict to the subcategory of $\mathsf{C}^\heartsuit$ consisting of flat objects.
\end{warning}

Returning to the example, and heeding the warning, as long as each $\pi_{n,n}(A)$ is a flat $\bk$-module then the $E_1$-bialgebra $\tau_{\leq 0}^\text{diag} A$ in the heart of $\mathsf{D}(\bk)^\bZ$ is also an $E_1$-bialgebra in $\mathsf{D}(\bk)^\bZ$ itself. We will usually enforce this flatness by taking $\bk$ to be a field.

\subsection{Thick subcategories}\label{sec:ThickSubcat}

In the stable categories $\gR\text{-}\mathsf{mod}$ or $\fil_*\gR\text{-}\mathsf{mod}$  (or indeed $\mathsf{D}(\bk)^\bZ$ or $\mathsf{fil}(\mathsf{D}(\bk)^\bZ)$) we say that a full subcategory $\mathsf{C}$ is \emph{thick} whenever:
\begin{enumerate}[(i)]
\item If $A \to B \to C$ is a cofibre sequence and two of these objects are in $\mathsf{C}$, then so is the third.

\item If $A$ is in $\mathsf{C}$ then so is $S^{n,d} \otimes A$ (or $S^{n,d,f} \otimes A$ in the case of filtered objects).

\item If $A$ is in $\mathsf{C}$ then so is any retract of it.
\end{enumerate}
Our second axiom differs slightly from the norm: usually $\mathsf{C}$ is only assumed closed under (de)suspensions, but we wish to allow shifts by all bigraded (filtered) spheres. We will say \emph{the thick subcategory generated by $X$} to mean the smallest thick subcategory containing $X$. We collect some elementary facts about them as follows.

\begin{lem}\label{lem:ThickSubcatFacts}
Let $U$ be in the thick subcategory of the monoidal unit and $f : U \otimes X \to X$ be a morphism: form its iterates $f^k$ by precomposing $f$ with $U \otimes f$ and so on.
\begin{enumerate}[(i)]
\item\label{it:ThickSubcatFacts:1} $X/f$ is in the thick subcategory generated by $X$.

\item\label{it:ThickSubcatFacts:2} For any $k \geq 1$, $X/f^k$ is in the thick subcategory generated by $X/f$.

\item\label{it:ThickSubcatFacts:3} If $f^k \simeq 0$ then $X$ is in the thick subcategory generated by $X/f$.

\item\label{it:ThickSubcatFacts:4} $X/f$ and $X/f^k$ generate the same thick subcategories.

\end{enumerate}
\end{lem}
\begin{proof}
As $U$ is in the thick subcategory of the monoidal unit, for any $Y$ it follows that $U \otimes Y$ is in the thick subcategory generated by $\bunit \otimes Y = Y$. We will use this principle often.

(\ref{it:ThickSubcatFacts:1})  The cofibre sequence $U \otimes X \overset{f}\to X \to X/f$ shows that $X/f$ is in the thick subcategory generated by $X$.

(\ref{it:ThickSubcatFacts:2}) The factorisation $f^{k} : U \otimes (U^{\otimes k-1}) \overset{A \otimes f^{k-1}}\to U \otimes X \overset{f}\to X$ gives a cofibre sequence $U \otimes X/f^{k-1} \to X/f^k \to X/f$. By induction $X/f^{k-1}$ is in the thick subcategory generated by $X/f$, so $U \otimes X/f^{k-1}$ is too, and hence by the cofibre sequence $X/f^k$ is too.

(\ref{it:ThickSubcatFacts:3}) If $f^k \simeq 0$ then $X/f^k \simeq X \oplus \Sigma (U \otimes X)$ contains $X$ as a retract. So $X$ is in the thick subcategory generated by $X/f^k$, so by (\ref{it:ThickSubcatFacts:2}) is in the thick subcategory generated by $X/f$.

(\ref{it:ThickSubcatFacts:4}) Consider the commutative diagram
\begin{equation*}
\begin{tikzcd}
U^{\otimes k+1} \otimes X \rar{U^{\otimes k} \otimes f} \dar{U^{\otimes k} \otimes f} & U^{\otimes k} \otimes X \rar \dar{U^{\otimes k-1} \otimes f} & \cdots \rar & U^{\otimes 2} \otimes X \rar{U \otimes f} \dar{U \otimes f} & U \otimes X \dar{f}\\
U^{\otimes k} \otimes X \rar{U^{\otimes k-1} \otimes f} & U^{\otimes k-1} \otimes X \rar & \cdots \rar & U \otimes X \rar{f}& X
\end{tikzcd}
\end{equation*}
and the total homotopy cofibre $C$ of the outer rectangle. By taking horizontal cofibres in the outer rectangle first there is a cofibre sequence $U \otimes X/f^k \to X/f^k \to C$, so $C$ is in the thick subcategory generated by $X/f^k$. On the other hand taking vertical cofibres first gives a cofibre sequence $U^{\otimes k} \otimes X/f \overset{\bar{f}^k}\to X/f \to C$. We will show that $\bar{f}^k$ is null when $k > 1$, so that $X/f$ is a retract of $C$ and so is in the thick subcategory generated by $X/f^k$.

To see this, note that $\bar{f}^k$ is indeed the $k$-fold iterate of the map $\bar{f} : U \otimes X/f \to X/f$ induced between vertical cofibres of the right-hand square (supposing of course that the homotopies making the other squares commute are obtained by applying $U^{\otimes i} \otimes -$ to this one). By a well-known principle such a map satisfies $\bar{f}^2 \simeq 0$: see e.g.\ \cite{RamziNote} for an extensive discussion.
\end{proof}

\subsection{Finiteness}

Recall that an object $X$ of an $\infty$-category $\mathsf{C}$ is called \emph{compact} if $\Map_\mathsf{C}(X,-) : \mathsf{C} \to \mathsf{Spaces}$ commutes with filtered colimits. If $\mathsf{C}$ is stable, so has mapping spectra, then we may instead ask that $\map_\mathsf{C}(X,-) : \mathsf{C} \to \mathsf{Sp}$ commutes with filtered colimits. As $\Map_\mathsf{C}(X,-) \simeq \Omega^\infty \map_\mathsf{C}(X,-)$ the latter implies the former, but as $\pi_n\map_\mathsf{C}(X,-) = \pi_0\Map_\mathsf{C}(\Sigma^n X,-) = \pi_0\Map_\mathsf{C}(X,\Sigma^{-n}-)$ and $\Sigma^{-n}$ preserves colimits the former implies the latter too.

In the stable categories $\gR\text{-}\mathsf{mod}$ or $\fil_*\gR\text{-}\mathsf{mod}$ an object is called \emph{finite} if it is in the thick subcategory generated by the monoidal unit. 

\begin{lem}
In $\gR\text{-}\mathsf{mod}$ or $\fil_*\gR\text{-}\mathsf{mod}$ the finite objects are compact.
\end{lem}
\begin{proof}
By their definition the compact objects clearly have the 2-out-of-3 property for cofibre sequences, and are closed under retracts and tensoring with $\otimes$-invertible objects such as $S^{n,d}$ or $S^{n,d,f}$. Thus they form a thick subcategory in the sense we gave in Section \ref{sec:ThickSubcat}. It suffices therefore to show that the monoidal units $\gR$ and $\fil_*\gR$ are compact. 

If $L$ is a left adjoint functor whose right adjoint preserves filtered colimits, then $L$ preserves compact objects. Each of
\begin{equation*}
\begin{tikzcd}
\mathsf{sSet} \rar{\bk\text{-chains}} & \mathsf{D}(\bk) \rar{0_*} & \mathsf{D}(\bk)^\bZ \rar{0_*} \dar{\gR \otimes -} & \mathsf{fil}(\mathsf{D}(\bk)^\bZ) \dar{\fil_*\gR \otimes -} \\
 & & \gR\text{-}\mathsf{mod} & \fil_* \gR\text{-}\mathsf{mod}
\end{tikzcd}
\end{equation*}
is such a left adjoint, and the compact object $* \in \mathsf{sSet}$ is sent to $\gR$ and $\fil_*\gR$.
\end{proof}

\begin{lem}\label{lem:TensorAndHomStaysThick}
If $\gF$ is a finite object in $\gR\text{-}\mathsf{mod}$, and $\gM$ is an object in this category, then $\gF \otimes_\gR \gM$, $\underline{\map}^r_\gR(\gF, \gM)$, and $\underline{\map}^l_\gR(\gF, \gM)$ are all in the thick subcategory generated by $\gM$. Similarly in the category $\fil_*\gR\text{-}\mathsf{mod}$.
\end{lem}
\begin{proof}
The functors $- \otimes_\gR \gM$, $\underline{\map}^r_\gR(-, \gM)$, and $\underline{\map}^l_\gR(-, \gM)$ all preserve (co)fibre sequences, send shifts to shifts, and (being functors) preserve retracts, so the class of $\gR$-modules which they send into the thick subcategory of $\gM$ is itself a thick subcategory. It certainly contains the monoidal unit, so contains all finite objects.
\end{proof}

\subsection{Duality}

When $\gR$ is an $E_2$-algebra the category $\gR\text{-}\mathsf{mod}$ (or $\fil_*\gR\text{-}\mathsf{mod}$) is only $E_1$-monoidal, so as discussed in Section \ref{sec:Modules} it has left and right internal mapping objects: thus we must speak of left and right duals as in \cite[Section 4.6.1]{HA}. By definition of the right mapping object there is an equivalence
$$\map_\gR(- \otimes_\gR \gM, -) \simeq \map_\gR(-, \underline{\map}_\gR^r(\gM, -)).$$
Using this one may produce a natural map
$$(-) \otimes_\gR \underline{\map}_\gR^r(\gM, \gR) \lra \underline{\map}_\gR^r(\gM, -),$$
and when $\gM$ is finite this is an equivalence (it is when evaluated at $\gR$, and both sides preserve colimits when $\gM$ is finite and hence compact). This exhibits $D^r(\gM) := \underline{\map}_\gR^r(\gM, \gR)$ as right dual to $\gM$; a similar calculation exhibits $D^l(\gM) := \underline{\map}_\gR^l(\gM, \gR)$ as left dual to $\gM$.

\begin{lem}\label{lem:DualInThick}
If $\gM$ and $\gN$ are dual finite objects in $\gR\text{-}\mathsf{mod}$, then they generate the same thick subcategory.
\end{lem}
\begin{proof}
The duality data
$$ev : \gN \otimes_\gR \gM \lra \gR \quad\quad\quad coev : \gR \lra \gM \otimes_\gR \gN$$
gives a factorisation $Id_\gM : \gM \overset{\gM \otimes coev}\to \gM \otimes_\gR \gN \otimes_\gR \gM \overset{ev \otimes \gM}\to \gM$ which expresses $\gM$ as a retract of $\gM \otimes_\gR \gN \otimes_\gR \gM$. As $\gM$ is a finite $\gR$-module, so is in the thick subcategory generated by $\gR,$ $\gM \otimes_\gR \gN \otimes_\gR \gM$ is in the thick subcategory generated by $\gN$, so $\gM$ is too. Similarly for the converse.
\end{proof}

The same discussion holds in $\fil_*\gR\text{-}\mathsf{mod}$.

\section{Stable homology is Bousfield localisation}\label{sec:localisation}

For an $\gR$-module $\gM$, we wish to develop the notion of the ``best approximation to $\gM$ up to slope $\lambda$''. As we are mainly interested in \emph{finite} $\gR$-modules, we propose that this be implemented as Bousfield localisation away from the (essentially small) class
$$\mathcal{A}^f_\lambda := \{\text{finite $\gR$-modules with a slope $\lambda$ vanishing line}\}.$$
We spell this out, closely following Miller \cite{MillerFin}.

\begin{defn}\label{defn:Local}\mbox{}
\begin{enumerate}[(i)]
\item A $\gR$-module $\gL$ is \emph{$\mathcal{A}^f_\lambda$-local} if $[\gT, \gL]_{\gR}=0$ for all $\gT \in \mathcal{A}_\lambda$.

\item A $\gR$-module $\gA$ is \emph{$\mathcal{A}^f_\lambda$-acyclic} if $[\gA, \gL]_{\gR}=0$ for every $\mathcal{A}^f_\lambda$-local $\gL$.

\item An $\gR$-module map $\phi: \gM \to \gL$ is an \emph{$\mathcal{A}^f_\lambda$-localisation} if $\gL$ is $\mathcal{A}^f_\lambda$-local and the fibre of $\phi$ is $\mathcal{A}^f_\lambda$-acyclic.

\end{enumerate}
\end{defn}

As $\mathcal{A}^f_\lambda$ is closed under shifts (i.e.\ $S^{n,d} \otimes -$ with $n,d \in \bZ$), in (i) it is equivalent to ask for the mapping object $\underline{\map}_\gR(\gT, \gL) \in \mathsf{D}(\bk)^\bZ$ to be trivial. As the $\mathcal{A}^f_\lambda$-local objects are closed under shifts, in (ii) it is equivalent to ask for $\underline{\map}_\gR(\gA, \gL) \simeq 0$.

\begin{prop}\label{prop:LocExists}
Any $\gR$-module has an $\mathcal{A}^f_\lambda$-localisation, which is unique up to homotopy equivalence. 
\end{prop}

In view of this proposition, we denote ``the'' localisation map by $\gM \to L_\lambda^f(\gM)$, and its fibre by $C_\lambda^f(\gM)$.

\begin{proof}[Proof of Proposition \ref{prop:LocExists}]
Let us first discuss uniqueness. Let
$$\gA_i \lra \gM \overset{\phi_i}\lra \gL_i \quad\quad i = 1,2$$
be fibre sequences with $\gL_i$ $\mathcal{A}^f_\lambda$-local and $\gA_i$ $\mathcal{A}^f_\lambda$-acyclic. As $[S^{n,d} \otimes \gA_1, \gL_2]_\gR=0$ for all $(n,d)$, because $\gL_2$ is $\mathcal{A}^f_\lambda$-local and $\gA_1$ is $\mathcal{A}^f_\lambda$-acyclic, the long exact sequence obtained by applying $[-,\gL_2]_\gR$ to the fibre sequence with $i=1$ gives a bijection 
\begin{equation}\label{eq:LocBij}
- \circ \phi_1 : [\gL_1, \gL_2]_\gR \overset{\sim}\lra [\gM, \gL_2]_\gR.
\end{equation}
Under this bijection $\phi_2$ corresponds to a map $\psi_{12} : \gL_1 \to \gL_2$ such that $\psi_{12} \circ \phi_1 \simeq \phi_2$. The same with the indices reversed gives a map $\psi_{21} : \gL_2 \to \gL_1$ such that $\psi_{21} \circ \phi_2 \simeq \phi_1$. Combining them gives $\psi_{21} \circ \psi_{12} \circ \phi_1 \simeq \mathrm{Id}_{\gL_1} \circ \phi_1 : \gM \to \gL_1$ which with the fact that \eqref{eq:LocBij} is injective shows $\psi_{21} \circ \psi_{12} \simeq \mathrm{Id}_{\gL_1}$; the other composition is treated similarly.

For existence, first write $\gM_0 := \gM$. Define maps $\gM_0 \to \gM_1 \to \gM_2 \to \cdots$ inductively by the cofibre sequences
\begin{equation}\label{eq:TelescopeDef}
\bigoplus_{\substack{[\gT] \in \mathcal{A}^f_\lambda/\sim \\ f : \gT \to \gM_i}} \gT \overset{ev}\lra \gM_i \lra \gM_{i+1},
\end{equation}
where the sum is over pairs of a homotopy type $[\gT]$ of object in $\mathcal{A}_\lambda^f$ (there is only a set of homotopy types of finite $\gR$-modules) and a map $f : \gT \to \gM_i$ from a representative of this homotopy type; the map $ev$ on the $([\gT], f)$-summand is given by $f$.  We claim that the map $\gM = \gM_0 \to \gM_\infty := \colim_{i \in \bN} \gM_i$ is an $\mathcal{A}^f_\lambda$-localisation. 

To see that $\gM_\infty$ is $\mathcal{A}^f_\lambda$-local, let $\gT \in \mathcal{A}^f_\lambda$ and consider a map $g : \gT \to \gM_\infty$. As $\gT$ is a finite $\gR$-module it is compact, so $g$ factors as $g : \gT \overset{f}\to \gM_i \to \gM_{i+1} \to \gM_\infty$. But the composition $\gT \overset{f}\to \gM_i \to \gM_{i+1}$ is nullhomotopic by definition of $\gM_{i+1}$, and so $g$ is nullhomotopic too.

To see that the fibre of $\gM \to \gM_{\infty}$ is $\mathcal{A}^f_\lambda$-acyclic, let $\gL$ be $\mathcal{A}^f_\lambda$-local and apply $\underline{\map}_\gR(-, \gL)$ to the cofibre sequences \eqref{eq:TelescopeDef}. As this sends sums to products and annihilates objects of $\mathcal{A}^f_\lambda$, it gives a sequence of equivalences
$$\underline{\map}_\gR(\gM_0, \gL) \overset{\sim}\longleftarrow \underline{\map}_\gR(\gM_1, \gL) \overset{\sim}\longleftarrow \underline{\map}_\gR(\gM_2, \gL) \overset{\sim}\longleftarrow \underline{\map}_\gR(\gM_3, \gL) \overset{\sim}\longleftarrow \cdots$$
and so, taking the limit, shows that 
$$\underline{\map}_\gR(\gM_\infty, \gL) \simeq \lim_{i \in \bN} \underline{\map}_\gR(\gM_i, \gL) \lra \underline{\map}_\gR(\gM_0, \gL) = \underline{\map}_\gR(\gM, \gL)$$
is an equivalence, for all $\mathcal{A}^f_\lambda$-local $\gL$. This is precisely the same as saying that the fibre of $\gM \to \gM_\infty$ is $\mathcal{A}^f_\lambda$-acyclic.
\end{proof}

Its main properties are summarised as follows: again we follow Miller \cite{MillerFin}.

\begin{prop}\label{prop:Localisation}\mbox{}
\begin{enumerate}[(i)]
\item\label{it:Localisation1} The class $\mathcal{A}^f_\lambda$ is closed under left or right tensoring with finite $\gR$-modules.

\item\label{it:Localisation2} The class of $\mathcal{A}^f_\lambda$-acyclic $\gR$-modules is closed under left or right tensoring with arbitrary $\gR$-modules.

\item\label{it:Localisation4} The localisation $L^f_\lambda$ is smashing, i.e.\ $L^f_\lambda(\gM) \simeq L^f_\lambda(\gR) \otimes_\gR \gM \simeq \gM \otimes_\gR L_\lambda^f(\gR)$. Thus $L^f_\lambda(\gR)$ is an idempotent $\gR$-algebra, and $L^f_\lambda$ is Bousfield localisation with respect to it.
\end{enumerate}
\end{prop}
\begin{proof}
For (\ref{it:Localisation1}), note that left or right tensoring an object $\gT$ by a finite $\gR$ module gives an object in the thick subcategory generated by $\gT$, as finite $\gR$-modules are in the thick subcategory generated by the monoidal unit $\gR$. If $\gT$ has a vanishing line of slope $\lambda$ then anything in the thick subcategory it generates does too.

For (\ref{it:Localisation2}), let $\gA$ be $\mathcal{A}^f_\lambda$-acyclic, $\gM$ be an $\gR$-module, and $\gL$ be $\mathcal{A}^f_\lambda$-local. Then
$$\underline{\map}_\gR(\gM \otimes_\gR \gA, \gL) \simeq \underline{\map}_\gR(\gM, \underline{\map}_\gR^r(\gA, \gL))$$
is trivial as $\underline{\map}_\gR(\gA, \gL) \simeq 0$ by definition of $\mathcal{A}^f_\lambda$-acyclic so the $\gR$-module $\underline{\map}^r_\gR(\gA, \gL)$ is trivial too. This holds for all $\mathcal{A}^f_\lambda$-local $\gL$, so $\gM \otimes_\gR \gA$ is $\mathcal{A}^f_\lambda$-acyclic. The same argument goes through for $\gA \otimes_\gR \gM$ using the left internal mapping object.

For (\ref{it:Localisation4}) we verify that the map $\gM = \gR \otimes_\gR \gM \to L_\lambda^f(\gR) \otimes_\gR \gM$ is an $\mathcal{A}^f_\lambda$-localisation. Its fibre is $C_\lambda^f(\gR) \otimes_\gR \gM$, which is $\mathcal{A}^f_\lambda$-acyclic by (\ref{it:Localisation2}) as $C_\lambda^f(\gR)$ is $\mathcal{A}^f_\lambda$-acyclic by definition. It remains to show that $L_\lambda^f(\gR) \otimes_\gR \gM$ is $\mathcal{A}^f_\lambda$-local. To see this, express $\gM$ as a filtered colimit $\colim_{\alpha} \gM_\alpha$ with $\gM_\alpha$ finite $\gR$-modules, and let $\gT \in \mathcal{A}^f_\lambda$. Then as $\gT$ is compact we have
\begin{align*}
\underline{\map}_\gR(\gT, L_\lambda^f(\gR) \otimes_\gR \gM) &\simeq \colim_\alpha \underline{\map}_\gR(\gT, L_\lambda^f(\gR) \otimes_\gR \gM_\alpha)\\
&\simeq \colim_\alpha \underline{\map}_\gR(\gT \otimes_\gR D^r(\gM_\alpha), L_\lambda^f(\gR))
\end{align*}
which is trivial as $\gT \otimes_\gR D^r(\gM_\alpha)$ is $\mathcal{A}^f_\lambda$-acyclic by (\ref{it:Localisation2}), and $L_\lambda^f(\gR)$ is $\mathcal{A}^f_\lambda$-local. The analogous argument goes through for $\gM \otimes_\gR L_\lambda^f(\gR)$, using left duals.
\end{proof}

It is sometimes convenient to have a more $\infty$-categorical perspective on the localisations $L^f_\lambda(-)$. The full subcategory $\mathcal{A}_\lambda^f\text{-}\mathsf{loc} \subseteq \gR\text{-}\mathsf{mod}$ of $\mathcal{A}^f_\lambda$-local objects is closed under homotopy limits. Furthermore, as $\mathcal{A}^f_\lambda$ consists of finite $\gR$-modules, which are in particular compact objects of $\gR\text{-}\mathsf{mod}$, the subcategory $\mathcal{A}^f_\lambda\text{-}\mathsf{loc}$ is closed under filtered homotopy colimits. Using that $\gR\text{-}\mathsf{mod}$ is presentable, it follows using the adjoint functor theorem \cite[Corollary 5.5.2.9]{HTT} (or alternatively from \cite[Proposition 5.5.4.15]{HTT}) that the inclusion of this full subcategory has a left adjoint
$$L_\lambda^f : \gR\text{-}\mathsf{mod} \lra \mathcal{A}^f_\lambda\text{-}\mathsf{loc}.$$
We momentarily overload our earlier notation by writing $\eta_\gM : \gM \to L_\lambda^f(\gM)$ for the unit of this adjunction, and $C_\lambda^f(\gM)$ for its fibre. If $\gL$ is $\mathcal{A}^f_\lambda$-local then in the fibre sequence
$$\underline{\map}_\gR(L_\lambda^f(\gM), \gL) \overset{- \circ \eta_\gM}\lra \underline{\map}_\gR(\gM, \gL) \lra \underline{\map}_\gR(C_\lambda^f(\gM), \gL)$$
the left-hand map is an equivalence by definition of $L_\lambda^f$ as a left adjoint, so $\map_\gR(C_\lambda^f(\gM), \gL)\simeq 0$. This holds for all $\mathcal{A}^f_\lambda$-local $\gL$'s, so $C_\lambda^f(\gM)$ is $\mathcal{A}^f_\lambda$-acyclic. Thus $\eta_\gM : \gM \to L_\lambda^f(\gM)$ is an $\mathcal{A}^f_\lambda$-localisation in the sense of Definition \ref{defn:Local}.

\begin{rem}
Nothing we have said so far is very particular to $\mathcal{A}_\lambda^f$: we have only used that it is an essentially small class of compact objects.
\end{rem}

\subsection{Examples}

In this section we describe some calculations of $L^f_{\lambda}(\gR)$ for examples such as those  discussed in Section \ref{sec:KindsOfStab}. The examples described in \ref{sec:ordinarystab} satisfy the hypotheses of the following proposition (usually with $\lambda=\frac{1}{2}$).

\begin{prop}\label{prop:SigmaStableHomology}
Let $\gR \in \mathsf{Alg}_{E_2}(\mathsf{D}(\bk)^\bZ)$, and suppose that there is a $\sigma \in \pi_{1,0}(\gR)$ such that $\gR/\sigma$ has a vanishing line of slope $\lambda$. Then there is an equivalence
$$\sigma^{-1} \gR \simeq L^f_\lambda(\gR)$$
of $\gR$-modules.
\end{prop}
\begin{proof}
Here $\sigma^{-1}\gR$ means the colimit of the telescope
$$\gR \overset{\sigma}\lra S^{-1,0} \otimes \gR \overset{\sigma}\lra S^{-2,0} \otimes \gR \overset{\sigma}\lra S^{-3,0} \otimes \gR \lra \cdots$$
in $\gR$-modules. We wish to recognise the map $\gR \to \sigma^{-1}\gR$ given by the inclusion of the initial term as $\mathcal{A}^f_\lambda$-localisation, so must verify that its fibre is $\mathcal{A}^f_\lambda$-acyclic and that $\sigma^{-1}\gR$ is $\mathcal{A}^f_\lambda$-local.

Considering $\gR$ as the colimit of the constant telescope, we see that the cofibre of $\gR \to \sigma^{-1}\gR$ may be expressed as the colimit of
$$0 \lra S^{-1,0} \otimes \gR/\sigma \lra S^{-2,0} \otimes \gR/\sigma^2 \lra S^{-3,0} \otimes \gR/\sigma^3 \lra \cdots.$$
The cofibres of these map all have the form $S^{-i,0} \otimes \gR/\sigma$ so are $\mathcal{A}^f_\lambda$-acyclic, and as $\mathcal{A}^f_\lambda$-acyclic objects are closed under colimits it follows that the cofibre of $\gR \to \sigma^{-1}\gR$ is $\mathcal{A}^f_\lambda$-acyclic, and hence so is its fibre.

Let $\gT \in \mathcal{A}^f_\lambda$ and $f : \gT \to \sigma^{-1}\gR$ be a map. Then the commutative diagram
\begin{equation*}
\begin{tikzcd}[column sep = 3em]
\gT \arrow[r, equals] & \gR \otimes_\gR \gT \rar{\gR \otimes_\gR f} \dar & \gR \otimes_\gR \sigma^{-1} \gR \arrow[r, equals] \dar{\simeq} & \sigma^{-1}\gR\\
\sigma^{-1}\gT \arrow[r, equals] & \sigma^{-1}\gR \otimes_\gR \gT \rar{\sigma^{-1}\gR \otimes_\gR f} & \sigma^{-1}\gR \otimes_\gR \sigma^{-1}\gR \arrow[r, equals] & \sigma^{-1}\gR
\end{tikzcd}
\end{equation*}
shows that $f$ factors through $\sigma^{-1}\gT$ up to homotopy. As $\sigma$ is a class of slope 0, and $\gT$ has a vanishing line of slope $>0$, we have $\sigma^{-1} \gT\simeq 0$ and hence $f=0$. Thus $\sigma^{-1} \gR$ is $\mathcal{A}^f_\lambda$-local.
\end{proof}

\begin{rem}
The proof of this proposition shows a little more. Namely, the cofibre of $\gR \to \sigma^{-1}\gR$ has a filtration with associated graded $\bigoplus_{i=1}^\infty S^{-i,0} \otimes \gR/\sigma$, so if $\pi_{n,d}(\gR/\sigma)=0$ for $d < \lambda n + \mu$ then the spectral sequence for this filtration shows that $\pi_{n,d}(\sigma^{-1}\gR, \gR)=0$ for $d < \lambda(n+1)+\mu$. So the fibre of the map $\gR \to L_\lambda^f(\gR)$ has a vanishing line of slope $\lambda$, and hence this map induces an isomorphism on homotopy groups below a line of slope $\lambda$: in other words, this map is a \emph{good} approximation in such a range. In Section \ref{sec:QualityApprox} we will give a general form of this argument.
\end{rem}

The example $\mathbf{RB}$ of configuration spaces of red-or-blue points in the plane described in Section \ref{sec:multistab} satisfies the hypotheses of the following proposition, with $\lambda=\tfrac{1}{2}$. We will spell this out in Example \ref{ex:RBLocalissation} below.

\begin{prop}\label{prop:redBlueStableHomology}
Let $\gR \in \mathsf{Alg}_{E_2}(\mathsf{D}(\bk)^\bZ)$, and suppose that there are elements $r, b \in \pi_{1,0}(\gR)$ such that $\gR/(r,b)$ has a vanishing line of slope $\lambda$. Then there is a homotopy cartesian square
\begin{equation}\label{eq:RedAndBlueLocalisation}
\begin{tikzcd}
L^f_\lambda(\gR) \rar \dar & r^{-1}\gR \dar\\
b^{-1}\gR \rar & r^{-1} b^{-1} \gR
\end{tikzcd}
\end{equation}
of $\gR$-modules.
\end{prop}
\begin{proof}
Write $\gP$ for the homotopy pullback of the punctured square, so that the localisation maps $\gR \to r^{-1} \gR$ and $\gR \to b^{-1} \gR$ assemble to a map $\gR \to \gP$. We show that this map is an $\mathcal{A}^f_\lambda$-localisation, by checking that the fibre of $\gR \to \gP$ is $\mathcal{A}^f_\lambda$-acyclic and that $\gP$ is $\mathcal{A}^f_\lambda$-local.

The fibre of $\gR \to \gP$ is the total homotopy fibre of the square obtained by replacing the top left corner of \eqref{eq:RedAndBlueLocalisation} by $\gR$. This square is obtained by tensoring together the arrows $\gR \to r^{-1}\gR$ and $\gR \to b^{-1}\gR$, so its total homotopy fibre is obtained by tensoring together the fibres $\gF_r$ and $\gF_b$ of these two maps. As in the proof of the previous proposition, $\gF_r$ is a telescope with all cofibres given by shifts of $\gR/r$, and similarly $\gF_b$ is a telescope with all cofibres given by shifts of $\gR/b$. From this we can express $\gF_r \otimes_\gR \gF_b$ as a telescope with all cofibres given by sums of shifts of $\gR/r \otimes_\gR \gR/b \simeq \gR/(r,b)$, which lies in $\mathcal{A}^f_\lambda$. Thus $\gF_r \otimes_\gR \gF_b$ is $\mathcal{A}^f_\lambda$-acyclic.

To verify that the pullback $\gP$ is $\mathcal{A}^f_\lambda$-local, for $\gT \in \mathcal{A}^f_\lambda$ we obtain a cartesian square of mapping objects
\begin{equation*}
\begin{tikzcd}
\underline{\map}_\gR(\gT, \gP) \rar \dar & \underline{\map}_\gR(\gT, r^{-1}\gR) \dar\\
\underline{\map}_\gR(\gT,b^{-1}\gR) \rar & \underline{\map}_\gR(\gT,r^{-1} b^{-1} \gR).
\end{tikzcd}
\end{equation*}
As in the proof of the previous proposition we have 
$$\underline{\map}_\gR(\gT, r^{-1}\gR) \simeq \underline{\map}_\gR(r^{-1}\gT, r^{-1}\gR) \simeq 0,$$
because $r^{-1}\gT\simeq 0$ (as $r$ has slope 0 and $\gT$ has a vanishing line of slope $>0$). Similarly $\underline{\map}_\gR(\gT,b^{-1}\gR)$ and $\underline{\map}_\gR(\gT,r^{-1} b^{-1} \gR)$ vanish, so as the square is cartesian it follows that $\underline{\map}_\gR(\gT, \gP) \simeq 0$. This holds for all $\gT \in \mathcal{A}^f_\lambda$, so $\gP$ is $\mathcal{A}^f_\lambda$-local. (More succinctly, each of $r^{-1}\gR$, $b^{-1}\gR$, and $r^{-1}b^{-1}\gR$ is $\mathcal{A}^f_\lambda$-local as inverting $r$ or $b$ kills all objects in $\mathcal{A}^f_\lambda$, and the $\mathcal{A}^f_\lambda$-local modules are closed under limits.)
\end{proof}

\begin{rem}\mbox{}
\begin{enumerate}[(i)]
\item Proposition \ref{prop:redBlueStableHomology} admits an evident generalisation: to a sequence $\sigma_1, \ldots, \sigma_r \in \pi_{*,*}(\gR)$ of elements such that $\gR/(\sigma_1, \ldots, \sigma_r)$ has a vanishing line of slope $\lambda$, and such that each $\sigma_i$ has slope $< \lambda$. Then $L_\lambda^f(\gR)$ is described via the $r$-cube obtained by tensoring together the maps $\gR \to \sigma_i^{-1} \gR$.

\item In the setting of (i) our discussion has a lot to do with local and \v{C}ech cohomology in their homotopical incarnations introduced by Greenlees and May \cite{GreenleesMay}, namely homotopical $I$-power torsion and homotopical localisation away from $I$. Concretely, in the setting of (i) and with $I := (\sigma_1, \ldots, \sigma_r) \subset \pi_{*,*}(\gR)$, the description in (i) shows that the fibre sequences
\begin{equation*}
\begin{tikzcd}[row sep = 0.4em]
C^f_\lambda(\gR) \rar & \gR \rar & L^f_\lambda(\gR)\\
\Gamma_{I}(\gR) \rar & \gR \rar & \gR[I^{-1}]
\end{tikzcd}
\end{equation*}
are identical.

The object $\gR[I^{-1}]$ may also be described as Bousfield localisation with respect to $\gR/(\sigma_1, \ldots, \sigma_r)$ (cf.\ \cite[Proposition 6.14]{DwyerGreenlees}), which lies in $\mathcal{A}^f_\lambda$ and so gives a factorisation $\gR \to \gR[I^{-1}] \to L^f_\lambda(\gR)$. The fact that the latter map is an equivalence reflects the fact that $\mathcal{A}_\lambda^f$ is generated as a thick subcategory by $\gR/(\sigma_1, \ldots, \sigma_r)$, which it is not too hard to prove manually. We will discuss a generalised form of this in Section \ref{sec:Monigenicity}.

\item Typical instances of higher-order homological stability are more complicated than the setting of (i), in that the higher-order stabilisation maps are usually only defined after taking cofibres for the lower-order ones: in particular one cannot ``invert them in $\gR$'' as they are not defined in $\gR$, so one cannot even formulate something like Proposition \ref{prop:redBlueStableHomology}. In Section \ref{sec:OrthCalcModel} we will discuss an analogue of the constructions of this section in this more general setting.
\end{enumerate}
\end{rem}

\begin{example}\label{ex:RBLocalissation}
The notation in Proposition \ref{prop:redBlueStableHomology} is inspired by the example $\mathbf{RB} = \gE_2(S^{1,0} r \oplus S^{1,0} b)$ discussed in Section \ref{sec:multistab}, which may be considered as the $\bk$-chains on the spaces of configurations of distinct unordered red-or-blue points in $[0,1]^2$. We mentioned the fact that $\mathbf{RB}/(r,b)$ has a vanishing line of slope $\tfrac{1}{2}$. In fact, the work of F.\ Cohen shows that as a $\bk$-algebra we have
$$\pi_{*,*}(\mathbf{RB}) \cong \bk[r,b] \otimes \mathrm{Sym}^*[\text{generators of slope $\geq \tfrac{1}{2}$}],$$
where $\mathrm{Sym}^*[-]$ denotes the free graded-commutative algebra. This shows that there is an isomorphism of bigraded $\bk$-modules
$$\pi_{*,*}(\mathbf{RB}/(r,b)) \cong \pi_{*,*}(\mathbf{RB})/(r,b) \cong \mathrm{Sym}^*[\text{generators of slope $\geq \tfrac{1}{2}$}],$$
from which one immediately sees the vanishing line of slope $\tfrac{1}{2}$. Using this, Proposition \ref{prop:redBlueStableHomology} gives a cartesian square
\begin{equation*}
\begin{tikzcd}
L^f_{1/2}(\mathbf{RB}) \rar \dar & r^{-1}\mathbf{RB} \dar\\
b^{-1}\mathbf{RB} \rar & r^{-1} b^{-1} \mathbf{RB},
\end{tikzcd}
\end{equation*}
mentioned in Example \ref{ex:RedBlueStabHom}. We have $\pi_{*,*}(\mathbf{RB}) \cong \bk[r,b] \otimes \pi_{*,*}(\mathbf{RB}/(r,b))$ as a $\bk[r,b]$-module. The homotopy pullback $P$ of
\begin{equation*}
\begin{tikzcd}
 & \bk[r^{\pm 1},b] \dar\\
\bk[r,b^{\pm 1}] \rar & \bk[r^{\pm 1},b^{\pm 1}]
\end{tikzcd}
\end{equation*}
in the derived category of $\bk[r,b]$-modules is no longer a discrete $\bk[r,b]$-module, but rather has
\begin{align*}
\pi_{*,0}(P) &\cong \bk[r,b]\\
\pi_{*,-1}(P) &\cong r^{-1}b^{-1}\bk[r^{-1}, b^{-1}],
\end{align*}
where the latter has the evident $\bk[r,b]$-module structure. This leads to the formula
$$\pi_{*,*}(L^f_{1/2}(\mathbf{RB})) \cong \pi_{*,*}(\mathbf{RB}) \oplus r^{-1}b^{-1}\bk[r^{-1}, b^{-1}] \otimes \pi_{*,*+1}(\mathbf{RB}/(r,b))$$
as a bigraded $\bk$-module. The second summand, which is $\pi_{*,*}(S^{0,1} \otimes C^f_{1/2} (\mathbf{RB}))$, is seen to vanish in bidegrees $(n,d)$ with $d < -1$ or $d < \tfrac{1}{2}n$.
\end{example}

\section{Applications of Smith--Toda complexes}\label{sec:STApp1}

By a \emph{Smith--Toda complex} we mean an $\gR$-module $\gR/(\alpha_1, \ldots, \alpha_{r})$ obtained as the iterated cofibre of a sequence of endomorphisms
$$\alpha_j : S^{n_j, d_j} \otimes \gR/(\alpha_1, \ldots, \alpha_{j-1}) \lra \gR/(\alpha_1, \ldots, \alpha_{j-1}).$$
We will call $\tfrac{d_j}{n_j}$ the \emph{slope} of $\alpha_j$. We will usually adopt some mild axioms and standardised notation for Smith--Toda complexes, as follows.

\begin{defn}
A Smith--Toda complex $\gR/(\alpha_1, \ldots, \alpha_{r})$ is \emph{admissible} if
\begin{enumerate}[(i)]
\item The slopes $\tfrac{d_1}{n_1}, \tfrac{d_2}{n_2}, \ldots$ are non-decreasing, and

\item it has a vanishing line of slope $\lambda$ with $\tfrac{d_r}{n_r} < \lambda$.

\end{enumerate}
We then write $\lambda' := \tfrac{d_r}{n_r}$, and let $\alpha_k, \ldots, \alpha_r$ denote the $\alpha$'s of slope precisely $\lambda'$.
\end{defn}

In Section \ref{sec:HigherStabMaps} we will prove that admissible Smith--Toda complexes exist for any $\lambda < 1$, as long as $\bk$ is a field of positive characteristic. In this section we explain what one can do with admissible Smith--Toda complexes, whether they are produced by the argument of Section \ref{sec:HigherStabMaps} or otherwise.

\begin{prop}\label{prop:SubOfAdmIsAdm}
If $\gR/(\alpha_1, \ldots, \alpha_{r})$ is an admissible Smith--Toda complex then $\gR/(\alpha_1, \ldots, \alpha_{k-1})$ is an admissible Smith--Toda complex with a vanishing line of slope $\lambda'$.
\end{prop}

This follows by downwards induction using the following general lemma.

\begin{lem}
Let $\gM$ be a finite $\gR$-module, and $\phi : S^{N,D} \otimes \gM \to \gM$ be an endomorphism of non-zero bidegree with $\lambda := \tfrac{D}{N}$. If $\pi_{n,d}(\gM/\phi)=0$ for $d < \lambda n + \kappa$, then $\pi_{n,d}(\gM)=0$ for $d < \lambda n + \kappa$ too.
\end{lem}
\begin{proof}
Consider the tower
\begin{equation*}
\begin{tikzcd}
\cdots S^{3N,3D} \otimes \gM \rar{\phi} &S^{2N,2D} \otimes \gM \rar{\phi} \dar & S^{N,D} \otimes \gM \rar{\phi} \dar & \gM \dar\\
& S^{2N,2D} \otimes \gM/\phi & S^{N,D} \otimes \gM/\phi & \gM/\phi
\end{tikzcd}
\end{equation*}
If $x \in \pi_{n,d}(\gM)$ with $d < \lambda n + \kappa$ then it vanishes when mapped to $\gM/\phi$ so $x = \phi(x')$ for some $x' \in \pi_{n-N, d-D}(\gM)$, and its bidegrees also satisfy $d-D < \lambda (n-N) + \kappa$ as $D = \lambda N$. Thus $x' = \phi(x'')$ for $x'' \in \pi_{n-2N, d-2D}(\gM)$, and so on. As $\gM$ is a finite $\gR$-module and $(N,D) \neq (0,0)$, $\pi_{n-kN, d-kD}(\gM)=0$ for $k \gg 0$, so $x=0$.
\end{proof}

\subsection{Quality of the approximation}\label{sec:QualityApprox}

One application of the existence of an admissible Smith--Toda complex with a slope $\lambda$ vanishing line is that it guarantees the high-connectivity of the map $\gR \to L^f_\lambda(\gR)$, as follows. 

\begin{thm}\label{thm:HigherStabRanges}
Let $\gR \in \mathsf{Alg}_{E_2}(\mathsf{D}(\bk)^\bZ)$. Let $\gR/(\alpha_1, \ldots, \alpha_{r})$ be an admissible Smith--Toda complex such that $\pi_{n,d}(\gR/(\alpha_1, \ldots, \alpha_{r}))=0$ for $d < \lambda n + \kappa$. Then the fibre $C_\lambda^f(\gR)$ of
$$\gR \lra L_\lambda^f(\gR)$$
has $\pi_{n,d}(C_\lambda^f(\gR))=0$ for $d < \lambda n + \kappa + \sum_{i=1}^r(\lambda n_i - d_i-1)$, and also for $d < \lambda' n + \tfrac{\lambda'}{\lambda}\kappa + \sum_{i=1}^r(\lambda' n_i - d_i-1)$.
\end{thm}

\begin{wrapfigure}[5]{r}{0.45\linewidth}
\vspace{-7.5ex}
\centering
\begin{tikzpicture}[scale=0.55]
  
  \fill[gray] (-3,-1) -- (2,1) -- (4,3) -- (-3,3);

  \draw[->] (-3,0) -- (4,0) node[right] {$n$};
  \draw[->] (0,-1) -- (0,3.2) node[above] {$d$};
  
  \draw (2,1) -- (4,3);
  \draw (2,1) -- (-3,-1);
  \draw (4.2,2) node {slope $\lambda$};
  \draw (-1.1,-0.9) node {slope $\lambda'$};
\end{tikzpicture}
\end{wrapfigure}

Note that the $\gR$-module $C_\lambda^f(\gR)$ need not be finite, so having a vanishing line of slope $\lambda'$ does not follow from having one of slope $\lambda$: as depicted in the adjacent figure, it is more information. 

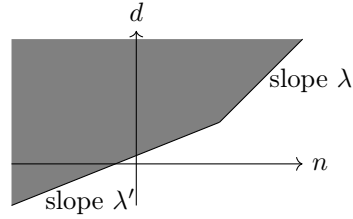
\begin{proof}[Proof of Theorem \ref{thm:HigherStabRanges}]
For any $\gR$-module $\gM$ write $\gM/(\alpha_1, \ldots, \alpha_{j-1}) := \gM \otimes_\gR \gR/(\alpha_1, \ldots, \alpha_{j-1})$, so there are induced maps
$$\alpha_j = \gM \otimes_{\gR} \alpha_j : S^{n_j, d_j} \otimes \gM/(\alpha_1, \ldots, \alpha_{j-1}) \lra \gM/(\alpha_1, \ldots, \alpha_{j-1}).$$
If $\gT \in \mathcal{A}^f_\lambda$ then $\gT/(\alpha_1, \ldots, \alpha_{j-1}) \in \mathcal{A}^f_\lambda$ as it is in the thick subcategory generated by $\gT$. As $\tfrac{d_j}{n_j} < \lambda$ it follows that
$$\pi_{*,*}(\alpha_j^{-1} \gT/(\alpha_1, \ldots, \alpha_{j-1})) = \alpha_j^{-1} \pi_{*,*}(\gT/(\alpha_1, \ldots, \alpha_{j-1})) = 0,$$
so $\alpha_j^{-1} \gT/(\alpha_1, \ldots, \alpha_{j-1}) = 0$.

Abbreviate $\gC := C_\lambda^f(\gR)$. By its construction in Proposition \ref{prop:LocExists}, $\gC$ is a telescope of maps whose cofibres are sums of elements of $\mathcal{A}^f_\lambda$. The construction $\alpha_j^{-1} (-)/(\alpha_1, \ldots, \alpha_{j-1}) = (-) \otimes_\gR \alpha^{-1}_j \gR/(\alpha_1, \ldots, \alpha_{j-1})$ commutes with colimits, so from the above discussion it follows that $\alpha_j^{-1} \gC/(\alpha_1, \ldots, \alpha_{j-1}) \simeq 0$.

From the directed colimit
$$\gC \overset{\alpha_1} \lra S^{-n_1, -d_1} \otimes \gC \overset{\alpha_1}\lra S^{-2n_1, -2d_1} \otimes \gC \lra \cdots \lra \alpha_1^{-1} \gC \simeq 0$$
we see that $(\alpha_1^{-1}\gC) / \gC \simeq S^{0,1} \otimes \gC$ has a filtration whose associated graded is
$$\bigoplus_{i_1 =1}^\infty S^{-i_1 n_1, -i_1 d_1} \otimes \gC/\alpha_1.$$
Similarly, from the directed colimit
$$\gC/\alpha_1 \overset{\alpha_2} \lra S^{-n_2, -d_2} \otimes \gC/\alpha_1 \overset{\alpha_2}\lra S^{-2n_2, -2d_2} \otimes \gC/\alpha_1 \lra \cdots \lra \alpha_2^{-1} \gC/\alpha_1 \simeq 0$$
we see that $(\alpha_2^{-1} \gC/\alpha_1) / (\gC/\alpha_1) \simeq S^{0,1} \otimes \gC/\alpha_1$ has a filtration whose associated graded is
$$\bigoplus_{i_2 = 1}^\infty S^{-i_2 n_2, -i_2 d_2} \otimes \gC/(\alpha_1, \alpha_2).$$
Continuing in this way, we see that $S^{0,1} \otimes \gC/(\alpha_{1}, \ldots, \alpha_{j-1})$ has a filtration whose associated graded is
$$\bigoplus_{i_j = 1}^\infty S^{-i_j n_j, -i_j d_j} \otimes \gC/(\alpha_1, \ldots, \alpha_{j})$$
for each $j \leq r$.

As $\gR/(\alpha_1, \ldots, \alpha_r)$ is a finite $\gR$-module with a vanishing line of slope $\lambda$ we have that $L_{\lambda}^f(\gR/(\alpha_1,  \ldots, \alpha_r))=0$ and so $\gC/(\alpha_1,  \ldots, \alpha_r) \simeq \gR/(\alpha_1,  \ldots, \alpha_r)$, which we have assumed has trivial homotopy groups in bidegrees $(n,d)$ satisfying $d < \lambda n + \kappa$. The spectral sequence for the final filtration obtained takes the form
$$E^1_{n,p,q} = \pi_{n,p+q}(S^{-p n_r, -p d_r} \otimes \gC/(\alpha_1, \ldots, \alpha_{r})) \Rightarrow \pi_{n,p+q}(S^{0,1} \otimes \gC/(\alpha_1, \ldots, \alpha_{r-1})),$$
defined for $p \geq 1$. This shows that
$$\pi_{n,d}(\gC/(\alpha_1, \ldots, \alpha_{r-1}))=0 \text{ for } d < \lambda n + \kappa + (\lambda n_r - d_r-1).$$
Running the series of analogous spectral sequences for the earlier filtrations gives
$$\pi_{n,d}(\gC)=0 \text{ for } d < \lambda n + \kappa + \sum_{i=1}^r(\lambda n_i - d_i-1).$$

As $\gC/(\alpha_1,  \ldots, \alpha_r) \simeq \gR/(\alpha_1,  \ldots, \alpha_r)$ is a \emph{finite} $\gR$-module with a slope $\lambda$ vanishing line, it also has a slope $\lambda'$ vanishing line. Specifically, using that its homotopy groups vanish in negative homological degree we also have $\pi_{n,d}(\gR/(\alpha_1,  \ldots, \alpha_r))=0$ for $d < \lambda' n + \tfrac{\lambda'}{\lambda} \kappa$. Running the same series of spectral sequences as above using this vanishing line shows first that 
$$\pi_{n,d}(\gC/(\alpha_1, \ldots, \alpha_{r-1}))=0 \text{ for } d < \lambda' n + \tfrac{\lambda'}{\lambda}\kappa + (\lambda' n_r - d_r-1)$$
and then by the same series of spectral sequences that
\begin{equation*}
\pi_{n,d}(\gC)=0 \text{ for } d < \lambda' n + \tfrac{\lambda'}{\lambda}\kappa + \sum_{i=1}^r(\lambda' n_i - d_i-1).\qedhere
\end{equation*}
\end{proof}

\subsection{Adams periodicity}\label{sec:AdamsPeriodicity}

In Theorem \ref{thm:HigherStabRanges} we have seen that the existence of an admissible Smith--Toda complex $\gR/(\alpha_1, \ldots, \alpha_{r})$ with a slope $\lambda$ vanishing line means the fibre of $\gR \to L^f_\lambda(\gR)$ has a certain vanishing line. If $\lambda' := \tfrac{d_r}{n_r}$ and $\alpha_k, \ldots, \alpha_r$ are the $\alpha$'s of slope precisely $\lambda'$, then $\gR/(\alpha_1, \ldots, \alpha_{k-1})$ has a slope $\lambda'$ vanishing line by Proposition \ref{prop:SubOfAdmIsAdm} and so the fibre of $\gR \to L^f_{\lambda'}(\gR)$ also has a certain vanishing line. Writing
$$M^f_{\lambda, \lambda'}(\gR) := \mathrm{fib}(L^f_{\lambda}(\gR) \to L^f_{\lambda'}(\gR)),$$
we will show below that this has a slope $\lambda'$ vanishing line, but moreover that the $\alpha_k, \ldots, \alpha_r$ induce certain periodicities on the homotopy groups of $M^f_{\lambda, \lambda'}(\gR)$, in a range of degrees.

To formulate the ``in a range of degrees'' part of this result, for an object $X \in \mathsf{D}(\bk)^\bZ$ we write $\tau^{\lambda'}_{< b}(X) \in \mathsf{D}(\bk)^\bZ$ for the object given in terms of Postnikov truncation by
$$\tau^{\lambda'}_{< b}(X)(n) := \tau_{< \lfloor b+\lambda'n\rfloor}(X(n)).$$
(This is not truncation with respect to a $t$-structure unless $\lambda'=1$, in which case it is truncation with respect to the diagonal $t$-structure introduced in Section \ref{sec:TStructures}.)

\begin{thm}\label{thm:AdamsPeriodicityGeneral}
Let $\gR \in \mathsf{Alg}_{E_2}(\mathsf{D}(\bk)^\bZ)$. Let $\gR/(\alpha_1, \ldots, \alpha_{r})$ be an admissible Smith--Toda complex with a slope $\lambda$ vanishing line, and write $\lambda' := \tfrac{d_r}{n_r}$ and let $\alpha_k, \ldots, \alpha_r$ be the $\alpha$'s of slope precisely $\lambda'$. Then

\begin{enumerate}[(i)]

\item\label{it:AdamsPeriodicityGeneral:1} $M^f_{\lambda, \lambda'}(\gR)$ has a slope $\lambda'$ vanishing line, and

\item\label{it:AdamsPeriodicityGeneral:2} if $\pi_{n,d}(M^f_{\lambda, \lambda'}(\gR))=0$ for $d < \lambda' n + \kappa'$, then setting
$$b := \kappa' + \min_{S \subsetneq \{1,2,\ldots,k-1\}}\left[ \sum_{s \not\in S}(\lambda' n_s - d_s-1) \right]$$
there are maps
$$\beta_{i} : S^{n_i, d_i}\otimes\tau_{< b}^{\lambda'} M_{\lambda, \lambda'}^f(\gR)/(\beta_k, \ldots, \beta_{i-1}) \lra  \tau_{\leq b}^{\lambda'} M_{\lambda, \lambda'}^f(\gR)/(\beta_k, \ldots, \beta_{i-1})$$
for $k \leq i \leq r$ with $\tau_{< b}^{\lambda'} M_{\lambda, \lambda'}^f(\gR)/(\beta_k, \ldots, \beta_r) \simeq 0$.
\end{enumerate}
\end{thm}

The meaning of this theorem is clearest when there is a single $\alpha$ of slope $\lambda'$, namely $\alpha_r$. It then says that there is a periodicity isomorphism
$$\beta_r : \pi_{n- n_r, d-d_r}(M_{\lambda, \lambda'}^f(\gR)) \overset{\sim}\lra \pi_{n,d}(M_{\lambda, \lambda'}^f(\gR))$$
defined for $d < \lambda'n + b$. This is only non-vacuous in the range where these homotopy groups are not known to vanish, i.e.\ for $b > \kappa'$, i.e.\ for 
$$\min_{S \subsetneq \{1,2,\ldots,r-1\}}\left[ \sum_{s \not\in S}((\lambda' n_s - d_s)-1) \right] > 0.$$
As $\tfrac{d_s}{n_s} < \lambda'$ for $s \in \{1,2,\ldots,r-1\}$, by the assumption that only $\alpha_r$ has slope $\lambda'$, we have $\lambda' n_s-d_s>0$ for each $s$, but a little more is required for the conclusion to be non-vacuous. Namely, we wish to find Smith--Toda complexes where these values are arbitrarily large: that will yield periodicity isomorphisms in arbitrarily wide bands of degrees parallel to the vanishing line (although the wider one wishes this band of homotopy groups to be, the longer one must typically take the period of the periodicity to be). When we establish the existence of Smith--Toda complexes in Section \ref{sec:HigherStabMaps} we will explain how this property may also be achieved.

\begin{rem}
 We call this ``Adams periodicity'' by analogy with his periodicity theorem \cite[Theorem 1.2]{AdamsPeriodicity} for $\Ext$ over the Steenrod algebra. Indeed, working with the $E_\infty$-algebra $\mathbf{a} := \Cobar(\mathcal{A}_*)$ (see Section \ref{sec:CobarOfHopf}) Adams' theorem may be recovered as an instance of this one.
\end{rem}

When there are several $\alpha$'s of slope $\lambda'$ it seems more difficult to give a concrete interpretation, unless one knows that the stabilisation maps $\alpha_k, \ldots, \alpha_r$ of slope $\lambda'$ can be defined simultanously as endomorphisms $\gR/(\alpha_1, \ldots, \alpha_{k-1})$, and arranged to form a coherently commutative cube. In Section \ref{sec:SimultaneousStabilisation} we will show that this is possible, at least when $\gR$ is an $E_3$-algebra.

\begin{addendum}\label{add:SimultaneousStabilisation}
If $\alpha_k, \ldots, \alpha_r$ can all be defined on $\gR/(\alpha_1, \ldots, \alpha_{k-1})$ so as to yield a coherently homotopy commutative $(r-k)$-cube, then $\beta_k, \ldots, \beta_r$ can all be defined on $\tau_{\leq b}^{\lambda'} M_{\lambda, \lambda'}^f(\gR)$ so as to yield a coherently commutative $(r-k)$-cube. In this case the conclusion of the Theorem is that this cube is homotopy (co)cartesian.
\end{addendum}

\begin{proof}[Proof of Theorem \ref{thm:AdamsPeriodicityGeneral} and Addendum \ref{add:SimultaneousStabilisation}]
For (\ref{it:AdamsPeriodicityGeneral:1}), it follows from Theorem \ref{thm:HigherStabRanges} that $C_\lambda^f(\gR)$ has a vanishing line of slope $\lambda$ as well as a vanishing line of slope $\lambda'$, and $C_{\lambda'}^f(\gR)$ has a vanishing line of slope $\lambda'$, so it follows that $M^f_{\lambda, \lambda'}(\gR) \simeq \Sigma \mathrm{fib}(C_\lambda^f(\gR) \to C_{\lambda'}^f(\gR))$ has a vanishing line of slope $\lambda'$. 

For (\ref{it:AdamsPeriodicityGeneral:2}), note that the object $\gR/(\alpha_1, \ldots, \alpha_{j-1})$ has a (skeletal) filtration with associated graded
$$\bigoplus_{S \subseteq \{1,2,\ldots, j-1\}} S^{\sum_{s \in S} n_s, \sum_{s \in S} (d_s+1)} \otimes \gR,$$
so the fibre of the projection $\gR/(\alpha_1, \ldots, \alpha_{j-1}) \to S^{\sum_{s=1}^{j-1} n_s, \sum_{s =1}^{j-1} (d_s+1)} \otimes \gR$ to the top cell has a filtration with associated graded given by the analogous sum but over proper subsets $S \subsetneq \{1,2,\ldots, j-1\}$. It follows that
$$\mathrm{fib}(S^{-\sum_{s=1}^{k-1} n_s, -\sum_{s=1}^{k-1}(d_s+1)} \otimes M_{\lambda, \lambda'}^f(\gR)/(\alpha_1, \ldots, \alpha_{k-1}) \to  M_{\lambda, \lambda'}^f(\gR))$$
has a filtration with associated graded
$$\bigoplus_{S \subsetneq \{1,2,\ldots, k-1\}} S^{-\sum_{s \not\in S} n_s, -\sum_{s \not\in S} (d_s+1)} \otimes M^f_{\lambda, \lambda'}(\gR)$$
and therefore has trivial $(n,d)$-th homotopy group as long as 
\begin{align*}
d &< \min_{S \subsetneq \{1,2,\ldots,k-1\}}\left[\lambda'(n + \sum_{s \not\in S}  n_s) - \sum_{s \not\in S}(d_s+1) + \kappa' \right]\\
& \quad\quad = \lambda' n + \kappa' + \min_{S \subsetneq \{1,2,\ldots,k-1\}}\left[ \sum_{s \not\in S}(( \lambda' n_s - d_s)-1) \right].
\end{align*}
In particular the homotopy groups of this fibre vanish for $d < \lambda'n+b$, so there is an equivalence
$$\tau_{< b}^{\lambda'} M_{\lambda, \lambda'}^f( S^{-\sum_{s=1}^{k-1}  n_s, -\sum_{s =1}^{k-1}(d_s+1)} \otimes \gR/(\alpha_1, \ldots, \alpha_{k-1}))\overset{\sim} \lra \tau_{< b}^{\lambda'} M_{\lambda, \lambda'}^f(\gR).$$
Given this equivalence, the maps $\beta_k, \ldots, \beta_r$ on $\tau_{< b}^{\lambda'} M_{\lambda, \lambda'}^f(\gR)$ are just those induced by applying $\tau_{< b}^{\lambda'} M_{\lambda, \lambda'}^f(-)$ to the sequence of maps $\alpha_k, \ldots, \alpha_r$ starting with $S^{-\sum_{s=1}^{k-1}  n_s, -\sum_{s =1}^{k-1}(d_s+1)} \otimes \gR/(\alpha_1, \ldots, \alpha_{k-1})$. The conclusion follows as
$$\tau_{< b}^{\lambda'} M_{\lambda, \lambda'}^f(\gR)/(\beta_k, \ldots, \beta_r) \simeq \tau_{< b}^{\lambda'} M_{\lambda, \lambda'}^f(S^{-\sum_{s=1}^{k-1}  n_s, -\sum_{s =1}^{k-1}(d_s+1)} \otimes \gR/(\alpha_1, \ldots, \alpha_r))$$
and $\gR/(\alpha_1, \ldots, \alpha_r)$ has a vanishing line of slope $\lambda$ so is annihilated by $M_{\lambda, \lambda'}^f(-)$.

For the Addendum, apply $\tau_{< b}^{\lambda'} M_{\lambda, \lambda'}^f(-)$ to the coherently commutative $(r-k)$-cube given by the endomorphisms $\alpha_k, \ldots, \alpha_r$ of $S^{-\sum_{s=1}^{k-1}  n_s, -\sum_{s =1}^{k-1}(d_s+1)} \otimes \gR/(\alpha_1, \ldots, \alpha_{k-1})$, to obtain a coherently commutative $(r-k)$-cube of endomorphisms  $\beta_k, \ldots, \beta_r$ of $\tau_{< b}^{\lambda'} M_{\lambda, \lambda'}^f(\gR)$. As above the total homotopy cofibre of this $(r-k)$-cube is identified with $\tau_{< b}^{\lambda'} M_{\lambda, \lambda'}^f(S^{-\sum_{s=1}^{k-1}  n_s, -\sum_{s =1}^{k-1}(d_s+1)} \otimes \gR/(\alpha_1, \ldots, \alpha_r))$ and vanishes.
\end{proof}

\subsection{Monogenicity and Quantisation}\label{sec:Monigenicity}

A further application of the existence of an admissible Smith--Toda complex with a slope $\lambda$ vanishing line is that it implies that $L^f_\lambda(-)$ is given by localisation away from a single object (namely the Smith--Toda complex), and also that there are only finitely-many different localisation functors at slopes $\leq \lambda$, as follows.

\begin{thm}\label{thm:STQuantisation}
If $\gR/(\alpha_1, \ldots, \alpha_{r})$ is an admissible Smith--Toda complex with a slope $\lambda$ vanishing line and $\bar{\lambda}$ satisfies $\tfrac{d_r}{n_r} < \bar{\lambda} \leq \lambda$, then any finite $\gR$-module with a slope $\bar{\lambda}$ vanishing line lies in the thick subcategory generated by $\gR/(\alpha_1, \ldots, \alpha_{r})$. In particular:
\begin{enumerate}[(i)]
\item\label{it:STQuantisation:1} Finite $\gR$-modules with a slope $\bar{\lambda}$ vanishing line have a slope $\lambda$ vanishing line.

\item\label{it:STQuantisation:2} $\gR/(\alpha_1, \ldots, \alpha_{r})$ generates the same thick subcategory as $\mathcal{A}_\lambda^f$, so $L_\lambda^f(-)$ agrees with Bousfield localisation away from $\gR/(\alpha_1, \ldots, \alpha_{r})$
\end{enumerate}
\end{thm}

\begin{proof}
Let $\gT \in \mathcal{A}^f_{\bar{\lambda}}$, and consider the finite $\gR$-module
$$\gT/(\alpha_1, \ldots, \alpha_r) := \gT \otimes_\gR \gR/(\alpha_1, \ldots, \alpha_r).$$
As $\gT$ has a slope $\bar{\lambda}$ vanishing line it follows that for each $i-1 \leq r$ the $\gR$-module $\gT/(\alpha_1, \ldots, \alpha_{i-1})$ has a slope $\bar{\lambda}$ vanishing line, as it is in the thick subcategory generated by $\gT$. Thus
$$\End_\gR(\gT/(\alpha_1, \ldots, \alpha_{i-1})) \simeq \gT/(\alpha_1, \ldots, \alpha_{i-1}) \otimes_\gR D^r(\gT/(\alpha_1, \ldots, \alpha_{i-1})) \in \ \mathcal{A}^f_{\bar{\lambda}}$$
too, as $D^r(\gT/(\alpha_1, \ldots, \alpha_{i-1}))$ is a finite $\gR$-module. As $\tfrac{d_i}{n_i} < \lambda$ it follows that the
$$\alpha_i \in \pi_{n_i, d_i}(\End_\gR(\gT/(\alpha_1, \ldots, \alpha_{i-1})))$$
are nilpotent for all $i \leq r$. Applying Lemma \ref{lem:ThickSubcatFacts} (\ref{it:ThickSubcatFacts:3}) it follows for each $i \leq r$ that $\gT/(\alpha_1, \ldots, \alpha_{i-1})$ is in the thick subcategory generated by $\gT/(\alpha_1, \ldots, \alpha_{i})$, so in particular $\gT$ lies in the thick subcategory generated by $\gT/(\alpha_1, \ldots, \alpha_i)$. 

On the other hand as $\gT$ is finite, i.e.\ in the thick subcategory of the monoidal unit, $\gT/(\alpha_1, \ldots, \alpha_i) = \gT \otimes_\gR \gR/(\alpha_1, \ldots, \alpha_r)$ lies in the thick subcategory generated by $\gR/(\alpha_1, \ldots, \alpha_i)$.
\end{proof}

\begin{cor}
If $\gR/(\alpha_1, \ldots, \alpha_{r})$ is an admissible Smith--Toda complex with a slope $\lambda$ vanishing line and $\bar{\lambda}$ satisfies $\tfrac{d_r}{n_r} < \bar{\lambda} \leq \lambda$, then $L^f_\lambda(-) = L^f_{\bar{\lambda}}(-)$.
\end{cor}
\begin{proof}
By Theorem \ref{thm:STQuantisation} the inclusion $\mathcal{A}^f_{\lambda} \subseteq \mathcal{A}^f_{\bar{\lambda}}$ is in fact an equality.
\end{proof}

\subsection{An explicit model for localisation}\label{sec:OrthCalcModel}

A consequence of Theorem \ref{thm:STQuantisation} is a somewhat explicit formula for $L^f_{\lambda}$. To give the formula, let $\gR/(\alpha_1, \ldots, \alpha_{r})$ be an admissible Smith--Toda complex with a slope $\lambda$ vanishing line. By admissibility the bidegrees $|\alpha_j|=(n_j, d_j)$ satisfy $\tfrac{d_j}{n_j} < \lambda$, but we will request the stronger property
\begin{equation}\label{eq:Wide}
\min_{\emptyset \neq S \subseteq \{1,2,\ldots, r\}} \left [ 1+ \sum_{s \in S} (\lambda n_j - d_j -1) \right] > 0.
\end{equation}
When we show, in Section \ref{sec:HigherStabMaps}, that admissible Smith--Toda complexes exist over fields of positive characteristic, we will also be able to arrange this property.

By Theorem \ref{thm:STQuantisation} (\ref{it:STQuantisation:2}) the localisation $L^f_{\lambda}(-)$ agrees with Bousfield localisation away from the single compact object $\gR/(\alpha_1, \ldots, \alpha_{r})$, and we will show that this localisation can be constructed by the following recipe. Define an $\gR$-module $\gF$ by the fibre sequence
$$\gF \lra \gR \lra \gR/(\alpha_1, \ldots, \alpha_{r})$$
where the right-hand map is the inclusion of the bottom cell. Mapping out of this cofibre sequence to some $\gR$-module $\gM$ gives a fibre sequence
$$\underline{\map}^r_\gR(\gR/(\alpha_1, \ldots, \alpha_{r}), \gM) \lra \gM \overset{\epsilon_\gM}\lra \underline{\map}^r_\gR(\gF, \gM) =: \tau(\gM).$$
Tautologically, the $\gR$-module $\gM$ is $\gR/(\alpha_1, \ldots, \alpha_{r})$-local if and only if the map $\epsilon_\gM : \gM \to \tau(\gM)$ is an equivalence. We define 
$$T_\lambda^f(\gM) := \colim(\gM \overset{\epsilon_\gM}\to \tau(\gM) \overset{\epsilon_{\tau(\gM)}}\to \tau^2(\gM) \overset{\epsilon_{\tau^2(\gM)}}\to \cdots).$$

\begin{prop}\label{prop:NullificationModel}
The map $\gM \to T_\lambda^f(\gM)$ is equivalent to Bousfield localisation away from $\gR/(\alpha_1, \ldots, \alpha_{r})$, so is equivalent to the map $\gM \to L_\lambda^f(\gM)$.
\end{prop}

\begin{rem}
The construction of the functor $T_\lambda^f(-)$ is deliberately reminiscent of Weiss' definition of the Taylor approximation functor $T_n(-)$ in orthogonal calculus \cite{Weiss, WeissErratum}. Indeed, orthogonal calculus (at least when valued in a stable category) may be set up in a manner which is highly parallel to the discussion we are giving here: we hope to explain this elsewhere.
\end{rem}

Before giving the proof of Proposition \ref{prop:NullificationModel}, we establish a basic lemma.

\begin{lem}\label{lem:LKillsSlopeLambda}
If $\pi_{n,d}(\gM)=0$ for $d < \lambda n + \kappa$ then $T_\lambda^f(\gM)=0$.
\end{lem}

\begin{proof}
The defining cofibre sequence $\gF \to \gR \to \gR/(\alpha_1, \ldots, \alpha_{r})$ and the standard $\gR$-module cell structure on $\gR/(\alpha_1, \ldots, \alpha_{r})$ shows that $\gF$ has a finite filtration with associated graded given by
$$\bigoplus_{\emptyset \neq S \subseteq \{1,2,\ldots, r\}} S^{\sum_{j \in S} n_j, -1+\sum_{j \in S} (d_j+1)} \otimes \gR.$$
Thus $\tau(\gM) = \underline{\map}^r_{\gR}(\gF, \gM)$ has a finite filtration with associated graded
$$\bigoplus_{\emptyset \neq S \subseteq \{1,2,\ldots, r\}} S^{-\sum_{j \in S} n_j, 1-\sum_{j \in S} (d_j+1)} \otimes \gM,$$
so $\pi_{n,d}(\tau(\gM))=0$ as long as
$$d < \lambda n + \kappa + \min_{\emptyset \neq S \subseteq \{1,2,\ldots, r\}} \left [ 1+ \sum_{s \in S} (\lambda n_j - d_j -1)  \right].$$
Writing $\epsilon \in \bR$ for this minimum, the assumption \eqref{eq:Wide} says that $\epsilon>0$. In other words the above shows that $\pi_{n,d}(\tau(\gM))=0$ for $d < \lambda n + (\kappa+\epsilon)$ with $\epsilon>0$. Iterating this, it follows that $\pi_{n,d}(\tau^s(\gM))=0$ for $d < \lambda n + (\kappa+s \epsilon)$, and as $\epsilon>0$ taking the colimit gives that $\pi_{n,d}(T_\lambda^f(\gM))=0$ for all $n$ and $d$.
\end{proof}

\begin{proof}[Proof of Proposition \ref{prop:NullificationModel}]
The fibre of $\epsilon_{\tau^s(\gM)} : \tau^s(\gM) \to \tau^{s+1}(\gM)$ is 
\begin{equation}\label{eq:fib}
\underline{\map}^r_\gR(\gR/(\alpha_1, \ldots, \alpha_{r}), \tau^s(\gM)) \simeq  \tau^s(\gM) \otimes_\gR D^r(\gR/(\alpha_1, \ldots, \alpha_{r})).
\end{equation}
By Lemma \ref{lem:DualInThick} the right dual $D^r(\gR/(\alpha_1, \ldots, \alpha_{r}))$ is in the thick subcategory generated by $\gR/(\alpha_1, \ldots, \alpha_{r})$ so is $\gR/(\alpha_1, \ldots, \alpha_{r})$-acyclic, and hence by the proof of Proposition \ref{prop:Localisation} (\ref{it:Localisation2}) the object \eqref{eq:fib} is $\gR/(\alpha_1, \ldots, \alpha_{r})$-acyclic too. Thus the fibre of $\gM \to T_\lambda^f(\gM)$ is also $\gR/(\alpha_1, \ldots, \alpha_{r})$-acyclic.

It remains to show that $T_\lambda^f(\gM)$ is $\gR/(\alpha_1, \ldots, \alpha_{r})$-local. For this we follow Weiss \cite{WeissErratum} extremely closely. We must show that the map
$$\epsilon_{T_\lambda^f(\gM)} : T_\lambda^f(\gM) \lra \tau(T_\lambda^f(\gM)) = \underline{\map}^r_\gR(\gF, T_\lambda^f(\gM))$$
is an equivalence. As $\gF$ is a finite $\gR$-module the functor $\tau(-)$ preserves filtered colimits, so $\epsilon_{T_\lambda^f(\gM)}$ is the map induced on horizontal colimits in the diagram
\begin{equation*}
\begin{tikzcd}
\gM \rar{\epsilon_{\gM}} \dar{\epsilon_{\gM}} & \tau(\gM) \rar{\epsilon_{\tau(\gM)}} \dar{\epsilon_{\tau(\gM)}} & \tau^2(\gM) \rar{\epsilon_{\tau^2(\gM)}} \dar{\epsilon_{\tau^2(\gM)}}& \cdots \rar & T_\lambda^f(\gM) \dar{\epsilon_{T_\lambda^f(\gM)}}\\
\tau(\gM) \rar{\tau(\epsilon_{\gM})} & \tau^2(\gM) \rar{\tau(\epsilon_{\tau(\gM)})} & \tau^3(\gM) \rar{\tau(\epsilon_{\tau^2(\gM)})} & \cdots \rar & \tau(T_\lambda^f(\gM))
\end{tikzcd}
\end{equation*}

We factorise the map from the $s$-th column to the last as
\begin{equation*}
\begin{tikzcd}
 \underline{\map}^r_\gR(\gR, \tau^s(\gM)) \rar{T_\lambda^f} \dar{\epsilon_{\tau^s(\gM)}}& \underline{\map}^r_\gR(T_\lambda^f(\gR), T_\lambda^f(\tau^s(\gM))) \rar \dar & \underline{\map}^r_{\gR}(\gR, T_\lambda^f(\gM)) \dar{\epsilon_{T_\lambda^f(\gM)}} \\
 \underline{\map}^r_\gR(\gF, \tau^s(\gM)) \rar{T_\lambda^f} & \underline{\map}^r_\gR(T_\lambda^f(\gF), T_\lambda^f(\tau^s(\gM))) \rar & \underline{\map}^r_{\gR}(\gF, T_\lambda^f(\gM)) 
\end{tikzcd}
\end{equation*}
where the middle vertical map is given by precomposition with $T_\lambda^f(\gF) \to T_\lambda^f(\gR)$, and the right-hand horizontal maps involve the identity $T_\lambda^f(\tau^s(\gM)) \simeq T_\lambda^f(\gM)$ coming from cofinality of the colimits defining them.

Now the functor $T_\lambda^f(-)$ is obtained by taking a filtered colimit of maps from finite $\gR$-modules, so it preserves cofibre sequences. But $\gR/(\alpha_1, \ldots, \alpha_{r})$ has a slope $\lambda$ vanishing line, so by Lemma \ref{lem:LKillsSlopeLambda} we have $T_\lambda^f(\gR/(\alpha_1, \ldots, \alpha_{r}))\simeq 0$ and therefore $T_\lambda^f(\gF) \overset{\sim}\to T_\lambda^f(\gR)$. Hence the middle vertical map is an equivalence.

The right-hand vertical map is the colimit of the left-hand vertical maps as $s \to \infty$: we have shown that this map is a retract of a colimit of equivalences, so is an equivalence.
\end{proof}

\begin{rem}
The proof shows that rather than the full strength of Lemma \ref{lem:LKillsSlopeLambda} we only use $T_\lambda^f(\gR/(\alpha_1, \ldots, \alpha_{r}))\simeq 0$. This is (the stable analogue of) the notion of a ``tidy map'' of Anel--Biedermann--Finster--Joyal \cite{ABFJtidy}, who are also inspired by Weiss' development of orthogonal calculus.
\end{rem}

\subsection{A ``telescope conjecture''}\label{sec:TelescopeConj}

In Section \ref{sec:localisation} we might instead have chosen to Bousfield localise away from
$$\mathcal{A}_\lambda := \{\text{$\gR$-modules with a slope $\lambda$ vanishing line}\},$$
i.e.\ remove the assumption that the modules be finite. Such a localisation still exists, but it can no longer be obtained by the telescope construction of Proposition \ref{prop:LocExists}. Instead, one may use that $\gR\text{-}\mathsf{mod}$ is presentable, so that by \cite[Proposition 5.5.4.15]{HTT} the inclusion
$\mathcal{A}_\lambda\text{-}\mathsf{loc}\subseteq\gR\text{-}\mathsf{mod}$ of the full subcategory of $\mathcal{A}_\lambda$-local objects has a left adjoint $L_\lambda : \gR\text{-}\mathsf{mod} \to \mathcal{A}_\lambda\text{-}\mathsf{loc} \subseteq \gR\text{-}\mathsf{mod}$. As $\mathcal{A}_\lambda$-local objects are in particular $\mathcal{A}^f_\lambda$-local, there is a comparison map
\begin{equation}\label{eq:TelConj}
L^f_\lambda(\gM) \lra L_\lambda(\gM).
\end{equation}

\begin{thm}[``Telescope conjecture'']
If an admissible Smith--Toda complex with a slope $\lambda$ vanishing line and satisfying \eqref{eq:Wide} exists, then \eqref{eq:TelConj} is an equivalence.
\end{thm}
\begin{proof}
It suffices to show that all $\gT \in \mathcal{A}_\lambda$ are $\mathcal{A}^f_\lambda$-acyclic. If  $\gR/(\alpha_1, \ldots, \alpha_r)$ is such a Smith--Toda complex, then Proposition \ref{prop:NullificationModel} shows that $L^f_\lambda(-)$ is given by the construction $T_\lambda^f(-)$ formed using this admissible Smith--Toda complex. But if $\gT \in \mathcal{A}_\lambda$ then  Lemma \ref{lem:LKillsSlopeLambda} shows that $T_\lambda^f(\gT)=0$, so $L^f_\lambda(\gT)=0$ and so $\gT$ is indeed $\mathcal{A}^f_\lambda$-acyclic.
\end{proof}

\subsection{Monochromatic layers}

In Section \ref{sec:AdamsPeriodicity} we have already discussed a special case of the $(\lambda, \lambda')$-monochromatic objects
$$M^f_{\lambda, \lambda'}(\gM) := \mathrm{fib}(L^f_{\lambda}(\gM) \to L^f_{\lambda'}(\gM)).$$
Let $\mathsf{Mono}_{\lambda, \lambda'} \subset \mathbf{\gR}\text{-}\mathsf{mod}$ denote the full subcategory spanned by such objects: more conceptually, it is the full subcategory on objects which are $\mathcal{A}_\lambda$-local and $\mathcal{A}_{\lambda'}$-acyclic. 

We briefly outline one way of looking at this category, omitting the details as we have not yet found a use for this perspective. If $\gR/(\beta_1, \ldots, \beta_s)$ is an admissible Smith--Toda complex with a slope $\lambda'$ vanishing line, then $L^f_{\lambda}(\gR/(\beta_1, \ldots, \beta_s))$ is $(\lambda, \lambda')$-monochromatic, and $\{S^{n,0} \otimes L^f_{\lambda}(\gR/(\beta_1, \ldots, \beta_s))\}_{n \in \bZ}$ forms a set of compact generators for $\mathsf{Mono}_{\lambda, \lambda'}$. Setting 
$$\gE := \End_\gR(L^f_{\lambda}(\gR/(\beta_1, \ldots, \beta_s))),$$
an $E_1$-algebra in $\gR\text{-}\mathsf{mod}$ and hence in $\mathsf{D}(\bk)^\bZ$, it follows from an appropriate form of the Schwede--Shipley theorem that there is an equivalence of categories
\begin{align*}
\mathsf{Mono}_{\lambda, \lambda'} \overset{\sim}\lra & \, \gE\text{-}\mathsf{mod}\\
\gM \longmapsto & \, \mathrm{map}_\gR(L^f_{\lambda}(\gR/(\beta_1, \ldots, \beta_s)), \gM)\\
L^f_{\lambda}(\gR/(\beta_1, \ldots, \beta_s)) \otimes_\gE \gX \longmapsfrom & \,\gX.
\end{align*}

\section{Existence of Smith--Toda complexes in positive characteristic}\label{sec:HigherStabMaps}

Our goal in this section is to prove that, as long as work over a field of positive characteristic, there exist admissible Smith--Toda complexes having vanishing lines of slopes arbitrarily close to 1. A precise statement is as follows:

\begin{thm}\label{thm:STExist}
Let $\bk$ be a field of positive characteristic $p$, and let $\gR \in \mathsf{Alg}_{E_2}(\mathsf{D}(\bk)^\bZ)$ satisfy (C), (SCE), and (F). For any given $\lambda < 1$, there exists an admissible Smith--Toda complex $\gR/(\alpha_1, \ldots, \alpha_r)$ having a vanishing line of slope $\lambda$. Furthermore:
\begin{enumerate}[(i)]
\item\label{it:STExist1} The slopes $\tfrac{d_i}{n_i}$ of the $\alpha_i$ are all of the form $\tfrac{k}{k+1}$ for $k \in \bN$.

\item\label{it:STExist2}  It in fact has a vanishing line of slope $\min\{\tfrac{k}{k+1} \, | \, k \in \bN \text{ with } \lambda \leq \tfrac{k}{k+1}\}$.

\item\label{it:STExist3}  If $W \in \bR$ is pre-specified, then the $\alpha_i$ may be chosen such that $\lambda n_i -d_i - 1 \geq W$ for $i=1,\ldots,r$.
\end{enumerate}
\end{thm}

This proves Theorem \ref{MainThm:A} (\ref{it:MainThm:A:1}) and (\ref{it:MainThm:A:2}). The role of Theorem \ref{thm:STExist} (\ref{it:STExist3}) is to ensure that the quantity ``$b$'' may be made arbitrarily large in Theorem \ref{thm:AdamsPeriodicityGeneral}, and also that the property \eqref{eq:Wide} may be achieved in Section \ref{sec:OrthCalcModel}.

In order to prove Theorem \ref{thm:STExist} we must construct endomorphisms $\alpha_i$ of inductively-defined $\gR$-modules $\gR/(\alpha_1, \ldots, \alpha_{i-1})$, and a basic difficulty is saying what these endomorphisms are supposed to achieve. To address this, we will consider the canonical multiplicative filtration $\fil_*\gR$ of $\gR$, and prove the analogue (Theorem \ref{thm:FiltSTExist}) of Theorem \ref{thm:STExist} in the category of $\fil_*\gR$-modules. The advantage of working in this category is that we can use the associated graded to describe what it is that a single endomorphism is supposed to achieve, which can be formulated simply because the associated graded of $\fil_*\gR$ is a free $E_2$-algebra and so the structure of its homotopy is completely understood by the work of F.~Cohen. In other words, $\pi_{*,*,*}(C\tau \otimes -)$ defines a homology theory on $\fil_*\gR$-modules, whose coefficients $\pi_{*,*,*}(C\tau \otimes \fil_*\gR)$ have a quite simple structure.

\subsection{The elements $x_i$}\label{sec:ElementsXi}

Recall from Section \ref{sec:CanMultFilt} that we write $\fil_*\gR = \fil_*^{E_2}\gR$ for the canonical multiplicative filtration of the augmented $E_2$-algebra $\gR$. Theorem \ref{thm:STExist} will follow from a more precise version in the category of $\fil_*\gR$-modules, which we shall formulate shortly. A distinguished role will be played by certain homotopy classes on the associated graded of $\fil_*\gR$, which we first describe.

Recall that we write $\gI := \mathrm{fib}(\epsilon : \gR \to \bk)$ for the augmentation ideal. There is an equivalence
$$C\tau \otimes \fil_*\gR \simeq \gE_2((-1)_* Q^{E_2^\text{nu}}\gI)$$
and, being a free $E_2$-algebra over $\bk$, the homotopy groups of this object may be completely described in terms of $\pi_{*,*}(Q^{E_2^\text{nu}}\gI) =: H_{*,*}^{E_2}(\gI)$ using the work of F.\ Cohen \cite[Chapter 3]{CLM}\footnote{Cohen considers $\bF_p$-coefficients whereas we wish to consider a general field $\bk$ of positive characteristic $p$. But as $\bk \otimes_{\bF_p} - : \mathsf{D}(\bF_p) \to \mathsf{D}(\bk)$ is essentially surjective, any free $E_k$-algebra in $\mathsf{D}(\bk)^\bZ$ is base-changed from one in $\mathsf{D}(\bF_p)^\bZ$, so by the K{\"u}nneth theorem Cohen's description of its homotopy continues to hold over $\bk$. As free algebras describe the available homology operations, it follows that $E_k$-algebras over $\bk$ have the same operations as those over $\bF_p$. There are some new features, for example Dyer--Lashof operations are only Frobenius-linear, but this does not affect the description of free $E_k$-algebras.}. We will have to assume familiarity with his results; they are too intricate to give a full description here, so we refer to \cite[Section 16]{e2cellsI} for an exposition. The description is
$$\pi_{*,*, *}(C\tau \otimes \mathrm{fil}_*\gR) \cong W_1((-1)_* H_{*,*}^{E_2}(\gI)),$$
the free $W_1$-algebra on the $E_2$-homology of $\gI$, placed in additional grading $-1$. Here $W_1(-)$ is an explicitly described monad, which is recalled in \cite[Section 16]{e2cellsI}. For our purposes the following features suffice: 
\begin{enumerate}[(i)]
\item $W_1((-1)_* H_{*,*}^{E_2}(\gI))$ is a free graded-commutative $\bk$-algebra.

\item Using the equivalences $\bk \otimes_\gR \bk \simeq B^{E_1}(\gR) \simeq \bk \oplus \Sigma Q^{E_1^\text{nu}}\gI$ from \eqref{eq:BarIsIndec}, axiom (SCE) is equivalent to $H^{E_1}_{n,d}(\gI)=0$ for $d < n-1$, and by \cite[Theorem 14.4]{e2cellsI} this implies that $H^{E_2}_{n,d}(\gI)=0$ for $d < n-1$ too. (In fact, by \cite[Theorem 14.6]{e2cellsI} the converse holds too.) Axiom (F) is equivalent to the $H^{E_1}_{n,n-1}(\gI)$ being finite-dimensional, and it follows from the bar spectral sequence as in the proof of \cite[Theorem 14.4]{e2cellsI} that the $H^{E_2}_{n,n-1}(\gI)$ are also finite-dimensional.

It is an elementary but laborious calculation using the description of a basis for $W_1((-1)_* H_{*,*}^{E_2}(\gI))$ in \cite[Section 16.2]{e2cellsI} that the free graded-commutative algebra generators in tridegrees $(n,d,f)$ having $d < n$ are given by the free restricted $\lambda_1$-algebra 
$$V_{*,*,*} := L_1((-1)_* \Big(\bigoplus_{n \geq 1} H_{n,n-1}^{E_2}(\gI)\Big))$$
on the trigraded vector space $(-1)_*(\bigoplus_{n \geq 1} H_{n,n-1}^{E_2}(\gI))$. As the $H^{E_2}_{n,n-1}(\gI)$ are finite-dimensional, $V_{*,*,*}$ has finite type. 

\end{enumerate}

\begin{defn}\label{defn:X}
If $p=2$ let $X_{*,*,*}=V_{*,*,*}$. If $p$ is odd, let $X_{*,*,*} \subseteq V_{*,*,*}$ be the subspace of those elements of even homological degree. 
\end{defn}

In either case, $X_{*,*,*}$ consists of the generators of the free graded-commutative algebra $\pi_{*,*,*}(C\tau \otimes \fil_*\gR)$ of tridegrees $(n,d,f)$ satisfying $d < n$ and which are not nilpotent. The tridegrees of these generators are constrained as follows:

\begin{lem}\label{lem:XOrdering}\mbox{}
\begin{enumerate}[(i)]
\item\label{it:XOrdering:1} If $p=2$ then $X_{N,D,F}$ vanishes unless $D=N-1$. If $p$ is odd then $X_{N,D,F}$ vanishes unless $D=N-1$, or $D=N-2$ and $N$ is even. In either case, if $X_{N,D,F}$ is non-zero then $\tfrac{D}{N}$ is a rational number of the form $\tfrac{k}{k+1}$.

\item\label{it:XOrdering:2} For each fixed $\lambda$ and $\mu$ there are finitely-many tuples $(N,D,F)$ having $\tfrac{D}{N} = \lambda$, $\tfrac{F}{D}=\mu$, and $X_{N,D,F}$ non-zero.
\end{enumerate}
\end{lem}
\begin{proof}
To prove (\ref{it:XOrdering:1}) we get into the details of what a free restricted $\lambda_1$-algebra is.

\vspace{1ex}

\noindent\textbf{Even characteristic.} In this case a restricted $\lambda_1$-algebra is a restricted Lie algebra, so that
$$V_{*,*,*} = \mathrm{Lie}^\text{res}((-1)_*\Big(\bigoplus_{n \geq 1} H_{n,n-1}^{E_2}(\gI)\Big)).$$ 
As such, $V_{*,*,*} = X_{*,*,*}$ is supported in tridegrees $(n,d,f)$ having $d=n-1$. (This is because $\bigoplus_{n \geq 1} H_{n,n-1}^{E_2}(\gI)$ is supported in such degrees, and the bracket satisfies $|[x,y]| = |x|+|y| + (0,1,0)$ and the restriction satisfies $|\xi(x)| = 2|x|+(0,1,0)$.)

\vspace{1ex}

\noindent\textbf{Odd characteristic.} In this case a restricted $\lambda_1$-algebra is a restricted Lie algebra with a further operation $\zeta$, so that
$$V_{*,*,*} = \mathrm{Lie}^{res}((-1)_*\Big(\bigoplus_{n \geq 1} H_{n,n-1}^{E_2}(\gI)\Big)) \oplus \zeta \cdot \mathrm{Lie}^{res}((-1)_*\Big(\bigoplus_{n \geq 1} H_{n,n-1}^{E_2}(\gI)\Big))_\text{odd}.$$
The first summand is supported in tridegrees $(n,d,f)$ with $d=n-1$, by the same argument as above.  The second summmand is supported in tridegrees with $d=n-2$ and $d$ even (so also $n$ is even). (This is because $|\zeta(x)| = p|x| + (0,p-2,0)$, so if $|x| = (n,n-1,f)$ with $n-1$ odd (so $n$ even) then $|\zeta(x)| = (pn, p(n-1) + p-2, pf) = (pn, pn-2, pf)$, and $pn-2$ is even.) Thus $X_{*,*,*}$ is also supported in such degrees.

\vspace{1ex}

We now prove (\ref{it:XOrdering:2}). In even characteristic if $X_{N,D,F}$ is non-zero then we have seen that $D=N-1$ in which case $\lambda = \tfrac{D}{N} = \tfrac{N-1}{N}$ determines $N$, and then $\mu = \tfrac{F}{N}$ determines $F$: in this case there is a single possibility for $(N,D,F)$. In odd characteristic  if $X_{N,D,F}$ is non-zero then we have seen that either $D=N-1$, in which case $\lambda = \tfrac{N-1}{N}$ determines $N$, or else $D=N-2$ with $N$ even, in which case $\lambda = \tfrac{N-2}{N}$ determines $N$. So there are two possibilities for the pair $(N,D)$, and for each of these $\mu = \tfrac{F}{N}$ determines $F$: in this case there are two possibilities for $(N,D,F)$.
\end{proof}

We define
$$X_{\lambda, \mu} := \bigoplus_{\substack{D/N=\lambda \\ F/D=\mu}} X_{N,D,F},$$
and totally order these vector spaces lexicographically by
\begin{enumerate}[(i)]
\item increasing $\lambda$, and then by

\item increasing $\mu$. (Bear in mind that $\mu$ is negative.)
\end{enumerate}
We now choose bases for each $X_{N,D,F}$, which are all finite-dimensional, and so (by ordering the summands anyhow) obtain bases for the $X_{\lambda, \mu}$, which are also finite-dimensional by Lemma \ref{lem:XOrdering} (\ref{it:XOrdering:2}). By concatenating these bases we obtain a totally ordered homogeneous basis
$$x_1, x_2, x_3, \ldots$$
for $X_{*,*,*}$ subordinate to the total ordering of the $X_{\lambda, \mu}$. We write
$$|x_i| = (n_i, d_i, f_i)$$
for their tridegrees, so that these elements define homotopy classes
$$x_i \in X_{n_i, d_i, f_i} \subseteq V_{n_i, d_i, f_i} \subseteq \pi_{n_i, d_i, f_i}(C\tau \otimes \fil_*\gR).$$
As a matter of language, we will say that $x_i$ has \emph{slopes} $(\tfrac{d_i}{n_i}, \tfrac{f_i}{d_i})$. More generally, we order slopes lexicographically, so say
$$\text{$x_i$ has slopes $< (\lambda, \mu)$} \Longleftrightarrow \tfrac{d_i}{n_i} < \lambda, \text{ or } \tfrac{d_i}{n_i} = \lambda \text{ and } \tfrac{f_i}{d_i} < \mu$$
and analogously for slopes $\leq (\lambda, \mu)$.

\begin{rem}
We will explain in Section \ref{sec:EkAlgBigk} some simplifications to this discussion which arise if $\gR$ is an $E_k$-algebra with $k \geq 3$ and we use the filtered $E_k$-algebra $\fil_*^{E_k}\gR$ instead.
\end{rem}

\subsection{The filtered form of Theorem \ref{thm:STExist}}

 As $\fil_*\gR$ is an $E_2$-algebra, for any $\fil_*\gR$-module $\fil_*\gM$ there is a unit map
$$u : \fil_*\gR \lra \End_{\fil_*\gR}(\fil_*\gM)$$
landing in the centre.

\begin{defn}\label{defn:xiselfmap}
A filtered map
$$\phi_i : S^{p^M n_i, p^M d_i, p^M f_i} \otimes \mathrm{fil}_*\gM \lra \mathrm{fil}_*\gM$$
is an \emph{$x_i$ self-map} if $C\tau \otimes {\phi_i} = (C\tau \otimes u)_*(x_i^{p^M})$ for some $M$.
\end{defn}

That is, $\phi_i$ is an $x_i$ self-map if it induces multiplication by a $p$-th power of $x_i$ on the associated graded. Using this notion we may formulate a---more precise---filtered form of Theorem \ref{thm:STExist} as follows.

\begin{thm}\label{thm:FiltSTExist}
Let $\bk$ be a field of positive characteristic $p$, and let $\gR \in \mathsf{Alg}_{E_2}(\mathsf{D}(\bk)^\bZ)$ satisfy (C), (SCE), and (F). Then there exists a sequence of endomorphisms
$$\phi_i : S^{p^{M_i} n_i, p^{M_i} d_i, p^{M_i} f_i} \otimes \fil_*\gR/(\phi_1, \ldots, \phi_{i-1}) \lra \fil_*\gR/(\phi_1, \ldots, \phi_{i-1})$$
for $i=1,2,\ldots$ such that:
\begin{enumerate}[(i)]
\item\label{it:FiltSTExist:1} $\phi_i$ is an $x_i$ self-map.

\item\label{it:FiltSTExist:2} $\tau^{-1} \fil_*\gR/(\phi_1, \ldots, \phi_{i})$ has a vanishing line of slope $\tfrac{d_{i+1}}{n_{i+1}}$.

\item\label{it:FiltSTExist:3} If $M$ is pre-specified then the $\phi_i$ may be chosen so that $M_i \geq M$ for all $i$.
\end{enumerate}
\end{thm}

We first explain how this implies Theorem \ref{thm:STExist}.

\begin{proof}[Proof of Theorem \ref{thm:STExist}]
Given $\lambda < 1$ let $x_1, \ldots, x_r$ be the set of all $x_i$'s with $\tfrac{d_i}{n_i}<\lambda$. Given a $W \in \bR$, choose $M$ large enough that $\lambda p^M n_i - p^M d_i-1 \geq W$ for all $i=1, \ldots, r$, which is possible as $\tfrac{d_i}{n_i} < \lambda$. Let $\phi_1, \ldots, \phi_r$ be the sequence of endomorphisms provided by Theorem \ref{thm:FiltSTExist}, with $M_i \geq M$, and set $\alpha_i := \tau^{-1}\phi_i$, an endomorphism of $\gR/(\alpha_1, \ldots, \alpha_{i-1})$. These endomorphisms have slopes $\tfrac{d_i}{n_i}$ which are all rational numbers of the form $\tfrac{k}{k+1}$ by Lemma \ref{lem:XOrdering} (\ref{it:XOrdering:1}). By Theorem \ref{thm:FiltSTExist} (\ref{it:FiltSTExist:2}) the object $\gR/(\alpha_1, \ldots, \alpha_r) = \tau^{-1} \fil_*\gR/(\phi_1, \ldots, \phi_r)$ has a vanishing line of slope $\tfrac{d_{r+1}}{n_{r+1}}$, which is $\geq \lambda$ by definition of $r$. Thus $\gR/(\alpha_1, \ldots, \alpha_r)$ is an admissible Smith--Toda complex satisfying properties (\ref{it:STExist1})--(\ref{it:STExist3}) of Theorem \ref{thm:STExist}.
\end{proof}

\subsection{Proof of Theorem \ref{thm:FiltSTExist}}

We suppose, for an induction, that a sequence of $x_j$ self-maps $\phi_j$ has been defined for $j<i$, so that $\mathrm{fil}_*\gR/({\phi}_1, \ldots, {\phi}_{i-1})$ has been constructed, and we consider its internal endomorphism object
$$\End_{\mathrm{fil}_*\gR}(\mathrm{fil}_*\gR/(\phi_1, \ldots, \phi_{i-1})).$$
This is a filtered object, and yields a multiplicative spectral sequence of signature
\begin{align*}
& E^1_{n,d,f} = \pi_{n,d,f}(C\tau \otimes \End_{\mathrm{fil}_*\gR}(\mathrm{fil}_*\gR/(\phi_1, \ldots, \phi_{i-1})))\\
& \quad\quad\quad\quad \Longrightarrow \pi_{n,d}(\End_{\gR}(\tau^{-1}\fil_*\gR/(\phi_1, \ldots, \phi_{i-1}))).
\end{align*}
We begin by establishing some vanishing estimates for $\pi_{*,*,*}(C\tau \otimes \mathrm{fil}_*\gR/({\phi}_1, \ldots, {\phi}_{i-1}))$ and for this $E^1$-page.

\begin{lem}\label{lem:VanishingEstimates}
\mbox{}
\begin{enumerate}[(i)]
\item\label{it:VanishingEstimates:1} There is a $\kappa$ such that $\pi_{n,d,f}(C\tau \otimes \mathrm{fil}_*\gR/(\phi_1, \ldots, \phi_{i-1}))=0$ for $d < \tfrac{d_i}{n_i}n+\kappa$.

\item\label{it:VanishingEstimates:2} For each $\tilde{\kappa}$ there is a $\rho$ such that $\pi_{n,d,f}(C\tau \otimes \mathrm{fil}_*\gR/(\phi_1, \ldots, \phi_{i-1}))=0$ whenever $d = \tfrac{d_i}{n_i} n + \tilde{\kappa}$ and $f < \tfrac{f_i}{d_i} d+ \rho$.

\item\label{it:VanishingEstimates:3} There is an $r_0$ such that $E^1_{p^Nn_i,p^Nd_i-1,p^Nf_i-r}=0$ for all $r \geq r_0$ and all $N$.
\end{enumerate}
\end{lem}
\begin{proof}
As we have discussed in Section \ref{sec:ElementsXi}, $\pi_{*,*,*}(C\tau \otimes \mathrm{fil}_*\gR)$ is a free graded-commutative $\bk$-algebra, and its free generators can be written as $V_{*,*,*} \oplus W'_{*,*,*}$ with $W'_{*,*,*}$ being supported in tridegrees $(n,d,f)$ with $d \geq n$, and $V_{*,*,*} = U_{*,*,*} \oplus X_{*,*,*}$ where $X_{*,*,*}$ is as in Definition \ref{defn:X}, and therefore $U_{*,*,*}$ are the multiplicative generators supported in tridegrees $(n,d,f)$ with $d < n$ which are nilpotent (i.e.\ those of odd homological degree when $p$ is odd, and zero if $p=2$). Let $x_i, \ldots, x_j$ be the $x_k$'s with with $i \leq k$ and $\tfrac{d_k}{n_k} = \tfrac{d_i}{n_i}$, so we have $\tfrac{f_i}{d_i} \leq \tfrac{f_{i+1}}{d_{i+1}} \leq \cdots \leq \tfrac{f_j}{d_j}$ by definition of the total order on the $x_k$'s. Let $W_{*,*,*}$ be spanned by $W'_{*,*,*}$ together with $x_{j+1}, x_{j+2}, \ldots$. We may therefore write
$$\pi_{*,*,*}(C\tau \otimes \mathrm{fil}_*\gR) = \Lambda^*_\bk[U_{*,*,*}] \otimes \bk[x_1, \ldots, x_j] \otimes \mathrm{Sym}_\bk^*[W_{*,*,*}].$$
As $\phi_k$ induces multiplication by $x_k^{p^{M_k}}$ on associated graded, and these form a regular sequence, $\pi_{*,*,*}(C\tau \otimes \mathrm{fil}_*\gR/(\phi_1, \ldots, \phi_{i-1}))$ is given by
\begin{equation}\label{eq:AssocGr}
\Lambda^*_\bk[U_{*,*,*}] \otimes \frac{\bk[x_1, \ldots, x_{i-1}]}{(x_1^{p^{M_1}}, \ldots, x_{i-1}^{p^{M_{i-1}}})} \otimes \bk[x_i,  \ldots, x_j] \otimes \mathrm{Sym}_\bk^*[W_{*,*,*}].
\end{equation}

To prove (\ref{it:VanishingEstimates:1}), we observe that the first two factors of \eqref{eq:AssocGr} combine to give a finite-dimensional trigraded vector space, which is therefore supported in tridegrees $(n,d,f)$ with $d \geq \tfrac{d_i}{n_i}n+ \kappa$ for some $\kappa$. It suffices to show that the remaining factors are supported in slopes $\geq \tfrac{d_i}{n_i}$, but the elements $x_i, \ldots, x_j$ have slopes $\tfrac{d_i}{n_i}$, and those of $W_{*,*,*}$ have slopes $> \tfrac{d_i}{n_i}$.

To prove (\ref{it:VanishingEstimates:2}), we consider \eqref{eq:AssocGr} as describing a free $\bk[x_i, \ldots, x_j]$-module, with module generators
$$G_{*,*,*} := \Lambda^*_\bk[U_{*,*,*}] \otimes \frac{\bk[x_1, \ldots, x_{i-1}]}{(x_1^{p^{M_1}}, \ldots, x_{i-1}^{p^{M_{i-1}}})} \otimes \mathrm{Sym}_\bk^*[W_{*,*,*}].$$
As the first two factors are finite-dimensional trigraded vector spaces, and the last is a free graded-commutative algebra on elements of slope $ > \tfrac{d_i}{n_i}$, it follows that for each $\tilde{\kappa}$ the vector space $\bigoplus_{d = \tfrac{d_i}{n_i}n +\tilde{\kappa}} G_{n, d, *}$ is finite-dimensional. Thus 
$$L_{\tilde{\kappa}} := \bigoplus_{d = \tfrac{d_i}{n_i}n +\tilde{\kappa}} \pi_{n, d, *}(C\tau \otimes \mathrm{fil}_*\gR/({\phi}_1, \ldots, {\phi}_{i-1}))$$
is a finitely-generated free $\bk[x_i, \ldots, x_j]$-module, say with generators $\{z_1, \ldots, z_p\}$ having tridegrees $|z_q| = (n'_q, d'_q, f'_q)$.
 
The element $z_q x_i^{\ell_i} \cdots x_j^{\ell_j}$ has tridegree 
$$(n,d,f)=(n'_q + \sum \ell_r n_r, d'_q + \sum \ell_r d_r, f'_q + \sum \ell_r f_r)$$
and so has
$$f - \tfrac{f_i}{d_i} d = (f'_q - \tfrac{f_i}{d_i} d'_q) + \sum_{r=i}^j \ell_r(f_r - \tfrac{f_i}{d_i} d_r) \geq f'_q - \tfrac{f_i}{d_i} d'_q$$
independently of the $\ell_r$, as $f_r - \tfrac{f_i}{d_i} d_r = d_r(\tfrac{f_r}{d_r} - \tfrac{f_i}{d_i}) \geq 0$ for all $i \leq r \leq j$. Thus every element of the free module $L_{\tilde{\kappa}}$ has tridegrees $(n,d,f)$ satisfying
$$f- \tfrac{f_i}{d_i} d \geq \min_{1 \leq q \leq p}\{f'_q -\tfrac{f_i}{d_i} d'_q\} =: \rho,$$
and so $L_{\tilde{\kappa}}$ vanishes in tridegrees $(n,d,f)$ satisfying $f < \tfrac{f_i}{d_i} d + \rho$.

To prove (\ref{it:VanishingEstimates:3}), we first note that $\End_{\mathrm{fil}_*\gR}(\mathrm{fil}_*\gR/(\phi_1, \ldots, \phi_{i-1}))$ lies in the thick subcategory generated by $\mathrm{fil}_*\gR/(\phi_1, \ldots, \phi_{i-1})$, by Lemma \ref{lem:TensorAndHomStaysThick}. It follows that $C\tau \otimes  \End_{\mathrm{fil}_*\gR}(\mathrm{fil}_*\gR/(\phi_1, \ldots, \phi_{i-1}))$ lies in the thick subcategory generated by $C\tau \otimes \mathrm{fil}_*\gR/(\phi_1, \ldots, \phi_{i-1})$. The property of satisfying the vanishing conditions described in parts (\ref{it:VanishingEstimates:1}) and (\ref{it:VanishingEstimates:2})  defines a thick subcategory, so it follows from those parts that $C\tau \otimes  \End_{\mathrm{fil}_*\gR}(\mathrm{fil}_*\gR/(\phi_1, \ldots, \phi_{i-1}))$ satisfies these vanishing conditions too. Spelling out the property in (\ref{it:VanishingEstimates:2}) with $\tilde{\kappa}=-1$, it follows that there is a $\rho$ such that
$$E^1_{p^Nn_i,p^Nd_i-1,p^Nf_i-r} = \pi_{p^Nn_i,p^Nd_i-1,p^Nf_i-r}(C\tau \otimes  \End_{\mathrm{fil}_*\gR}(\mathrm{fil}_*\gR/(\phi_1, \ldots, \phi_{i-1})))$$
vanishes as long as $p^N f_i-r < \tfrac{f_i}{d_i}(p^Nd_i-1) + \rho$. In particular it vanishes for all
$$r \geq r_0 := \tfrac{f_i}{d_i}-\rho + 1 > p^N f_i - \tfrac{f_i}{d_i}(p^Nd_i-1) - \rho $$
independently of $N$.
\end{proof}

We make use of these vanishing estimates as follows. The unit map
$$u: \mathrm{fil}_*\gR \lra \End_{\mathrm{fil}_*\gR}(\mathrm{fil}_*\gR/(\phi_1, \ldots, \phi_{i-1}))$$
lands in the centre, and hence gives a map of spectral sequences of algebras landing in the centre (of each page), which on the $E^1$-page has the form
$$(C\tau \otimes u)_* : W_1((-1)_* H_{*,*}^{E_2}(\gI)) = \pi_{*,*,*}(C\tau \otimes \mathrm{fil}_*\gR) \lra E^1_{*,*,*}.$$
In particular we have a class $(C\tau \otimes u)_*(x_i) \in E^1_{n_i, d_i, f_i}$ lying in the centre.

We claim that some power $(C\tau \otimes u)_*(x_i)^{p^{M_i}}$ is a permanent cycle in the spectral sequence $\{E^r_{*,*,*}\}_r$. As it comes from a filtered $E_1$-algebra, this spectral sequence is one of associative rings, i.e.\ the differentials satisfy the Leibniz rule. Thus if a class $z \in E^r_{*,*,*}$ lies in the centre then $d^r(z^p) = p z^{p-1} d^r(z) = 0$ so $z^p$ survives until $E^{r+1}_{*,*,*}$. This may be used to find $p$-th powers of $(C\tau \otimes u)_*(x_i)$ which survive until arbitrarily late pages of the spectral sequence. In particular some $(C\tau \otimes u)_*(x_i)^{p^{M_i}}$ survives until $E^{r_0}_{p^{M_i} n_i, p^{M_i} d_i, p^{M_i} f_i}$ for the $r_0$ given by Lemma \ref{lem:VanishingEstimates} (\ref{it:VanishingEstimates:3}). By that Lemma we have $E^{r_0}_{p^{M_i} n_i, p^{M_i} d_i-1, p^{M_i} f_i-r_0}=0$ and so $d^{r_0}((C\tau \otimes u)_*(x_i)^{p^{M_i}})=0$ and $(C\tau \otimes u)_*(x_i)^{p^{M_i}}$ survives to $E^{r_0+1}_{p^{M_i} n_i, p^{M_i} d_i, p^{M_i} f_i}$. But the conclusion of Lemma \ref{lem:VanishingEstimates} (\ref{it:VanishingEstimates:3}) holds for all $r \geq r_0$, so the class $(C\tau \otimes u)_*(x_i)^{p^{M_i}}$ continues to be a cycle on all later pages, and survives to $(C\tau \otimes u)_*(x_i)^{p^{M_i}} \in E^{\infty}_{p^{M_i} n_i, p^{M_i} d_i, p^{M_i} f_i}$. It detects a filtered map
$$\phi_i : S^{p^{M_i} n_i, p^{M_i} d_i, p^{M_i} f_i} \lra \End_{\mathrm{fil}_*\gR}(\mathrm{fil}_*\gR/({\phi}_1, \ldots, {\phi}_{i-1})),$$
which by construction is an $x_i$ self-map. 

This establishes part (\ref{it:FiltSTExist:1}) of Theorem \ref{thm:FiltSTExist}. By raising $\phi_i$ to further $p$-th powers, we may assume that $M_i \geq M$, which establishes part (\ref{it:FiltSTExist:3}). For part (\ref{it:FiltSTExist:2}) we observe that Lemma \ref{lem:VanishingEstimates} (\ref{it:VanishingEstimates:1}) gives that there is a $\kappa$ such that $\pi_{n,d,f}(C\tau \otimes \mathrm{fil}_*\gR/({\phi}_1, \ldots, {\phi}_{i}))=0$ for $d < \tfrac{d_{i+1}}{n_{i+1}} n + \kappa$, so running the spectral sequence for the filtered object $\mathrm{fil}_*\gR/({\phi}_1, \ldots, {\phi}_{i})$ shows that $\pi_{n,d}(\tau^{-1}\mathrm{fil}_*\gR/({\phi}_1, \ldots, {\phi}_{i}))=0$ for $d < \tfrac{d_{i+1}}{n_{i+1}} n + \kappa$ too. This finishes the proof of Theorem \ref{thm:FiltSTExist}.

\begin{rem}\label{rem:PowerOpsInSS}
Rather than relying on centrality of $(C\tau \otimes u)_*(x_i)$ and the Leibniz rule in the spectral sequence $\{E^r_{*,*,*}\}_r$ associated to the filtered $E_1$-algebra $\End_{\mathrm{fil}_*\gR}(\mathrm{fil}_*\gR/(\phi_1, \ldots, \phi_{i-1}))$, we could instead proceed as follows. In a filtered $E_2$-algebra such as $\fil_*\gR$, if $y \in E^r_{*,*,*}$ then an element $Q^s(y)$ obtained by applying a Dyer--Lashof operation survives until $E^{pr}_{*,*,*}$ (and $d^{pr}(Q^s(y))$ is represented by $Q^s(d^r(y))$), see \cite[Theorem 16.8]{e2cellsI}. In particular, $x_i \in E^1_{n_i, d_i, f_i}(\fil_*\gR)$ so $x_i^{p} = Q^{s}(x_i)$ (with $s=d_i$ if $p=2$ and $s=d_i/2$ if $p$ is odd) survives until $E^{p}_{*,*,*}(\fil_*\gR)$, and by the same principle $(C\tau \otimes u)_*(x_i^{p^{M_i}})$ survives until $E^{p^{M_i}}_{*,*,*}$. The (only) advantage of this over the Leibniz rule method is that a lower power of $(C\tau \otimes u)_*(x_i)$ needs to be taken for it to survive to a given page.
\end{rem}

\subsection{$E_k$-algebras, $k\geq 3$}\label{sec:EkAlgBigk}

If $\gR$ is an augmented $E_k$-algebra with $k \geq 3$, then we can run the discussion and proofs of this section using the filtered $E_k$-algebra $\fil_*^{E_k} \gR$ instead, which has associated graded
$$C\tau \otimes \fil_*^{E_k} \gR \simeq \gE_k((-1)_* Q^{E_k^\text{nu}} \gI).$$
It is still the case that (SCE) implies that $H^{E_k}_{n,d}(\gI)=0$ for $d<n-1$. The work of F.\ Cohen in this case describes the homotopy groups of this object as the free $W_{k-1}$-algebra on the trigraded vector space $(-1)_*H^{E_k}_{*,*}(\gI)$, and it is again a free graded-commutative $\bk$-algebra. But now the Browder bracket or $\xi$ or $\zeta$ applied to elements of $H^{E_k}_{*,*}(\gI)$ end up in tridegrees $\pi_{n,d,f}(C\tau \otimes \fil_*^{E_k} \gR)$ with $d \geq n$, and this improves further with repeated applications. The multiplicative generators below the diagonal are therefore just those admissible iterated Dyer--Lashof operations applied to elements of $\bigoplus_{n \geq 1}(-1)_*H^{E_k}_{n,n-1}(\gI)$ which happen to lie below the diagonal. We record this as follows, and omit the proof which is elementary but laborious.

\begin{lem}
If $p$ is even, then the multiplicative generators below the diagonal are the elements $Q^{2^{r-1}n} \cdots Q^{2n} Q^n(x)$ with $x \in H^{E_k}_{n,n-1}(\gI)$.

If $p$ is odd, then the multiplicative generators below the diagonal are the elements $\beta^{\epsilon}Q^{p^{r-1}n/2} \cdots Q^{pn/2} Q^{ n/2}(x)$ with $x \in H^{E_k}_{n,n-1}(\gI)$ (implicitly $n$ is even if $r > 0$).\qed
\end{lem}

This describes a basis for $V_{*,*,*}$. If $p$ is even then this is $X_{*,*,*}$. If $p$ is odd then $X_{*,*,*}$ is spanned by the elements we have described which have even homological degree (i.e.\ the $x$'s with $n$ odd, and the $\beta Q^{p^{r-1}n/2} \cdots Q^{ pn/2} Q^{ n/2}(x)$'s with $n$ even). In either case we may take these elements, appropriately ordered, as the $x_i$.

\section{Efficient Smith--Toda complexes}\label{sec:EfficientSTComplexes}

The argument presented in Section \ref{sec:HigherStabMaps} produces, for a given slope $\lambda < 1$, a Smith--Toda complex $\gR/(\alpha_1, \ldots, \alpha_r)$ with a slope $\lambda$ vanishing line, but the endomorphisms $\alpha_i$ can be nilpotent, or even zero. The corresponding filtered endomorphism $\phi_i$ of $\fil_*\gR/(\phi_1, \ldots, \phi_{i-1})$ is never nilpotent, as in the associated graded it is detected by (some power of) the non-nilpotent element $x_i$: this is precisely why we had to kill them, in order to force $x_i$ to become nilpotent.

One feels that if $\alpha_i$ is nilpotent then it ought not to be necessary to kill it: doing so does not take one out of the thick subcategory one resides in, so cannot (by itself) improve the slope of a vanishing line. In this section we wish to explain how to make Smith--Toda complexes where we only kill those $x_i$'s for which the corresponding $\alpha_i$'s are non-nilpotent. In fact, we will explain how to do slightly better than this.

We will do this by making more refined arguments in the category of $\fil_*\gR$-modules, inspired by recent work with synthetic spectra. The essential idea is the following. To show that a power of $u_*(x_i)$ survives in the spectral sequence for $\End_{\fil_*\gR}(\fil_*\gM)$, for some $\fil_*\gR$-module $\fil_*\gM$, it is enough to know that the vanishing property of Lemma \ref{lem:VanishingEstimates} (\ref{it:VanishingEstimates:3}) holds on \emph{some} page of the spectral sequence for the filtered object $\fil_*\gM$, not necessarily on the $E^1$-page. This is best managed by considering the filtered homotopy groups of $\fil_*\gM$ (analogous to the synthetic homotopy groups), and working with lines beyond which they are uniformly bounded $\tau$-power torsion. Indeed, a filtered endomorphism $\phi_i$ yields a nilpotent $\alpha_i$ precisely when some power of $\phi_i$ is $\tau$-torsion.

\subsection{$y$ self-maps}

Recall that there is a filtered algebra map
$$u: \mathrm{fil}_*\gR \lra \End_{\mathrm{fil}_*\gR}(\mathrm{fil}_*\gM)$$
landing in the centre. Generalising Definition \ref{defn:xiselfmap}, we have the following.

\begin{defn}
For a $y \in \pi_{N,D,F}(C\tau \otimes \fil_*\gR)$, a filtered map
$$\psi : S^{p^M N, p^M D, p^M F} \otimes \mathrm{fil}_*\gM \lra \mathrm{fil}_*\gM$$
is a \emph{$y$ self-map} if $C\tau \otimes {\psi} = (C\tau \otimes u)_*(y^{p^M})$.
\end{defn}

Our first goal is to describe a condition which guarantees the existence (and later uniqueness up to nilpotence) of $y$ self-maps.

\begin{defn}\label{defn:WeaklyTypei}
For $\lambda \geq 0$ and $\mu \leq 0$, let $\mathsf{W}(\lambda, \mu) \subseteq \fil_*\gR\text{-}\mathsf{mod}$ be the full subcategory consisting of the finite $\fil_*\gR$-modules $\fil_*\gM$ for which there exists an $r$ such that: 
\begin{enumerate}[(I)]
\item\label{it:WeaklyTypei:1} There is a $\kappa$ such that $E^{r+1}_{n,d,*}(\fil_*\gM)=0$ for $d < \lambda n + \kappa$.

\item\label{it:WeaklyTypei:2} For each $\tilde{\kappa}$ there is a $\rho$ such that $E^{r+1}_{n,d,f}(\fil_*\gM)=0$ whenever $d = \lambda n + \tilde{\kappa}$ and $f < \mu d + \rho$.
\end{enumerate}
\end{defn}

It will be convenient to also work with the following alternative formulation. The equivalence between them is an elaboration of \cite[Proposition 11.6]{BHS}, cf.\ \cite{HPS}.

\begin{prop}\label{prop:WeaklyTypeiEq}
These conditions are equivalent to asking that:
\begin{enumerate}[(I)]
\item[(A)]\label{it:WeaklyTypei:1prime} There is a $\kappa$ such that every class in $\pi_{n,d,*}(\fil_*\gM)$ with $d < \lambda n + \kappa$ is $\tau^r$-torsion.

\item[(B)]\label{it:WeaklyTypei:2prime} For each $\tilde{\kappa}$ there is a $\rho$ such that every class in $\pi_{n,d,f}(\fil_*\gM)$
with $d = \lambda n + \tilde{\kappa}$ and $f < \mu d + \rho$ is $\tau^r$-torsion.
\end{enumerate}
\end{prop}
\begin{proof}
Tensoring $\fil_* \gM$ with the cofibre sequence $S^{0,0,1} \overset{\tau}\to S^{0,0,0} \to C\tau$ and taking homotopy groups gives the exact couple
\begin{equation*}
\begin{tikzcd}
\pi_{n,d,f-1}(\fil_* \gM) \arrow[rr, "\tau"] && \pi_{n,d,f}(\fil_* \gM) \arrow[dl, "q"]\\
& \pi_{n,d,f}(C\tau \otimes \fil_* \gM) \arrow[ul, dashed, "\partial"] \arrow[r, equals] & E^1_{n,d,f}(\fil_* \gM)
\end{tikzcd}
\end{equation*}
which defines the spectral sequence $\{E^{r}_{*,*,*}(\fil_* \gM)\}_r$. We explained in Section \ref{sec:CanMultFilt} that when $\gR$ satisfies axiom (C), which we are assuming, the filtered object $\fil_*\gR$ is complete, and therefore the finite $\fil_*\gR$-module $\fil_*\gM$ is complete too. Recall that $x \in E^1_{n,d,f}(\fil_* \gM)$ survives to $E^r_{n,d,f}(\fil_* \gM)$ if and only if $\partial(x) = \tau^{r-1} \cdot y$ for some $y \in \pi_{n,d,f-r}(\fil_* \gM)$, in which case $d^r([x]) = [q(y)]$. 

The argument is deliberately repetitive: the equivalence (II) $\Leftrightarrow$ (B) is an elaboration of the equivalence (I) $\Leftrightarrow$ (A), and for clarity we give both in detail.

Suppose (I) holds with parameter $\kappa$, and let $x \in \pi_{n,d,f}(\fil_*\gM)$ with $d < \lambda n + \kappa$. Then $q(x) \in E^1_{n,d,f}$ is a permanent cycle, but as $E^{r+1}_{n,d,f}=0$ there must be a differential $d^s([u]) = [q(x)]$ for some $s \leq r$. That is, $\partial(u) = \tau^{s-1} \cdot z$ and $q(z)=q(x)$. Then $x - z = \tau \cdot x'$, and as $\tau \cdot \partial(-)=0$ it follows that $0 = \tau^s(x  - \tau \cdot x')$, so multiplying by $\tau^{r-s}$ gives $\tau^r \cdot x = \tau^{r+1} \cdot x'$. The same reasoning applies to $x'$, showing that $\tau^r \cdot x = \tau^{r+2} \cdot x''$, and so on: it follows that $\tau^r \cdot x$ is infinitely divisible by $\tau$, so it lifts to an element of the limit $\lim_p \pi_{n,d,f+r-p}(\fil_* \gM)$. But the Milnor sequence
$$0 \to {\lim_p}^1 \pi_{n,d+1,f+r-p}(\fil_* \gM) \to \pi_{n,d}(\lim \fil_* \gM) \to \lim_p \pi_{n,d,f+r-p}(\fil_* \gM) \to 0$$
and the fact that $\fil_*\gM$ is complete, i.e.\ that $\lim \fil_* \gM \simeq 0$, means that this limit is zero, and so $\tau^r \cdot x=0$. Thus (A) holds with parameter $\kappa$.

Conversely, suppose (A) holds with parameter $\kappa$, and let $[u] \in E^{r+1}_{n,d,f}$ with $d < \lambda n + \kappa$. As $u \in E^{1}_{n,d,f}$ survives until $E^{r+1}$, we have $\partial(u) = \tau^r \cdot z$ for some $z \in \pi_{n,d-1, f-1-r}(\fil_*\gM)$. Now $d-1 < \lambda n + \kappa$ too so by (A) we have $\tau^r \cdot z=0$ and hence $\partial(u)=0$, meaning that $u = q(x)$ for $x \in \pi_{n,d,f}(\fil_*\gM)$. As $d < \lambda n + \kappa$ we have $\tau^r \cdot x=0$. Let $s \leq r$ be such that $\tau^{s-1} \cdot x \neq 0$ but $\tau^s \cdot x=0$. Then $\tau^{s-1} \cdot x = \partial(v)$, meaning that $d^s([v]) = [q(x)] = [u]$, so that $[u]=0 \in E^{r+1}_{n,d,f}$. Thus (I) holds with parameter $\kappa$.

Now suppose (II) holds and choose a $\tilde{\kappa}$, and let $x \in \pi_{n,d,f}(\fil_*\gM)$ with $d = \lambda n + \tilde{\kappa}$. Then $q(x) \in E^1_{n,d,f}$ is a permanent cycle. Applying (II) with parameter $\tilde{\kappa}$ we find that there is a $\rho$ such that $E^{r+1}_{n,d,f}=0$ as long as $f < \mu d + \rho$. Supposing then that this is the case, we must have a differential $d^s([u]) = [q(x)]$ for some $s \leq r$, so $\partial(u) = \tau^{s-1} \cdot z$ with $q(z)=q(x)$, and as above we deduce that $\tau^r \cdot x = \tau^{r+1} \cdot x'$. This $x'$ has the same grading and homological degree, but smaller filtration than $x$, so the above reasoning continues to apply to it: it follows that $\tau^r \cdot x$ is infinitely divisible by $\tau$, so as above it vanishes. We have shown that any $x \in \pi_{n,d,f}(\fil_*\gM)$ with $d = \lambda n + \tilde{\kappa}$ and $f < \mu d + \rho$ is $\tau^r$-torsion, so (B) holds.

Conversely, suppose (B) holds and choose a $\tilde{\kappa}$, and let $[u] \in E^{r+1}_{n,d,f}$ with $d = \lambda n + \tilde{\kappa}$. We have $\partial(u) = \tau^r \cdot z$ for some $z \in \pi_{n,d-1, f-1-r}(\fil_*\gM)$. Applying (B) with parameter  $\tilde{\kappa}' := \tilde{\kappa}-1$ implies that there is a $\rho'$ such that $\pi_{n,d-1, f-1-r}(\fil_*\gM)$ is $\tau^r$-torsion as long as $f-1-r < \mu (d-1) + \rho'$, i.e.\ for $f < \mu d + (\rho'-\mu+1+r)$. Supposing then that this is the case, we have $\partial(u) = \tau^r \cdot z = 0$ and so $u = q(x)$ with $x \in \pi_{n,d,f}(\fil_*\gM)$. Applying (B) again with parameter $\tilde{\kappa}'' := \tilde{\kappa}$ implies that there is a $\rho''$ such that $\pi_{n,d,f}(\fil_*\gM)$ is $\tau^r$-torsion as long as $f < \mu d + \rho''$. Supposing that this is also the case, there is a $s \leq r$ with $\tau^{s-1} \cdot x \neq 0$ and $\tau^s \cdot x=0$, meaning that there is a $v$ with $\partial(v) = \tau^{s-1} \cdot x$, and hence a differential $d^s([v]) = [q(x)]=[u]$, so that $[u]=0 \in E^{r+1}_{n,d,f}$. We have shown that any $[u] \in E^{r+1}_{n,d,f}$ with $d = \lambda n + \tilde{\kappa}$ and $f < \mu d + \min\{\rho'-\mu+1+r, \rho''\}$ vanishes, so (II) holds.
\end{proof}

\begin{defn}
Let $\mathsf{T} \subseteq \fil_*\gR\text{-}\mathsf{mod}$ be the full subcategory consisting of the finite $\fil_*\gR$-modules $\fil_*\gM$ such that $\{\pi_{n,d,*}(\fil_*\gM)\}_{n,d}$ are uniformly bounded $\tau$-power torsion (i.e.\ are all annihilated by $\tau^r$ for some fixed $r$)
\end{defn}

\begin{prop}\label{prop:WeaklyTypeiProperties}\mbox{}
\begin{enumerate}[(i)]
\item\label{it:WeaklyTypeiProperties:1} $\mathsf{W}(\lambda, \mu)$ is a thick subcategory.

\item\label{it:WeaklyTypeiProperties:2} We have $\mathsf{W}(\lambda, \mu) \subseteq \mathsf{W}(\bar{\lambda}, \bar{\mu})$ if $\lambda > \bar{\lambda}$, or if $\lambda = \bar{\lambda}$ and $\mu \geq \bar{\mu}$.

\item\label{it:WeaklyTypeiProperties:3} $\mathsf{T} \subseteq \mathsf{W}(\lambda,\mu)$ for any $\lambda$ and $\mu$.
\end{enumerate}
\end{prop}
\begin{proof}
For (\ref{it:WeaklyTypeiProperties:1}), first note that objects of $\mathsf{W}(\lambda, \mu)$ are clearly closed under retracts, and under shifts by any $S^{n,d,f}$. Let $\fil_*\gM' \to \fil_*\gM \to \fil_*\gM''$ be a cofibre sequence where the outer terms are in $\mathsf{W}(\lambda, \mu)$. To verify (A) for $\fil_*\gM$, note that there are $r'$, $\kappa'$, $r''$ and $\kappa''$ such that $\pi_{n,d,*}(\fil_*\gM')$ is $\tau^{r'}$-torsion for $d < \lambda n+\kappa'$ and $\pi_{n,d,*}(\fil_*\gM'')$ is $\tau^{r''}$-torsion for $d < \lambda n+\kappa''$. From the long exact sequence it follows that $\pi_{n,d,*}(\fil_*\gM)$ is $\tau^{r'+r''}$-torsion for $d < \lambda n+\min\{\kappa',\kappa''\}$, verifying (A) with $r := r'+r''$. To verify (B) for $\fil_*\gM$, let $\tilde{\kappa}$ be given. Then there are $\rho'$ and $\rho''$ such that the filtered homotopy groups of $\fil_*\gM'$ are $\tau^{r'}$-torsion for $d = \lambda n + \tilde{\kappa}$ and $f < \mu d + \rho'$, and those of $\fil_*\gM''$ are $\tau^{r''}$-torsion for $d = \lambda n + \tilde{\kappa}$ and $f < \mu d + \rho''$, then the long exact sequence shows that those of $\fil_*\gM$ are $\tau^{r'+r''}$-torsion for $d = \lambda n + \tilde{\kappa}$ and $f < \mu d + \min\{\rho', \rho''\}$. This verifies (B) with $r = r'+r''$ too.

For (\ref{it:WeaklyTypeiProperties:2}), first note that as $\bar{\lambda} \leq \lambda$, and the homotopy groups $\pi_{n,d,*}(\fil_*\gM)$ vanish for $n\ll 0$ or $d \ll 0$ by Axiom (C) and the fact that $\fil_*\gM$ is a finite $\fil_*\gR$-module, property (A) for $\bar{\lambda}$ follows from that for $\lambda$. For property (B), there are two cases. 

If $\lambda = \bar{\lambda}$ then along each line $d = \lambda n + \tilde{\kappa}$ we know that there is a $\rho$ such that $\pi_{n,d,f}(\fil_* \gM)$ is $\tau^r$-torsion for $f < \mu d+\rho$. Using again that these homotopy groups vanish for $d \ll 0$ and the fact that $\bar{\mu} \leq \mu$, it follows that they are also $\tau^r$-torsion for $f < \bar{\mu} d + \rho'$ for some $\rho'$.

If instead $\lambda > \bar{\lambda}$ then by (A) each line $d= \bar{\lambda} n + \tilde{\kappa}$ only contains finitely-many $\pi_{n,d,*}(\fil_*\gM)$'s which are not $\tau^r$-torsion. As these are finitely-generated $\bk[\tau]$-modules, each of them vanishes in small enough filtration. Thus $\rho$ may be chosen sufficiently small so that all these finitely-many $\bk[\tau]$-modules vanish in filtration $f < \rho$, and hence also for $f < \bar{\mu} d+\rho$ (as $\bar{\mu} \leq 0$). Having done this, it follows that $\pi_{n,d,f}(\fil_*\gM)$ is $\tau^r$-torsion for $d = \bar{\lambda} n + \tilde{\kappa}$ and $f < \bar{\mu}d+\rho$.

Item (\ref{it:WeaklyTypeiProperties:3}) is immediate from (A) and (B).
\end{proof}

\begin{thm}\label{thm:xiSelfMapsExist}
Let $y \in \pi_{N,D,F}(C\tau \otimes \fil_*\gR)$ be given. Then all objects of $\mathsf{W}(\tfrac{D}{N}, \tfrac{F}{D})$ admit a $y$ self-map.
\end{thm}
\begin{proof}
Let $\fil_* \gM \in \mathsf{W}(\tfrac{D}{N}, \tfrac{F}{D})$. As it is a finite $\fil_*\gR$-module its endomorphism object $\End_{\fil_*\gR}(\fil_*\gM)$ lies in the thick subcategory generated by $\fil_*\gM$ and so is also in $\mathsf{W}(\tfrac{D}{N}, \tfrac{F}{D})$. By (II) applied with $\tilde{\kappa}=-1$ there is an $r$ and a $\rho$ such that
\begin{equation}\label{eq:VanishingRangeSS}
E^{r+2}_{p^M N, p^M D-1, p^M F - s}(\End_{\fil_*\gR}(\fil_*\gM))=0
\end{equation}
for $-s < \rho$ and all $M$. Set $M_0 := \max\{r+2, 1-\rho\}$.

The central class 
$$u_*(y^{p^{M_0}}) \in E^1_{p^{M_0} N, p^{M_0} D, p^{M_0} F}(\End_{\fil_*\gR}(\fil_*\gM))$$
survives until $E^{{M_0}}_{p^{M_0} N, p^{M_0} D, p^{M_0} F}(\End_{\fil_*\gR}(\fil_*\gM))$, by applications of the Leibniz rule.  For each $s \geq {M_0}$ the class
$$d^s(u_*(y^{p^{M_0}})) \in E^{s}_{p^{M_0} N, p^{M_0} D-1, p^{M_0} F - s}(\End_{\fil_*\gR}(\fil_*\gM))$$
must vanish by \eqref{eq:VanishingRangeSS}. Thus $u_*(y^{p^{M_0}})$ is a permanent cycle, and detects a filtered homotopy class $\psi \in \pi_{p^{M_0} N, p^{M_0} D, p^{M_0} F}(\End_{\fil_*\gR}(\fil_*\gM))$ which is a $y$ self-map.
\end{proof}

Uniqueness up to nilpotence and several other useful properties of $y$ self-maps are collected as follows.

\begin{thm}\label{thm:xiSelfMapsUniqueAndNatural}
Let $y \in \pi_{N,D,F}(C\tau \otimes \fil_*\gR)$ be given.
\begin{enumerate}[(i)]
\item\label{it:Centre} If $\psi$ is a $y$ self-map of $\fil_*\gM \in \mathsf{W}(\tfrac{D}{N}, \tfrac{F}{D})$, then after perhaps taking a $p$-th power it lies in the (graded) centre of $\pi_{*,*,*}(\End_{\fil_* \gR}(\fil_* \gM))$.

\item\label{it:Unique} If $\psi$ and $\psi'$ are $y$ self-maps of  $\fil_*\gM \in \mathsf{W}(\tfrac{D}{N}, \tfrac{F}{D})$, then they agree after perhaps taking $p$-th powers.

\item\label{it:Natural} If $\fil_*\gM$ and $\fil_*\gN$ are in $\mathsf{W}(\tfrac{D}{N}, \tfrac{F}{D})$ and have $y$ self-maps $\psi^\gM$ and $\psi^\gN$ of the same tridegrees, then after perhaps taking $p$-th powers they intertwine all $\fil_*\gR$-module maps $f : S^{n,d,f} \otimes \fil_*\gM \to \fil_*\gN$.

\item\label{it:Cofib} If $\fil_*\gM \overset{f}\to \fil_*\gN \overset{g}\to \fil_* \gP \overset{\partial}\to \Sigma \fil_* \gM$ is a distinguished triangle of $\fil_*\gR$-modules, and $\fil_*\gM$ and $\fil_*\gN$ are in $\mathsf{W}(\tfrac{D}{N}, \tfrac{F}{D})$, then so is $\fil_* \gP$. Furthermore they admit $y$ self-maps $\psi^\gM$, $\psi^\gN$, and $\psi^\gP$ which assemble into a morphism of distinguished triangles.
\end{enumerate}
\end{thm}
\begin{proof}
For (\ref{it:Centre}) we follow the idea of \cite[Lemma 3.5]{HopkinsSmith}, though expressing ourselves in terms of filtered objects and their filtered homotopy groups rather than the associated spectral sequence. Consider the commuting endomorphisms
$$\psi \circ (-), (-) \circ \psi : S^{p^M N, p^M D, p^M F} \otimes \End_{\fil_* \gR}(\fil_* \gM) \lra \End_{\fil_* \gR}(\fil_* \gM).$$
On associated graded they are given by post- and pre-composition with $u_*(y^{p^M})$, which lies in the graded centre of $C\tau \otimes \End_{\fil_* \gR}(\fil_* \gM)$, so they are homotopic up to a sign of $(-1)^{p^M D} = (-1)^D \in \bF_p^\times \subseteq \bk^\times$. Thus the (signed) difference may be divided by $\tau$ to give a filtered map
$$[\psi \circ (-) - (-1)^D(-) \circ \psi] : S^{p^M N, p^M D, p^M F-1} \otimes \End_{\fil_* \gR}(\fil_* \gM) \lra \End_{\fil_* \gR}(\fil_* \gM)$$
of filtration 1 less (which is no longer a difference of two maps). Consider its $p^L$-th power
$$[\psi \circ (-) - (-1)^D(-) \circ \psi]^{p^L} : S^{p^{M+L} N, p^{M+L} D, p^{M+L} F-p^L} \lra \End_{\fil_* \gR}(\End_{\fil_* \gR}(\fil_* \gM)),$$
for some $L$ to be chosen shortly.

The endomorphism object $\End_{\fil_* \gR}(\fil_* \gM)$ is again in $\mathsf{W}(\tfrac{D}{N}, \tfrac{F}{D})$ by Proposition \ref{prop:WeaklyTypeiProperties} (i), so $\End_{\fil_* \gR}(\End_{\fil_* \gR}(\fil_* \gM))$ is too by the same reasoning. Thus by property (B) there is an $r$ and a $\rho$ such that
$$\pi_{n, \tfrac{D}{N} n, f}(\End_{\fil_* \gR}(\End_{\fil_* \gR}(\fil_* \gM))) \text{ is $\tau^r$-torsion for } f < \tfrac{F}{N} n + \rho \text{ and all } n.$$
If we choose $L$ large enough that $-p^L < \rho$ then it follows that 
$$[\psi \circ (-) - (-1)^D(-) \circ \psi]^{p^L} \in \pi_{p^{M+L} N, p^{M+L} D, p^{M+L} F-p^L}(\End_{\fil_* \gR}(\End_{\fil_* \gR}(\fil_* \gM)))$$
is $\tau^r$-torsion, and if in addition $L$ is large enough that $p^L \geq r$ then multiplying by $\tau^{p^L}$  gives
$$(\psi \circ (-) - (-1)^D(-) \circ \psi)^{p^L} = 0 \in \pi_{p^{M+L} N, p^{M+L} D, p^{M+L} F}(\End_{\fil_* \gR}(\End_{\fil_* \gR}(\fil_* \gM))).$$
 As the endomorphisms $\psi \circ (-)$ and $(-) \circ \psi$ of $\End_{\fil_* \gR}(\fil_* \gM)$ commute, the binomial expansion shows that $(\psi \circ (-) - (-1)^D(-) \circ \psi)^{p^L} = \psi^{p^L} \circ (-) - (-1)^D(-) \circ \psi^{p^L}$, or in other words that $\psi^{p^L} \circ (-) \simeq (-1)^{p^{M+L}D}(-) \circ \psi^{p^L}$ as endomorphisms of $\End_{\fil_* \gR}(\fil_* \gM)$. In particular $\psi^{p^L}$ lies in the graded centre of $\pi_{*,*,*}(\End_{\fil_* \gR}(\fil_* \gM))$.

For (\ref{it:Unique}) the argument is highly similar. After raising $\psi$ and $\psi'$ to perhaps different $p$-th powers, we may suppose they have the same tridegree, say $(p^M N, p^M D, p^M F)$, and by (\ref{it:Centre}) are both central in $\pi_{*,*,*}(\End_{\fil_* \gR}(\fil_* \gM))$. Then the filtered map
$$\psi' - \psi : S^{p^M N, p^M D, p^M F} \lra \End_{\fil_* \gR}(\mathrm{fil}_*\gM)$$
is given by $u_*(y^{p^M})-u_*(y^{p^M})=0$ on associated graded, so it may be divided by $\tau$ to give a map
$$[\psi' - \psi] : S^{p^M N, p^M D, p^M F-1} \lra \End_{\fil_* \gR}(\mathrm{fil}_*\gM)$$
of filtration 1 less (which is no longer a difference of two maps). Thus its $p^L$-th power is
$$[\psi' - \psi]^{p^L} : S^{p^{M+L} N, p^{M+L} D, p^{M+L} F - p^L} \lra \End_{\fil_* \gR}(\mathrm{fil}_*\gM).$$
Again, $\End_{\fil_* \gR}(\mathrm{fil}_*\gM)$ is in $\mathsf{W}(\tfrac{D}{N}, \tfrac{F}{D})$ so by (B) there is an $r$ and a $\rho$ such that
$$\pi_{n, \tfrac{D}{N}n, f}(\End_{\fil_* \gR}(\mathrm{fil}_*\gM)) \text{ is $\tau^r$-torsion for } f < \tfrac{F}{N}n + \rho \text{ and all } n.$$
Choosing $L$ large enough that $-p^L < \rho$ and $p^L \geq r$, it follows that $[\psi' - \psi]^{p^L}$ is $\tau^r$-torsion and so is $\tau^{p^L}$-torsion, and hence that $(\psi' - \psi)^{p^L}=0$. As $\psi$ is central we can expand this binomially to obtain $(\psi')^{p^L} = (\psi)^{p^L} \in \pi_{*,*,*}(\End_{\fil_* \gR}(\mathrm{fil}_*\gM))$.

For (\ref{it:Natural}), we consider the $\fil_*\gR$-module $\Hom_{\fil_* \gR}(\fil_* \gM, \fil_* \gN)$. As $\fil_* \gM$ is a finite $\fil_*\gR$-module and $\fil_* \gN$ is in $\mathsf{W}(\tfrac{D}{N}, \tfrac{F}{D})$, $\Hom_{\fil_* \gR}(\fil_* \gM, \fil_* \gN)$ is also in $\mathsf{W}(\tfrac{D}{N}, \tfrac{F}{D})$. Suppose that $\psi^\gM$ and $\psi^\gN$ both have tridegree $(p^M N, p^M D, p^M F)$. Then the maps
$$(-) \circ \psi^\gM : S^{p^M N, p^M D, p^M F} \otimes \Hom_{\fil_* \gR}(\fil_* \gM, \fil_* \gN) \lra \Hom_{\fil_* \gR}(\fil_* \gM, \fil_* \gN)$$
$$\psi^\gN \circ (-) : S^{p^M N, p^M D, p^M F} \otimes \Hom_{\fil_* \gR}(\fil_* \gM, \fil_* \gN) \lra \Hom_{\fil_* \gR}(\fil_* \gM, \fil_* \gN)$$
are both $y$ self-maps, so by part (\ref{it:Unique}) they agree after perhaps taking further $p$-th powers. In particular there is a fixed $p^L$ such that $f \circ (\psi^\gM)^{p^L} \simeq (\psi^\gN)^{p^L} \circ f$ for all $\fil_*\gR$-module maps $f : S^{n,d,f} \otimes \fil_*\gM \to \fil_*\gN$.

For (\ref{it:Cofib}), we have already explained that $\mathsf{W}(\tfrac{D}{N}, \tfrac{F}{D})$ is a thick subcategory, so $\fil_*\gP$ is in $\mathsf{W}(\tfrac{D}{N}, \tfrac{F}{D})$. Choosing any $y$ self-maps for all three objects, of the same tridegree, after raising them to a suitable $p$-th power it follows from part (\ref{it:Natural}) that they commute with $f$, $g$, and $\partial$, so give a morphism of distinguished triangles.
\end{proof}

\subsection{Nilpotence up to $\tau$-power torsion}

The following lemma formalises how when working modulo objects of $\mathsf{T}$, i.e.\ finite $\fil_*\gR$-modules whose homotopy groups are uniformly bounded $\tau$-power torsion, we may treat endomorphism which are nilpotent up to $\tau$-torsion as though they were nilpotent. We then give some applications of this.

\begin{lem}\label{lem:NilpConeDoesNotLeave}
Let $f$ be a self-map of a finite $\fil_*\gR$-module $\fil_*\gM$ such that some iterate of $f$ is $\tau$-power torsion. Then $\fil_*\gM$ lies in the thick subcategory generated by $\fil_*\gM/f$ and $\mathsf{T}$.

In particular, if $\fil_*\gM$ is finite and $\fil_*\gM/f$ is in $\mathsf{W}(\tfrac{D}{N}, \tfrac{F}{D})$ then so is $\fil_*\gM$.
\end{lem}
\begin{proof}
Suppose that $\tau^r f^K \simeq 0$. There is a cofibre sequence
$$S^{0,0,r} \otimes C\tau^{r} \otimes \fil_*\gM \lra \fil_*\gM /(\tau^{r} f^{K}) \lra \fil_*\gM /(f^{K})$$
upon which we make the following observations:
\begin{enumerate}[(a)]
\item The left term is a finite $\fil_*\gR$-module with uniformly bounded $\tau$-power torsion homotopy groups (as they are $\tau^{2r}$-torsion).

\item The middle term contains $\fil_*\gM$ as a retract (as $\tau^{r}f^K$ is null).

\item The right term lies in the thick subcategory generated by $\fil_*\gM/f$ (by Lemma \ref{lem:ThickSubcatFacts} (\ref{it:ThickSubcatFacts:2})).
\end{enumerate}
The claim follows.
\end{proof}

\begin{lem}\label{lem:Generation} Fix slopes $(\lambda, \mu)$.
\begin{enumerate}[(i)]
\item\label{it:Generation:1} Let $x_s, \ldots, x_t$ be those $x$'s having slopes $(\lambda, \mu)$. If $\fil_*\gR/(\phi_1, \ldots, \phi_{s-1})$ is any filtered Smith--Toda complex obtained by taking the iterated cofibre of a sequence of $x_k$ self-maps $\phi_k$, then  $\fil_*\gR/(\phi_1, \ldots, \phi_{s-1})$ lies in $\mathsf{W}(\lambda, \mu)$, and

\item\label{it:Generation:2} Let $\fil_*\gR/(\psi_1, \ldots, \psi_r)$ be any filtered Smith--Toda complex which lies in $\mathsf{W}(\lambda, \mu)$, and such that each $\psi_i$ has slopes $< (\lambda, \mu)$. Then $\mathsf{W}(\lambda, \mu)$ is generated as a thick subcategory by $\fil_*\gR/(\psi_1, \ldots, \psi_r)$ and $\mathsf{T}$.
\end{enumerate}
\end{lem}
\begin{proof}
By Lemma \ref{lem:VanishingEstimates} (\ref{it:VanishingEstimates:1}) there is a $\kappa$ such that $\pi_{n,d,f}(C\tau \otimes \fil_*\gR/(\phi_1, \ldots, \phi_{s-1}))=0$ for $d < \lambda n+\kappa$, and by Lemma \ref{lem:VanishingEstimates} (\ref{it:VanishingEstimates:2}) for each $\tilde{\kappa}$ there is a $\rho$ such that $\pi_{n,d,f}(C\tau \otimes \fil_*\gR/(\phi_1, \ldots, \phi_{s-1}))=0$ for $d = \lambda n + \tilde{\kappa}$ and $f < \mu d+\rho$.  In particular it satisfies (I) and (II) of Definition \ref{defn:WeaklyTypei} with $r=0$, so $\fil_*\gR/(\phi_1, \ldots, \phi_{s-1}) \in \mathsf{W}(\tfrac{D}{N}, \tfrac{F}{D})$.

To prove (\ref{it:Generation:2}), suppose $\fil_*\gM \in \mathsf{W}(\tfrac{D}{N}, \tfrac{F}{D})$. We observe that for each $1 \leq k \leq r$ the endomorphism $\mathrm{Id} \otimes \psi_k$ of $\fil_*\gM \otimes_{\fil_*\gR} \fil_*\gR/(\psi_1, \ldots, \psi_{k-1})$ has an iterate which is $\tau$-power torsion, because this object lies in $\mathsf{W}(\lambda, \mu)$ and our assumption on the slopes of $\psi_k$. By many applications of Lemma \ref{lem:NilpConeDoesNotLeave} it follows that $\fil_*\gM$ is in the thick subcategory generated by $\fil_*\gM \otimes_{\fil_*\gR} \fil_*\gR/(\psi_1, \ldots, \psi_{k})$ and $\mathsf{T}$. As $\fil_*\gM \otimes_{\fil_*\gR} \fil_*\gR/(\psi_1, \ldots, \psi_{k})$ is in particular in the thick subcategory generated by $\fil_*\gR/(\psi_1, \ldots, \psi_{k})$, the claim follows.
\end{proof}

\begin{defn}
Let $y \in \pi_{N,D,F}(C\tau \otimes \fil_*\gR)$ be given. Write $\mathsf{N}(y) \subset \mathsf{W}(\tfrac{D}{N}, \tfrac{F}{D})$ for the full subcategory of those objects such that each $y$ self-map (which exists by Theorem \ref{thm:xiSelfMapsExist} and is unique up to taking powers by Theorem \ref{thm:xiSelfMapsUniqueAndNatural} (\ref{it:Unique})) has an iterate which is $\tau$-power torsion.
\end{defn}

\begin{lem}\label{lem:NilProperties}\mbox{}
\begin{enumerate}[(i)]
\item\label{it:NilProperties:1} $\mathsf{N}(y)$ is a thick subcategory.

\item\label{it:NilProperties:2} $\mathsf{T} \subseteq \mathsf{N}(y)$ for any $y$.

\item\label{it:NilProperties:3} If $\fil_*\gM \in \mathsf{W}(\tfrac{D}{N}, \tfrac{F}{D})$ and $\psi$ is a $y$ self-map, then $\fil_*\gM/\psi \in \mathsf{N}(y)$.
\end{enumerate}
\end{lem}
\begin{proof}
Let $\fil_*\gM' \overset{i}\to \fil_*\gM \overset{q} \to \fil_*\gM''$ be a cofibre sequence in $\mathsf{W}(\tfrac{D}{N}, \tfrac{F}{D})$, so all three terms admit $y$ self-maps $\psi'$, $\psi$, and $\psi''$ by Theorem \ref{thm:xiSelfMapsExist}. Assume the outer terms are in $\mathsf{N}(y)$. By Theorem \ref{thm:xiSelfMapsUniqueAndNatural} (\ref{it:Cofib}) these self-maps may be assumed to have the same tridegree and to commute with the maps $i$ and $q$. Raising to appropriate powers, the maps $\psi'$ and $\psi''$ may be assumed to be $\tau^r$-torsion, giving a map of cofibre sequences of the form
\begin{equation*}
\begin{tikzcd}
S^{p^{M} N, p^{M} D, p^{M} F+r} \otimes \fil_*\gM' \dar{i} \rar{0} & \fil_*\gM' \dar{i}\\
S^{p^{M} N, p^{M} D, p^{M} F+r} \otimes \fil_*\gM \dar{q} \rar{\tau^r \psi} & \fil_*\gM \dar{q}\\
S^{p^{M} N, p^{M} D, p^{M} F+r} \otimes \fil_*\gM'' \rar{0} & \fil_*\gM''
\end{tikzcd}
\end{equation*}
and so a factorisation
$$S^{p^{M} N, p^{M} D, p^{M} F+r} \otimes \fil_*\gM \overset{q}\lra S^{p^{M} N, p^{M} D, p^{M} F+r} \otimes \fil_*\gM'' \overset{\text{induced}}\lra \fil_*\gM' \overset{i}\lra \fil_*\gM$$
of $\tau^r \psi$. This implies that $(\tau^r \psi)^2\simeq 0$, as $q \circ i\simeq 0$, so $\psi^2$ is $\tau^{2r}$-torsion. Finally, any $y$ self-map of $\fil_*\gM$ agrees with $\psi$ up to taking iterates, by Theorem \ref{thm:xiSelfMapsUniqueAndNatural} (\ref{it:Unique}), so the analogous conclusion holds for all of them. 

It is clear that $\mathsf{N}(y)$ is closed under shifts, so to show that it is a thick subcategory it remains to show it is closed under retracts. If $\fil_*\gM' \oplus \fil_*\gM''$ is in $\mathsf{N}(y)$, then each summand is in $\mathsf{W}(\tfrac{D}{N}, \tfrac{F}{D})$, so has a $y$ self-map. Raising these to appropriate powers we may suppose they have the same degree, hence giving $\psi' \oplus \psi''$ a $y$ self-map of $\fil_*\gM' \oplus \fil_*\gM''$. This has some iterate which is $\tau$-power torsion, from which it follows that both $\psi'$ and $\psi''$ do too.

To prove (\ref{it:NilProperties:2}), first note that by Proposition \ref{prop:WeaklyTypeiProperties} (\ref{it:WeaklyTypeiProperties:3}) $\mathsf{T} \subseteq \mathsf{W}(\tfrac{D}{N}, \tfrac{F}{D})$. As $\mathsf{T}$ is a thick subcategory of finite $\fil_*\gR$-modules, if $\fil_*\gM \in \mathsf{T}$ then $\End_{\fil_*\gR}(\fil_*\gM) \in \mathsf{T}$ too. So every endomorphism of $\fil_*\gM$ is $\tau$-power torsion, in particular $y$ self-maps are.

To prove (\ref{it:NilProperties:3}), note that 
$C\tau \otimes (\fil_*\gM/\psi) \simeq (C\tau \otimes \fil_*\gM)/(y^{p^M})$ on which $y$ acts nilpotently. Thus the zero endomorphism of $\fil_*\gM/\psi$ is a $y$ self-map, and hence by Theorem \ref{thm:xiSelfMapsUniqueAndNatural} (\ref{it:Unique}) any $y$ self-map of $\fil_*\gM/\psi$ is nilpotent.
\end{proof}

\subsection{Omitting nilpotent endomorphisms}\label{sec:Omitting}

As a first step towards efficiency of Smith--Toda complexes, we make a slight improvement to Theorem \ref{thm:FiltSTExist}. Let us say that an $x_i$ is \emph{redundant} if the $x_i$ self-map
$$\phi_i : S^{p^{M_i}n_i, p^{M_i}d_i, p^{M_i}f_i} \otimes \fil_*\gR/(\phi_1, \ldots, \phi_{i-1}) \lra \fil_*\gR/(\phi_1, \ldots, \phi_{i-1})$$
 is nilpotent after inverting $\tau$, i.e.\ has an iterate which is $\tau$-power torsion. Equivalently, $(\tau \phi_i)^{-1}\fil_*\gR/(\phi_1, \ldots, \phi_{i-1}) \simeq 0$. The latter formulation makes it clear that this property does not depend on the choices of $x_j$ self-maps $\phi_1, \ldots, \phi_{i-1}$ or $\phi_i$ as follows. If $\fil_*\gR/(\phi'_1, \ldots, \phi'_{i-1})$ is another filtered Smith--Toda complex with an $x_i$ self map $\phi'_i$ then by Lemma \ref{lem:Generation} (\ref{it:Generation:2}) it lies in the thick subcategory generated by $\fil_*\gR/(\phi_1, \ldots, \phi_{i-1})$ and $\mathsf{T}$. Using that inverting $\tau$ kills $\mathsf{T}$, and the uniqueness (Theorem \ref{thm:xiSelfMapsUniqueAndNatural} (\ref{it:Unique})) and naturality (Theorem \ref{thm:xiSelfMapsUniqueAndNatural} (\ref{it:Natural})) up to taking $p$-th powers of $x_i$ self-maps, it follows that $(\tau \phi'_i)^{-1}$ annihilates $\fil_*\gR/(\phi'_1, \ldots, \phi'_{i-1})$ too.

\begin{thm}\label{thm:OmittingRedundantXs}
Let $x_{j_1}, x_{j_2}, \ldots, x_{j_s}$ be the subsequence of $x_1, x_2, \ldots, x_{i}$ consisting of non-redundant $x$'s. 
\begin{enumerate}[(i)]
\item\label{it:OmittingRedundantXs:1} There is a filtered Smith--Toda complex $\fil_*\gR/(\phi_{j_1}, \ldots, \phi_{j_s})$ which lies in the thick subcategory generated by $\fil_*\gR/(\phi_1, \ldots, \phi_i)$ and $\mathsf{T}$.

In particular $\tau^{-1}\fil_*\gR/(\phi_{j_1}, \ldots, \phi_{j_s})$ lies in the thick subcategory generated by $\tau^{-1}\fil_*\gR/(\phi_1, \ldots, \phi_i)$ so has a vanishing line of slope $\tfrac{d_{i+1}}{n_{i+1}}$.

\item\label{it:OmittingRedundantXs:2}  $\fil_*\gR/(\phi_1, \ldots, \phi_i)$ lies in the thick subcategory generated by $\fil_*\gR/(\phi_{j_1}, \ldots, \phi_{j_s})$.

\item\label{it:OmittingRedundantXs:3}  For each $1 \leq r \leq s$ the endomorphism $\phi_{j_r}$ of $\fil_*\gR/(\phi_{j_1}, \ldots, \phi_{j_{r-1}})$ is non-nilpotent after inverting $\tau$.
\end{enumerate}
\end{thm}

\begin{proof}
The Smith--Toda complexes in the statements depend on the choices of self maps $\phi$, but the claims about them do not: using Lemma \ref{lem:ThickSubcatFacts} (\ref{it:ThickSubcatFacts:4}), Theorem \ref{thm:xiSelfMapsUniqueAndNatural} (\ref{it:Unique}) and Theorem \ref{thm:xiSelfMapsUniqueAndNatural} (\ref{it:Natural}) the thick subcategory generated by a $\fil_*\gR/(\phi_1, \ldots, \phi_k)$ is independent of the specific $\phi$'s used. We will not comment any further on this point.

For (\ref{it:OmittingRedundantXs:1}) we will build $\fil_*\gR/(\phi_{j_1}, \phi_{j_2}, \ldots, \phi_{j_{r}})$ inductively on $r$, showing that they lie in the thick subcategory generated by $\fil_*\gR/(\phi_1, \phi_2, \ldots, \phi_{j_{r+1}-1})$ and $\mathsf{T}$. To start the induction, we show that $\fil_*\gR$ lies in the thick subcategory generated by $\fil_*\gR/(\phi_1, \phi_2, \ldots, \phi_{j_{1}-1})$ and $\mathsf{T}$. This is because $x_1, x_2, \ldots x_{j_1 - 1}$ are redundant, so the $\phi_1, \phi_2, \ldots \phi_{j_1 - 1}$ are nilpotent up to $\tau$-power torsion, and hence the claim holds by iterated application of Lemma \ref{lem:NilpConeDoesNotLeave}.

Suppose for the inductive step that $\fil_*\gR/(\phi_{j_1}, \phi_{j_2}, \ldots, \phi_{j_{r-1}})$ exists and lies in the thick subcategory generated by $\fil_*\gR/(\phi_1, \phi_2, \ldots, \phi_{j_{r}-1})$ and $\mathsf{T}$.  In particular $\fil_*\gR/(\phi_{j_1}, \phi_{j_2}, \ldots, \phi_{j_{r-1}})$ lies in $\mathsf{W}(\tfrac{d_{j_r}}{n_{j_r}}, \tfrac{f_{j_r}}{d_{j_r}})$, so it has a $x_{j_r}$ self-map $\phi_{j_r}$ by Theorem \ref{thm:xiSelfMapsExist}. By the asymptotic uniqueness of $x_{j_r}$ self-maps as well as their compatibility with cofibre sequences, it follows that the mapping cone $\fil_*\gR/(\phi_{j_1}, \phi_{j_2}, \ldots, \phi_{j_{r}})$ of this map lies in the thick subcategory generated by $\fil_*\gR/(\phi_1, \phi_2, \ldots, \phi_{j_r})$ and $\mathsf{T}$. As $x_{j_{r}+1}, x_{j_r+2}, \ldots, x_{j_{r+1} -1}$ are redundant an iterated application of Lemma \ref{lem:NilpConeDoesNotLeave} implies that $\fil_*\gR/(\phi_{j_1}, \phi_{j_2}, \ldots, \phi_{j_{r}})$ also lies in the thick subcategory generated by $\fil_*\gR/(\phi_1, \phi_2, \ldots, \phi_{j_{r+1}-1})$ and $\mathsf{T}$, as required.

For (\ref{it:OmittingRedundantXs:2}) note that by (\ref{it:OmittingRedundantXs:1}) and Lemma \ref{lem:Generation} (\ref{it:Generation:2}) the object $\fil_*\gR/(\phi_{j_1}, \phi_{j_2}, \ldots, \phi_{j_s})$ lies in $\mathsf{W}(\tfrac{d_{i+1}}{n_{i+1}}, \tfrac{f_{i+1}}{d_{i+1}})$, and hence it admits $x_j$ self-maps for all $1 \leq j \leq i$ by Theorem \ref{thm:xiSelfMapsExist}. We may therefore construct a $\fil_*\gR/(\phi_1, \phi_2, \ldots, \phi_i)$ from $\fil_*\gR/(\phi_{j_1}, \phi_{j_2}, \ldots, \phi_{j_s})$ by killing $x_j$ self-maps for each $j \in \{1,2,\ldots, i\} \setminus \{j_1, \ldots, j_r\}$.

For (\ref{it:OmittingRedundantXs:3}), note that as $\fil_*\gR/(\phi_1, \phi_2, \ldots, \phi_{j_{r}-1})$ lies in the thick subcategory generated by $\fil_*\gR/(\phi_{j_1}, \phi_{j_2}, \ldots, \phi_{j_{r-1}})$ and $\mathsf{T}$, if $(\tau \phi_{j_r})^{-1}$ annihilated the latter then it would annihilate the former, and it does not because $x_{j_r}$ is not redundant.
\end{proof}

\subsection{Changing coordinates}\label{sec:ChangingCoords}

The discussion so far allows us to omit those $x_i$ which lead to self-maps $\phi_i$ of $\fil_*\gR/(\phi_1, \ldots, \phi_{i-1})$ that are nilpotent after inverting $\tau$. This goes some way towards efficiency of Smith--Toda complexes, but is still not satisfactory as the following example indicates.

\begin{example}\label{ex:RedAndBlueModRedBlue}
With $\bk=\bF_2$ let $\gR := \gE_\infty(S^{1,0}r \oplus S^{1,0}b) \cup^{E_\infty}_{r\cdot b} D^{2,1}\rho$. This has $Q^{E_\infty^\text{nu}} \gI \simeq S^{1,0}r \oplus S^{1,0}b \oplus S^{2,1} \rho$, and so for the $x_i$'s (considered as in Section \ref{sec:EkAlgBigk} to be associated to the canonical $E_\infty$-multiplicative filtration) we may take a sequence starting
$$r, b, Q^1(r), Q^1(b), \rho, \ldots.$$
We have $\pi_{*,0}(\gR) = \bF_2[r,b]/(rb)$, so $r \cdot -$ acts non-nilpotently here and so $x_1=r$ is not redundant. But now $\pi_{*,0}(\gR/r) = \bF_2[r,b]/(rb, r) = \bF_2[b]$ and $b \cdot -$ acts non-nilpotently here, so $x_2=b$ is also not redundant.

However, if we take an alternative sequence $x_i'$ starting
$$r+b, b, Q^1(r), Q^1(b), \rho, \ldots$$
then $\pi_{*,0}(\gR/(r+b)) = \bF_2[r,b]/(rb, r+b) = \bF_2[b]/(b^2)$ and $b \cdot -$ acts nilpotently here. In fact, using the spectral sequence for the canonical multiplicative filtration one may easily calculate that $\pi_{n,d}(\gR/(r+b))=0$ for $d < \tfrac{1}{2}(n-1)$, so any $\gR$-module endomorphism of $\gR/(r+b)$ of slope $< \tfrac{1}{2}$ acts nilpotently: in particular $b \cdot -$ does. So now $x'_2 = b$ is redundant, and can proceed to $x'_3$. 
\end{example}

We wish to explain how to make systematic the kind of improvement pointed out in this example. This will include changing basis for the spaces $X_{N,D,F} \subset X_{*,*,*} \subset V_{*,*,*}$, as in this example, but also something more general than just changing basis. However, as we will explain in Example \ref{ex:Inefficiency2} at the end of this subsection, this is still not completely satisfactory from the point of view of efficiency of Smith--Toda complexes. In Section \ref{sec:HigherStabMapsRevisited} we will explain a different method which, when it applies, allows us to produce truly efficient Smith--Toda complexes.

\subsubsection{Nilpotence of canonical self-maps}

For fixed slopes $(\lambda, \mu)$, let $x_i, \ldots, x_j$ be the $x$'s of slopes $(\lambda, \mu)$, which we recall are a homogeneous basis of 
$$X_{\lambda, \mu} := \bigoplus_{\substack{D/N=\lambda \\ F/D=\mu}} X_{N,D,F}.$$
By Theorem \ref{thm:xiSelfMapsExist} a $\fil_*\gM \in \mathsf{W}(\lambda, \mu)$ admits central $x_i, \ldots, x_j$ self-maps $\phi_i, \ldots, \phi_j$, and our goal in this section is to define and study a radical ideal
$$I_{\lambda, \mu}(\fil_* \gM) \subseteq \bk[x_i, \ldots, x_j] = \mathrm{Sym}_\bk^*[ X_{\lambda, \mu}]$$
which roughly speaking records the polynomials in $\phi_i, \ldots, \phi_j$ for which the induced self-map has an iterate which is $\tau$-power torsion. Importantly, we will see that this ideal does not depend on the choice of homogeneous basis $x_i, \ldots, x_j$ of $X_{\lambda, \mu}$.

We define $I_{\lambda, \mu}(\fil_* \gM)$ as the kernel of a ring homomorphism
$$\chi : \mathrm{Sym}_\bk^*[X_{\lambda, \mu}] = \bk[x_i, \ldots, x_j]  \lra Z(\tau^{-1} \pi_{*,*,*}(\End_{\fil_*\gR}(\fil_*\gM)))_\text{perf},$$
whose target is the (colimit) perfection of the centre of the $\tau$-inverted endomorphism ring of $\fil_*\gM$. This is obtained by choosing $\phi_i, \ldots, \phi_j \in Z(\pi_{*,*,*}(\End_{\fil_*\gR}(\fil_*\gM)))$ with $C\tau \otimes \phi_k = u_*(x_k^{p^{M_k}})$, using Theorem \ref{thm:xiSelfMapsExist} and Theorem \ref{thm:xiSelfMapsUniqueAndNatural} (\ref{it:Centre}), using these to form the homomorphism
$$\chi^\text{pre} : \bk[x_i^{p^{M_i}}, \ldots, x_j^{p^{M_j}}] \lra Z(\tau^{-1}\pi_{*,*,*}(\End_{\fil_*\gR}(\fil_*\gM)))$$
given by $x_k^{p^{M_k}} \mapsto \phi_k$, then taking perfections and precomposing with 
$$x_k \mapsto (x_k^{p^{M_k}})^{1/p^{M_k}} : \bk[x_i, \ldots, x_j] \lra \bk[x_i^{p^{M_i}}, \ldots, x_j^{p^{M_j}}]_\text{perf}.$$
So concretely $\chi(x_k) := (\phi_k)^{1/p^{M_k}}$. It is elementary to check using Theorem \ref{thm:xiSelfMapsUniqueAndNatural} that this definition of $\chi$ does not depend on the choice of self-maps $\phi_i, \ldots, \phi_j$, and---when its domain is viewed as the symmetric algebra on $X_{\lambda, \mu}$---does not depend on the choice of homogeneous basis $x_i, \ldots, x_j$  for $X_{\lambda, \mu}$.

\begin{lem}\label{lem:NilIdealProperties}\mbox{}
\begin{enumerate}[(i)]
\item\label{it:NilIdealProperties:1} If $\fil_*\gM$ and $\fil_*\gN$ are in $\mathsf{W}(\lambda, \mu)$ and $\fil_*\gN$ is in the thick subcategory generated by $\fil_*\gM$ and $\mathsf{T}$, then
$$I_{\lambda, \mu}(\fil_*\gM) \subseteq I_{\lambda, \mu}(\fil_*\gN).$$

\item\label{it:NilIdealProperties:2} If $\fil_*\gR/(\phi_1, \ldots, \phi_{i-1})$ and $\fil_*\gR/(\phi'_1, \ldots, \phi'_{i-1})$ are two Smith--Toda complexes obtained by iteratedly killing a sequence of $x_1, \ldots, x_{i-1}$ self-maps, then $I_{\lambda, \mu}(\fil_*\gR/(\phi_1, \ldots, \phi_{i-1})) = I_{\lambda, \mu}(\fil_*\gR/(\phi'_1, \ldots, \phi'_{i-1}))$. In view of this we write
$$I_{\lambda, \mu} := I_{\lambda, \mu}(\fil_*\gR/(\phi_1, \ldots, \phi_{i-1})).$$

\item\label{it:NilIdealProperties:3} If $\fil_*\gM \in \mathsf{W}(\lambda, \mu)$ then $I_{\lambda, \mu} \subseteq I_{\lambda, \mu}(\fil_* \gM)$.

\item\label{it:NilIdealProperties:4} If $\fil_*\gM \in \mathsf{W}(\lambda, \mu)$ and $x_{j+1}$ has slopes $(\bar{\lambda}, \bar{\mu})$, then $\fil_*\gM \in \mathsf{W}(\bar{\lambda}, \bar{\mu})$ if and only if $I_{\lambda, \mu}(\fil_*\gM) = (x_i, \ldots, x_j)$.

\end{enumerate}
\end{lem}

\begin{proof}
For (\ref{it:NilIdealProperties:1}), note that if $y \in I_{\lambda, \mu}(\fil_*\gM)$ then $\fil_*\gM \in \mathsf{N}(y)$, so by Lemma \ref{lem:NilProperties} it follows that $\fil_*\gN \in \mathsf{N}(y)$ too, i.e.\ $y \in I_{\lambda, \mu}(\fil_*\gN)$.

For (\ref{it:NilIdealProperties:2}), by Lemma \ref{lem:Generation} each of $\fil_*\gR/(\phi_1, \ldots, \phi_{i-1})$ and $\fil_*\gR/(\phi'_1, \ldots, \phi'_{i-1})$ lies in the thick subcategory generated by the other and $\mathsf{T}$, so the claim follows by (\ref{it:NilIdealProperties:1}).

For (\ref{it:NilIdealProperties:3}) we apply Lemma \ref{lem:Generation} and (\ref{it:NilIdealProperties:1}).

For the reverse direction of (\ref{it:NilIdealProperties:4}) first observe that $\mathsf{W}(\bar{\lambda}, \bar{\mu})$ is generated as a thick subcategory by $\fil_*\gR/(\phi_1, \ldots, \phi_{j})$ and $\mathsf{T}$ by Lemma \ref{lem:Generation}, so if $\fil_*\gM \in \mathsf{W}(\bar{\lambda}, \bar{\mu})$ then by (\ref{it:NilIdealProperties:1}) we have
$$I_{\lambda, \mu}(\fil_*\gR/(\phi_1, \ldots, \phi_{j})) \subseteq I_{\lambda, \mu}(\fil_*\gM).$$
But it follows from Lemma \ref{lem:NilProperties} (\ref{it:NilProperties:3}) that $x_i, \ldots, x_j \in I_{\lambda, \mu}(\fil_*\gR/(\phi_1, \ldots, \phi_{j}))$, so $I_{\lambda, \mu}(\fil_*\gM) = (x_i, \ldots, x_j)$ as required.

For the forwards direction of (\ref{it:NilIdealProperties:4}) suppose that $\fil_*\gM \in \mathsf{W}(\lambda, \mu)$ and $I_{\lambda, \mu}(\fil_*\gM) = (x_i, \ldots, x_j)$. Firstly, the $\fil_*\gR$-module $\fil_*\gR/(\phi_1, \ldots, \phi_{k-1})$ is finite, so the objects $\fil_*\gM \otimes_{\fil_*\gR} \fil_*\gR/(\phi_1, \ldots, \phi_{k-1})$ all lie in the thick subcategory generated by $\fil_*\gM$, so in particular lie in $\mathsf{W}(\lambda, \mu)$. Thus for $1 \leq k \leq i-1$ the endomorphism $\mathrm{Id} \otimes\phi_k$ of $\fil_*\gM \otimes_{\fil_*\gR} \fil_*\gR/(\phi_1, \ldots, \phi_{k-1})$ has an iterate which is $\tau$-power torsion. By Lemma \ref{lem:NilpConeDoesNotLeave} it follows that $\fil_*\gM$ lies in the thick subcategory generated by $\fil_*\gM \otimes_{\fil_*\gR} \fil_*\gR/(\phi_1, \ldots, \phi_{i-1})$ and $\mathsf{T}$. Secondly, as $\fil_*\gM \otimes_{\fil_*\gR} \fil_*\gR/(\phi_1, \ldots, \phi_{i-1})$ lies in the thick subcategory generated by $\fil_*\gM$, it follows from (\ref{it:NilIdealProperties:1}) that
$$(x_i, \ldots, x_j) = I_{\lambda, \mu}(\fil_*\gM) \subseteq I_{\lambda, \mu}(\fil_*\gM \otimes_{\fil_*\gR} \fil_*\gR/(\phi_1, \ldots, \phi_{i-1})),$$
i.e.\ for $i \leq k \leq j$ the endomorphism $\mathrm{Id} \otimes \phi_k$ of $\fil_*\gM \otimes_{\fil_*\gR} \fil_*\gR/(\phi_1, \ldots, \phi_{k-1})$ has an iterate which is $\tau$-power torsion. By Lemma \ref{lem:NilpConeDoesNotLeave} it follows that the object $\fil_*\gM \otimes_{\fil_*\gR} \fil_*\gR/(\phi_1, \ldots, \phi_{i-1})$ lies in the thick subcategory generated by $\fil_*\gM \otimes_{\fil_*\gR} \fil_*\gR/(\phi_1, \ldots, \phi_{j})$ and $\mathsf{T}$. Thirdly, as $\fil_*\gM$ is finite it follows that $\fil_*\gM \otimes_{\fil_*\gR} \fil_*\gR/(\phi_1, \ldots, \phi_{j})$ lies in the thick subcategory generated by $\fil_*\gR/(\phi_1, \ldots, \phi_{j})$, so in particular lies in $\mathsf{W}(\bar{\lambda}, \bar{\mu})$. Putting this all together, it follows that $\fil_*\gM$ lies in $\mathsf{W}(\bar{\lambda}, \bar{\mu})$.
\end{proof}

\subsubsection{The elements $y_i$}\label{sec:RevisitingXi}

In view of the discussion in the previous section, we may improve upon the list of elements $x_1, x_2, \ldots$. Namely, given slopes $(\lambda, \mu)$ we choose elements 
$$y_1^{\lambda, \mu}, \ldots, y^{\lambda, \mu}_{s_{\lambda, \mu}} \in \mathrm{Sym}_\bk^*[X_{\lambda, \mu}]$$
such that the composition
$$\bk[y_1^{\lambda, \mu}, \ldots, y^{\lambda, \mu}_{s_{\lambda, \mu}}] \lra \mathrm{Sym}_\bk^*[X_{\lambda, \mu}] \lra \mathrm{Sym}_\bk^*[X_{\lambda, \mu}]/I_{\lambda, \mu}$$
is a Noether normalisation (i.e.\ is injective and the target is module-finite over the source). The number $s_{\lambda, \mu}$ is the Krull dimension of $\mathrm{Sym}_\bk^*[X_{\lambda, \mu}]/I_{\lambda, \mu}$, which by the discussion in the previous section is canonically associated to $\gR$. The elements $y_1^{\lambda, \mu}, \ldots, y^{\lambda, \mu}_{s_{\lambda, \mu}}$ will serve as a replacement for the basis $x_i, \ldots, x_j$ of $X_{\lambda, \mu}$. There may be fewer of them: this is to be considered an advantage. We let
$$y_1, y_2, y_3, \ldots$$
 be the sequence obtained by concatenating the bases $y_1^{\lambda, \mu}, \ldots, y^{\lambda, \mu}_{s_{\lambda, \mu}}$ in the lexicographic order of $(\lambda, \mu)$. Let the tridegree of $y_i$ be denoted by $(N_i, D_i, F_i)$.

\begin{thm}\label{thm:EfficientSTComplexes}
Let $\lambda  < 1$ be given, and let $y_1, \ldots, y_r$ be those $y$'s with $\tfrac{D_i}{N_i} < \lambda$. Then there is a sequence of endomorphisms
$$\psi_s : S^{p^{M_s} N_s, p^{M_s} D_s, p^{M_s} F_s} \otimes \fil_*\gR/(\psi_1, \ldots, \psi_{s-1}) \lra \fil_*\gR/(\psi_1, \ldots, \psi_{s-1})$$
for $s=1, \ldots, r$ such that:
\begin{enumerate}[(i)]
\item\label{it:EfficientSTComplexes:1} $\psi_s$ is a $y_s$ self-map.
\item\label{it:EfficientSTComplexes:2} $\fil_*\gR/(\psi_1, \ldots, \psi_{s-1})$ lies in $\mathsf{W}(\tfrac{D_{s}}{N_{s}}, \tfrac{F_{s}}{D_{s}})$ for all $0 \leq s-1 \leq r$. 

\item\label{it:EfficientSTComplexes:4} $\pi_{*,*}(\tau^{-1} \fil_*\gR/(\psi_1, \ldots, \psi_r))$ has a vanishing line of slope $\tfrac{D_{r+1}}{N_{r+1}}$ and so in particular of slope $\lambda$.

\item\label{it:EfficientSTComplexes:5} If $M$ is pre-specified then the $\psi_s$ may be chosen so that $M_s \geq M$ for all $s$.

\end{enumerate}
\end{thm}
\begin{proof}
Observe that (\ref{it:EfficientSTComplexes:4}) follows from (\ref{it:EfficientSTComplexes:2}) by setting $r-1=s$ and inverting $\tau$. 

We will make a sequence $\fil_*\gR/(\psi_1, \ldots, \psi_{s})$ for all $s$, satisfying (\ref{it:EfficientSTComplexes:1}), (\ref{it:EfficientSTComplexes:2}), and (\ref{it:EfficientSTComplexes:5}). We do this inductively, so suppose for the inductive step that $\fil_*\gR/(\psi_1, \ldots, \psi_{s-1})$ exists satisfying (\ref{it:EfficientSTComplexes:1}), (\ref{it:EfficientSTComplexes:2}), and (\ref{it:EfficientSTComplexes:5}). By (\ref{it:EfficientSTComplexes:2}) it lies in $\mathsf{W}(\tfrac{D_{s}}{N_{s}}, \tfrac{F_{s}}{D_{s}})$, so by Theorem \ref{thm:xiSelfMapsExist} it admits a $y_s$ self-map $\psi_s$. By raising it to a further $p$-th power, we may suppose that $M_s \geq M$ as required by (\ref{it:EfficientSTComplexes:5}). Using this we may form $\fil_*\gR/(\psi_1, \ldots, \psi_{s})$, which satisfies (\ref{it:EfficientSTComplexes:1}) by definition, and lies in $\mathsf{W}(\tfrac{D_{s}}{N_{s}}, \tfrac{F_{s}}{D_{s}})$ as this is a thick subcategory. 

If $\tfrac{D_{s}}{N_{s}} = \tfrac{D_{s+1}}{D_{s+1}}$ and $\tfrac{F_{s}}{D_{s}}=\tfrac{F_{s+1}}{D_{s+1}}$ then this means that (\ref{it:EfficientSTComplexes:2}) is satisfied. If not, then the sequence $y_1, \ldots, y_s$ is the union of sequences $y^1_{\lambda, \mu}, \ldots, y_{\lambda, \mu}^{s_{\lambda, \mu}}$ for all $(\lambda, \mu) < (\tfrac{D_{s+1}}{N_{s+1}}, \frac{F_{s+1}}{D_{s+1}})$. Let $x_1, \ldots, x_i$ be the $x$'s with slopes $< (\tfrac{D_{s+1}}{N_{s+1}}, \frac{F_{s+1}}{D_{s+1}})$, and consider
$$\fil_*\gR/(\psi_1, \ldots, \psi_{s}) \otimes_{\fil_*\gR} \fil_*\gR/(\phi_1, \ldots, \phi_{i}).$$
The second factor is in the thick subcategory $\mathsf{W}(\tfrac{D_{s+1}}{N_{s+1}}, \tfrac{F_{s+1}}{D_{s+1}})$ by Lemma~\ref{lem:Generation}~(\ref{it:Generation:1}), so the whole object is too. 

\begin{claim}
For each $1 \leq k \leq i$ the endomorphism $\mathrm{Id} \otimes \phi_k$ of $\fil_*\gR/(\psi_1, \ldots, \psi_{s}) \otimes_{\fil_*\gR} \fil_*\gR/(\phi_1, \ldots, \phi_{k-1})$ has an iterate which is $\tau$-power torsion.
\end{claim}

\begin{proof}[Proof of Claim]
There are two cases. If $x_k$ has slopes $< (\tfrac{D_{s}}{N_{s}}, \tfrac{F_{s}}{D_{s}})$, then because $\fil_*\gR/(\psi_1, \ldots, \psi_{s}) \otimes_{\fil_*\gR} \fil_*\gR/(\phi_1, \ldots, \phi_{k-1}) \in \mathsf{W}(\tfrac{D_{s}}{N_{s}}, \tfrac{F_{s}}{D_{s}})$ (as the first factor is), it immediately follows that some iterate of $\mathrm{Id} \otimes \phi_k$ is $\tau$-power torsion. So consider those $x_j, \ldots, x_i$ which have slopes $(\tfrac{D_{s}}{N_{s}}, \tfrac{F_{s}}{D_{s}})$, and let $y_t, \ldots, y_s$ be the $y$'s of the same slopes. We make two observations about the ideal $I_{D_s/N_s, F_s/D_s}(\fil_*\gR/(\psi_1, \ldots, \psi_{s}))$:
\begin{enumerate}[(i)]
\item $y_t, \ldots, y_s \in I_{D_s/N_s, F_s/D_s}(\fil_*\gR/(\psi_1, \ldots, \psi_{r}))$ by Lemma \ref{lem:NilProperties} (\ref{it:NilProperties:3}).

\item By Lemma \ref{lem:NilIdealProperties} (\ref{it:NilIdealProperties:4}), $I_{D_s/N_s, F_s/D_s} \subseteq I_{D_s/N_s, F_s/D_s}(\fil_*\gR/(\psi_1, \ldots, \psi_{r}))$.
\end{enumerate}
As $y_t, \ldots, y_s$ was by definition chosen to be a Noether normalisation for the algebra $\bk[x_j, \ldots, x_i]/I_{D_s/N_s, F_s/D_s}$, and $I_{D_s/N_s, F_s/D_s}(\fil_*\gR/(\psi_1, \ldots, \psi_{s}))$ is a radical ideal, it follows that $I_{D_s/N_s, F_s/D_s}(\fil_*\gR/(\psi_1, \ldots, \psi_{s})) = (x_j, \ldots, x_i)$. By Lemma \ref{lem:NilIdealProperties} (\ref{it:NilIdealProperties:1}) it follows that
$$I_{D_s/N_s, F_s/D_s}(\fil_*\gR/(\psi_1, \ldots, \psi_{s}) \otimes_{\fil_*\gR} \fil_*\gR/(\phi_1, \ldots, \phi_{j-1}))=(x_j, \ldots, x_i)$$
too, so the endomorphisms $\phi_j, \ldots, \phi_i$ have iterates which are $\tau$-power torsion too.
\end{proof}

Using this claim, several applications of Lemma \ref{lem:NilpConeDoesNotLeave} gives that $\fil_*\gR/(\psi_1, \ldots, \psi_{s})$ lies in the thick subcategory generated by $\fil_*\gR/(\psi_1, \ldots, \psi_{s}) \otimes_{\fil_*\gR} \fil_*\gR/(\phi_1, \ldots, \phi_{i})$ and $\mathsf{T}$, so lies in $\mathsf{W}(\tfrac{D_{s+1}}{N_{s+1}}, \tfrac{F_{s+1}}{D_{s+1}})$, which finishes the proof of (\ref{it:EfficientSTComplexes:2}).
\end{proof}

In the sequence of endomorphisms provided by Theorem \ref{thm:EfficientSTComplexes}, it may be that a $\psi_s$ is nilpotent up to $\tau$-power torsion as an endomorphism of $\fil_*\gR/(\psi_1, \ldots, \psi_{s-1})$. By the argument given in Section \ref{sec:Omitting}, these may be omitted.

\subsection{Scope for further efficiency}

The following example indicates that the discussion so far is still not completely satisfactory from the point of view of efficiency of Smith--Toda complexes.

\begin{example}\label{ex:Inefficiency2}
With $\bk=\bF_2$ let $\gR := \gE_\infty(S^{1,0}\sigma \oplus S^{2,1}x) \cup^{E_\infty}_{x \cdot Q^1(\sigma)} D^{4,3}\rho$, and $\fil_*\gR$ denote its canonical $E_\infty$-multiplicative filtration. Write $\hat{\sigma} \in \pi_{1,0,-1}(\fil_*\gR)$ and $\hat{x} \in \pi_{2,1,-1}(\fil_*\gR)$ for filtered representatives of $\sigma$ and $x$ of (minimal) filtration $-1$, and write $\sigma', x' \in \pi_{*,*,*}(C\tau \otimes \fil_*\gR)$ for their reductions modulo $\tau$.  For the $x_i$'s we may take a sequence starting
$$\sigma', Q^1(\sigma'), x', Q^2Q^1(\sigma'), Q^2(x'), \ldots \in \pi_{*,*,*}(C\tau \otimes \fil_*\gR).$$
Certainly $\sigma \cdot -$ is non-nilpotent on $\gR$, and we must form $\fil_*\gR/\hat{\sigma}$. Now $Q^1(\sigma')$ has tridegree $(2,1,-2)$ and $x'$ has tridegree $(2,1,-1)$, and these are the only $x_i$'s of $\tfrac{d}{n}$-slope $\tfrac{1}{2}$. As they have different tridegrees they cannot be considered together and must be considered in increasing order of $\tfrac{f}{d}$-slope.

We must first construct a $Q^1(\sigma')$ self-map of $\fil_*\gR/\hat{\sigma}$, which we may take to be the filtered map $Q^1(\hat{\sigma}) \cdot -$ as $\hat{\sigma}$ lifts $\sigma'$. By base-change along the $E_\infty$-map $\gR \to \gE_\infty(S^{1,0}\sigma)$ which kills $x$ and $\rho$, we see that $Q^1(\sigma) \cdot -$ is non-nilpotent on $\gR/\sigma$.

We now construct a $x'$ self-map of $\fil_*\gR/(\hat{\sigma}, Q^1(\hat{\sigma}))$, which can be taken to be $\hat{x} \cdot -$ because $\hat{x}$ lifts $x'$. By base-change along the $E_\infty$-map $\gR \to \gE_\infty(S^{2,1}x)$ which kills $\sigma$ and $\rho$, one sees that $x \cdot -$ is non-nilpotent on $\gR/(\sigma, Q^1(\sigma))$. Then $\gR/(\sigma, Q^1(\sigma), x)$ has a slope $\tfrac{1}{2}$ vanishing line.

But this example is highly similar to Example \ref{ex:RedAndBlueModRedBlue}, in that $\gR/(\sigma, Q^1(\sigma) + x)$ already has a slope  $\tfrac{1}{2}$ vanishing line. This is because $(Q^1(\sigma) + x)\cdot -$ is nilpotent on $\gR/(\sigma, Q^1(\sigma) + x)$, but also $x \cdot Q^1(\sigma) =0 \in \pi_{*,*}(\gR)$, and together these show that $Q^1(\sigma) \cdot -$ and $x \cdot -$ are both nilpotent on $\gR/(\sigma, Q^1(\sigma) + x)$.

However the slope $\frac{1}{2}$ Smith--Toda complex $\gR/(\sigma, Q^1(\sigma) + x)$ can never arise from the mechanism we have discussed so far: the mechanism obliges us to kill $Q^1(\sigma') \cdot - = (C\tau \otimes u)_*(Q^1(\sigma'))$ first because it has lower $\tfrac{f}{d}$-slope than $x'$ and is non-nilpotent, so if we do not kill it then we will not generally know that the element $(C\tau \otimes u)_*(x')$ of higher $\tfrac{f}{d}$-slope is a permanent cycle in the spectral sequence for $\End_{\fil_*\gR}(\fil_*\gR/\hat{\sigma})$. But in this example we see that $(C\tau \otimes u)_*(x')$ is indeed a permanent cycle in this spectral sequence as $u_*(\hat{x})$ is a filtered homotopy class lifting it, and furthermore we see that choosing the lift $u_*(\hat{x} + Q^1(\hat{\sigma})) \in \pi_{2,1,-1}(\End_{\fil_*\gR}(\fil_*\gR/\hat{\sigma}))$ to a filtered homotopy class is wiser than choosing the lift $u_*(\hat{x}) \in \pi_{2,1,-1}(\End_{\fil_*\gR}(\fil_*\gR/\hat{\sigma}))$.
\end{example}

\section{Changing rings and detecting nilpotence}

Which $x_i$ can be omitted, the ideals $I_{\lambda, \mu}$, and which $y_i$ may be omitted, are all described in terms of nilpotence of endomorphisms up to $\tau$-power torsion, or equivalently in terms of nilpotence after inverting $\tau$. To make working with these concepts effective, we need tools for testing the (non-)nilpotence of endomorphisms of $\gR$-modules. The following gives us a way to check nilpotence after base-changing along a suitable map $\gR \to \gS$, which can be used to reduce the question to working over a simpler $E_k$-algebra. We will explain in Section \ref{sec:StabHopfAlg} that there is a very good choice of $\gS$ to which we can apply this theorem.

\begin{thm}\label{thm:DetNilp}
Let $\bk$ be a commutative ring, $\gR \to \gS$ be a morphism in $\mathsf{Alg}_{E_2}(\mathsf{D}(\bk)^\bZ)$ between objects satisfying (C), and suppose that $\theta \geq 0$ is such that
\begin{equation}\label{eq:DetAss}\tag{$\dagger$}
H_{n,d}^{\gR}(\gS) \text{ is just $\bk[0,0]$ in the range of bidegrees $d < \theta n + 1$ or $n < 1$}.
\end{equation}
\begin{enumerate}[(i)]

\item\label{it:DetNilp:1} Let $\gM$ be an $\gR$-module, $\lambda \geq 0$ and $\kappa$ be given, and $N$ be such that $\pi_{n,d}(\gM)=0$ for $n < N$. Suppose that $\theta \geq \lambda$. Then 
\begin{equation*}
\pi_{n,d}(\gM)=0 \text{ for } d < \lambda n + \kappa \iff \pi_{n,d}(\gS \otimes_\gR \gM)=0 \text{ for } d < \lambda n + \kappa.
\end{equation*}
If these (equivalent) conditions hold, then the map
$$\pi_{n,d}(\gM) \lra \pi_{n,d}(\gS \otimes_\gR \gM)$$
is an epimorphism for $d < \lambda n + \kappa + (\theta-\lambda)+1$, and an isomorphism for $d < \lambda n + \kappa + (\theta-\lambda)$.

\item\label{it:DetNilp:2} Let $\gE$ be a unital associative algebra in (the homotopy category of) $\gR\text{-}\mathsf{mod}$ and $\lambda \geq 0$, $\kappa$, and $N$ be such that $\pi_{n,d}(\gE)=0$ for $d < \lambda n + \kappa$ or for $n < N$. Suppose that $\theta > \lambda$. Then for $\tfrac{d}{n} = \lambda$ elements in the kernel of
\begin{equation*}
\pi_{n,d}(\gE) \lra \pi_{n,d}(\gS \otimes_\gR \gE)
\end{equation*}
are nilpotent.

\item\label{it:DetNilp:3} In the setting of (\ref{it:DetNilp:2}), suppose also that $\bk$ is a field of positive characteristic $p$. If $x \in \pi_{n,d}(\gS)$ has slope $\tfrac{d}{n} = \lambda$ then some $p$-th power of $(\gS \otimes_\gR u)_*(x) \in \pi_{n,d}(\gS \otimes_\gR \gE)$ lifts to a $\psi \in \pi_{*,*}(\gE)$. 

\item \label{it:DetNilp:4} In the setting of (\ref{it:DetNilp:3}), if $\gE$ is in addition a finite $\gR$-module then some $p$-th power of the lift $\psi$ lies in the centre of $\pi_{*,*}(\gE)$, and any two lifts $\psi$ and $\psi'$ agree after taking further $p$-th powers.
\end{enumerate}
\end{thm}

We will prove this theorem by using the $\gS$-based Adams spectral sequence in the category of $\gR$-modules, so before giving the proof we begin with a general discussion of this spectral sequence.

We choose to construct this spectral sequence using the ``canonical'' Adams resolution, as follows. Let $\overline{\gS}$ denote the fibre of $\gR \to \gS$, an $\gR$-module, and set
\begin{align*}
F_s\gM &:= \overline{\gS}^{\otimes_\gR s} \otimes_\gR \gM\\
K_s \gM &:= \gS \otimes_\gR \overline{\gS}^{\otimes_\gR s} \otimes_\gR \gM
\end{align*}
to obtain the tower
\begin{equation}\label{eq:DescentASSTower}
\begin{tikzcd}
\gM \arrow[r, equals] & F_0 \gM \dar & F_1 \gM \lar \dar& F_2 \gM \lar \dar& \cdots\lar\\
& K_0 \gM & K_1 \gM & K_2 \gM
\end{tikzcd}
\end{equation}
Taking homotopy groups yields an exact couple having
$$E_1^{n,s,t} := \pi_{n, t-s}(K_s \gM).$$
This exact couple gives a spectral sequence
\begin{equation}\label{eq:DescentASS}
E_1^{n,s,t}(\gM) \Longrightarrow \pi_{n,t-s}(\gM),
\end{equation}
with differentials $d_r : E_r^{n,s,t}(\gM) \to E_r^{n, s+r, t+r-1}(\gM)$.

There are pairings $E_r^{n,s,t}(\gM) \otimes_\bk E_r^{n', s', t'}(\gN) \to E_r^{n+n', s+s', t+t'}(\gM \otimes_\gR \gN)$ of spectral sequences converging to the tensor product map $\pi_{n,t-s}(\gM) \otimes_\bk \pi_{n',t'-s'}(\gN) \to \pi_{n+n', t+t'-s-s'}(\gM \otimes_\gR \gN)$. This is established precisely as for the classical Adams spectral sequence, e.g.\ \cite[\S 4]{Adams58}. In particular, if $\gE$ is equipped with a product $\mu : \gE \otimes_\gR \gE \to \gE$ then it induces a product on $\{E_r^{*,*,*}(\gE)\}_r$, enjoying the same unitality or associativity properties as $\gE$.

Strong convergence is guaranteed in our setting by the following (which can certainly be improved upon, but suffices for our kinds of applications).

\begin{lem}\label{lem:AdamsSSConverges}
If $(\dagger)$ holds (for any $\theta \geq 0$) and there is an $N$ such that $\pi_{n,d}(\gM)=0$ for $n < N$ then this spectral sequence converges strongly.
\end{lem}
\begin{proof}
We refer to Boardman \cite{Boardman}: in the language of that paper, this is a half-plane spectral sequence with entering differentials. We will use \cite[Theorem 7.3]{Boardman}. We will first show that $\pi_{n,d}(F_s \gM)=0$ for all $s \gg 0$, implying that the spectral sequence is conditionally convergent. We will then show that $E_1^{n,s,d+s}(\gM) = \pi_{n,d}(K_s\gM) = 0$ for all $s \gg 0$, implying that there are only finitely-many differentials out of each position, and hence that $RE_\infty=0$. By \cite[Theorem 7.3]{Boardman} the spectral sequence converges strongly.

From $(\dagger$) we see that $H^\gR_{0,d}(\gS)$ is just $\bk$ for $d=0$, and $H^\gR_{n,0}(\gS)$ is just $\bk$ for $n=0$. From this it follows that $\gS$ can be constructed as an $\gR$-module from $\gR$ by attaching $(n,d)$-$\gR$-cells with $n > 0$ and $d>0$. Taking the skeletal filtration of such a model, we have associated graded
$$\mathrm{gr}(\gS) \simeq \gR \oplus \bigoplus_{\alpha} S^{n_\alpha, d_\alpha} \otimes_\bk \gR$$
where $n_\alpha >0$ and $d_\alpha >0$. It gives an associated filtration of $\overline{\gS}$ with $\mathrm{gr}(\overline{\gS}) \simeq \bigoplus_{\alpha} S^{n_\alpha, d_\alpha-1} \otimes_\bk \gR$, and hence a filtration of $F_s\gM$ with
\begin{align*}
\mathrm{gr}(F_s \gM) \simeq  (\bigoplus_{\alpha} S^{n_\alpha, d_\alpha-1} \otimes_\bk \gR)^{\otimes_\gR s} \otimes_\gR \gM \simeq (\bigoplus_{\alpha} S^{n_\alpha, d_\alpha-1} )^{\otimes_\bk s} \otimes_\bk \gM.
\end{align*}
As $n_\alpha > 0$ and $d_\alpha-1 \geq 0$, $ (\bigoplus_{\alpha} S^{n_\alpha, d_\alpha-1} )^{\otimes s}$ is connective and is supported in gradings $\geq s$. Under the given assumption on $\gM$ it follows that $\pi_{n,d}(\mathrm{gr}(F_s \gM))=0$ for $n < N+s$, so by strong convergence of the skeletal spectral sequence that $\pi_{n,d}(F_s \gM)=0$ for $n < N+s$, as required.

The analogous filtration for $K_s\gM$ shows that $E_1^{n, s, d+s}(\gM) = \pi_{n,d}(K_s\gM)=0$ for $n < N+s$, showing that $RE_\infty=0$ as required.
\end{proof}

Before starting on the proof of Theorem \ref{thm:DetNilp}, we have the following.

\begin{lem}\label{lem:OverlineSVanish}
If $(\dagger)$ holds then $\pi_{n,d}(\overline{\gS} \otimes_\gR \bk)=0$ for $n < 1$ or $d < \theta n$.
\end{lem}
\begin{proof}
We use the cofibre sequence
$$\overline{\gS} \otimes_\gR \bk \lra {\gR} \otimes_\gR \bk (\simeq \bk) \lra \gS \otimes_\gR \bk$$
and its long exact sequence
$$\cdots\pi_{n,d+1}(\bk) \lra  \pi_{n,d+1}(\gS \otimes_\gR \bk) \lra \pi_{n,d}(\overline{\gS} \otimes_\gR \bk) \lra \pi_{n,d}(\bk) \lra \pi_{n,d}(\gS \otimes_\gR \bk) \cdots,$$
in which $\pi_{0,0}(\bk) = \bk$ cancels against $\pi_{0,0}(\gS \otimes_\gR \bk)=\bk$.
\end{proof}

\begin{proof}[Proof of Theorem \ref{thm:DetNilp}]
For the forwards implication of (\ref{it:DetNilp:1}), we consider $\gS \otimes_\gR \gM$ and to begin with we Postnikov filter $\gM$, giving a filtered object $\tau_{\geq -*} \gM$. As $\gR$ is connective, this is a filtration by $\gR$-modules. Its associated graded is $\gr(\tau_{\geq -*} \gM) = \bigoplus_{d \in \bZ} \Sigma^d \pi_{*,d}(\gM)$, where the $d$-th summand is in grading $-d$. This filtration induces a filtration of $\gS \otimes_\gR \gM$ having
$$\gr(\gS \otimes_\gR \tau_{\geq -*} \gM) = \bigoplus_{d \in \bZ} \gS \otimes_\gR \Sigma^d \pi_{*,d}(\gM).$$
Here $\pi_{*,d}(\gM)$ has the structure of an $\gR$-module via the truncation $\gR \to \tau_{\leq 0}\gR = \pi_{*,0}(\gR)$, so we may write the above as
\begin{equation}\label{eq:PartiallyFiltered}
\bigoplus_{d \in \bZ} (\gS \otimes_\gR \pi_{*,0}(\gR)) \otimes_{\pi_{*,0}(\gR)} \Sigma^d \pi_{*,d}(\gM).
\end{equation}
We now filter each $\pi_{*,d}(\gM)$ by its grading, i.e.\ 
$$\fil_f \pi_{n,d}(\gM) = \begin{cases}
 \pi_{n,d}(\gM) & f \geq n\\
 0 & f < n.
\end{cases}$$
As $\pi_{n,0}(\gR)=0$ for $n<0$, this is a filtration by $\pi_{*,0}(\gR)$-modules. The associated graded $\gr(\fil_f \pi_{*,d}(\gM))$ can be identified with $\pi_{*,d}(\gM)$, but the $\pi_{*,0}(\gR)$-module structure now factors over the augmentation $\epsilon : \pi_{*,0}(\gR) \to \pi_{0,0}(\gR)=\bk$. Taking the induced filtration on each summand of \eqref{eq:PartiallyFiltered}, the associated graded may be expressed by the same formula, but now the $\pi_{*,0}(\gR)$-module structure on $\Sigma^d \pi_{*,d}(\gM)$ is via the augmentation. The formula can therefore be manipulated into
$$\bigoplus_{d \in \bZ} (\gS \otimes_\gR \pi_{*,0}(\gR)) \otimes_{\pi_{*,0}(\gR)} \bk \otimes_\bk \Sigma^d \pi_{*,d}(\gM) \simeq \bigoplus_{d \in \bZ} (\gS \otimes_\gR \bk) \otimes_\bk \Sigma^d \pi_{*,d}(\gM).$$
Using the K{\"u}nneth spectral sequence to calculate the homotopy groups of the latter, and neglecting several internal gradings, we therefore have a chain of (strongly convergent) spectral sequences 
$$E^2_{n,p,q} = \Tor^\bk_{p}(\pi_{*,*}(\gS \otimes_\gR \bk), \pi_{*,*}(\gM))_{n,q} \Longrightarrow \cdots \Longrightarrow \pi_{n, p+q}(\gS \otimes_\gR \gM).$$
Assumption ($\dagger$) is precisely that $\pi_{n,d}(\gS \otimes_\gR \bk)=0$ for $n < 1$ or $d < \theta n+1$ except that it is $\bk$ for $(n,d)=(0,0)$. Using the assumption $\pi_{n,d}(\gM)=0$ for $d < \lambda n + \kappa$ too, it follows that
$$E^2_{n,p,q}=0 \text{ for } q < \min(\lambda n + \kappa, \min_{\substack{a+b=n\\ a \geq 1}}(\theta a + 1 +  \lambda b + \kappa),$$
so in particular for $p+q < \lambda n + \kappa$ using that $\theta \geq \lambda$ and $p \geq 0$. Running these spectral sequences shows that $\pi_{n,d}(\gS \otimes_\gR \gM)=0$ for $d < \lambda n + \kappa$ as required.

For the reverse implication of (\ref{it:DetNilp:1}), suppose that $\pi_{n,d}(\gS \otimes_\gR \gM)=0$ for $d < \lambda n + \kappa$. We can write
$$K_s\gM = \gS \otimes_\gR \overline{\gS}^{\otimes_\gR s} \otimes_\gR \gM = (\gS \otimes_\gR \overline{\gS}^{\otimes_\gR s}) \otimes_\gS (\gS \otimes_\gR \gM).$$
We apply the same two filtrations as above to $\gS \otimes_\gR \gM$: first Postnikov filter it, then filter it by grading. These induce iterated filtrations of $K_s\gM$, whose final associated graded is
\begin{align*}
&((\gS \otimes_\gR \overline{\gS}^{\otimes_\gR s}) \otimes_\gS \bk) \otimes_\bk \left(\bigoplus_{d \in \bZ}\Sigma^d \pi_{*,d}(\gS \otimes_\gR \gM)\right)\\
 &\quad\quad\simeq ( \overline{\gS}^{\otimes_\gR s} \otimes_\gR \bk) \otimes_\bk\left(\bigoplus_{d \in \bZ}\Sigma^d \pi_{*,d}(\gS \otimes_\gR \gM)\right)\\
&\quad\quad\simeq (\overline{\gS} \otimes_\gR \bk)^{\otimes_\bk s} \otimes_\bk \left(\bigoplus_{d \in \bZ}\Sigma^d \pi_{*,d}(\gS \otimes_\gR \gM)\right).
\end{align*}
(For the first equivalence we have used that $\gS$ is central in $\gR\text{-}\mathsf{mod}$, cf.\ Lemma \ref{lem:E2AlgebraIsModuleCentral}.) Using the K{\"u}nneth spectral sequence to calculate the homotopy groups of the latter, and neglecting several internal gradings, we therefore have a chain of (strongly convergent) spectral sequences 
$$E^2_{n,p,q} = \Tor^\bk_p(\pi_{*,*}((\overline{\gS} \otimes_\gR \bk)^{\otimes_\bk s}), \pi_{*,*}(\gS \otimes_\gR \gM))_{n,q} \Longrightarrow \cdots \Longrightarrow \pi_{n,p+q}(K_s\gM).$$
By Lemma \ref{lem:OverlineSVanish} we have $\pi_{n,d}(\overline{\gS} \otimes_\gR \bk)=0$ for $n < 1$ or $d < \theta n$, so running several K{\"u}nneth spectral sequences it follows that $\pi_{n,d}((\overline{\gS} \otimes_\gR \bk)^{\otimes_\bk s})=0$ for $n < s$ or $d <\theta n$. Using the assumption $\pi_{n,d}(\gS \otimes_\gR \gM)=0$ for $d < \lambda n + \kappa$ too, it follows that
$$E^2_{n,p,q}=0 \text{ for } q < \min_{\substack{a+b=n\\ a \geq s}}(\theta a  +  \lambda b + \kappa),$$
so in particular for $p+q < \lambda n + \kappa +s (\theta-\lambda)$, using that $\theta \geq \lambda$ and $p \geq 0$. Running these spectral sequences shows that $\pi_{n,d}(K_s\gM)=0$ for $d < \lambda n + \kappa +s (\theta-\lambda)$. By our assumption that $\pi_{n,d}(\gM)=0$ for $n<N$, Lemma \ref{lem:AdamsSSConverges} shows that the Adams spectral sequence converges strongly, so running it shows that $\pi_{n,d}(\gM)=0$ for $d < \lambda n + \kappa$ as required.

For the addendum to (\ref{it:DetNilp:1}), note that the map $\pi_{n,d}(\gM) \to \pi_{n,d}(\gS \otimes_\gR \gM) = \pi_{n,d}(K_0\gM) = E_1^{n,0,d}(\gM)$ is the edge homomorphism of the Adams spectral sequence. By strong convergence, for it to be an epimorphism it suffices that all possible targets $E_r^{n,r,d+r-1}(\gM)$ of differentials out of $E_r^{n,0,d}(\gM)$ vanish, and for this it suffices for $E_1^{n,r,d+r-1}(\gM) = \pi_{n,d-1}(K_r\gM)$ to vanish for all $r \geq 1$, and the estimate above shows this happens as long as $d-1 < \lambda n + \kappa + (\theta-\lambda)$. For it to be a monomorphism means to have $E_\infty^{n,s,d+s}(\gM)=0$ for all $s \geq 1$, and for this it suffices to have $E_1^{n,s,d+s}(\gM)=\pi_{n,d}(K_s\gM)=0$ for all $s \geq 1$, and the estimate above shows this happens as long as $d < \lambda n + \kappa + (\theta-\lambda)$.

For (\ref{it:DetNilp:2}), we consider the Adams spectral sequence
\begin{equation*}
E_1^{n,s,t}(\gE) \Longrightarrow \pi_{n,t-s}(\gE),
\end{equation*}
which is equipped with a unital and associative multiplicative structure as $\gE$ is an $E_1$-algebra in $\gR\text{-}\mathsf{mod}$. By our assumption that $\pi_{n,d}(\gE)=0$ for $n<N$, Lemma \ref{lem:AdamsSSConverges} shows that it converges strongly. We have $E_1^{n,0,d}(\gE) = \pi_{n,d}(\gS \otimes_\gR \gE)$, and the edge homomorphism 
$$\Phi : \pi_{n,d}(\gE) \lra E_1^{n,0,d}(\gE) = \pi_{n,d}(\gS \otimes_\gR \gE)$$
is the canonical map. Our assumption is that a class $\Phi(x) \in E_1^{n,0,d}(\gE)$ is nilpotent: after replacing $x$ by a power, we may assume that $\Phi(x)=0 \in E_1^{n,0,d}(\gE)$. That is, the class $x \in \pi_{n,d}(\gE)$ has $\gS$-based Adams filtration $\geq 1$. To conclude the argument, we will show:

\begin{claim}
For each $\tilde{\kappa} \in \bZ$ there is an $s_0$ such that $E_1^{\ell n, s, \ell d+\tilde{\kappa}+s}(\gE)=0$ for all $s \geq s_0$ and all $\ell$.
\end{claim}

It follows from this Claim with $\tilde{\kappa}=0$, and strong convergence of the spectral sequence, that a class in $\pi_{\ell n, \ell d}(\gE)$ of $\gS$-Adams filtration $\geq s_0$ is zero. As $x \in \pi_{n,d}(\gE)$ has $\gS$-Adams filtration $\geq 1$, the class $x^\ell \in \pi_{\ell n, \ell d}(\gE)$ has $\gS$-Adams filtration $\geq \ell$, so as long as $\ell \geq s_0$ we have $x^\ell=0$. Thus $x$ is nilpotent as required.

\begin{proof}[Proof of Claim]
By definition $E_1^{\ell n, s, \ell d+\tilde{\kappa}+s}(\gE)= \pi_{\ell n, \ell d+\tilde{\kappa}}(\gS \otimes_\gR \overline{\gS}^{\otimes_\gR s} \otimes_\gR \gE)$. Filtering $\gE$ by its grading induces a filtration of $\gS \otimes_\gR \overline{\gS}^{\otimes_\gR s} \otimes_\gR \gE$ with associated graded $(\gS \otimes_\gR \overline{\gS}^{\otimes_\gR s} \otimes_\gR \bk) \otimes_\bk \gE$. As $- \otimes_\gR \bk : \gR\text{-}\mathsf{mod} \to \bk \text{-}\mathsf{mod}$ is strong monoidal, we have
$$\gS \otimes_\gR \overline{\gS}^{\otimes_\gR s} \otimes_\gR \bk \simeq (\gS \otimes_\gR \bk) \otimes_\bk (\overline{\gS} \otimes_\gR \bk)^{\otimes_\bk s}.$$
By ($\dagger$) we have $\pi_{n,d}(\gS \otimes_\gR \bk)=0$ for $n < 0$ or $d < \theta n + 1$ except it is $\bk$ for $(n,d)=(0,0)$, and by Lemma \ref{lem:OverlineSVanish} we have $\pi_{n,d}(\overline{\gS} \otimes_\gR \bk)=0$ for $n<1$ or $d < \theta n$. Putting these together it follows that
$$\pi_{n,d}(\gS \otimes_\gR \overline{\gS}^{\otimes_\gR s} \otimes_\gR \bk)=0 \text{ for } n < s \text{ or } d < \theta n.$$
By our assumption that $\pi_{n,d}(\gE)=0$ for $d < \lambda n + \kappa$, and that $\theta \geq \lambda$ it follows that
$$\pi_{\ell n, \ell d+\tilde{\kappa}}(\gS \otimes_\gR \overline{\gS}^{\otimes_\gR s} \otimes_\gR \gE)=0 \text { for } \ell d+\tilde{\kappa} < \min_{\substack{a+b=n\ell \\ a \geq s}}(\theta a + \lambda b + \kappa)$$
so as $\theta \geq \lambda$ it in particular vanishes for $\ell d + \tilde{\kappa} < \lambda(n \ell) + \kappa + (\theta-\lambda)s$. As $\tfrac{d}{n}=\lambda$ this condition may be written as 
$$0=\ell (d - \lambda n)< \kappa - \tilde{\kappa}+(\theta-\lambda)s$$
and by our assumption that $\theta > \lambda$ this condition is satisfied for all $s \gg 0$, independently of $\ell$.
\end{proof}

For (\ref{it:DetNilp:3}), we consider the map of $\gS$-based Adams spectral sequences for the unit $u: \gR \to \gE$. This map is central, as $\gE$ is a $\gR$-algebra by assumption, so the maps $E_r^{*,*,*}(\gR) \to E_r^{*,*,*}(\gE)$ all land in the centre. In particular the class $z := (\gS \otimes_\gR u)_*(x) \in \pi_{n,d}(\gS \otimes_\gR \gE) =  E_1^{n,0,d}(\gE)$ and its powers are central in $E_1^{*,*,*}(\gE)$, and hence also in any page of the spectral sequence they survive to. We wish to show that some $p$-th power of this class lifts along the edge homomorphism $\Phi : \pi_{*,*}(\gE) \to E^{*,0,*}_1(\gE)$, which by convergence is the same as showing it is a permanent cycle. As in the proof of Theorem \ref{thm:FiltSTExist}, because $z$ is central it follows that we may find $p$-th powers of it surviving until arbitrarily late pages of the spectral sequence.

Applying the Claim with $\tilde{\kappa} = -1$, there is an $s_0$ such that $E_1^{\ell n, s, \ell d-1+s}(\gE)=0$ for all $s \geq s_0$ and all $\ell$. Choose $N$ large enough that $z^{p^N}$ survives until $E_{s_0}^{p^Nn, 0, p^Nd}(\gE)$. A potential differential $d_s(z^{p^N})$ with $s \geq s_0$ lands in $E_{s}^{p^Nn, s, p^Nd-1+s}(\gE)$ but these all vanish by our choice of $s_0$, so $z^{p^N}$ is a permanent cycle as required.

For (\ref{it:DetNilp:4}), we repeat the proof of Theorem \ref{thm:xiSelfMapsUniqueAndNatural} (\ref{it:Centre}) and (\ref{it:Unique}). Let $\psi \in \pi_{n,d}(\gE)$ be some lift of $(\gS \otimes_\gR u)_*(x^{p^N})$ provided by (\ref{it:DetNilp:3}). As $\gE$ is a finite $\gR$-module, $\End_\gR(\gE)$ lies in the thick subcategory generated by $\gE$ and so also satisfies the assumptions of (\ref{it:DetNilp:2}). The elements $\psi \circ -, - \circ \psi \in \pi_{n,d}(\End_\gR(\gE))$ map to $(\gS \otimes_\gR u)_*(x^{p^N}) \circ -, - \circ (\gS \otimes_\gR u)_*(x^{p^N}) \in \pi_{n,d}(\gS \otimes_\gR \End_\gR(\gE))$, which are equal up to a suitable sign as $(\gS \otimes_\gR u)_* : \pi_{*,*}(\gS) \to \pi_{*,*}(\gS \otimes_\gR \End_\gR(\gE))$ lands in the graded centre. By (\ref{it:DetNilp:2}) it follows that the appropriately signed difference $(\psi \circ -) \pm (- \circ \psi) \in \pi_{n,d}(\End_\gR(\gE))$ is nilpotent, then as $\psi \circ -$ and $- \circ \psi$ commute it follows that $(\psi^{p^M} \circ -) \simeq \pm(- \circ \psi^{p^M})$, i.e.\ $\psi^{p^M}$ lies in the graded centre of $\pi_{*,*}(\gE)$. Replace $\psi$ by $\psi^{p^M}$ so that it lies in the centre. If $\psi' \in \pi_{n', d'}(\gE)$ is a lift of $(\gS \otimes_\gR u)_*(x^{p^{N'}})$ then $(\psi')^{p^N} - (\psi)^{p^{N'}} \in \Ker(\pi_{*,*}(\gE) \to \pi_{*,*}(\gS \otimes_\gR \gE))$ so by (\ref{it:DetNilp:2}) again this element is nilpotent. As $\psi$ is in the centre, this means that $\psi'$ agrees with $\psi$ after taking suitable $p$-th powers.
\end{proof}

\section{The stability Hopf algebra}\label{sec:StabHopfAlg}

Recall from Section \ref{sec:BarCobar} there there is an adjunction
\begin{equation*}
\begin{tikzcd}
\Bar : {\mathsf{Alg}_{E_2}^\mathrm{aug}(\mathsf{D}(\bk)^\bZ)} \arrow[rr, ""{name=G1}, shift left=.8ex] 
& &  {\mathsf{Alg}_{E_{1}}(\mathsf{coAlg}_{E_1}^\mathrm{aug}(\mathsf{D}(\bk)^\bZ))} : \Cobar \arrow[ll, ""{name=F1}, shift left=.8ex] 
\arrow[phantom, from=F1, to=G1,  "\scalebox{0.7}{$\dashv$}" rotate=-90 ]
\end{tikzcd}
\end{equation*}
which by Proposition \ref{prop:BarCobarInverse} restricts to an equivalence of categories when we take objects satisfying suitable connectivity hypotheses on each side. This allows us to consider the $E_1$-bialgebra $\Bar(\gR)$ in place of an $E_2$-algebra $\gR$. We will show that when $\gR$ satisfies (SCE) this allows us to extract a connected graded Hopf algebra $\Delta_\gR$ from $\gR$, and we will explain how it controls the stability properties of $\gR$.

\subsection{$E_2$-algebras from Hopf algebras}\label{sec:CobarOfHopf}

In Section \ref{sec:TStructures} we have discussed the diagonal $t$-structure on $\mathsf{D}(\bk)^\bZ$, and the fact that its heart is the abelian 1-category of $\bZ$-graded $\bk$-modules. If $A$ is a ($\bN$-)graded connected Hopf algebra over $\bk$ which is a flat $\bk$-module in each degree, then we may consider it as an $E_1$-bialgebra object in $\bk\text{-}\mathsf{mod}^\bZ = (\mathsf{D}(\bk)^\bZ)^\heartsuit$ and hence, using flatness and Warning \ref{warn:flatness}, as an $E_1$-bialgebra object in $\mathsf{D}(\bk)^\bZ$. We may then form
$$\mathbf{a} := \Cobar(A) \in \mathsf{Alg}_{E_2}(\mathsf{D}(\bk)^\bZ).$$
The homotopy groups of this object are tautologically given by
$$\pi_{n,d}(\mathbf{a}) = \Cotor_{A}^{n-d}(\bk, \bk)_n.$$
Using the fact that $A$ is connected, and calculating $\Cotor$ with the reduced Cobar complex, we see that $\mathbf{a}$ satisfies (C). Observing that $A$ satisfies the hypothesis of Proposition \ref{prop:BarCobarInverse} we see that $\Bar(\mathbf{a}) \simeq \Bar(\Cobar(A)) \simeq A$ is supported in diagonal bidegrees, so $\mathbf{a}$ satisfies (SCE). Furthermore we see that $\mathbf{a}$ satisfies (F) precisely when $A$ has finite type. 

If $A$ is commutative, then it can be considered as an $E_1$-coalgebra in $E_\infty$-algebras in $\mathsf{D}(\bk)^\bZ$, and hence $\mathbf{a} = \Cobar(A)$ promotes to an $E_\infty$-algebra.

\begin{rem}
The category $\mathbf{a}\text{-}\mathsf{mod}$ is equivalent to the category $\mathsf{Stable}(A)$ introduced by Hovey \cite[Section 2.5]{HoveyBook}, \cite{HoveyComodules}, in its $\infty$-categorical incarnation \cite[Definition 4.9]{BHV}. To see this, note that Koszul duality identifies the category of finite $\mathbf{a}$-modules with the category of $A$-comodules in $\mathsf{D}(\bk)^\bZ$ whose homotopy groups are finite $\bk$-modules in total. These are precisely the $A$-comodules which are dualisable in $\mathsf{D}(\bk)^\bZ$: in other words the dualisable objects in the category of $A$-comodules equipped with the monoidal structure $\otimes_\bk$. $\mathsf{Stable}(A)$ is defined as $\mathsf{Ind}$ of this category, but $\mathbf{a}\text{-}\mathsf{mod}$ is  $\mathsf{Ind}$ of the category of finite $\mathbf{a}$-modules.
\end{rem}

\subsection{The stability Hopf algebra}

Let $\gR \in \mathsf{Alg}_{E_2}(\mathsf{D}(\bk)^\bZ)$ satisfy (C) and (SCE). The underlying object of $\Bar(\gR)$ is $\bk \otimes_\gR \bk$, which by axiom (SCE) has trivial homotopy groups below the diagonal, i.e.\ in bidegrees $(n,d)$ with $d<n$. That is, it is connective with respect to the diagonal $t$-structure. Applying the strong symmetric monoidal functor $\tau_{\leq 0}^\mathrm{diag} : \mathsf{D}(\bk)^\bZ_{\geq 0} \to (\mathsf{D}(\bk)^\bZ)^\heartsuit = \bk\text{-}\mathsf{mod}^\bZ$ we obtain a bialgebra object
$$\Delta := \tau^\mathrm{diag}_{\leq 0} (\Bar(\gR)).$$
As it lies in the heart, it is equivalent to the data of its homotopy groups
$$\bigoplus_{n \geq 0} \pi_{n,n}(\bk \otimes_\gR \bk),$$
which are equipped with the structure of a positively graded and connected bialgebra (so in fact a Hopf algebra). We identify these objects, and call them $\Delta_\gR$.

Supposing that the $\bk$-modules $\pi_{n,n}(\bk \otimes_\gR \bk)$ are all flat (cf. Warning \ref{warn:flatness}) then $\Delta_\gR$ may be considered as an augmented $E_1$-bialgebra in $\mathsf{D}(\bk)^\bZ$, and truncation
$$\Bar(\gR) \lra \tau_{\leq 0}^\mathrm{diag}(\Bar(\gR)) = \Delta_\gR$$
is a map of such. As $\gR$ satisfies (C) it in particular satisfies the hypothesis of Proposition \ref{prop:BarCobarInverse}, so $\gR \simeq \Cobar(\Bar(\gR))$. Applying $\Cobar$ to the map above therefore gives a map of $E_2$-algebras\footnote{I am grateful to Kelly Wang for suggesting this convenient and flexible notation.}
$$\gR \lra \Cobar(\Delta_\gR) =: \mathbf{r}.$$

The following theorem formalises the slogan that ``$\gR$ and $\mathbf{r}$ enjoy precisely the same kinds of homological stability in slopes $< 1$''. Its purpose is that it reduces the problem of determining which kinds of homological stability $\gR$ has to a problem in the cohomology of the Hopf algebra $\Delta_\gR$. We will take this principle further in Section \ref{sec:HigherStabMapsRevisited}.

\begin{thm}\label{thm:IdealsSameForSmallAlgebra}
Let $\gR$ satisfy (C) and (SCE).
\begin{enumerate}[(i)]
\item\label{it:IdealsSameForSmallAlgebra:1} The map $\bigoplus_{n}H^{E_2}_{n,n-1}(\gI_\gR) \to \bigoplus_{n}H^{E_2}_{n,n-1}(\gI_\mathbf{r})$ is an isomorphism. In particular, when $\bk$ is a field of positive characteristic $p$ and $\gR$ satisfies (F) we may speak of the same sequence $x_1, x_2, \ldots$ for both $\gR$ and $\mathbf{r}$.

\item\label{it:IdealsSameForSmallAlgebra:2} The map $\gR \to \mathbf{r}$ satisfies assumption $(\dagger)$ of Theorem \ref{thm:DetNilp} with $\theta=1$.

\item\label{it:IdealsSameForSmallAlgebra:3} Let $\gM$ be a finite $\gR$-module with a slope $\lambda < 1$ vanishing line. Then for $\tfrac{d}{n} = \lambda$ elements in the kernel of
$$\pi_{n,d}(\End_\gR(\gM)) \lra \pi_{n,d}(\End_{\mathbf{r}}(\mathbf{r} \otimes_\gR \gM))$$
are nilpotent.

\item\label{it:IdealsSameForSmallAlgebra:4} Let $\bk$ be a field of positive characteristic $p$ and $\gR$ satisfy (F). Fix slopes $(\lambda, \mu)$, let $x_i, \ldots, x_j$ be the $x$'s of slopes $(\lambda,\mu)$, and suppose $\fil_*\gM \in \mathsf{W}(\lambda, \mu)$. Then there is an equality
$$I_{\lambda, \mu}(\fil_*\gM) = I_{\lambda, \mu}(\fil_* \mathbf{r} \otimes_{\fil_*\gR} \fil_*\gM)$$
of ideals of $\bk[x_i, \ldots, x_j]$, where the latter is calculated in the category of  $\fil_* \mathbf{r}$-modules.
\end{enumerate}
\end{thm}

Part (\ref{it:IdealsSameForSmallAlgebra:1}) is an immediate consequence of the following.

\begin{lem}\label{lem:CellsEstimateTruncation}
We have $H^{E_2}_{n,d}(\gI_{\mathbf{r}}, \gI_{\gR})=0$ for $d < n+1$, and in particular the map $H^{E_2}_{n,d}(\gI_\gR) \to H^{E_2}_{n,d}(\gI_{\mathbf{r}})$ is an isomorphism in bidegrees satisfying $d \leq n-1$.
\end{lem}
\begin{proof}
Recall that $\bk \otimes_\gR \bk \simeq \Bar(\gR) \simeq \bk \oplus \Sigma Q^{E_1^\text{nu}}(\gI_{\gR})$ by \eqref{eq:BarIsIndec}, and similarly for $\mathbf{r}$. By construction the map
$$\bk \otimes_\gR \bk \lra \bk \otimes_{\mathbf{r}} \bk \simeq \Delta_\gR$$
is an isomorphism on homotopy groups in bidegrees $d \leq n$, and an epimorphism in bidegrees $d \leq n+1$ (as the target vanishes in degrees $(n,n+1)$). Thus $H^{E_1}_{n,d}(\gI_{\mathbf{r}}, \gI_{\gR})=0$ for $d < n+1$. We also have that $H^{E_1}_{n,d}(\gI_{\mathbf{r}})$ and $H^{E_1}_{n,d}(\gI_{\gR})$ vanish for $d < n-1$. We apply \cite[Proposition 14.5]{e2cellsI} with $\rho(n)=n$ and $\sigma(n)=n+1$, which satisfy $\rho*\rho=\rho$ and $\rho*\sigma = \sigma$, to deduce that the $E_2$-homology satisfies $H^{E_2}_{n,d}(\gI_{\mathbf{r}}, \gI_{\gR})=0$ for $d < n+1$.
\end{proof}

Part (\ref{it:IdealsSameForSmallAlgebra:2}) is as follows.

\begin{cor}
We have $H_{n,d}^{\gR}(\mathbf{r}) = 0$ for $n < 1$ or $d < n+1$, except that it is $\bk$ in bidegree $(0,0)$.
\end{cor}
\begin{proof}
 By \cite[Theorem 15.9]{e2cellsI} applied with $\rho(n)=n$ and $\sigma(n)=n+1$ it follows that there is a map
$$H^{\gR}_{n,d}(\mathbf{r}, \gR) \lra H^{E_2}_{n,d}(\gI_{\mathbf{r}}, \gI_{\gR})$$
which is an isomorphism for $d < n+1$, so in particular the domain vanishes in this range. The difference between this and the stated result is precisely $H_{*,*}^\gR(\gR) = \bk[0,0]$.
\end{proof}

Then part (\ref{it:IdealsSameForSmallAlgebra:3}) follows from this and Theorem \ref{thm:DetNilp} (\ref{it:DetNilp:2}). 

For part (\ref{it:IdealsSameForSmallAlgebra:4}), first note that the containment $I_{\lambda, \mu}(\fil_*\gM) \subseteq I_{\lambda, \mu}(\fil_* \mathbf{r} \otimes_{\fil_*\gR} \fil_*\gM)$ is immediate, as a map which has an iterate which is $\tau$-power torsion continues to have this property upon applying $\fil_* \mathbf{r} \otimes_{\fil_*\gR} -$. For the converse, let $y \in I_{\lambda, \mu}(\fil_*\mathbf{r} \otimes_{\fil_*\gR} \fil_*\gM) \subset \bk[x_i, \ldots, x_j]$ and let $\psi$ be a $y$ self-map of $\fil_*\gM$. We wish to show that $\tau^{-1}\psi$ is a nilpotent endomorphism of the $\gR$-module $\tau^{-1}\fil_*\gM$. As we chose $y$ in the ideal $I_{\lambda, \mu}(\fil_*\mathbf{r} \otimes_{\fil_*\gR} \fil_*\gM)$ we know that $\mathbf{r} \otimes_\gR \tau^{-1}\psi = \tau^{-1}(\fil_* \mathbf{r} \otimes_{\fil_*\gR} \psi)$ is nilpotent, but by part (\ref{it:IdealsSameForSmallAlgebra:3}) it follows that $\tau^{-1}\psi$ is nilpotent, i.e.\ that $y \in I_{\lambda, \mu}(\fil_*\gM)$.

\subsection{Maps of Hopf algebras}

In practice the Hopf algebra $\Delta_\gR$ may not be completely known, and one must work with partial information. In particular, it will often be convenient to pass to a quotient Hopf algebra $\Delta_\gR \twoheadrightarrow Q$. Here we provide some tools for doing so.

If $f : A \to B$ is a morphism of ($\bk$-flat) Hopf algebras then there is a corresponding map of $E_2$-algebras
$$\mathbf{a} := \Cobar(A) \lra \mathbf{b} :=\Cobar(B).$$
This may be used to compare calculations over $\mathbf{a}$ with calculations over $\mathbf{b}$, in a range of degrees, as follows.

\begin{prop}\label{prop:ChangeOfHopfAlg}\mbox{}
\begin{enumerate}[(i)]
\item\label{it:ChangeOfHopfAlg:1} If $f$ is surjective and $\Ker(f)$ is supported in gradings $\geq N$, then this map satisfies hypothesis $(\dagger)$ of Theorem \ref{thm:DetNilp} with slope $\theta = \tfrac{N-1}{N}$.

\item\label{it:ChangeOfHopfAlg:2} If $f$ is injective and $\mathrm{Coker}(f)$ is supported in gradings $\geq N$, then this map satisfies hypothesis $(\dagger)$ of Theorem \ref{thm:DetNilp} with slope $\theta = \tfrac{N-2}{N}$.
\end{enumerate}
\end{prop}

\begin{proof}
As $H^{E_1}_{*,*}(\gI_\mathbf{a})$ is $A_n$ in bidegree $(n,n-1)$ and trivial otherwise, and similarly for $\mathbf{b}$, we see that
$$H^{E_1}_{n,d}(\gI_\mathbf{b}, \gI_\mathbf{a}) = \begin{cases}
\mathrm{Coker}(A_n \to B_n) & d=n-1\\
\Ker(A_n \to B_n) & d=n\\
0 & \text{else}.
\end{cases}$$
Under assumption (i) these are supported in bidegrees $(n,n)$ and vanish for $n < N$. In particular they vanish in bidegrees $(n,d)$ satisfying $d < \tfrac{N-1}{N}n + 1$. Applying \cite[Proposition 14.5]{e2cellsI} with $\rho(n)=n$ and $\sigma(n) = \tfrac{N-1}{N}n + 1$, it follows that $H^{E_2}_{n,d}(\gI_\mathbf{b}, \gI_\mathbf{a})$ for $d < \tfrac{N-1}{N}n + 1$ too. Invoking \cite[Theorem 15.9]{e2cellsI} with the same $\rho$ and $\sigma$ gives the same vanishing range for $H^{\mathbf{a}}_{n,d}(\mathbf{b}, \mathbf{a})$. It follows that hypothesis ($\dagger$) of Theorem \ref{thm:DetNilp} holds with $\theta = \tfrac{N-1}{N}$.

Under assumption (ii) the relative $E_1$-homology groups are supported in bidegrees $(n,n-1)$ and vanish for $n < N$, so in particular they vanish in bidegrees $(n,d)$ satisfying $d < \tfrac{N-2}{N}n + 1$. The same line of reasoning gives the conclusion.
\end{proof}

If $A$ is a connected graded Hopf algebra, and ${P}$ is a set of primitive elements in $A$, then the algebra quotient $A \twoheadrightarrow A/({P})$ by the two-sided ideal generated by ${P}$ canonically obtains the structure of a Hopf algebra. Conversely, if $f: A \twoheadrightarrow B$ is a surjective map of Hopf algebras, and $a$ is an element of $\Ker(f)$ of minimal degree, then $a$ is primitive. Both of these are elementary to check, and it follows from the latter that a general quotient of $A$ is obtained by killing a (possibly transfinite) \emph{primitive sequence} $p_1, p_2, \ldots$, where $p_i$ is primitive modulo the two-sided ideal generated by the previous $p_j$'s.

\subsection{The underived canonical multiplicative filtration}

The canonical multiplicative filtration $\fil_* \mathbf{r} = \fil_*^{E_2} \mathbf{r}$ has 
$$\bk \otimes_{\fil_*^{E_2} \mathbf{r}} \bk = B^{E_1}(\fil_*^{E_2}\mathbf{r}) \simeq \fil^{E_1}_* \Delta_\gR$$
by Proposition \ref{prop:BarOnCanMultFilt}. Although $\Delta_\gR$ is an essentially discrete object, as it is in the heart of the diagonal $t$-structure, its canonical $E_1$-multiplicative filtration is usually not a filtration by subobjects. Indeed, the associated graded of this filtration is
$$\gr(\fil_*^{E_1} \Delta_\gR) \simeq E_1((-1)_* Q^{E_1^\text{nu}} \bar{\Delta}_\gR)$$
for $\bar{\Delta}_\gR := \Ker(\epsilon : \Delta_\gR \to \bk)$, and the indecomposables $Q^{E_1^\text{nu}} \bar{\Delta}_\gR$ will only be a discrete object if $\Delta_\gR$ is a free associative algebra. Our purpose in this section is to compare this canonical $E_1$-multiplicative filtration with a filtration by subobjects, namely the filtration of $\Delta_\gR$ by (literal, not derived) powers of the ideal $\bar{\Delta}_\gR$. To do so let us suppose that $\bk$ is a field, to avoid having to worry about flatness.

In fact we may as well discuss the setting of a connected graded Hopf algebra $A$ and its corresponding $E_2$-algebra $\mathbf{a} := \Cobar(A)$. The augmentation ideal $\bar{A} := \Ker(\epsilon : A \to \bk)$ defines a filtered $E_1$-algebra by
$$\fil_p^\mathrm{aug}A := \begin{cases}
A & p \geq 0\\
\bar{A}^{-p} &  p < 0.
\end{cases}$$
and similarly $\fil_*^\mathrm{aug} \bar{A}$. These filtered objects are both complete. We have $\bar{A} = \fil_{-1}^\mathrm{aug} \bar{A}$, and as $(-1)_*^{E_1}$ is left adjoint to evaluation at $-1$ by adjunction we obtain a filtered non-unital $E_1$-algebra map
$$\fil^{E_1}_* \bar{A} \lra \fil_*^\mathrm{aug}\bar{A},$$
and, by unitalising, a filtered $E_1$-algebra map
$$\fil^{E_1}_* {A} \lra \fil_*^\mathrm{aug}{A}.$$
They both induce an equivalence on colimits. The latter may be checked to in fact be $E_1$-bialgebra map, using that the target is in the heart of the diagonal $t$-structure. Taking the cobar construction again gives a filtered $E_2$-algebra map
$$\fil_* \mathbf{a} \lra \mathrm{Cobar}(\fil_*^{\mathrm{aug}} A) =: \fil_*^\mathrm{aug} \mathbf{a},$$
which then also induces an equivalence on colimits.

The spectral sequence associated to $\fil_*^\mathrm{aug} \mathbf{a}$ computes $\Cotor^*_{A}(\bk, \bk)$ starting from $\Cotor$ over the associated graded Hopf algebra 
$$\gr(\fil_*^\mathrm{aug} A) = 0_*\bk \oplus \bigoplus_{n \geq 1} (-n)_* \bar{A}^n/\bar{A}^{n+1}.$$
This kind of spectral sequence was first considered by Ivanovsky \cite{Ivanovsky} for $A$ the dual Steenrod algebra; it is in a sense dual, and in practical terms quite similar, to the May spectral sequence \cite{MaySS}: a further useful reference is Bajer--Sadofsky \cite{BajerSadofsky}. The associated graded Hopf algebra $\gr(\fil_*^\mathrm{aug} A)$ is generated in grading $-1$ and so is primitively generated, and is therefore the universal restricted enveloping algebra of its restricted Lie algebra of primitive elements \cite[Theorem 6.11]{MilnorMoore}. In particular it is cocommutative. 

Supposing now that $\bk$ is $\bQ$ or $\bF_p$ (or more generally a perfect field), it follows from a theorem of Borel \cite[Theorem 7.11]{MilnorMoore} that the dual $\gr(\fil_*^\mathrm{aug} A)^\vee$ is isomorphic \emph{as an algebra} to a tensor product of monogenic Hopf algebras. That is, of the connected Hopf algebras
\begin{align*}
\bk[x] & \quad\text{with $x$ of even degree}\\
\Lambda_\bk[y] & \quad \text{with $y$ of odd degree}\\
\bF_p[x]/(x^{p^\ell}) & \quad \text{(with $x$ of even degree if $p$ is odd)}
\end{align*}
with $x$ or $y$ primitive. Knowing $\gr(\fil_*^\mathrm{aug} A)^\vee$ as an algebra suffices to compute 
$$\Ext^*_{\gr(\fil_*^\mathrm{aug} A)^\vee}(\bk,\bk) \cong \Ext^*_{\gr(\fil_*^\mathrm{aug} A)\text{-}\mathsf{comod}}(\bk,\bk) \cong \Cotor^*_{\gr(\fil_*^\mathrm{aug} A)}(\bk,\bk)$$
as an algebra (and the result will be a commutative algebra because $\gr(\fil_*^\mathrm{aug} A)$ is in fact a Hopf algebra), however it does not suffice to determine the Browder bracket or Dyer--Lashof operations, as these are encoded by the product on $\gr(\fil_*^\mathrm{aug} A)$, i.e.\ the coproduct on $\gr(\fil_*^\mathrm{aug} A)^\vee$. The monogenic Hopf algebras have
\begin{align*}
\Ext^*_{\bk[x]}(\bk, \bk) &\cong \Lambda_{\bk}[\underline{x}^\vee]\\
\Ext^*_{\Lambda_{\bk}[y]}(\bk, \bk) &\cong \bk[\underline{y}^\vee]\\
\Ext^*_{\bF_p[x]/(x^{p^\ell})}(\bF_p, \bF_p) &\cong \Lambda_{\bF_p}[\underline{x}^\vee] \otimes \bF_p[t] \quad\text{$p$ odd or $p=2$ and $\ell>1$}
\end{align*}
where we write $\underline{z}^\vee \in \Ext^1$ for the dual to the indecomposable element $z$, and in the final case $t$ can be described as the $p^\ell$-fold Massey product $\langle \underline{x}^\vee, \underline{x}^\vee, \ldots, \underline{x}^\vee \rangle$: in particular, in our usual grading we have $|\underline{z}^\vee| = (|z|-1, |z|)$ and $|t| = (p^\ell |x| - 2, p^\ell |x|)$.

If the Hopf algebra $\gr(\fil_*^\mathrm{aug} A)$ is in addition commutative---which will certainly be the case if $A$ is commutative, which in turn will certainly be the case if it is $\Delta_\gR$ for $\gR$ an $E_3$-algebra---then by the classification of connected primitively generated abelian Hopf algebras over a perfect field \cite[Theorem 7.16]{MilnorMoore} it is isomorphic to a locally finite tensor product of monogenic Hopf algebras \emph{as a Hopf algebra}. The monogenic Hopf algebras have
\begin{align*}
\Cotor^*_{\bQ[x]}(\bQ, \bQ) &\cong \Lambda_{\bQ}[\underline{x}]\\
\Cotor^*_{\Lambda_{\bk}[y]}(\bk, \bk) &\cong \bk[\underline{y}]\\
\Cotor^*_{\bF_2[x]/(x^{2^\ell})}(\bF_2, \bF_2) &\cong \bF_2[\underline{x}, \xi(\underline{x}), \xi^2(\underline{x}), \ldots \xi^{\ell-1}(\underline{x})]\\
\Cotor^*_{\bF_p[x]/(x^{p^\ell})}(\bF_p, \bF_p) &\cong \Lambda_{\bF_p}[\underline{x},  \xi(\underline{x}), \ldots \xi^{\ell-1}(\underline{x})] \otimes \bF_p[\zeta(\underline{x}), \zeta(\xi(\underline{x}), \ldots, \zeta(\xi^{\ell-1}(\underline{x})))]
\end{align*}
where we write $\underline{z} \in \Cotor^1$ for the class of a primitive element $z$, and the case of polynomial algebras over finite fields is included by formally allowing $\ell=\infty$. This calculation is expressed as the homology of an $E_2$-algebra, and the operations $\xi$ and $\zeta$ have their usual meanings from this context. In fact, as the monogenic Hopf algebras are all commutative, their cobar constructions are all $E_\infty$-algebras so we may express the result in terms of Dyer--Lashof operations by the notational substitutions $\xi(t) = Q^{1+|t|}(t)$ for $p=2$, and $\xi(t) = Q^{(1+|t|)/2}(t)$ and $\zeta(t) = \beta Q^{(1+|t|)/2}(t)$ for $p$ odd. (This also accounts for why no Browder brackets arise.)

In either case this allows us to fully describe the $E_1$-page of the spectral sequence
$$\pi_{*,*,*}(\gr(\fil_*^\mathrm{aug} \mathbf{a})) = \Cotor^*_{\gr(\fil_*^\mathrm{aug} A)}(\bk, \bk)\Longrightarrow \Cotor^*_{A}(\bk, \bk) = \pi_{*,*}(\mathbf{a})$$
as a commutative algebra, and similarly for the spectral sequence associated to a filtered Smith--Toda complex $\fil^\mathrm{aug}_* \mathbf{a}/(\psi_1, \ldots, \psi_s)$ considered as a module over the above. Furthermore, being a spectral sequence associated to a purely algebraic filtration it is in principle algorithmically computable: methods developed to study the May spectral sequence \cite{MaySS, Tangora, Lin} will surely be useful here.

\subsection{Detecting periodic families}\label{sec:DetectingPeriodic}

There is a further kind of application of the Hopf algebra $\Delta_\gR$, not directly to establishing homological stability properties of $\gR$ but rather to finding periodic families of elements in $\pi_{*,*}(\gR)$. The idea is very simple: if $\Delta_\gR \to Q$ is a (quotient) map of Hopf algebras, then there are maps of $E_2$-algebras
$$\gR \lra \mathbf{r} = \Cobar(\Delta_\gR) \lra \Cobar(Q) =: \mathbf{q},$$
which we may try to use to detect elements of $\pi_{*,*}(\gR)$. This will be pursued elsewhere.

\section{Existence of Smith--Toda complexes revisited}\label{sec:HigherStabMapsRevisited}

In this section we wish to explain another construction of Smith--Toda complexes, building on the Hopf algebra methods introduced in the previous section. The method is more \emph{ad hoc}, but is somewhat stronger in that it is able to construct all stabilisation maps of a given slope at once. That is, for a finite $\gR$-module $\gM$ with a slope $\lambda$ vanishing line, it can construct slope $\lambda$ endomorphisms $\delta_s, \ldots, \delta_t$ of $\gM$ which coherently homotopy commute with each other, and such that $\gM/(\delta_s, \ldots, \delta_t)$ has a vanishing line of slope $>\lambda$.

We will describe this method when $\gR$ is an $E_3$-algebra (what will will say in fact works whenever the Hopf algebra $\Delta_\gR$ is known to be commutative), and comment in Remark \ref{rem:RevistedMethodForE2Alg} on what will need to be done to extend this to $E_2$-algebras.

\subsection{Commutative Hopf algebras}

We collect here the facts about graded connected commutative Hopf algebras that we will use.

\begin{thm}\label{thm:FinIndecAndPrim}
Let $A$ be a graded connected Hopf algebra over a field of positive characteristic, such that its indecomposables and primitives are both finite-dimensional. Suppose that $A$ is commutative (or cocommutative). Then $A$ is finite-dimensional.
\end{thm}
\begin{proof}
By dualising if necessary, we may suppose that $A$ is commutative. We will show that every element of $A$ of strictly positive degree is nilpotent: in particular the finitely-many algebra generators are nilpotent, from which it follows that $A$ is finite-dimensional.

Supposing the conclusion did not hold, so there are non-nilpotent elements of strictly positive degree, let $x \in A$ be one of minimal degree. Then $\psi(x) = 1 \otimes x + \sum_i x_i' \otimes x_i'' + x \otimes 1$ where $x_i', x_i'' \in A$ have strictly positive degrees which are strictly smaller than the degree of $x$. Thus $x_i'$ and $x_i''$ are nilpotent, so for $N \gg 0$ $\psi(x^{p^N}) = 1 \otimes x^{p^N} + x^{p^N} \otimes 1$, and hence $x^{p^N}$ is primitive. But if $x$ is non-nilpotent then $\{x^{p^M}\}_{M \geq N}$ is an infinite linearly-independent set of non-zero primitive elements, so the primitives of $A$ are not finite-dimensional: this is a contradiction.
\end{proof}

The following combines \cite[Theorem A]{Wilkerson} and \cite[Theorem 1.2]{BajerSadofsky}.

\begin{thm}[Wilkerson, Bajer--Sadofsky]\label{thm:FinDimHopfAlg}
Let $A$ be a finite-dimensional graded connected Hopf algebra over a field $\bk$ of positive characteristic. Suppose that $A$ is commutative. Then 
\begin{enumerate}[(i)]
\item $\Cotor_A^*(\bk, \bk)_*$ is finitely-generated as a $\bk$-algebra.

\item The spectral sequence $\Cotor_{\gr(A)}^*(\bk, \bk)_* \Rightarrow \Cotor_A^*(\bk, \bk)_*$, associated to the filtration of $A$ by powers of the augmentation ideal, collapses at a finite page.
\end{enumerate}
\qed
\end{thm}

\begin{rem}\label{rem:RevistedMethodForE2Alg}
To extend the methods we are about to describe to $E_2$-algebras, one needs Theorems \ref{thm:FinIndecAndPrim} and \ref{thm:FinDimHopfAlg} with the assumption ``Suppose that $A$ is commutative'' removed. 

For Theorem \ref{thm:FinIndecAndPrim} experts in Hopf algebras tell me that it is plausible, but I have not succeeded in making progress. 

For Theorem \ref{thm:FinDimHopfAlg} the problem is merely technical. What is needed is to prove \cite[Proposition 3.2]{Wilkerson}, which Wilkerson does using power operations in the Cartan--Eilenberg spectral sequence for a conormal extension of commutative Hopf algebras. But the operations required are those that exist in a spectral sequence of $E_2$-algebras, and as $\Cobar(A)$ is an $E_2$-algebra for any Hopf algebra $A$ it seems extremely likely that the Cartan--Eilenberg filtration can be lifted to a filtered $E_2$-algebra.
\end{rem}

\subsection{The method}

Let $\gR$ be an $E_3$-algebra satisfying (C), (SCE), and (F). Fix a slope $\lambda < 1$. Let $x_s, \ldots, x_t$ be the $x$'s of $\tfrac{d}{n}$-slope precisely $\lambda$. Choose an integer $\delta$ large enough that $n_1, \ldots, n_t < \delta$ and $\lambda < \tfrac{\delta-1}{\delta+1}$. Form the Hopf algebra quotient
$$\Delta_\gR \lra Q$$
by killing all primitives of grading $>\delta$, and repeating this infinitely often: thus $Q$ has no primitives of grading $>\delta$. Taking $\Cobar$, there are maps of filtered $E_3$-algebras
$$\rho: \fil_* \gR \lra \fil_* \mathbf{r} \lra \fil_* \mathbf{q} \lra \fil_*^\mathrm{aug} \mathbf{q}.$$

\begin{lem}\label{lem:ConstructGammas}
There are filtered homotopy classes $\xi_1, \ldots, \xi_t \in \pi_{*,*,*}(\fil_*^\mathrm{aug} \mathbf{q})$ such that $C\tau \otimes \xi_i$ is a $p$-th power of $(C\tau \otimes \rho)_*(x_i)$. 

The object $\tau^{-1}\fil_*^\mathrm{aug} \mathbf{q}/(\xi_1, \ldots, \xi_t)$ has a vanishing line of slope $>\lambda$.
\end{lem}
\begin{proof}
Form the Hopf subalgebra $A \to Q$ generated by $\bigoplus_{n \leq \delta} Q_n$ (in other words, the quotient $Q^\vee \to A^\vee$ is obtained by iteratedly killing all primitives of grading $>\delta$). There are maps of filtered $E_3$-algebras $\fil_* \gR \overset{\rho}\to \fil_*^\mathrm{aug} \mathbf{q} \leftarrow \fil_*^\mathrm{aug} \mathbf{a}$.
As $\gR$ is an $E_3$-algebra, $\Delta_\gR$ is a commutative Hopf algebra and hence so are $Q$ and $A$.
The commutative Hopf algebra $A$ has finite-dimensional primitives and indecomposables, so is finite-dimensional by Theorem \ref{thm:FinIndecAndPrim}. Thus by Theorem \ref{thm:FinDimHopfAlg} the spectral sequence
$$\Cotor_{\gr(\fil^\mathrm{aug}_* A)}^{d-n}(\bk, \bk)_n \Longrightarrow \Cotor_A^{d-n}(\bk, \bk)_n = \pi_{n,d}(\mathbf{a})$$
collapses at a finite page.

On associated graded, the maps above induce
$$\pi_{*,*,*}(\gE_3((-1)_*Q^{E_3^\text{nu}}(\gI_\gR))) \overset{\rho_*}\lra \Cotor_{\gr(\fil^\mathrm{aug}_* Q)}^{*}(\bk, \bk)_* \longleftarrow \Cotor_{\gr(\fil^\mathrm{aug}_* A)}^{*}(\bk, \bk)_*.$$
The $x_1, \ldots, x_t$ map to elements $x_1^Q, \ldots, x_t^Q \in \Cotor_{\gr(\fil^\mathrm{aug}_* Q)}^{*}(\bk, \bk)_*$, and as $A \to Q$ is an isomorphism in gradings $<\delta$, so $\gr(\fil^\mathrm{aug}_* A) \to \gr(\fil^\mathrm{aug}_* Q)$ is too, these lift uniquely to classes $x_1^A, \ldots, x_t^A \in \Cotor_{\gr(\fil^\mathrm{aug}_* A)}^{*}(\bk, \bk)_*$. As the spectral sequence for $\fil_*^\mathrm{aug}\mathbf{a}$ collapses at a finite page, the $p^M$-th power of these are permanent cycles for some $M \gg 0$, detecting classes $\xi_1^A, \ldots, \xi_t^A \in \pi_{*,*,*}(\fil_*^\mathrm{aug} \mathbf{a})$: these map to the required classes in $\pi_{*,*,*}(\fil_*^\mathrm{aug} \mathbf{q})$.

To prove the second part, first note that as $A$ is a finite-dimensional commutative Hopf algebra, $\gr(\fil^\mathrm{aug}_* A)$ is a tensor product of monogenic Hopf algebras and so $\Cotor_{\gr(\fil^\mathrm{aug}_* A)}^{*}(\bk, \bk)_*$ is a tensor product of a polynomial and perhaps an exterior algebra. The elements $x_1^A, \ldots, x_t^A$ can be taken as the polynomial generators of $\tfrac{d}{n}$-slope $< \lambda$, so that
$$\Cotor_{\gr(\fil^\mathrm{aug}_* A)}^{*}(\bk, \bk)_* = \bk[x_1^A, \ldots, x_r^A] \otimes \Lambda_\bk[x'_1, \ldots, x'_{r'}].$$
The spectral sequence for $\fil_*^\mathrm{aug} \mathbf{a}/(\xi_1^A, \ldots, \xi_t^A)$ therefore takes the form
$$\frac{\bk[x_1^A, \ldots, x_r^A]}{((x_1^A)^{p^M}, \ldots, (x_t^A)^{p^M})} \otimes \Lambda_\bk[x'_1, \ldots, x'_{r'}] \Longrightarrow \pi_{*,*}(\tau^{-1} \fil_*^\mathrm{aug} \mathbf{a}/(\xi_1^A, \ldots, \xi_s^A)),$$
showing that $\tau^{-1} \fil_*^\mathrm{aug} \mathbf{a}/(\xi_1^A, \ldots, \xi_t^A)$ has a slope $>\lambda$ vanishing line. The claim now follows by applying $\mathbf{q} \otimes_\mathbf{a} -$, and appealing to Theorem \ref{thm:DetNilp} (\ref{it:DetNilp:1}) via Proposition \ref{prop:ChangeOfHopfAlg} (\ref{it:ChangeOfHopfAlg:2}). The latter applies to the map of Hopf algebras $A \to Q$ with $N = \delta+1$, yielding $\theta = \tfrac{\delta-1}{\delta+1}$ which is $>\lambda$ by our choice of $\delta$.
\end{proof}

The $\xi_1, \ldots, \xi_t$ produced in the proof of this lemma are not very canonical, but we choose them once and for all. We use the same notation for the images of these classes in $\pi_{*,*}(\mathbf{q})$ given by inverting $\tau$. The idea is now to repeat the kind of argument from Sections \ref{sec:HigherStabMaps} and \ref{sec:EfficientSTComplexes}, but instead of constructing self-maps $\phi_i$ for $s \leq i \leq t$ characterised by $C\tau \otimes \phi_i$ being a $p$-th power of $(C\tau \otimes u)_*(x_i)$, we will construct self-maps $\gamma_i$ characterised by $\mathbf{q} \otimes_{\gR} \gamma_i$ being a $p$-th power of $u_*(\xi_i)$. 

To formalise this, for a $\beta \in \pi_{*,*}(\mathbf{q})$ let us say that an endomorphism $\psi$ of a $\gR$-module $\gM$ is a \emph{$\beta$ self-map} if 
$$\mathbf{q} \otimes_{\gR} \psi = u_*(\beta^{p^M}) \in \pi_{*,*}(\End_{\mathbf{q}}(\mathbf{q} \otimes_{\gR} \gM))$$
for some $M$. We will prove an analogue of Theorems \ref{thm:xiSelfMapsExist} and \ref{thm:xiSelfMapsUniqueAndNatural} on the existence of of $\beta$ self-maps, and their uniqueness and so on. The class of objects to which the argument applies is the following.

\begin{defn}
Let $\mathsf{W}(\lambda) \subseteq \gR\text{-}\mathsf{mod}$ be the full subcategory consisting of the finite $\gR$-modules $\gM$ for which there is a $\kappa$ such that $\pi_{n,d}(\gM)=0$ for all $d < \lambda n + \kappa$. (The objects of this category are precisely $\mathcal{A}^f_\lambda$ from Section \ref{sec:localisation}.)
\end{defn}

Just as in Proposition \ref{prop:WeaklyTypeiProperties}, this is a thick subcategory, and there are containments $\mathsf{W}(\lambda) \subseteq \mathsf{W}(\bar{\lambda})$ when $\lambda \geq \bar{\lambda}$.

\begin{thm}\label{thm:RevisedXiSelfMaps}
Let $\beta \in \pi_{N,D}(\mathbf{q})$ be given, with $\tfrac{D}{N}=\lambda$.
\begin{enumerate}[(i)]
\item\label{it:RevisedXiSelfMaps:1} All objects $\gM \in \mathsf{W}(\lambda)$ admit a $\beta$ self-map.

\item\label{it:RevisedXiSelfMaps:2} If $\psi$ is a $\beta$ self-map of $\gM \in \mathsf{W}(\lambda)$ then after perhaps taking a $p$-th power it lies in the graded centre of $\pi_{*,*}(\End_{\gR}(\gM))$.

\item\label{it:RevisedXiSelfMaps:3} If $\psi$ and $\psi'$ are $\beta$ self-maps of $\gM \in \mathsf{W}(\lambda)$ then they agree after perhaps taking $p$-th powers.

\item\label{it:RevisedXiSelfMaps:4} If $\gM$ and $\gN$ are in $\mathsf{W}(\lambda)$ and have $\beta$ self-maps $\psi^\gM$ and $\psi^\gN$ of the same tridegrees, then after perhaps taking $p$-th powers they intertwine all $\gR$-module maps $f : S^{n,d} \otimes \gM \to \gN$.

\item\label{it:RevisedXiSelfMaps:5} If $\gM \overset{f}\to \gN \overset{g}\to \gP \overset{\partial}\to \Sigma \gM$ is a distinguished triangle of $\gR$-modules, and $\gM$ and $\gN$ are in $\mathsf{W}(\lambda)$, then so is $\gP$. Furthermore they admit $\beta$ self-maps $\psi^\gM$, $\psi^\gN$, and $\psi^\gP$ which assemble into a morphism of distinguished triangles.
\end{enumerate}
\end{thm}
\begin{proof}
The map $\gR \to \mathbf{q}$ satisfies property ($\dagger$) of Theorem \ref{thm:DetNilp} with $\theta = \tfrac{\delta}{\delta+1}$ by Theorem \ref{thm:IdealsSameForSmallAlgebra} (\ref{it:IdealsSameForSmallAlgebra:2}) and Proposition \ref{prop:ChangeOfHopfAlg}, and we had $\lambda < \tfrac{\delta-1}{\delta+1} < \tfrac{\delta}{\delta+1}$ by assumption.

To prove (\ref{it:RevisedXiSelfMaps:1}) we apply Theorem \ref{thm:DetNilp} (\ref{it:DetNilp:3}) with $\gE := \End_{\gR}(\gM)$ and $x = \beta$, which precisely gives a $\beta$ self-map $\psi$. Then (\ref{it:RevisedXiSelfMaps:2}) and (\ref{it:RevisedXiSelfMaps:3}) follow from Theorem \ref{thm:DetNilp} (\ref{it:DetNilp:4}), and (\ref{it:RevisedXiSelfMaps:4}) and (\ref{it:RevisedXiSelfMaps:5}) are then deduced in the same way as the corresponding parts of Theorem \ref{thm:xiSelfMapsUniqueAndNatural}.
\end{proof}

\subsection{Simultaneous stabilisation}\label{sec:SimultaneousStabilisation}

If $\gM \in \mathsf{W}(\lambda)$ then by Theorem \ref{thm:RevisedXiSelfMaps} (i) there are $\xi_i$ self-maps $\gamma_i$ of $\gM$ for $s \leq i \leq t$, and by Theorem \ref{thm:RevisedXiSelfMaps} (ii) they can be chosen to be central in the endomorphism ring, so in particular to commute with each other up to homotopy. There is a strengthening of this, which was in fact our motivation for this method: they can be chosen to simultaneously commute up to coherent homotopy, in the following sense.

\begin{prop}\label{prop:SimStab}
For a sequence $\beta_i \in \pi_{N_i,D_i}(\mathbf{q})$ with $1 \leq i \leq r$, having $\tfrac{D_i}{N_i}=\lambda$, and $\gM \in \mathsf{W}(\lambda)$, there are $\beta_i$-self maps $\psi_i$ and a commutative cube
$$[1]^{r} \lra \gR\text{-}\mathsf{mod}$$
whose restriction to each 1-morphism $\{\epsilon_1\} \times \cdots \times \{\epsilon_{i-1}\} \times [1] \times \{\epsilon_{i+1}\} \times \cdots \times \{\epsilon_{r}\}$ is (a shift of) $\psi_i$, for all $\epsilon$'s in $\{0,1\}$.
\end{prop}
\begin{proof}
We proceed by induction on $r$, so suppose that $\beta_1, \ldots, \beta_{r-1}$ and the required $(r-1)$-cube $c: [1]^{r-1} \to \gR\text{-}\mathsf{mod}$ has been constructed. The category $\mathsf{Fun}([1]^{r-1},\gR\text{-}\mathsf{mod})$ of such cubes is (left) tensored over the presentable category $\gR\text{-}\mathsf{mod}$ and hence admits mapping objects in that category. We then let
$$\gE := \End_{\mathsf{Fun}([1]^{r-1},\gR\text{-}\mathsf{mod})}(c) \in \gR\text{-}\mathsf{mod};$$
composition equips this with the structure of a unital associative algebra in the homotopy category of $\gR$-modules (which could be promoted to an $E_1$-structure, but this suffices). This object may be expressed as the limit of a finite diagram whose entries are of the form $\Hom_{\gR}(S^{n,d} \otimes \gM, S^{n',d'} \otimes \gM)$. All such objects are in the thick subcategory $\mathsf{W}(\lambda)$, so a finite limit of them is too: thus $\mathbf{E} \in \mathsf{W}(\lambda)$. By Theorem \ref{thm:RevisedXiSelfMaps} (i) there is a $\psi_r \in \pi_{*,*}(\mathbf{E})$ such that 
$$\mathbf{q}\otimes_{\gR} \psi_r = u_*(\beta_r^{p^M}) \in \pi_{*,*}(\End_{\mathbf{q}}(\mathbf{q}\otimes_{\gR} \gE)).$$
Interpreted using the definition of this endomorphism object, this $\psi_r$ is exactly a morphism $(S^{p^MN_r, p^MD_r} \otimes \gR) \otimes_\gR c \to c$ in $\mathsf{Fun}([1]^{r-1},\gR\text{-}\mathsf{mod})$, i.e.\ an $r$-cube in $\gR\text{-}\mathsf{mod}$, which by naturality induces a $\beta_r$ self-map when restricted to each vertex of $[1]^{r-1}$, as required.
\end{proof}

By the discussion so far, for $\gM \in \mathsf{W}(\lambda)$ we may find $\xi_i$ self-maps $\gamma_i$ for $s \leq i \leq t$, by Proposition \ref{prop:SimStab} we may arrange these into a commutative $(s-t)$-cube, and we may then let
$$\gM/(\gamma_s, \ldots, \gamma_t)$$
denote the total homotopy cofibre of this cube.

\begin{prop}\label{prop:HAMethodHighConn}
$\gM/(\gamma_s, \ldots, \gamma_t)$ lies in $\mathsf{W}(\bar{\lambda})$ for some $\bar{\lambda} > \lambda$.
\end{prop}
\begin{proof}
By Theorem \ref{thm:DetNilp} (i) it suffices to check this after applying $\mathbf{q} \otimes_\gR -$. The $\mathbf{q}$-module $\gN := \mathbf{q} \otimes_\gR \gM$ has a slope $\lambda$ vanishing line also by Theorem \ref{thm:DetNilp} (i), and 
$$\mathbf{q} \otimes_\gR (\gM/(\gamma_s, \ldots, \gamma_t)) \simeq \gN \otimes_\mathbf{q} \mathbf{q}/(\xi_s^{p^{M_s}}, \ldots, \xi_t^{p^{M_t}}).$$
The latter is in the thick subcategory of $\gN \otimes_\mathbf{q} \mathbf{q}/(\xi_s, \ldots, \xi_t)$, so it suffices to show that this has a slope $>\lambda$ vanishing line.

We also have available the classes $\xi_1, \ldots, \xi_{s-1} \in \pi_{*,*}(\mathbf{q})$, letting us consider
$$\gN \otimes_\mathbf{q} \mathbf{q}/(\xi_1, \ldots, \xi_t).$$
For $1 \leq i < s$ the endomorphism $\xi_i \cdot -$ of $\gN \otimes_\mathbf{q} \mathbf{q}/(\gamma_{i+1}, \ldots, \gamma_t)$ is nilpotent, as $\xi_i$ has slope $< \lambda$ but this object has a slope $\lambda$ vanishing line as it is in the thick subcategory generated by $\gN$. It follows by downwards induction that $\gN \otimes_\mathbf{q} \mathbf{q}/(\xi_s, \ldots, \xi_t)$ is in the thick subcategory of $\gN \otimes_\mathbf{q} \mathbf{q}/(\xi_1, \ldots, \xi_t)$, and as $\gN$ is a finite $\mathbf{q}$-module $\gN \otimes_\mathbf{q} \mathbf{q}/(\xi_s, \ldots, \xi_t)$ is therefore in the thick subcategory of $\mathbf{q}/(\xi_1, \ldots, \xi_t)$. The latter has a vanishing line of slope $> \lambda$ by Lemma \ref{lem:ConstructGammas}, so $\gN \otimes_\mathbf{q} \mathbf{q}/(\xi_s, \ldots, \xi_t)$ does too. 
\end{proof}

\subsection{Efficiency}

Just as in Sections \ref{sec:Omitting} and \ref{sec:ChangingCoords} we can often get away with fewer stabilisation maps. By Lemma \ref{lem:ConstructGammas} the object $\mathbf{q}/(\xi_1, \ldots, \xi_t)$ has a vanishing line of slope $>\lambda$, so choose polynomials
$$\zeta_1, \ldots, \zeta_r \in \bk[\xi_s, \ldots, \xi_t]$$
such that
\begin{enumerate}[(i)]
\item $\mathbf{q}/(\xi_1, \ldots, \xi_{s-1}, \zeta_1, \ldots, \zeta_r)$ has a vanishing line of slope $>\lambda$, and

\item this cannot be achieved using fewer than $r$ $\zeta$'s.
\end{enumerate}

If $\gM \in \mathsf{W}(\lambda)$ then by Theorem \ref{thm:RevisedXiSelfMaps} (i) there are $\zeta_1,\ldots, \zeta_r$ self-maps $\delta_1, \ldots, \delta_{r}$ of $\gM$, which can again be arranged into a homotopy commutative cube with total homotopy cofibre $\gM/(\delta_1, \ldots, \delta_{r})$. 

\begin{prop}\label{prop:HAMethodEfficient}
$\gM/(\delta_1, \ldots, \delta_{r})$ lies in $\mathsf{W}(\bar{\lambda})$ for some $\bar{\lambda} > \lambda$.
\end{prop}
\begin{proof}
By the same method as the proof of Proposition \ref{prop:HAMethodHighConn}, we reduce to showing that the $\mathbf{q}$-module $\mathbf{q}/(\xi_1, \ldots, \xi_{s-1}, \zeta_1, \ldots, \zeta_r)$ has a vanishing line of slope $> \lambda$, but this is true by assumption.
\end{proof}

Finally, we provide an analogue of Theorem \ref{thm:OmittingRedundantXs}.

\begin{prop}\label{prop:ZetasNonNilp}
Suppose that $\gM \in \mathsf{W}(\lambda)$ generates the same thick subcategory as $\gR/(\phi_1,\ldots, \phi_{s-1})$. Then the endomorphism $\delta_i$ of $\gM/(\delta_1, \ldots, \delta_{i-1})$ is not nilpotent.
\end{prop}
\begin{proof}
The subcategory of $\mathsf{W}(\lambda)$ of those objects whose $\zeta_i$ self-map $\delta_i$ is nilpotent is thick by Theorem \ref{thm:RevisedXiSelfMaps} (\ref{it:RevisedXiSelfMaps:5}), so it suffices to show that $\delta_i$ is non-nilpotent on $\gR/(\phi_1,\ldots, \phi_{s-1}, \delta_1, \ldots, \delta_{i-1})$. By Theorem \ref{thm:DetNilp} (\ref{it:DetNilp:2}) it suffices to show that $\zeta_i \cdot -$ is non-nilpotent on $\mathbf{q}/(\xi_1, \ldots, \xi_{s-1}, \zeta_1, \ldots, \zeta_{i-1})$. If it were, then the $\mathbf{q}$-module $\mathbf{q}/(\xi_1, \ldots, \xi_{s-1}, \zeta_1, \ldots, \zeta_{i-1}, \zeta_{i+1}, \ldots, \zeta_r)$ would lie in the thick subcategory of $\mathbf{q}/(\xi_1, \ldots, \xi_{s-1}, \zeta_1, \ldots, \zeta_{r})$. Then $\mathbf{q}/(\xi_1, \ldots, \xi_{s-1}, \zeta_1, \ldots, \zeta_{i-1}, \zeta_{i+1}, \ldots, \zeta_r)$ would have a vanishing line of slope $> \lambda$, but only use $(r-1)$ $\zeta$'s which contradicts our minimality assumption.
\end{proof}

\begin{rem}
In practice one might have $\mathbf{q}/(\xi_1, \ldots, \xi_{s-1}, \zeta_1, \ldots, \zeta_r)$ with a vanishing line of slope $>\lambda$, but find it hard to know that there is no possible shorter sequence of $\zeta$'s achieving this. If one checks that each $\mathbf{q}/(\xi_1, \ldots, \xi_{s-1}, \zeta_1, \ldots, \zeta_{i-1}, \zeta_{i+1}, \ldots, \zeta_r)$ obtained by omitting a single $\zeta$ does not have a vanishing line of slope $>\lambda$, then the argument and conclusion of Proposition \ref{prop:ZetasNonNilp} still holds.

In particular, we can always find a subsequence of $\xi_s, \ldots, \xi_t$ inducing a sequence of non-nilpotent self-maps and achieving a vanishing line of slope $>\lambda$, but as in Example \ref{ex:Inefficiency2} such a sequence will not always have the minimal length possible.
\end{rem}

\section{Miscellany}

We collate here some extended examples and more speculative discussion which would have been distracting in the middle of the text.

\subsection{Some features of homological multi-stability}

The following example shows that in the setting of multi-stability described in Section \ref{sec:multistab} one must not be too ambitious with what one tries to say about the structure of $\pi_{*,d}(\gR)$ as a $\pi_{*,0}(\gR)$-module.

\begin{example}\label{ex:MVConnectingMap}
Let $\bk=\bQ$, let $\ell \geq 2$ and $N \gg 0$ be given, and set
$$\gR := \gE_\infty(S^{1,0} r \oplus S^{1,0} b \oplus \bigoplus_{i=1}^N S^{\ell+1,\ell} x_i) \cup^{E_\infty} D^{\ell+2,\ell+1} y_1 \cup^{E_\infty} \cdots\cup^{E_\infty} D^{\ell+2,\ell+1} y_{N+1}$$
where $y_1$ is attached along $r \cdot x_1$, $y_j$ is attached along $r \cdot x_j - b \cdot x_{j-1}$ for $1 < j < N+1$, and $y_{N+1}$ is attached along $-b \cdot x_N$. This $E_\infty$-algebra satisfies (C), (SCE), and (F), and one calculates that $\pi_{n,d}(\gR/(r,b))=0$ for $d < \tfrac{\ell}{\ell+1}n$, as $\gR/(r,b)$ is equivalent to $\gE_\infty(\bigoplus_{i=1}^N S^{\ell+1,\ell} x_i \oplus \bigoplus_{j=1}^{N+1} S^{\ell+2,\ell+1} y_j)$.

Now $\pi_{*,0}(\gR) = \bQ[r,b]$, and we calculate that
\begin{align*}
\pi_{*,\ell}(\gR) &= \bQ[r,b]\{x_1, \ldots, x_N\}/(r\cdot x_1, r \cdot x_i - b \cdot x_{i-1}, -b \cdot x_N)\\
\pi_{*,\ell+1}(\gR) &= \bQ[r,b]\{b^N \cdot y_1 + b^{N-1} r \cdot y_2 + \cdots + b r^{N-1} \cdot y_N + r^N \cdot y_{N+1}\}
\end{align*}
as $\bQ[r,b]$-modules, with the next non-trivial homotopy group $\pi_{*,2\ell}(\gR)$. The calculation will look the same for $\bQ$ replaced by $\bF_p$ with $p \gg \ell$. We see several sobering features, all manifestations of the same point. 

Firstly, while $\pi_{*,\ell+1}(\gR)$ is a free $\bQ[r,b]$-module on one generator, the grading of this generator is $\ell+2+N$. So whatever one might mean by ``stability'' for a graded $\bQ[r,b]$-module the stable range here can be made arbitrarily bad by increasing $N$.

Secondly, we see that while $\pi_{*,\ell}(\gR)$ is generated in gradings $\leq \ell$ and presented in gradings $\leq \ell+1$, it has $\Tor_2^{\bQ[r,b]}(\pi_{*,\ell}(\gR), \bQ)$ non-trivial in grading $\ell+2+N$. (At one time Castelnuovo--Mumford regularity had seemed like a good measure of the ``stable range'' of a graded $\bQ[r,b]$-module, but here we see that the regularity is $\ell+N$, so can be made arbitrarily bad by changing the \emph{number} of $E_\infty$-cells in $\gR$ but not the \emph{bidegrees} in which they sit. Indeed it is a well-known feature of Castelnuovo--Mumford regularity over polynomial rings in more than one variable that one cannot estimate regularity only in terms of the gradings of generators and of relations: one must use the numbers of them too.)

Thirdly, we see that in the highly cartesian square analogous to \eqref{eq:RedBlueStab} the Mayer--Vietoris connecting map
$$\partial : \pi_{N+\ell+2,\ell+1}(\gR) \lra \pi_{N+\ell,\ell}(\gR),$$
which exists because for $N \gg 0$ these groups are in the range in which the square is cartesian, is non-trivial. Indeed, one may check that
$$\partial(b^N \cdot y_1 + b^{N-1} r \cdot y_2 + \cdots + b r^{N-1} \cdot y_N + r^N \cdot y_{N+1}) = [b^{N-1}\cdot x_1] \neq 0.$$
\end{example}

\begin{rem}
This example cannot arise by $\bQ$-linearising an $E_2$-algebra in $\mathsf{Top}$. If it did then as $\pi_{*,0}(\gR)=\bQ[r,b]$ it would follow that it can be given a $\bN \times \bN$-grading in which $r$ and $b$ have gradings $(1,0)$ and $(0,1)$ respectively. But one can easily check that for $N > \ell+2$ the $x_i$ and $y_j$ cannot be given $\bN \times \bN$-gradings which make all the attaching maps homogeneous.

In fact, if $\gR$ is an $E_2$-algebra in $\bN^k$-graded chain complexes having $\pi_{*,\ldots,*,0}(\gR) = \bk[\sigma_1, \ldots, \sigma_k]$, with $\sigma_i$ of grading $(0,\ldots,0,1,0,\ldots,0)$ with 1 in the $i$-th position, then assuming (SCE) there is a stability range (of slope $\tfrac{1}{2}$) for each $\sigma_i \cdot -$ expressed in terms of the $i$-th grading. This may be proved analogously to \cite[Theorem 18.1]{e2cellsI}.
\end{rem}

\subsection{Going beyond slope 1}\label{sec:NotBeyondSlope1}

The following example shows that for an $E_2$-algebra $\gR$ satisfying (C), (SCE) and (F) one cannot generally find a Smith--Toda complex $\gR/(\alpha_1, \ldots, \alpha_r)$ having a vanishing line of slope $\geq 1$.

\begin{example}\label{ex:NotBeyondSlope1}
Let $\bk=\bF_2$ and $\gR = \gE_2(S^{1,0}\sigma)$. We will show that a finite $\gR$-module $\gM$ with a slope 1 vanishing line must be trivial.

The finite $\gR$-module $\gM$ may be filtered by skeleta, giving a spectral sequence
$$E^1_{*,*} = \pi_{*,*}(\gR) \otimes_{\bF_2} V_{*,*} \Longrightarrow \pi_{*,*}(\gM)$$
of $\pi_{*,*}(\gR)$-modules, with $V_{*,*}$ a finite-dimensional bigraded $\bF_2$-vector space. This spectral sequence collapses at a finite stage, and $\pi_{*,*}(\gM)$ is obtained from $E^\infty_{*,*}$ by finitely-many extensions, as the skeletal filtration is finite. We have
\begin{equation}\label{eq:FreeE2AtF2}
\pi_{*,*}(\gR) = \bF_2[\sigma, Q^1(\sigma), Q^2Q^1(\sigma), Q^4Q^2Q^1(\sigma), \ldots],
\end{equation}
with the class $Q^{2^{s-1}} \cdots Q^2Q^1(\sigma)$ having bidegree $(2^s, 2^s-1)$. We now make use of the theory of coherent modules and rings: \cite[Chapter I]{Cohen} is a good reference for the audience of this paper. The ring \eqref{eq:FreeE2AtF2} is coherent, for example by \cite[Proposition 1.5]{Cohen}. Therefore a $\pi_{*,*}(\gR)$-module is coherent if and only if it is finitely-presented \cite[Proposition 1.4]{Cohen}. As $V_{*,*}$ is finite-dimensional, it follows that $E^1_{*,*}$ is a coherent $\pi_{*,*}(\gR)$-module. As kernels and cokernels of maps between coherent modules are coherent \cite[Proposition 1.3]{Cohen}, it follows that each page $E^r_{*,*}$ is a coherent $\pi_{*,*}(\gR)$-module. As the spectral sequence collapses at a finite stage with finitely-many extensions, and extensions of coherent modules are coherent \cite[Proposition 1.2]{Cohen}, it follows that $\pi_{*,*}(\gM)$ is a coherent $\pi_{*,*}(\gR)$-module. Therefore it is a finitely-presented $\pi_{*,*}(\gR)$-module.

As $\pi_{*,*}(\gM)$ has a slope 1 vanishing line, each element $Q^{2^{s-1}} \cdots Q^2Q^1(\sigma)$ acts nilpotently on it, because they have slope $\tfrac{2^s-1}{2^s}<1$. As it is a finitely-presented and so in particular finitely-generated $\pi_{*,*}(\gR)$-module, it follows that $\pi_{*,*}(\gM)$ is in fact finite-dimensional over $\bF_2$.

To finish the argument we will show that a non-zero finite-dimensional module over the ring \eqref{eq:FreeE2AtF2} cannot be finitely-presented, giving a contradiction. If $N_{*,*}$ is such a module then each element $Q^{2^{s-1}} \cdots Q^2Q^1(\sigma)$ with $s$ large enough acts trivially on $N_{*,*}$, and so the module structure is induced along the quotient map
$$\bF_2[\sigma, Q^1(\sigma), Q^2Q^1(\sigma), \ldots] \lra \frac{\bF_2[\sigma, Q^1(\sigma), Q^2Q^1(\sigma), \ldots]}{(Q^{2^{s-1}} \cdots Q^2Q^1(\sigma) : s \geq K)} =: T_{*,*}$$
for some $K$. Thus
$$\Tor_1^{\bF_2[\sigma, Q^1(\sigma), Q^1Q^2(\sigma), \ldots]}(\bF_2, N_{*,*}) \cong \Tor_1^{T_{*,*}}(\bF_2, N_{*,*}) \otimes \bF_2\{Q^{2^{s-1}} \cdots Q^1(\sigma) \, : \, s \geq K\}.$$
By considering the right-hand side we see that if this is non-zero then it is infinite-dimensional: but it is non-zero, because the non-zero finite-dimensional $\bF_2$-vector space $N_{*,*}$ is certainly not a free module. As the left-hand side measures minimal relations in a resolution of $N_{*,*}$, this module is indeed not finitely-presented.
\end{example}

\subsection{A counterexample in characteristic zero}\label{sec:CharZeroCounterexample}

Theorem \ref{MainThm:A}, and its proof, requires $\bk$ to have positive characteristic. The following example, which we learnt from Robert Burklund, shows that this assumption is essential at least if we wish to keep Theorem \ref{MainThm:A} (\ref{it:MainThm:A:NonNilp}).

Let $\bk=\bQ$ and consider the $E_\infty$-algebra
$$\gR := \gE_\infty(S^{1,0}\sigma \oplus S^{2,1} z ) \cup^{E_\infty}_{\sigma z} D^{3,2} x \cup^{E_\infty}_{\sigma^2} D^{2,1} y.$$
This satisfies (C), (F), and (SCE). The homotopy groups of $\gR$ are given by the homology of the bigraded cdga
$$(\mathrm{Sym}^*_\bQ[\sigma_{1,0}, z_{2,1}, x_{3,2}, y_{2,1}], dx = \sigma z, dy = \sigma^2).$$
One calculates that its homology has basis
$$1, \sigma, z, x\sigma-yz, xz,  x(x\sigma-yz), x^2z, x^2(x\sigma-yz), x^3z, , \ldots$$
ordered by grading. Thus $\pi_{*,*}(\gR)$ has a vanishing line of slope $\tfrac{2}{3}$. But also it vanishes in positive gradings divisible by 3, so every element of $\pi_{*,*}(\gR)$ is nilpotent, and hence $\gR$ has no non-nilpotent $\gR$-module endomorphisms other than scalar multiples of the identity. 

The formula for $\gR$ makes sense for $\bk=\bZ$ and so can be reduced modulo $p$. Let us suppose $p$ is odd. The spectral sequence for $\fil_*^{E_\infty} \gR_{\bF_p}$ takes the form
$$E^1_{*,*,*} = \pi_{*,*,*}(\gE_\infty(S^{1,0,-1} \sigma \oplus S^{2,1,-1} z \oplus S^{3,2,-1} x \oplus S^{2,1,-1} y)) \Longrightarrow \pi_{*,*}(\gR_{\bF_p}).$$
There is a differential $d^1(x) = \sigma z$ and so (by the principle explained in Remark \ref{rem:PowerOpsInSS}, cf.\ \cite[Theorem 16.8]{e2cellsI}) the class $x^p$ survives until $E^p_{3p,2p,-p}$ and then
$$d^p(x^p) = d^p(Q^{1}(x)) = Q^1(\sigma z) = Q^1(\sigma) Q^0(z) + Q^0(\sigma) Q^1(z),$$
which vanishes because $Q^0(z)=0$ by instability, and $Q^0(\sigma) = \sigma^p$ vanishes as there is an earlier differential $d^1(y) = \sigma^2$. In fact, considering $E^2_{*,*,*}$ we see that there are no possible targets for differentials on $x^p$ so it is a permanent cycle: thus there is an filtered homotopy class $\xi_p \in \pi_{3p,2p,-p}(\fil_* \gR_{\bF_p})$ detected by $x^p$. It is easy to see that $\tau^{-1}(\fil_* \gR_{\bF_p}/\xi_p)$ has a vanishing line of slope $\geq \tfrac{3}{4}$. 

Here we see another reason why an integral or rational statement is not possible: an endomorphism of $\gR_\bQ$ is defined after inverting finitely-many primes, and so can be reduced modulo $p$ for the infinitely-many remaining primes. But the minimal period of a non-nilpotent slope $\tfrac{2}{3}$ endomorphism of $\gR_{\bF_p}$ depends on $p$.

\subsection{Change-of-coefficients}

The ring homomorphisms $\bQ \leftarrow \mathbb{Z} \to \bF_p$ induce adjunctions
\begin{equation*}
\begin{tikzcd}
{\mathsf{D}(\bQ)} \arrow[r, swap, "U_\bQ"{name=G1}, shift left=-.8ex] 
&  \mathsf{D}(\bZ) \arrow[l, swap, "\bQ \otimes_\bZ -"{name=F1}, shift left=-.8ex] 
\arrow[r, "\bF_p \otimes_\bZ -"{name=F}, shift left=.8ex] 
& \mathsf{D}(\bF_p). \arrow[l, "U_{\bF_p}"{name=G}, shift left=.8ex] 
\arrow[phantom, from=F1, to=G1, "\scalebox{0.7}{$\dashv$}" rotate=-90]
\arrow[phantom, from=F, to=G, "\scalebox{0.7}{$\dashv$}" rotate=-90]
\end{tikzcd}
\end{equation*}
The left adjoints are strong symmetric monoidal, so the right adjoints are lax symmetric monoidal (in fact $U_\bQ$ is strong, as $\bQ \otimes_\bZ -$ is a localisation).

Suppose that $\gR \in \mathsf{Alg}_{E_2}(\mathsf{D}(\bZ)^\bZ)$, i.e.\ is defined over the integers. For $\bk \in \{\bF_p,\bQ\}$ we may form $\gR_{\bk} := \gR \otimes_\bZ \bk$, the tensor product implicitly derived. This may be considered as lying in either $\mathsf{Alg}_{E_2}(\mathsf{D}(\bk)^\bZ)$ or via the lax symmetric monoidal functor $U_\bk$ as lying in $\mathsf{Alg}_{E_2}(\mathsf{D}(\bZ)^\bZ)$: in the latter case, there is an $E_2$-algebra map $\gR \to \gR_\bk$.

Let $L^f_\lambda(-)$ denote the finite localisation as previously discussed applied in the category of $\gR$-modules, and $L^{f, \bk}_\lambda(-)$ denote the analogue in $\gR_\bk$-modules, using the class $\mathcal{A}_\lambda^{f,\bk}$ of finite $\gR_\bk$-modules with a vanishing line of slope $\lambda$. Similarly, let $L_\lambda(-)$ denote the non-finite localisation discussed in Section \ref{sec:TelescopeConj}, using the class $\mathcal{A}_\lambda$ of all $\gR$-modules with a slope $\lambda$ vanishing line, and $L^{ \bk}_\lambda(-)$ the same in $\gR_\bk$-modules.

\begin{prop}\label{prop:BaseChange}
In the following we implicitly forget down to $\gR\text{-}\mathsf{mod}$.
\begin{enumerate}[(i)]

\item\label{it:BaseChange:1} There are factorisations $\gR_\bk \to L_\lambda^f(\gR_\bk) \to L_\lambda^{f, \bk}(\gR_\bk)$ and $\gR_\bk \to L_\lambda(\gR_\bk) \to L_\lambda^{\bk}(\gR_\bk)$.

\item\label{it:BaseChange:2} $L^f_\lambda(\gR_\bk) \simeq L^f_\lambda(\gR) \otimes_\bZ \bk$.

\item\label{it:BaseChange:3} $L_\lambda(\gR_{\bF_p}) \simeq L_\lambda(\gR)\otimes_\bZ \bF_p$.

\item\label{it:BaseChange:4} When $\bk=\bF_p$, the map $L^f_\lambda(\gR_{\bF_p}) \to L^{f, \bF_p}_\lambda(\gR_{\bF_p})$ is an equivalence.

\item\label{it:BaseChange:5} When $\bk=\bQ$, the map $L_\lambda(\gR_{\bQ}) \to L^{ \bQ}_\lambda(\gR_{\bQ})$ is an equivalence.

\end{enumerate}
\end{prop}
\begin{proof}
For (\ref{it:BaseChange:1}), if $\gT \in \mathcal{A}^f_\lambda$ then $\gT_\bk \in \mathcal{A}_\lambda^{f,\bk}$ by the Universal Coefficient Theorem, so $\map_\gR(\gT, L_\lambda^{f, \bk}(\gR_\bk)) \simeq \map_{\gR_\bk}(\gT_\bk, L_\lambda^{f, \bk}(\gR_\bk)) \simeq 0$, and therefore $L_\lambda^{f, \bk}(\gR_\bk)$ is $\mathcal{A}_\lambda$-local: this gives a factorisation
$$\gR_\bk \lra L_\lambda^f(\gR_\bk) \lra L_\lambda^{f, \bk}(\gR_\bk).$$
The argument for the non-finite localisation is the same.

For (\ref{it:BaseChange:2}) we use that $L_\lambda^f$ is smashing to write $L^f_\lambda(\gR_\bk) = L^f_\lambda(\gR) \otimes_\gR \gR_\bk = L^f_\lambda(\gR) \otimes_\gR \gR \otimes_\bZ \bk = L^f_\lambda(\gR) \otimes_\bZ \bk$. For (\ref{it:BaseChange:3}) we apply the left adjoint $L_\lambda$ to the cofibre sequence $\gR \overset{p\cdot -}\to \gR \to \gR_{\bF_p}$, to get a cofibre sequence $L_\lambda(\gR) \overset{p\cdot -}\to L_\lambda(\gR) \to L_\lambda(\gR_{\bF_p})$ in the $\mathcal{A}_\lambda$-local category. But cofibres in the $\mathcal{A}_\lambda$-local category may be computed in the ambient category (as they are suspended fibres), so $L_\lambda(\gR_{\bF_p})$ is the cofibre of multiplication by $p$ on $L_\lambda(\gR)$.

For (\ref{it:BaseChange:4}), if $\gT^{\bF_p}$ is a finite $\gR_{\bF_p}$-module with a slope $\lambda$ vanishing line, then it is also finite as an $\gR$-module, because $\gR_{\bF_p}$ is a finite $\gR$-module (namely the cofibre of $p \cdot - : \gR \to \gR$). Thus it lies in $\mathcal{A}_\lambda^f$, so $\mathcal{A}^f_\lambda$-local objects are in particular $\mathcal{A}_\lambda^{f, \bF_p}$-local, giving a map $L_\lambda^{f, \bF_p}(\gR_{\bF_p}) \to L_\lambda^{f}(\gR_{\bF_p})$. The universal properties show that it is a homotopy inverse to the natural map.

For (\ref{it:BaseChange:5}), first note that $L_\lambda(\gR_\bQ)$ is a $\gR_\bQ$-algebra, and so is rational. If $\gT^\bQ$ is a $\gR_\bQ$-module with a slope $\lambda$ vanishing line, then it is also a $\gR$-module with a slope $\lambda$ vanishing line. As rationalising is a localisation
$$0 \simeq \map_{\gR}(\gT^\bQ, L_\lambda(\gR_\bQ)) \simeq \map_{\gR_\bQ}(\gT^\bQ, L_\lambda(\gR_\bQ))$$
so $L_\lambda(\gR_\bQ)$ is in particular $\mathcal{A}^{\bQ}_\lambda$-local, giving a map $L^\bQ_\lambda(\gR_\bQ) \to L_\lambda(\gR_\bQ)$. The universal properties show that it is a homotopy inverse to the natural map.
\end{proof}

The following example shows that the conclusion of Proposition \ref{prop:BaseChange} (\ref{it:BaseChange:4}) cannot hold for $\bk=\bQ$.

\begin{example}
Let $\bk=\bZ$ and  $\gR := \gE_2(S^{1,0} \sigma)$. We first claim that any finite $\gR$-module $\gM$ with a slope 1 vanishing line is rationally trivial. This is because the finite $\gR_{\bF_2}$-module $\gM_{\bF_2}$ again has a slope 1 vanishing line, so by Example \ref{ex:NotBeyondSlope1} is trivial. Thus 2 is invertible on $\gM$, so its homotopy groups are $\bZ[\tfrac{1}{2}]$-modules, but they are also finitely-generated abelian groups so must be finite, and therefore $\gM$ is rationally trivial.

If $\gT \in \mathcal{A}_1^f$ it follows from the claim that $\gT_\bQ=0$, which means that $\map_\gR(\gT, \gR_\bQ) \simeq \map_{\gR_\bQ}(\gT_\bQ, \gR_\bQ)=0$. Thus $\gR_\bQ$ is $\mathcal{A}_1^f$-local, so $L^f_1(\gR_\bQ) = \gR_\bQ$ and therefore $L^f_1(\gR_\bQ/\sigma) = \gR_\bQ/\sigma$. On the other hand we have
$$\pi_{*,*}(\gR_\bQ) = \bQ[\sigma] \otimes \Lambda_\bQ[\,[\sigma,\sigma]\,],$$
so $\pi_{*,*}(\gR_\bQ/\sigma) \cong \Lambda_\bQ[\,[\sigma,\sigma]\,]$. This is finite-dimensional over $\bQ$ so has a vanishing line of slope 1 (or in fact of any slope), so $\gR_\bQ/\sigma \in \mathcal{A}_1^{f,\bQ}$ and hence $L_1^{f, \bQ}(\gR_\bQ/\sigma) =0 \neq L^f_1(\gR_\bQ/\sigma)$. It follows that 
$$\text{the map }L_1^f(\gR_\bQ) \lra L_1^{f, \bQ}(\gR_\bQ) \text{ is not an equivalence}.$$
(In fact one may show that this is the map $\gR_\bQ \to \sigma^{-1} \gR_\bQ$.)
\end{example}

\subsection{Further structure on Smith--Toda complexes}

Let us briefly indicate some further things which can most likely be done with our constructions of Smith--Toda complexes, though we have not checked in detail. Let us write $x_1, \ldots, x_i$ for the $x$'s of $\tfrac{d}{n}$-slope $< \lambda$.

Firstly, following \cite[Proposition 4.11]{HoveyStrickland}, it seems that one may produce filtered Smith--Toda complexes $\fil_*\gR/(\phi_1, \ldots, \phi_i)$ equipped with a (non-associative!) left-unital multiplication $\mu$ in $\fil_*\gR\text{-}\mathsf{mod}$. Such filtered Smith--Toda complexes are ``atomic'' in the sense used there, because $C\tau \otimes (\fil_*\gR/(\phi_1, \ldots, \phi_i)) = (C\tau \otimes \fil_*\gR)/(x_1^{p^{M_1}}, \ldots, x_i^{p^{M_i}})$ and $x_1^{p^{M_1}}, \ldots, x_i^{p^{M_i}}$ is a regular sequence in $\pi_{*,*,*}(C\tau \otimes \fil_*\gR)$, so any endomorphism of $C\tau \otimes (\fil_*\gR/(\phi_1, \ldots, \phi_i))$ under $C\tau \otimes \fil_*\gR$ is an equivalence, and hence the same is true before applying $C\tau \otimes -$. The (modest) advantage of having such a multiplication is that to produce a self-map of $\fil_*\gR/(\phi_1, \ldots, \phi_i)$ it now suffices to give it on the bottom cell, i.e.\ to give a filtered homotopy class. Using the methods of \cite{BurklundMoore} one may probably even construct $\fil_*\gR/(\phi_1, \ldots, \phi_i)$ as an $E_2$-algebra.

Secondly, following \cite[Section 4, Proposition 7.10 (a)]{HoveyStrickland} \cite[Section 3]{MahowaldSadofsky}, if $\fil_*\gR/(\phi_1, \ldots, \phi_i)$ is a filtered Smith--Toda complex obtained by the proof of Theorem \ref{thm:FiltSTExist}, where $C\tau \otimes \phi_j = (C\tau \otimes u)_*(x_j^{p^{M_j}})$, then it is not hard to show using Theorem \ref{thm:xiSelfMapsUniqueAndNatural}  that there are divergent sequences $M_1(n), \ldots, M_i(n)$ and a sequence of filtered Smith--Toda complexes $\fil_*\gR/(\phi_1(n), \ldots, \phi_i(n))$ with $C\tau \otimes \phi_j(n) = (C\tau \otimes u)_*(x_j^{p^{M_j(n)}})$ and starting with $\phi_j(0)=\phi_j$, assembling into a tower
$$\cdots \lra \fil_*\gR/(\phi_1(2), \ldots, \phi_i(2)) \lra \fil_*\gR/(\phi_1(1), \ldots, \phi_i(1)) \lra \fil_*\gR/(\phi_1, \ldots, \phi_i)$$
under $\fil_*\gR$. Inverting $\tau$ and then dualising with $D^r(-) = \underline{\map}_\gR^r(-, \gR)$ gives a direct system of $\gR$-modules over $\gR$, whose colimit can be shown to be $C_\lambda^f(\gR) \to \gR$. As in \cite[Proposition 4.18]{HoveyStrickland} these filtered Smith--Toda complexes can be constructed to be self-dual up to a shift, so this colimit describing $C_\lambda^f(\gR)$ is quite explicit. Its cofibre therefore gives another explicit model for $L_\lambda^f(\gR)$ different from that of Section \ref{sec:OrthCalcModel}. 

Thirdly, when a filtered Smith--Toda complex $\fil_*\gR/(\phi_1, \ldots, \phi_i)$ exists, we may Bousfield localise away from it in $\fil_*\gR\text{-}\mathsf{mod}$: provisionally call this $L_\lambda^{f, \mathrm{fil}}(-)$. Constructing this localisation by the analogue of Section \ref{sec:OrthCalcModel}, or perhaps as in the previous paragraph, one finds that $\tau^{-1}L_\lambda^{f, \mathrm{fil}}(\fil_*\gR)  = L_\lambda^f(\gR)$ and that $C\tau \otimes L_\lambda^{f, \mathrm{fil}}(\fil_*\gR)$ is Bousfield localisation of $C\tau \otimes \fil_*\gR$ away from $(C\tau \otimes \fil_*\gR)/(x_1^{p^{M_1}}, \ldots, x_i^{p^{M_i}})$. The latter may be described as the homotopy pull back in $C\tau \otimes \fil_*\gR$-modules of the punctured $i$-cube obtained by tensoring together the maps $C\tau \otimes \fil_*\gR \to x_j^{-1}(C\tau \otimes \fil_*\gR)$ for $j=1,\ldots,i$, as in Example \ref{ex:RBLocalissation}. The completeness of the filtered object $L_\lambda^{f, \mathrm{fil}}(\fil_*\gR)$ does not seem obvious, so convergence of the associated spectral sequence in unclear: but in principle this gives some access to $L_\lambda^f(\gR)$. See \cite[Section 7]{MahowaldSadofsky} for the analogue in stable homotopy theory.

\section{Worked examples}\label{sec:examples}

Our goal in this section is to explain how our discussion applies to various well-known examples, and how several recent results about these examples can be efficiently encoded in statements about their stability Hopf algebras. While this leads to some new results, that is not the goal of this section: rather we simply wish to illustrate what has been discussed. Readers looking for more substantial applications are referred to the work of Wang \cite{wang}, which we will also mention often below.

\subsection{Free $E_k$-algebras}

Let $\gR := \gE_k(S^{1,0}_\bk \sigma)$ be the free $E_k$-algebra on one generator, with $k \geq 2$. Then
$$\bk \otimes_\gR \bk \simeq \gE_{k-1}(S^{1,1}_\bk \bar{\sigma})$$
as an $E_{k-1}$-algebra, so taking its diagonal homotopy groups we have
$$\Delta_\gR \cong \begin{cases}
\bk[\bar{\sigma}] & \text{$k=2$, or $\bk=\bF_2$ and $k \geq 2$}\\
\bk[\bar{\sigma}]/(\bar{\sigma}^2) & \text{$k >2$ and $\bk$ has odd characteristic}
\end{cases}$$
as an algebra. In both cases $\bar{\sigma}$ must be primitive, which determines the Hopf algebra structure completely. In the first case, as $\bar{\sigma}$ has odd degree and is primitive it follows that $\bar{\sigma}^2$ is primitive too. Writing $x$ and $y$ for the duals to these elements, the dual algebra is
$$\Delta_\gR^\vee \cong \begin{cases}
\Lambda_\bk[x] \otimes \Gamma_\bk[y] & \text{$k=2$, or $\bk=\bF_2$ and $k \geq 2$}\\
\Lambda_\bk[x] & \text{$k >2$ and $\bk$ has odd characteristic.}
\end{cases}$$
Thus for $\mathbf{r} = \Cotor(\Delta_\gR)$ we may calculate
$$\pi_{*,*}(\mathbf{r}) = \begin{cases}
\bk[\sigma, \xi(\sigma), \xi\xi(\sigma),\ldots] & \text{$\bk=\bF_2$ and $k \geq 2$}\\
\!\begin{aligned}
       & \bk[\sigma, \zeta[\sigma,\sigma], \zeta\xi[\sigma,\sigma], \ldots]  \\
       & \,\,\,\,\otimes \Lambda_\bk[[\sigma,\sigma], \xi[\sigma,\sigma], \xi\xi[\sigma,\sigma] \ldots]
    \end{aligned}
& \text{$\bk=\bF_p$ with $p$ odd and $k = 2$}\\
\bk[\sigma] & \text{$k >2$ and $\bk$ has odd characteristic.}
\end{cases}$$

\begin{rem}
$\gR := \gE_2(S^{1,0}\sigma)$ satisfies $\bk \otimes_\gR \bk \simeq \gE_{1}(S^{1,1}_\bk \bar{\sigma})$, which lies in the heart of the diagonal $t$-structure and so is identified with the polynomial algebra $\Delta_\gR = \bk[\bar{\sigma}]$. Thus $\mathbf{r} = \Cobar(\Delta_\gR) \simeq \gR$ in this case.

This has a curious consequence when $\bk=\bF_2$ and $k=\infty$, as then $\gR := \gE_\infty(S^{1,0}\sigma)$ has $\Delta_\gR =  \bF_2[\bar{\sigma}]$ and so $\mathbf{r} = \Cobar(\Delta_\gR) \simeq \gE_2(S^{1,0}\sigma)$. But being Cobar of a commutative Hopf algebra this is an $E_\infty$-algebra. In other words, $\gE_2(S^{1,0}\sigma)$ has an $E_\infty$-structure, and $\gR \to \mathbf{r}$ has the form of an $E_\infty$-map $\gE_\infty(S^{1,0}\sigma) \to \gE_2(S^{1,0}\sigma)$ in this case. The same happens for any $2 \leq k \leq \infty$. (If $\bk$ does not have characteristic 2 then the algebra $\bk[\bar{\sigma}]$ is \emph{not} commutative, as the Koszul sign rule would mean that $2 \bar{\sigma}^2=0$.)
\end{rem}

With these calculations we can produce Smith--Toda complexes for $\gR$, as follows. In each case the claimed vanishing line can be checked after base-change along $\gR \to \mathbf{r}$, using Theorem \ref{thm:DetNilp} (\ref{it:DetNilp:1}).

\begin{enumerate}[(i)]

\item If $\bk=\bF_2$ then we may take the sequence of Smith--Toda complexes to be
$$\gR/(\sigma, \xi(\sigma), \ldots, \xi^{s-1}(\sigma)),$$
whose homotopy groups vanish in bidegrees $(n,d)$ satisfying $d < \tfrac{2^s-1}{2^s}n$. 

\item If $\bk=\bF_p$ with $p$ odd and $k=2$ then we may take the sequence of Smith--Toda complexes to be
$$\gR/(\sigma, \zeta[\sigma,\sigma], \zeta\xi[\sigma,\sigma], \ldots, \zeta\xi^{s-2}[\sigma,\sigma])$$
whose homotopy groups vanish below the line of slope $\tfrac{2p^s-2}{2p^s} = \tfrac{p^s-1}{p^s}$ passing through the element $[\sigma,\sigma] \cdot \xi[\sigma,\sigma] \cdots \xi^{s-1}[\sigma, \sigma]$ of bidegree $(2\tfrac{p^s-1}{p-1}, 2\tfrac{p^s-1}{p-1}-s)$, i.e.\ for bidegrees $(n,d)$ satisfying $d-(2\tfrac{p^s-1}{p-1}-s) < \tfrac{p^s-1}{p^s}(n - 2\tfrac{p^s-1}{p-1})$. 

\item If $\bk=\bF_p$ with $p$ odd and $k>2$ then $\gR/\sigma$ has vanishing homotopy groups in bidegrees $(n,d)$ satisfying $d < n$.
\end{enumerate}

\subsection{Revisiting \cite[Section 18]{e2cellsI}}\label{sec:RevisitSec18}

In Theorem 18.1 of \cite{e2cellsI} Galatius, Kupers, and the author established a ``generic homological stability theorem'' intended to cover some of the most common examples. Here we discuss the hypotheses and conclusions of that Theorem from the point of view of stability Hopf algebras.

Let $\gR \in \mathsf{Alg}_{E_2}(\mathsf{D}(\bF_p)^\bZ)$ satisfy (C) and (SCE), and suppose furthermore that $\pi_{*,0}(\gR) = \bF_p[\sigma]$ for a class $\sigma$ of grading 1. In low degrees the $E^1$-page of the bar spectral sequence 
$$E^1_{n,s,t} = \pi_{n,t}{(\gR^{\otimes s})} \Longrightarrow \pi_{n, s+t}(\bF_p \otimes_\gR \bF_p)$$
takes the form
\begin{equation*}
\begin{tikzcd}
& \pi_{*,1}(\gR) & \bF_p[\sigma] \otimes \pi_{*,1}(\gR) \oplus \pi_{*,1}(\gR) \otimes \bF_p[\sigma] \lar & \cdots \lar\\
\bF_p & \bF_p[\sigma] \lar & \bF_p[\sigma] \otimes \bF_p[\sigma] \lar& \cdots \lar
\end{tikzcd}
\end{equation*}
depicted in the $(s,t)$-plane. Using this we find that $\Delta_\gR(0) = \bF_p\{1\}$ and $\Delta_\gR(1) = \bF_p\{\bar{\sigma}\}$. Now $\gr(\fil^\mathrm{aug}_* \Delta_\gR)$ is a primitively generated Hopf algebra, so by the theorem of Borel we have
$$\gr(\fil^\mathrm{aug}_* \Delta_\gR)^\vee = 
\begin{cases}
\bF_2[\bar{\sigma}^\vee]/((\bar{\sigma}^\vee)^{2^\ell}) & p=2\\
\Lambda_{\bF_p}[\bar{\sigma}^\vee] & p \text{ odd}
\end{cases} \otimes \left\{\substack{\text{monogenic Hopf algebras} \\ \text{generated in gradings $\geq 2$}}\right\}$$
as an algebra. It follows that there is a spectral sequence converging to $\pi_{*,*}(\mathbf{r})$ and starting with
$$\begin{cases}
\Lambda_{\bF_2}[\sigma] \otimes \bF_2[t] & p=2 \text{ and } \ell>1\\
\bF_p[\sigma] & p \text{ odd, or } p=2 \text{ and } \ell=1
\end{cases} \otimes \left\{\substack{\text{free graded-commutative algebra} \\ \text{on generators of slope $\geq \tfrac{1}{2}$}}\right\},$$
where in the first case $t$ has bidegree $(2^\ell - 2, 2^\ell)$. Now the class $\sigma \in \pi_{*,*}(\mathbf{r})$ is not nilpotent (if it were then it would also be nilpotent in $\pi_{*,*}(\gR)$ by Theorem \ref{thm:DetNilp}), so we see that the second case must occur: that is, if $p=2$ then we must have $\ell=1$. It follows that there is a spectral sequence
$$\left\{\substack{\text{free graded-commutative algebra} \\ \text{on generators of slope $\geq \tfrac{1}{2}$}}\right\} \Longrightarrow \pi_{*,*}(\mathbf{r}/\sigma)$$
so $\pi_{n,d}(\mathbf{r}/\sigma)$ for  $d < \tfrac{1}{2}n$. This recovers the first part of \cite[Theorem 18.1]{e2cellsI}.

For the second part of that Theorem we furthermore assume that the map $\sigma \cdot - : \pi_{1,1}(\gR) \to \pi_{2,1}(\gR)$ is surjective. In the bar spectral sequence, for  an element $x \in \pi_{*,1}(\gR)$ we have $d^1(\sigma \otimes x) = -\sigma \cdot x$ so $d^1 : E^1_{*,2,1} \to E^1_{*,1,1}$ is surjective. Furthermore $E^2_{*,2,0} = \mathrm{Tor}_2^{\bF_p[\sigma]}(\bF_p, \bF_p)=0$, so it follows that $\pi_{*,2}(\bF_p \otimes_\gR \bF_p)=0$ and hence that $\Delta_\gR(2)=0$. As above we then have
$$\gr(\fil^\mathrm{aug}_* \Delta_\gR)^\vee = \Lambda_{\bF_p}[\bar{\sigma}^\vee]  \otimes \left\{\substack{\text{monogenic Hopf algebras} \\ \text{generated in gradings $\geq 3$}}\right\}$$
leading to a spectral sequence
$$\left\{\substack{\text{free graded-commutative algebra} \\ \text{on generators of slope $\geq \tfrac{2}{3}$}}\right\} \Longrightarrow \pi_{*,*}(\mathbf{r}/\sigma)$$
so that $\pi_{n,d}(\mathbf{r}/\sigma)$ for  $d < \tfrac{2}{3}n$. This recovers the second part of \cite[Theorem 18.1]{e2cellsI}.

Thus the two parts of this Theorem correspond precisely to knowing that the stability Hopf algebra of $\gR$ has one of the following forms:
\setlength{\tabcolsep}{10pt} 
\renewcommand{\arraystretch}{1.5} 
\begin{table}[h]
\begin{tabular}{l|lll}
$n$ & 0 & 1 & 2   \\ \hline
$\Delta_\gR(n)$ & $\bF_p\{1\}$  & $\bF_p\{\bar{\sigma}\}$  & ? \\
$\Delta_\gR(n)$ & $\bF_p\{1\}$  & $\bF_p\{\bar{\sigma}\}$  & 0  
\end{tabular}
\end{table}

\subsection{General linear groups}

Let $\mathcal{O}$ be a Dedekind domain. Galatius, Kupers, and the author \cite[Section 18.2]{e2cellsI} have constructed and analysed an $E_\infty$-algebra $\mathbf{BGL}$ in $\mathsf{D}(\bk)^\bZ$ having
$$\pi_{n,d}(\mathbf{BGL}) \cong \bigoplus_{\substack{[M]\\ \rk(M)=n}}H_d(\GL(M) ; \bk),$$
where the sum is over isomorphism classes finitely-generated projective $\mathcal{O}$-modules of rank $n$. It satisfies (C), and satisfies (SCE) by a theorem of Charney. It need not satisfy (F) (for example $H_1(\GL_1(\bQ);\bQ) = \bQ^\infty$), but will satisfy it if $\mathcal{O}$ is a finite field, or if it is the ring of integers in a number field (using \cite{Raghunathan}).\footnote{In any case, axiom (F) was only necessary for the general construction of Smith--Toda complexes: other things we have discussed can still be done.} The terms in its stability Hopf algebra may be described as
$$\Delta_\mathbf{BGL}(n) = \bigoplus_{\substack{[M]\\ \rk(M)=n}} H_0(\GL(M) ; \mathrm{St}^{E_1}(M)),$$
the coinvariants of the so-called split (or $E_1$-) Steinberg modules.

Clausen and Jansen \cite{ClausenJansen} have introduced a new form of unstable algebraic $K$-theory, in the guise of the reductive Borel--Serre spaces $|\mathrm{RBS}(M)|$, which come with maps $B\GL(M) \to |\mathrm{RBS}(M)| \to \Omega^\infty_0 \mathrm{K}(\cO)$. Jansen \cite{Jansen2} has explained how these too may be assembled into an $E_\infty$-algebra $\mathbf{RBS}$ having
$$\pi_{n,d}(\mathbf{RBS}) \cong \bigoplus_{\substack{[M]\\ \rk(M)=n}}H_d(|\mathrm{RBS}(M)| ; \bk).$$
This again satisfies (C), and satisfies (SCE) by a theorem of Jansen \cite[Theorem 10.11]{Jansen2}. Again, it might not satisfy (F).  The terms in its stability Hopf algebra may be described as
$$\Delta_\mathbf{RBS}(n) = \bigoplus_{\substack{[M]\\ \rk(M)=n}} H_0(\GL(M) ; \mathrm{St}(M)),$$
the coinvariants of the classical Steinberg modules.

\subsubsection{$\mathbf{RBS}$ for fields}
If $\mathcal{O}$ is a field then it follows from a theorem of Lee--Szczarba \cite[Theorem 4.1]{LS2} that $H_0(\GL(\mathcal{O}^n) ; \mathrm{St}(\mathcal{O}^n))=0$ for $n \geq 2$, meaning that the stability Hopf algebra is $\Delta_\mathbf{RBS} = \Lambda_\bk[\bar{\sigma}]$.  It follows that $\mathbf{rbs} = \bk[\sigma]$, so $\mathbf{rbs}/\sigma \simeq \bk$ has a slope $\infty$ vanishing line, and hence by Theorem \ref{thm:IdealsSameForSmallAlgebra} (\ref{it:IdealsSameForSmallAlgebra:2}) and Theorem \ref{thm:DetNilp} (\ref{it:DetNilp:1}) we have $\pi_{n,d}(\mathbf{RBS}/\sigma)=0$ for $d < n$.

\subsubsection{Infinite fields}
If $\mathcal{O}$ is an infinite field then the natural map $\mathbf{BGL} \to \mathbf{RBS}$ is an equivalence by \cite[Theorem 1.3]{ClausenJansen}, as $\mathcal{O}$ is a ``ring with many units''. In particular their stability Hopf algebras agree, so by the discussion of $\mathbf{RBS}$ above we have $\Delta_\mathbf{BGL} = \Lambda_\bk[\bar{\sigma}]$ too and so $\pi_{n,d}(\mathbf{BGL}/\sigma)=0$ for $d < n$ by the same reasoning. Note however that in this setting much stronger results are known, see \cite{e2cellsIV}.

\subsubsection{Finite fields}
If $\mathcal{O}=\bF_q$ is a finite field with $q \neq 2$, then it follows from \cite[Theorem 5.1]{e2cellsIII} that $\Delta_\mathbf{BGL} = \Lambda_\bk[\bar{\sigma}]$, which as above gives $\pi_{n,d}(\mathbf{BGL}/\sigma)=0$ for $d < n$. This recovers \cite[Theorem A]{e2cellsIII} and \cite{SprehnWahl}.

The situation for $\mathcal{O}=\bF_2$ is much more interesting, and has been analysed in detail by Wang \cite{wang}.

\subsubsection{Principal ideal domains}\label{sec:PIDs}
For $\mathcal{O}$ a Dedekind domain of class number 1 (i.e.\ a PID), it was shown in \cite[Section 18.2]{e2cellsI} that if $H_1(\GL_2(\cO);\GL_1(\cO);\bk)=0$ then $\pi_{n,d}(\mathbf{BGL})=0$ for $d < \tfrac{2}{3}n$. This was a consequence of the ``generic homological stability theorem'', which we have explained from the point of view of the stability Hopf algebra in Section \ref{sec:RevisitSec18}.

A more elaborate argument was given by Kupers--Miller--Patzt \cite[Theorem C]{KMP} to obtain a similar conclusion for $\bk=\bZ$ under the assumption that $\cO$ is a Euclidean domain for which $\cO/2$ is generated as an abelian group by the units of $\cO$. This applies for example to $\cO=\bZ$. Let us explain this case from the point of view of the stability Hopf algebra. All the subtlety concerns the case $\bk=\bF_2$. 

The hypotheses are used to argue \cite[Lemma 6.4]{KMP} that $\mathbf{BGL}$ can be constructed as an $E_\infty$-algebra from $\gE_\infty(S^{1,0}\sigma)$ using no $(2,1)$-cells, and that $\sigma \cdot Q^1(\sigma) \in \pi_{3,1}(\mathbf{BGL})$ must destabilise twice, so there must be a $(3,2)$-cell $\delta$ attached along an element of the form $\sigma Q^1(\sigma)-\sigma^2 u$ for some $u \in \pi_{1,1}(\mathbf{BGL})$. Results of Wang \cite[Theorem 5.6]{wang} allow us to translate such cell-attachments into a decription of the stability Hopf algebra in a range of degrees, as
\setlength{\tabcolsep}{10pt} 
\renewcommand{\arraystretch}{1.5} 
\begin{table}[h]
\begin{tabular}{l|llll}
$n$ & 0 & 1 & 2 &3  \\ \hline
$\Delta_\mathbf{BGL}(n)$ & $\bF_2\{1\}$  & $\bF_2\{\bar{\sigma}\}$  & $\bF_2\{\bar{\sigma}^2\}$ & $\bF_2\{\delta\} \oplus ?$ 
\end{tabular}
\end{table}

\noindent with $\bar{\sigma}$ and $\bar{\sigma}^2$ primitive, and $\psi(\delta) = \delta \otimes 1 + \bar{\sigma} \otimes \bar{\sigma} ^2 + 1 \otimes \delta$. This is obtained just as in \cite[Section 6.1]{wang}. From this Hopf algebra structure we see that
$$\gr(\fil_*^\text{aug} \Delta_\mathbf{BGL}) = \bF_2[\bar{\sigma}]/(\bar{\sigma}^{2^\ell}) \otimes \left\{\substack{\text{monogenic Hopf algebras} \\ \text{generated in gradings $\geq 3$}}\right\}$$
for some $\ell \geq 2$, and so the associated spectral sequence for $\mathbf{bgl}/(\sigma, Q^1(\sigma))$ has the form
\begin{align*}
&\bF_2[\sigma, Q^1(\sigma)]/(\sigma, Q^1(\sigma)) \otimes \left\{\substack{\text{free graded-commutative algebra} \\ \text{on generators of slope $\geq \tfrac{2}{3}$}}\right\}\\
&\quad \Longrightarrow \pi_{*,*}(\mathbf{bgl}/(\sigma, Q^1(\sigma))),
\end{align*}
and hence $\pi_{n,d}(\mathbf{bgl}/(\sigma, Q^1(\sigma)))=0$ for $d < \tfrac{2}{3}n$. Now $\mathbf{bgl}/(\sigma, Q^1(\sigma)^3)$ has a filtration with associated graded
$$\mathbf{bgl}/(\sigma, Q^1(\sigma)) \oplus \Sigma^{2,1} \mathbf{bgl}/(\sigma, Q^1(\sigma)) \oplus \Sigma^{4,2} \mathbf{bgl}/(\sigma, Q^1(\sigma)),$$
so $\pi_{n,d}(\mathbf{bgl}/(\sigma, Q^1(\sigma)^3))=0$ for $d-2 < \tfrac{2}{3}(n-4)$, i.e.\ for $d < \tfrac{2}{3}(n-1)$.

On the other hand, applying $Q^2(-)$ to the relation $\sigma Q^1(\sigma) = \sigma^2 u \in \pi_{3,1}(\mathbf{BGL})$ shows that
$$Q^2(\sigma)Q^1(\sigma)^2 + Q^1(\sigma)^3 + \sigma^2 Q^2 Q^1(\sigma) = Q^2(\sigma^2 u) \in \pi_{6,3}(\mathbf{BGL}).$$
When this is mapped to $\mathbf{bgl}$ we may use the vanishing of
\begin{align*}
u \in \pi_{1,1}(\mathbf{bgl}) = \Cotor^0_{\Delta_\mathbf{BGL}}(\bF_2, \bF_2)_1=0\\
Q^2(\sigma) \in \pi_{2,2}(\mathbf{bgl}) = \Cotor^0_{\Delta_\mathbf{BGL}}(\bF_2, \bF_2)_2=0
\end{align*}
to simplify it to $Q^1(\sigma)^3 = \sigma^2 Q^2 Q^1(\sigma) \in \pi_{6,3}(\mathbf{bgl})$. It follows that the map
$$Q^1(\sigma)^3 \cdot - : S^{6,3} \otimes \mathbf{bgl}/\sigma \lra \mathbf{bgl}/\sigma$$
is nullhomotopic, since the map $\sigma^2 \cdot -$ is nullhomotopic on $\mathbf{bgl}/\sigma$, and therefore that $\mathbf{bgl}/\sigma$ is a retract of $\mathbf{bgl}/(\sigma, Q^1(\sigma)^3)$. But then $\pi_{n,d}(\mathbf{bgl}/\sigma)=0$ for $d < \tfrac{2}{3}(n-1)$, and so $\pi_{n,d}(\mathbf{BGL}/\sigma)=0$ for $d < \tfrac{2}{3}(n-1)$ by Theorem \ref{thm:IdealsSameForSmallAlgebra} (\ref{it:IdealsSameForSmallAlgebra:2}) and Theorem \ref{thm:DetNilp} (\ref{it:DetNilp:1}), which recovers the range in \cite[Theorem C]{KMP}.

The situation for $\mathbf{RBS}$ in the case of Euclidean domains is especially good. In that case the theorem of Lee--Szczarba \cite[Theorem 4.1]{LS2} implies that
$$\Delta_\mathbf{RBS}(n) = H_0(\GL(\cO^n) ; \mathrm{St}(\cO^n))=0 \text{ for } n \geq 2,$$
so that $\Delta_\mathbf{RBS} = \Lambda_\bk[\bar{\sigma}]$. Just as in the case of fields, it follows that $\mathbf{rbs}/\sigma$ has a slope $\infty$ vanishing line and hence that $\pi_{*,*}(\mathbf{RBS}/\sigma)=0$ for $d < n$. This recovers a theorem of Jansen \cite[Theorem 10.22]{Jansen2}. The argument given there uses a theorem of Jansen--Miller \cite[Theorem 1]{JansenMiller} whose assumption can be phrased as ``$\Delta_\gR = \Lambda_\bk[\bar{\sigma}]$'': what we have just explained re-proves the latter theorem. This delivers on our promise in \cite[Remark 1.3]{RWDedekind}.

There are several papers \cite{KMP, BernardMillerSroka, Scalamandre} implying the vanishing of $H_0(\GL_n(\cO) ; \mathrm{St}(\cO^n))$ or $H_0(\GL_n(\cO) ; \mathrm{St}^{E_1}(\cO^n))$ for $n \geq 2$, often after inverting 2. The above strategy gives a direct route to slope 1 homological stability in these cases.

\subsubsection{Dedekind domains}
In \cite{RWDedekind} we considered the case of general Dedekind domains. The new feature is that $A_{\mathrm{Pic}(\mathcal{O})} := \pi_{*,0}(\mathbf{BGL}) \simeq \tau_{\leq 0}\mathbf{BGL} \simeq \tau_{\leq 0}\mathbf{RBS}$ is itself an interesting ring, encoding the classification of finitely-generated projective $\mathcal{O}$-modules. It is given explicitly by 
$$A_{\mathrm{Pic}(\mathcal{O})} = \bk[\mathrm{Pic}(\mathcal{O})]/(\rho \cdot \rho' - \sigma \cdot (\rho \otimes \rho')),$$
 where $\sigma=[\mathcal{O}] \in \mathrm{Pic}(\mathcal{O})$ is the identity element. We explained \cite[Corollary 2.2]{RWDedekind} that the quadratically presented algebra $A_{\mathrm{Pic}(\mathcal{O})}$ is in fact Koszul, so that $\bk \otimes_{A_{\mathrm{Pic}(\mathcal{O})}} \bk$ has homotopy groups supported along the diagonal and given by the quadratic dual coalgebra $A^!_{\mathrm{Pic}(\mathcal{O})}$ (as $A_{\mathrm{Pic}(\mathcal{O})}$ is in fact commutative, its quadratic dual has the structure of a commutative Hopf algebra). In particular, taking bar constructions of the truncation map $\mathbf{BGL} \to \mathbf{RBS} \to \tau_{\leq 0}\mathbf{RBS} = A_{\mathrm{Pic}(\mathcal{O})}$ gives 
$$\bk \otimes_\mathbf{BGL} \bk \lra \bk \otimes_\mathbf{RBS} \bk \lra \bk \otimes_{A_{\mathrm{Pic}(\mathcal{O})}} \bk \simeq A^!_{\mathrm{Pic}(\mathcal{O})}$$
inducing on diagonal truncations maps of Hopf algebras
$$\Delta_\mathbf{BGL} \lra \Delta_\mathbf{RBS} \lra A^!_{\mathrm{Pic}(\mathcal{O})}.$$
We showed \cite[Theorem 3.5, Corollary 3.6]{RWDedekind} that the second map is surjective, and that the composition is surjective in gradings $\leq 4$, but follow-up work of Armeanu--Miller \cite{ArmeanuMiller} shows that the composition is surjective in all gradings. So this gives examples where the stability Hopf algebras $\Delta_\mathbf{BGL}$ or $\Delta_\mathbf{RBS}$ must be somewhat complicated. (It also demonstrates that having a somewhat complicated stability Hopf algebra does not preclude having good homological stability properties: the Hopf algebra $A^!_{\mathrm{Pic}(\mathcal{O})}$ looks relatively complex, but $\Cobar(A^!_{\mathrm{Pic}(\mathcal{O})}) \simeq A_{\mathrm{Pic}(\mathcal{O})}$ and $A_{\mathrm{Pic}(\mathcal{O})}/\sigma$ has a slope $\infty$ vanishing line.)

\begin{rem}\label{rem:WangNonvanishing}
Using the methods of this paper, Wang has shown \cite[Proposition 8.5 (i)]{wang} that if the Dedekind domain $\cO$ surjects onto $\bZ/2$ then $\mathbf{BGL}_{\bF_2}/\sigma$ does \emph{not} have a vanishing line of slope $>\tfrac{2}{3}$.
\end{rem}

\subsubsection{An extensive example}

To go a little further, consider \cite[Example 1.3]{RWDedekind}, where $\mathcal{O} :=\bZ[\sqrt{-5}]$, and so $\mathrm{Pic}(\mathcal{O}) = \bZ\{\sigma, \lambda\}$ with $\lambda$ represented by the invertible ideal $\mathfrak{l}=(2, 1+\sqrt{-5})$, and hence $A_{\mathrm{Pic}(\mathcal{O})} = \bk[\sigma, \lambda]/(\lambda^2 - \sigma^2)$. The linear dual algebra $(A_{\mathrm{Pic}(\mathcal{O})}^!)^\vee$ to the quadratic dual coalgebra $A_{\mathrm{Pic}(\mathcal{O})}^!$ has presentation $(A_{\mathrm{Pic}(\mathcal{O})}^!)^\vee = T_\bk[s,\ell]/(s^2 + \ell^2, s \ell + \ell s)$. Here $s$ and $\ell$ are primitive, and are the dual basis to $\bar{\sigma}, \bar{\lambda} \in \pi_{1,1}(\bk \otimes_{A_{\mathrm{Pic}(\mathcal{O})}} \bk)$. From here it is a problem of linear algebra to describe the Hopf algebra $A_{\mathrm{Pic}(\mathcal{O})}^!$.

With coefficients $\bk=\bZ$ the groups $\pi_{*,1}(\mathbf{BGL})$ are tabulated in \cite[Example 1.3]{RWDedekind}. We first observe that they are all 2-torsion, so base-changing to $\bF_p$ with $p$ odd we have $\pi_{*,1}(\mathbf{BGL})=0$. Thus $A_{\mathrm{Pic}(\mathcal{O})} \simeq \tau_{\leq 0} \mathbf{BGL}$ can be obtained from $\mathbf{BGL}$ by attaching $E_\infty$-cells of homological degree $\geq 3$, and it follows that $\Delta_\mathbf{BGL} \to A^!_{\mathrm{Pic}(\mathcal{O})}$ is an isomorphism in gradings $<3$. It follows from Theorem \ref{thm:IdealsSameForSmallAlgebra} (\ref{it:IdealsSameForSmallAlgebra:2}) and Proposition \ref{prop:ChangeOfHopfAlg} (\ref{it:ChangeOfHopfAlg:1}) that $\mathbf{BGL} \to \mathbf{bgl} \to A_{\mathrm{Pic}(\mathcal{O})}$ satisfies property ($\dagger$) with $\theta = \tfrac{2}{3}$, and hence by Theorem \ref{thm:DetNilp} (\ref{it:DetNilp:1}) that $\pi_{n,d}(\mathbf{BGL}/\rho)=0$ for $d < \tfrac{2}{3}(n-1)$ for any $\rho \in \mathrm{Pic}(\mathcal{O}) \subset \pi_{1,0}(\mathbf{BGL})$. This improves on \cite[Theorem 1.1]{RWDedekind} in the case of $\bF_p$-coefficients with $p$ odd (but could also have been done with the methods of that paper).

More interesting is the situation for $\bk=\bF_2$. In this case, writing $\hat{\ell} := \ell+s$ we have $(A_{\mathrm{Pic}(\mathcal{O})}^!)^\vee =\bF_2[s, \hat{\ell}]/(\hat{\ell}^2)$ with $s$ and $\hat{\ell}$ primitive. Thus it is a tensor product of Hopf algebras, and $A_{\mathrm{Pic}(\mathcal{O})}^! = \Gamma_{\bF_2}[\hat{\sigma}] \otimes \Lambda_{\bF_2}[\bar{\lambda}]$ as Hopf algebras with $\hat{\sigma} := \bar{\sigma} + \bar{\lambda}$. On the other hand, we may understand a cellular model for $\mathbf{BGL}$ in low degrees using the tabulated homology groups and the following, which refers to the notation of \cite[Appendix A]{RWDedekind} and also to the classes 
\begin{align*}
X &:=[-1] \in H_1(\GL(\cO);\bZ) \subset \pi_{1,1}(\mathbf{BGL})\\
X' &: =[-1] \in H_1(\GL(\mathfrak{l});\bZ) \subset \pi_{1,1}(\mathbf{BGL}).
\end{align*}

\begin{lem}\label{lem:DLInBGL}
There are identities
\begin{align*}
Q^1_\bZ(\sigma) &= \sigma X - 3T\\
Q^1_\bZ(\lambda) &= C - (\sigma X- 3T)
\end{align*}
in $\pi_{2,1}(\mathbf{BGL})$.
\end{lem}
\begin{proof}
To follow this we refer to \cite[Appendix A]{RWDedekind}. The element $Q^1_\bZ(\sigma)$ is represented by the matrix $\left( \begin{smallmatrix}
0 & 1\\
1 & 0
\end{smallmatrix} \right) = EA \in \GL(\cO \oplus \cO)$ so using the relations in that Appendix we have $Q^1(\sigma) = \sigma X - 3T$. The element $Q^1_\bZ(\lambda)$ is represented by the matrix $\left(\begin{smallmatrix}
0 & 1\\
1 & 0
\end{smallmatrix} \right) \in \GL(\mathfrak{l} \oplus \mathfrak{l})$ which swaps the two copies of $\mathfrak{l}$, but we must identify this group with $\GL(\cO \oplus \cO)$. This may be done using the isomorphism
\begin{align*}
\cO \oplus \cO &\overset{\sim}\lra \mathfrak{l} \oplus \mathfrak{l}\\
(x,y) &\longmapsto x\cdot(1+\sqrt{-5}, 2) + y\cdot(2, 1-\sqrt{-5}),
\end{align*}
under which the swap matrix corresponds to 
$$\left( \begin{smallmatrix}
-2\sqrt{-5} & -4-\sqrt{-5}\\
-4+\sqrt{-5} & 2\sqrt{-5}
\end{smallmatrix} \right) = C A^{-1} E^{-1} \in \GL(\cO \oplus \cO).$$ 
Using the relations in that appendix we obtain $Q^1_\bZ(\lambda) = C - (\sigma X- 3T)$.
\end{proof}

It follows that over $\bF_2$ there is a model of the form
\begin{align*}
\mathbf{BGL} &\simeq \gE_\infty(S^{1,0}\sigma \oplus S^{1,0}\lambda \oplus S^{1,1}X \oplus S^{1,1} X' \oplus S^{2,1} U \oplus S^{2,1} B \oplus S^{2,1} C' \oplus S^{2,1} D')\\
& \quad\quad \cup^{E_\infty}_{\lambda^2-\sigma^2} D^{2,1}\rho \cup^{E_\infty} \cdots,
\end{align*}
where we omit the elements $T$ and $C$ as these are $E_\infty$-decomposable by Lemma \ref{lem:DLInBGL}. 

For $\mathbf{RBS}$ we may develop the analogous tabulation of $\pi_{*,1}(\mathbf{RBS})$ using \cite[Lemma 10.14]{Jansen2}, which says that the map
$$H_1(\GL(M);\bk) \lra H_1(|\mathrm{RBS}(M)|;\bk)$$
is surjective, and its kernel is generated by classes $[P]$ where $P$ preserves some splittable flag of $M$. In the case $\cO$ is a Dedekind domain and $M$ has rank $2$, preserving a splittable flag is precisely the same as $P$ having an eigenvector when acting on $M \otimes_{\cO} \mathrm{frac}(\cO)$, or equivalently $P$ having an eigenvalue in $\mathrm{frac}(\cO)$. Consulting the matrices in \cite[Appendix A]{RWDedekind} one immediately sees that in $\GL(\cO \oplus \cO)$ the matrices $J,T,U,A,E$ all have eigenvalues $\pm 1$. Less obvious is that $TB$ and $CEA$ both have eigenvalue 1 too. Similarly, in $\GL(\cO \oplus \mathfrak{l})$ the matrices $J,A,V,C,D,E$ all have eigenvalues $\pm  1$. It follows that the classes $U,T,B,C,C',D' \in \pi_{2,1}(\mathbf{BGL})$ all die under the surjection to $\pi_{2,1}(\mathbf{RBS})$. (The classes $\sigma^i X$ and $\sigma^i X'$ do not, as $\mathbf{BGL}$ and $\mathbf{RBS}$ have the same stable homology.) Thus we have
$$\pi_{*,1}(\mathbf{RBS}) = A_{\mathrm{Pic}(\mathcal{O})}\{X, X'\}/(\lambda X - \sigma X', \sigma X - \lambda X')$$
as an $A_{\mathrm{Pic}(\mathcal{O})}$-module. Using Lemma \ref{lem:DLInBGL} we see that $Q^1(\sigma) = \sigma X$ and $Q^1(\lambda) = -\sigma X$. It follows that 
over $\bF_2$ there is a model of the form
\begin{align*}
\mathbf{RBS} \simeq& \gE_\infty(S^{1,0}\sigma \oplus S^{1,0}\lambda \oplus S^{1,1}X \oplus S^{1,1} X')\\
&\quad\quad  \cup^{E_\infty}_{\lambda^2-\sigma^2} D^{2,1}\rho \cup^{E_\infty}_{Q^1(\sigma)-\sigma X} D^{2,2}\rho_2 \cup^{E_\infty}_{Q^1(\lambda)+\sigma X} D^{2,2}\rho_3 \cup \cdots,
\end{align*}

From these two models, using the formulas of Wang \cite[Theorem 5.6]{wang} we determine the stability Hopf algebras in a range:
\setlength{\tabcolsep}{5pt} 
\renewcommand{\arraystretch}{1.5} 
\begin{table}[h]
\begin{tabular}{l|llll}
$n$ & 0 & 1 & 2 & 3   \\ \hline
$\Delta_\mathbf{BGL}(n)$ & $\bF_2\{1\}$  & $\bF_2\{\bar{\sigma}, \bar{\lambda}\}$  & $\bF_2\{\bar{\sigma}^2, \bar{\sigma}\bar{\lambda}, \bar{\lambda}^2,  \bar{U}, \bar{B}, \bar{C}', \bar{D}', \bar{\rho}\}$  & $\bF_2\{?\}$  \\
$\Delta_\mathbf{RBS}(n)$ & $\bF_2\{1\}$  & $\bF_2\{\bar{\sigma}, \bar{\lambda}\}$  & $\bF_2\{ \bar{\sigma}\bar{\lambda},  \bar{\rho}\}$  & $\bF_2\{?\}$  \\
$A^!_{\mathrm{Pic}(\mathcal{O})}$ & $\bF_2\{1\}$  & $\bF_2\{\hat{\sigma}, \bar{\lambda}\}$  & $\bF_2\{\hat{\sigma}^{[2]}, \bar{\lambda} \bar{\sigma}' \}$  & $\bF_2\{\hat{\sigma}\hat{\sigma}^{[2]}, \bar{\lambda} \hat{\sigma}^{[2]}\}$  
\end{tabular}
\end{table}

\noindent The only interesting coproduct is $\psi(\bar{\rho}) = \bar{\rho} \otimes 1 + \bar{\lambda} \otimes \bar{\lambda} + \bar{\sigma} \otimes \bar{\sigma}  + 1 \otimes \bar{\rho}$. The maps to $A^!_{\mathrm{Pic}(\mathcal{O})}$ are given by $\bar{\sigma} \mapsto \hat{\sigma}+ \bar{\lambda}$ and $\bar{\lambda} \mapsto \bar{\lambda}$, $\bar{U}, \bar{B}, \bar{C}', \bar{D}' \mapsto 0$, and $\bar{\rho} \mapsto \hat{\sigma}^{[2]} + \bar{\lambda}\hat{\sigma}$ (which may be seen by taking coproducts). In particular, we see that $\Delta_\mathbf{RBS} \to A^!_{\mathrm{Pic}(\mathcal{O})}$ is an isomorphism in gradings $\leq 2$, and it is an epimorphism in all gradings as discussed above. This has the following consequence for its homological stability.

\begin{cor}
For $\cO = \bZ[\sqrt{-5}]$ and any $[L] \in \mathrm{Pic}(\cO)$, the map
$$(L \oplus -)_* : H_d(|\mathrm{RBS}(M)|;\bk) \lra H_d(|\mathrm{RBS}(L \oplus M)|;\bk)$$
is an epimorphism for $d< \tfrac{2}{3}\mathrm{rk}(M)$, and an isomorphism for $d< \tfrac{2}{3}\mathrm{rk}(M)-1$.
\end{cor}
\begin{proof}
By Theorem \ref{thm:IdealsSameForSmallAlgebra} (\ref{it:IdealsSameForSmallAlgebra:2}) and Proposition \ref{prop:ChangeOfHopfAlg} (\ref{it:ChangeOfHopfAlg:1}) the composition $\mathbf{RBS} \to \mathbf{rbs} \to A_{\mathrm{Pic}(\cO)}$ satisfies hypothesis ($\dagger$) of Theorem~\ref{thm:DetNilp}~(\ref{it:DetNilp:1}) with $\theta=\tfrac{2}{3}$, and hence by that theorem for any $\rho \in \mathrm{Pic}(\cO) \subset \pi_{1,0}(\mathrm{RBS})$ the object $\mathbf{RBS}/\rho$ has the same vanishing line as $A_{\mathrm{Pic}(\mathcal{O})}/\rho$ up to slope $\tfrac{2}{3}$. The latter vanishes for $d < \tfrac{2}{3}(n-1)$, so $\mathbf{RBS}/\rho$ does too.
\end{proof}

\begin{rem}
It seems too optimistic to hope that $\Delta_\mathbf{RBS} \to A^!_{\mathrm{Pic}(\mathcal{O})}$ is an isomorphism for general Dedekind domains, or even for rings of integers in number fields (though we have seen in Section \ref{sec:PIDs} that it is the case for Euclidean domains): if this were so, then $\mathbf{rbs} = A_{\mathrm{Pic}(\mathcal{O})}$ and so $\mathbf{RBS}$ would have slope 1 stability. It would be nice to have a concrete counterexample: probably \cite{MillerPatztWilsonYasaki} is a good starting place.
\end{rem}

By Remark \ref{rem:WangNonvanishing} the $E_\infty$-algebra $\mathbf{BGL}$ cannot have homological stability with slope $> \tfrac{2}{3}$, but the situation for slopes $\leq \tfrac{2}{3}$ is not yet clear. Each of $U,T,B,C,C',D' \in \pi_{2,1}(\mathbf{BGL})$ induce secondary stabilisation maps of $\mathbf{BGL}/\sigma$ or $\mathbf{BGL}/\lambda$. An argument parallel to that of Section \ref{sec:PIDs} shows that $T = \sigma X + Q^1(\sigma)$ and $C = Q^1(\sigma) + Q^1(\lambda)$ induce nilpotent maps on $\mathbf{bgl}/\sigma$ and $\mathbf{bgl}/\lambda$, but the fate of the stabilisation maps $U\cdot-, B\cdot -, C'\cdot -, D'\cdot -$ on these depends on the structure of $\Delta_\mathbf{BGL}$ in higher degrees (in the form of questions such as: does the term $\bar{U} \otimes \bar{U}$ arise in the coproduct of any elements of degree 4?) and cannot be determined from the data in the table above.

The classes $U,T,B,T+C,C',D' \in \pi_{2,1}(\mathbf{BGL})$ are all annihilated by $\sigma$ and by $\lambda$, so there must be $E_\infty$-$(3,2)$-cells enforcing this. It follows that there are elements of $\Delta_\mathbf{BGL}(3)$ whose reduced coproduct is each of 
$$\bar{\sigma} \otimes \bar{U}, \bar{\sigma} \otimes \bar{\sigma}^2, \bar{\sigma} \otimes \bar{B}, \bar{\sigma} \otimes \bar{\lambda}^2, \bar{\sigma} \otimes \bar{C}', \bar{\sigma} \otimes \bar{D}'$$
and similarly for $\bar{\lambda} \otimes -$; similarly also for the factors reversed. With the fact that it also surjects onto $A^!_{\mathrm{Pic}(\mathcal{O})}(3)$ it follows that $\Delta_\mathbf{BGL}(3)$ is at least $26$-dimensional.

\subsection{Mapping class groups}

Galatius, Kupers, and the author \cite{e2cellsII} have constructed and analysed an $E_2$-algebra $\mathbf{MCG}$ in $\mathsf{D}(\bZ)^\bZ$ having
$$\pi_{g,d}(\mathbf{MCG}) \cong H_d(\Gamma_{g,1} ; \bZ),$$
where $\Gamma_{g,1}$ denotes the mapping class group of a genus $g$ surface with one boundary component. It satisfies axiom (C) trivially, satisfies (SCE) by \cite[Corollary 3.5]{e2cellsII}, and also satisfies (F) (because e.g.\ the $\Gamma_{g,1}$ are virtual duality groups). The endomorphism
$$\sigma : S^{1,0} \otimes \mathbf{MCG} \lra \mathbf{MCG},$$
given by multiplication by the canonical element $\sigma \in \pi_{1,0}(\mathbf{MCG}) = H_0(\Gamma_{1,1};\bZ)$, has $\pi_{g,d}(\mathbf{MCG}/\sigma)=0$ for $d < \tfrac{2}{3}g$, and there is an endomorphism
$$\varphi : S^{3,2} \otimes \mathbf{MCG}/\sigma \lra \mathbf{MCG}/\sigma$$
with $\pi_{g,d}(\mathbf{MCG}/(\sigma, \varphi))=0$ for $d < \tfrac{3}{4}g$ (and even $\pi_{g,d}(\mathbf{MCG}/(\sigma, \varphi)) \otimes \bQ=0$ for $d < \tfrac{4}{5}g$). Applying Theorem \ref{thm:HigherStabRanges} gives:

\begin{cor}\label{cor:MCGConnectivities}
The fibre of $\mathbf{MCG} \to L^f_{2/3}(\mathbf{MCG}) = \sigma^{-1}\mathbf{MCG}$ has trivial homotopy groups in bidegrees $(g,d)$ satisfying $d < \tfrac{2g-1}{3}$ or $d < 0$.

The fibre of $\mathbf{MCG} \to L^f_{3/4}(\mathbf{MCG})$ has trivial homotopy groups in bidegrees $(g,d)$ satisfying $d < \tfrac{3g-4}{4}$ or $d < \tfrac{2g-4}{3}$.
\end{cor}

Our first goal in this section is to spell out the consequences of Adams Periodicity (Theorem \ref{thm:AdamsPeriodicityGeneral}) in this example. This concerns the homotopy groups of the object
$$\gF := \mathrm{fib}(\mathbf{MCG} \to \sigma^{-1} \mathbf{MCG})$$
in a slope $\tfrac{3}{4}$ range of degrees. We consider this as describing the ``purely unstable'' part of the homology of mapping class groups. It has $\pi_{g,d}(\gF) = H_{d+1}(\Gamma_{\infty,1}, \Gamma_{g,1} ;\bZ)$, and from the exact sequence
$$ H_{d+1}(\Gamma_{g,1} ;\bZ) \to H_{d+1}(\Gamma_{\infty,1} ;\bZ) \overset{\partial}\to \pi_{g,d}(\gF) \to H_d(\Gamma_{g,1} ;\bZ) \to H_d(\Gamma_{\infty,1}; \bZ)$$
we see that $\pi_{g,*}(\gF)$ has a contribution from the homology of $\Gamma_{g,1}$ which is in the kernel of stabilisation, as well as from the stable homology which does not destabilise to genus $g$.

\subsubsection{The case $\tfrac{1}{10} \in \bk$}

In this case then after base-change to $\bk$ the endomorphism $\varphi$ may be realised as multiplication by a certain class $\tfrac{1}{10}\lambda \in \pi_{3,2}(\mathbf{MCG}_\bk)$, cf.\ \cite[Lemma 5.10]{e2cellsII}. In particular, the endomorphism $\varphi$ is defined \emph{before} taking the quotient by $\sigma$. This drastically simplifies matters, and we deal with this case first.

\begin{prop}
Suppose that $\tfrac{1}{10} \in \bk$. Then there is a map
$$\tfrac{1}{10}\lambda \cdot - : \pi_{g-3, d-2}(\gF_\bk) \lra \pi_{g,d}(\gF_\bk)$$
which is an epimorphism for $d < \tfrac{3g-1}{4}$ and an isomorphism for $d < \tfrac{3g-5}{4}$. If $\bk=\bQ$ then it is an epimorphism for $d < \tfrac{4g-1}{5}$ and an isomorphism for $d < \tfrac{4g-6}{5}$.
\end{prop}
\begin{proof}
Writing $\varphi : S^{3,2} \otimes \mathbf{MCG}_\bk \to \mathbf{MCG}_\bk$ for multiplication by $\tfrac{1}{10}\lambda$, the claim is that $\pi_{g,d}(\gF_\bk/\varphi)=0$ for $d < \tfrac{3g-1}{4}$. This can be approached similarly to the proof of Theorem \ref{thm:HigherStabRanges}. Namely we have $\sigma^{-1}\gF_\bk=0$, as $\mathbf{MCG}_\bk \to \sigma^{-1}\mathbf{MCG}_\bk$ becomes an equivalence when we invert $\sigma$, and so $\sigma^{-1}(\gF_\bk/\varphi) = (\sigma^{-1} \gF_\bk)/\varphi=0$. As in the proof of Theorem \ref{thm:HigherStabRanges} it follows that $S^{0,1} \otimes \gF_\bk/\varphi$ has a filtration with associated graded
$$\bigoplus_{p \geq 1} S^{-p, 0} \otimes \gF_\bk/(\sigma, \varphi).$$
But $\gF_\bk/\sigma \overset{\sim}\to \mathbf{MCG}_\bk/\sigma$, as $(\sigma^{-1} \mathbf{MCG}_\bk)/\sigma=0$, so $\gF_\bk/(\sigma, \varphi) \simeq \mathbf{MCG}_\bk/(\sigma, \varphi)$. The homotopy groups of this object vanish for $d < \tfrac{3g}{4}$, or for $d < \tfrac{4g}{5}$ if $\bk=\bQ$. The spectral sequence then shows that $S^{0,1} \otimes \gF_\bk/\varphi$ has trivial homotopy groups in bigradings $(g,d)$ satisfying $d < \tfrac{3g+3}{4}$, or for $d < \tfrac{4g+4}{5}$ if $\bk=\bQ$. Desuspending once gives the claimed result.
\end{proof}

\subsubsection{The cases $\bk \in \{\bF_2, \bF_5\}$}

In these cases the secondary stabilisation map $\varphi$ \emph{cannot} be obtained from an endomorphism of $\mathbf{MCG}_\bk$ itself, but only of $\mathbf{MCG}_\bk/\sigma$. The fibre of the induced map
$$\gF = \mathrm{fib}(\mathbf{MCG} \to \sigma^{-1} \mathbf{MCG}) \lra M_{3/4}^f(\mathbf{MCG}) = \mathrm{fib}(L_{3/4}^f(\mathbf{MCG}) \to L_{2/3}^f(\mathbf{MCG}))$$
is the same as that of $\mathbf{MCG} \to L_{3/4}^f(\mathbf{MCG})$, so it follows from Corollary \ref{cor:MCGConnectivities} that it has trivial homotopy groups in bidegrees $(g,d)$ satisfying $d < \tfrac{3g-4}{4}$ or $d < \tfrac{2g-4}{3}$. So in this range of degrees we may as well work with the monochromatic object $M_{3/4}^f(\mathbf{MCG})$, and here we may apply our Adams periodicity results. These involve a $\varphi$-self map of $\mathbf{MCG}_{\bF_p}/(\sigma^{p^r})$, which exists

\begin{lem}\label{lem:MCGlongSTcx}
Let $p=2$ or 5. For each $r \geq 0$ there are maps 
$$\varphi_{r} : S^{3 \cdot p^{r}, 2 \cdot p^{r}} \otimes \mathbf{MCG}_{\bF_p}/(\sigma^{p^r}) \lra \mathbf{MCG}_{\bF_p}/(\sigma^{p^r})$$
whose cofibre $\mathbf{MCG}_{\bF_p}/(\sigma^{p^r}, \varphi_{r})$ has trivial homotopy groups in bidegrees $(g,d)$ satisfying $d  <  \tfrac{3g -4(p^r-1)}{4}$.
\end{lem}
\begin{proof}
We use the map
$$\gA := \gE_2(S^{1,0} \sigma \oplus S^{1,1} t) \cup^{E_2}_{10\sigma t} \gD^{2,2}\rho_1 \cup^{E_2}_{Q^1_\bZ(\sigma)-3\sigma t} \gD^{2,2} \rho_2 \cup^{E_2}_{\sigma^2 t} \gD^{3,2} \rho_3 \lra \mathbf{MCG}$$
constructed in \cite[\S 5.3.2]{e2cellsII}, which satisfies $H_{g,d}(\mathbf{MCG}, \gA) = 0$ for $4d \leq 3g - 1$ by \cite[Lemma 5.11]{e2cellsII}. (In fact $\rho_1$ becomes trivially attached when we tensor with $\bF_p$ and so could be omitted.) It suffices to construct $\varphi_{r} : S^{3 \cdot p^{r}, 2 \cdot p^{r}} \otimes \gA_{\bF_p}/(\sigma^{p^r}) \to \gA_{\bF_p}/(\sigma^{p^r})$ having the stated property, and then base change to $\mathbf{MCG}_{\bF_p}$.

We discuss the case $p=2$, and mention the minor changes in the case $p=5$ at the end. Consider the canonical multiplicative filtration of $\gA_{\bF_2}$ and the filtered map $\tilde{\sigma} : S^{1,0,-1} \to \fil_*\gA_{\bF_2}$ detected by $\sigma \in \pi_{1,0,-1}(C\tau \otimes\fil_* \gA_{\bF_2})$. The $E^1$-page $\pi_{*,*,*}(C\tau \otimes \fil_* \gA_{\bF_2})$ of the spectral sequence of the filtered object $\fil_* \gA_{\bF_2}$ is
\begin{align*}
& \pi_{*,*,*}(\gE_2(S^{1,0,-1}\sigma \oplus S^{1,1,-1}t \oplus S^{2,2,-1} \rho_1 \oplus S^{2,2,-1} \rho_2 \oplus S^{3,2,-1} \rho_3))\\
&\quad\quad = \bF_2[\sigma, t, Q^1(\sigma), \rho_2, \rho_3] \otimes \left\{\substack{\text{free graded-commutative algebra} \\ \text{on generators of slope $\geq \tfrac{3}{4}$}}\right\},
\end{align*}
and there are differentials $d^1(\rho_2) = Q^1(\sigma)-\sigma t$, $d^1(\rho_3)=0$, and $d^2(\rho_3) = \sigma^2 t$. 

The class $\rho_3$ survives until $E^2_{3,2,-1}(\fil_* \gA_{\bF_2})$, and so the class $\rho_3^{2^r}$ survives until $E^{2 \cdot 2^{r}}_{3 \cdot 2^r,2 \cdot 2^r,- 2^r}(\fil_* \gA_{\bF_2})$; so its image $[\rho_3^{2^r}]$ survives until $E^{2 \cdot 2^{r}}_{3 \cdot 2^r,2 \cdot 2^r,- 2^r}(\fil_* \gA_{\bF_2}/(\tilde{\sigma}^{2^r}))$. We wish to show that it is a permanent cycle in the latter spectral sequence. For this, and for later use, we establish two kinds of vanishing ranges.

\begin{samepage}
\begin{claim}\mbox{}
\begin{enumerate}[(i)]
\item $E^2_{g,d,f}(\fil_*\gA_{\bF_2}/(\tilde{\sigma}^{2^r}))=0$ for $2d - g + f < -2(2^r-1)$,

\item $\pi_{g,d,f}(\fil_*\gA_{\bF_2}/(\tilde{\sigma}^{2^r}))$ is $\tau$-torsion for $d < \tfrac{2}{3}(g-(2^{r}-1))$.

\end{enumerate}
\end{claim}
\end{samepage}

\begin{proof}[Proof of Claim]
For (i), consider the linear function $\ell(g,d,f) := 2d-g+f$ on trigradings $(g,d,f)$. The multiplicative generators of
$$E^1_{*,*,*}(\fil_* \gA_{\bF_2}) = W_1(S^{1,0,-1}\sigma \oplus S^{1,1,-1}t \oplus S^{2,2,-1} \rho_1 \oplus S^{2,2,-1} \rho_2 \oplus S^{3,2,-1} \rho_3)$$
are $\sigma, t, \rho_1, \rho_2, \rho_3$, then Lie words on these given by the Browder bracket, then certain Dyer--Lashof (and ``top'') operations applied to these. We have
$$\ell(\sigma) = -2 \quad \ell(t)=0 \quad \ell(\rho_1) = 1 \quad \ell(\rho_2) = 1 \quad \ell(\rho_3) = 0$$
as well as $\ell([x_1,x_2]) = 2 + \ell(x_1) + \ell(x_2)$ and $\ell(Q^s(x)) = 2\ell(x)+2s$ and $\ell(\xi(x)) = 2\ell(x) + 2 + 2d(x)$. Thus $\ell(Q^1(\sigma)) = -2$, but all other multiplicative generators have $\ell \geq 0$. (Here we use that $[\sigma, \sigma]=0$, and that $Q^s Q^1(\sigma)$ is a multiplicative generator only when $s \geq 2$, in which case $\ell(Q^s Q^1(\sigma)) = 2\cdot(-2) + 2s \geq 0$.) The differential $d^1(\rho_1) = Q^1(\sigma)-\sigma t$ shows that $E^2_{*,*,*}(\fil_* \gA_{\bF_2}/(\tilde{\sigma}^{2^r}))$ is a subquotient of
$$\bF_2[\sigma]/(\sigma^{2^r}) \otimes \left\{\substack{\text{free graded-commutative algebra} \\ \text{on generators with $\ell \geq 0$}}\right\}.$$
As $\ell(\sigma)=-2$, this means that $E^2_{*,*,*}(\fil_* \gA_{\bF_2}/(\tilde{\sigma}^{2^r}))$ is supported in tridegrees $(g,d,f)$ with $\ell(g,d,f) \geq -2(2^r-1)$, so it vanishes for $2d-g+f < -2(2^r-1)$.

For (ii), a much simpler consideration using the linear function $3d-2g$ shows that $E^2_{*,*,*}(\fil_* \gA_{\bF_2}/(\tilde{\sigma}^{2^r}))$ is a subquotient of
$$\bF_2[\sigma]/(\sigma^{2^r}) \otimes \left\{\substack{\text{free graded-commutative algebra} \\ \text{on generators with $\tfrac{d}{g} \geq \tfrac{2}{3}$}}\right\},$$
so that it vanishes in tridegrees $(g,d,f)$ satisfying $d < \tfrac{2}{3}(g-(2^{r}-1))$. The vanishing of the $E^2$-page in this range implies the filtered homotopy groups being $\tau$-torsion, as in the proof of (I) $\Rightarrow$ (A) in Proposition \ref{prop:WeaklyTypeiEq}.
\end{proof}

In the spectral sequence $E_{g,d,f}^*(\fil_* \gA_{\bF_2}/(\tilde{\sigma}^{2^r}))$ a potential differential $d^s([\rho_3^{2^{r}}])$ lands in $E^s_{3 \cdot 2^r,2 \cdot 2^r-1,- 2^r-s}(\fil_* \gA_{\bF_2}/(\tilde{\sigma}^{2^r}))$. We have
$$2(2 \cdot 2^{r}-1)-(3 \cdot 2^{r}) + (-2^{r}-s)  < -2(2^r-1)$$
as long as $s > 2 \cdot 2^{r}-4$, so in particular for $s \geq 2 \cdot 2^{r}$. By part (i) of the Claim it follows that for $s\geq 2 \cdot 2^{r}$ the element $d^s([\rho_3^{2^{r}}])$ lies in a group that is zero, and therefore $[\rho_3^{2^{r}}] \in E^{2^{r+1}}_{3 \cdot 2^r,2 \cdot 2^r,- 2^r}(\fil_* \gA_{\bF_2}/(\tilde{\sigma}^{2^r}))$ is a permanent cycle. It may be represented by a filtered map
$$\tilde{\varphi}_{r}' : S^{3 \cdot 2^{r}, 2 \cdot 2^{r}, -2^{r}} \otimes \fil_* \gA_{\bF_2} \lra \fil_* \gA_{\bF_2}/(\tilde{\sigma}^{2^r}).$$

The composite
$$S^{3 \cdot 2^{r}, 2 \cdot 2^{r}, -2^{r}} \otimes S^{2^r, 0, -2^r} \otimes \fil_* \gA_{\bF_2} \overset{\tilde{\sigma}^{2^r}}\lra S^{3 \cdot 2^{r}, 2 \cdot 2^{r}, -2^{r}} \otimes \fil_* \gA_{\bF_2} \overset{\tilde{\varphi}'_{r}}\lra \fil_* \gA_{\bF_2}/(\tilde{\sigma}^{2^r})$$
represents a class in the filtered homotopy group $\pi_{4 \cdot 2^{r}, 2 \cdot 2^{r}, -2 \cdot 2^{r}}(\fil_* \gA_{\bF_2}/(\tilde{\sigma}^{2^r}))$, and as $2 \cdot 2^r < \tfrac{2}{3}(4 \cdot 2^r - (2^r-1))$ it follows from part (ii) of the Claim that this filtered homotopy group is $\tau$ torsion. Thus $(\tilde{\varphi}'_{r} \cdot \tau)\circ\tilde{\sigma}^{2^r}$ is nullhomotopic as a filtered map, so the map $\tilde{\varphi}'_{r} \cdot \tau$ extends to a filtered endomorphism
$$\tilde{\varphi}_{r} : S^{3 \cdot 2^{r}, 2 \cdot 2^{r}, -2^{r}+1} \otimes \fil_* \gA_{\bF_2}/(\tilde{\sigma}^{2^r}) \lra \fil_* \gA_{\bF_2}/(\tilde{\sigma}^{2^r}).$$
We let $\varphi_r$ denote the map induced by $\tilde{\varphi}_{r}$ on colimits.

As $\pi_{*,*,*}(C\tau \otimes \fil_* \gA_{\bF_2}/(\tilde{\sigma}^{2^r})) = \pi_{*,*,*}(C\tau \otimes \fil_* \gA_{\bF_2})/({\sigma}^{2^r})$ we see that $C\tau \otimes \tilde{\varphi}_{r}$ induces the zero map, so we have
$$\pi_{*,*,*}(C\tau \otimes \fil_* \gA_{\bF_2}/(\tilde{\sigma}^{2^r}, \tilde{\varphi}_{r})) = \pi_{*,*,*}(C\tau \otimes \fil_* \gA_{\bF_2})/({\sigma}^{2^r}) \otimes \Lambda[x]$$
with $x$ of tridegree $(3 \cdot 2^{r}, 2 \cdot 2^{r}+1, -2^{r}+1)$ supporting a differential $d^1(1 \otimes x) = \rho_3^{2^r}\otimes 1$. With the differential $d^1(\rho_2 \otimes 1) = (Q^1(\sigma)-\sigma t) \otimes 1$, this implies that $E^{2}_{*,*,*}(\fil_* \gA_{\bF_2}/(\tilde{\sigma}^{2^r}, \tilde{\varphi}_{r}))$ is a subquotient of
$$\bF_2[\sigma, \rho_3]/(\sigma^{2^r}, \rho_3^{2^r}) \otimes \left\{\substack{\text{free graded-commutative algebra} \\ \text{on generators with $\tfrac{d}{g} \geq \tfrac{3}{4}$}}\right\}.$$
From this we see that $\gA_{\bF_2}/(\sigma^{2^r}, \varphi_r)$ has a vanishing line of slope $\tfrac{3}{4}$ passing through the bidegree of the class $\sigma^{2^r-1} \rho_3^{2^{r}-1}$, namely $(2^r-1 + 3(2^{r}-1), 2(2^{r}-1))$. That is, 
$$\pi_{g,d}(\gA_{\bF_2}/(\sigma^{2^r}, \varphi_r))=0 \text{ for } d-2(2^{r}-1) < \tfrac{3}{4}(g-(2^r-1) - 3(2^{r}-1)),$$
which upon rearranging gives the claimed range.

If $p=5$ then the argument is identical, except that $Q^1(\sigma)$ should be denoted by $\tfrac{1}{2}[\sigma,\sigma]$.
\end{proof}

\begin{cor}
For $p=2$ or $5$ and any $r \geq 0$ there are isomorphisms
$$(\varphi_r)_* : \pi_{g-3 \cdot p^r, d-2 \cdot p^r}(M_{3/4}^f(\mathbf{MCG}_{\bF_p})) \lra \pi_{g,d}(M_{3/4}^f(\mathbf{MCG}_{\bF_p}))$$
defined as long as $d <  \tfrac{2g + 2 \cdot p^r -4}{3}$. Thus there are isomorphisms
$$(\varphi_r)_* : \pi_{g-3 \cdot p^r, d-2 \cdot p^r}(\gF_{\bF_p}) \lra \pi_{g,d}(\gF_{\bF_p})$$
defined as long as $d < \tfrac{3g-4-p^r}{4}$, and $d <  \tfrac{2g + 2 \cdot p^r -4}{3}$.
\end{cor}
\begin{proof}
The first part will follow from Theorem \ref{thm:AdamsPeriodicityGeneral} using the Smith--Toda complexes $\mathbf{MCG}_{\bF_p}/(\sigma^{p^r}, \varphi_r)$ produced by Lemma \ref{lem:MCGlongSTcx}. Concretely, it follows from Corollary \ref{cor:MCGConnectivities} that $M_{3/4}^f(\mathbf{MCG}_{\bF_p})$ has trivial homotopy groups in bidegrees $(g, d)$ satisfying $d < \tfrac{2g-1}{3} = \tfrac{2}{3}g - \tfrac{1}{3}$, so we take $(\lambda',\kappa')=(\tfrac{2}{3}, - \tfrac{1}{3})$. Then Theorem \ref{thm:AdamsPeriodicityGeneral} says to take
$$b := - \tfrac{1}{3} + (\tfrac{2}{3} p^r-0-1) = \tfrac{2 \cdot p^r-4}{3}$$
and that $\varphi_r$ induces an equivalence
$$S^{3 \cdot p^r, 2 \cdot p^r} \otimes \tau_{\leq b}^{2/3} M_{3/4}^f(\mathbf{MCG}_{\bF_p}) \overset{\sim}\lra \tau_{\leq b}^{2/3} M_{3/4}^f(\mathbf{MCG}_{\bF_p}).$$
Taking homotopy groups gives the first part of the statement. The second part follows using that the fibre of $\gF \to M_{3/4}^f(\mathbf{MCG}_{\bF_p})$ has trivial homotopy groups in bidegrees $(g,d)$ satisfying $d < \tfrac{3g-4}{4}$ or $d < \tfrac{2g-4}{3}$.
\end{proof}

\subsubsection{Detecting a periodic family}\label{sec:MCGPeriodicExample}

In a different direction, we will now demonstrate the idea suggested in Section \ref{sec:DetectingPeriodic}. (A much more sophisticated demonstration is given by Wang \cite{wang}.) We use it to prove the following.

\begin{thm}\label{thm:MCGLambdaAndXHit}
There is a surjective map of graded-commutative rings
$$\pi_{*,*}(\mathbf{MCG}_\bQ) \lra \mathbb{Q}[\sigma, \lambda],$$
splitting the natural map in the other direction, as well as a surjective map 
$$\pi_{*,*}(\mathbf{MCG}_\bQ/\sigma) \lra \mathbb{Q}[\lambda]$$
of modules over the above, splitting the natural map in the other direction. The secondary stabilisation map $\varphi$ of $\mathbf{MCG}_\bQ/\sigma$ intertwines multiplication by $\lambda$.
\end{thm}

\begin{proof}
Just before \cite[Lemma 6.3]{e2cellsII} we produced a map
$$\gB := \gE_2(S^{1,0}\sigma \oplus S^{2,3}\lambda) \cup^{E_2}_{[\sigma,\sigma]} D^{2,2} \rho \cup^{E_2}_{[\sigma,\lambda]} D^{4,4}\rho' \lra \mathbf{MCG}_\bQ$$
of unital $E_2$-algebras, such that $H^{E_2}_{g,d}(\gI_{\mathbf{MCG}}, \gI_\gB)=0$ for $d < \tfrac{4}{5}g$. Adapting \cite[Theorem 4.11 (ii), Theorem 5.6 (ii)]{wang} to $\bQ$-coefficients, the corresponding stability Hopf algebra is
$$\Delta_\gB = T_\bQ[\bar{\sigma}, \bar{\lambda}]/(\bar{\sigma}^2, \bar{\sigma}\bar{\lambda} + \bar{\lambda}\bar{\sigma}),$$
and there is a map of Hopf algebras $\Delta_\gB \to \Delta_\mathbf{MCG}$. By the vanishing range for relative $E_2$-cells, this map is an isomorphism in gradings $\leq 4$. Thus we have
\setlength{\tabcolsep}{10pt} 
\renewcommand{\arraystretch}{1.5} 
\begin{table}[h]
\begin{tabular}{l|lllllll}
$n$ & 0 & 1 & 2 & 3 & 4 & 5 &6 \\ \hline
$\Delta_\gB(n)$ & $\bQ\{1\}$  & $\bQ\{\bar{\sigma}\}$  & 0  & $\bQ\{\bar{\lambda}\}$  & $\bQ\{\bar{\sigma}\bar{\lambda}\}$  & 0 & $\bQ\{\bar{\lambda}^2\}$ \\
$\Delta_\mathbf{MCG}(n)$ & $\bQ\{1\}$  & $\bQ\{\bar{\sigma}\}$  & 0  & $\bQ\{\bar{\lambda}\}$  & $\bQ\{\bar{\sigma}\bar{\lambda}\}$  & $V$ & $W$
\end{tabular}
\end{table}

\noindent for some vector spaces $V$ and $W$. If $v \in V$ has coproduct $\psi(v) = 1 \otimes v + A \cdot \bar{\sigma} \otimes \bar{\sigma}\bar{\lambda} + B \cdot \bar{\sigma}\bar{\lambda} \otimes \bar{\sigma} + v \otimes 1$ then writing out the formula for coassociativity of the coproduct we find that $A=B=0$, so $V$ consists entirely of primitive elements. If $w \in W$ has coproduct
$\psi(w) = w \otimes 1 + \bar{\sigma} \otimes v + A \cdot \bar{\lambda} \otimes \bar{\lambda} + v' \otimes \bar{\sigma} + 1 \otimes w$
for $v, v' \in V$, then one may check that there are no constraints on $v,v',A$ coming from coassociativity. However:

\begin{claim}
Such elements have $A=0$.
\end{claim}

\begin{proof}[Proof of Claim]
We have elements
\begin{align*}
\sigma = [\bar{\sigma}] &\in \Cotor^1_{\Delta_\mathbf{MCG}}(\bQ,\bQ)_1 = \pi_{1,0}(\mathbf{mcg})\\
\lambda = [\bar{\lambda}] &\in \Cotor^1_{\Delta_\mathbf{MCG}}(\bQ,\bQ)_3 = \pi_{3,2}(\mathbf{mcg})\\
\underline{v} = [v], \underline{v}' = [v'] &\in \Cotor^1_{\Delta_\mathbf{MCG}}(\bQ,\bQ)_5 = \pi_{5,4}(\mathbf{mcg})
\end{align*}
and, by calculating in the cobar complex, a coproduct $\psi(w)$ of the above form implies the relation
$$\sigma \cdot \underline{v} + A \cdot \lambda^2 + \underline{v}' \cdot {\sigma} = 0 \in \Cotor^2_{\Delta_\mathbf{MCG}}(\bQ,\bQ)_6 = \pi_{6,4}(\mathbf{mcg}),$$
so if $A \neq 0$ then it implies that inverting $\lambda$ in the ring $\pi_{*,*}(\mathbf{mcg})$ also inverts $\sigma$. But then we would have
\begin{align*}
\lambda^{-1} \mathbf{mcg}/\sigma &= (\lambda^{-1} \mathbf{mcg}) \otimes_{\mathbf{mcg}} \mathbf{mcg}/\sigma = (\sigma^{-1} \lambda^{-1} \mathbf{mcg}) \otimes_{\mathbf{mcg}} \mathbf{mcg}/\sigma\\
&= (\sigma^{-1}\lambda^{-1} \mathbf{mcg}) \otimes_{\sigma^{-1}\mathbf{mcg}} (\sigma^{-1}\mathbf{mcg}/\sigma)=0
\end{align*}
and so the endomorphism $\lambda \cdot -$ would be nilpotent on $\mathbf{mcg}/\sigma$. By Theorem~\ref{thm:IdealsSameForSmallAlgebra}~(\ref{it:IdealsSameForSmallAlgebra:3}) it would also be nilpotent on $\mathbf{MCG}_\bQ/\sigma$. With e.g.\ \cite[Corollary 6.4]{e2cellsII} this would imply that mapping class groups have rational homological stability with slope $\tfrac{4}{5}$, but in fact slope $\tfrac{2}{3}$ is sharp, by e.g. \cite[Corollary 6.7]{e2cellsII}.
\end{proof}

Form a quotient Hopf algebra $Q$ of $\Delta_\mathbf{MCG}$ by repeatedly quotienting out the two-sided ideal generated by the primitive elements of grading $ \geq 5$, until there are none left. Set $\mathbf{q} := \Cobar(Q)$, so there are unital $E_2$-algebra maps
\begin{equation*}
\begin{tikzcd}
\gB \rar \dar& \mathbf{MCG}_\bQ \dar\\
\mathbf{b} \rar & \mathbf{mcg} \rar & \mathbf{q}.
\end{tikzcd}
\end{equation*}

We now wish to analyse the structure of the Hopf algebra $Q$. By construction 
$$\Prim(Q) = \bQ\{\bar{\sigma}, \bar{\lambda}\}$$
of degrees $1$ and $3$. Passing to dual Hopf algebras, this means there is a surjection $T_\bQ[\bar{\sigma}^\vee, \bar{\lambda}^\vee] \to Q^\vee$. We easily verify that in $Q^\vee$ we have $\bar{\sigma}^\vee \bar{\sigma}^\vee=0$ and $\bar{\sigma}^\vee\bar{\lambda}^\vee + \bar{\lambda}^\vee\bar{\sigma}^\vee =0$, as these are in the range of degrees in which we have complete information about $Q$. Claim 1 above implies that $\bar{\lambda}^\vee\bar{\lambda}^\vee=0$ too, giving a surjection 
$$T_\bQ[\bar{\sigma}^\vee, \bar{\lambda}^\vee]/((\bar{\sigma}^\vee)^2, \bar{\sigma}^\vee\bar{\lambda}^\vee + \bar{\lambda}^\vee\bar{\sigma}^\vee, (\bar{\lambda}^\vee)^2) \lra Q^\vee,$$
which is in fact an isomorphism as any further relation would require killing one of $\bar{\sigma}^\vee$ or $\bar{\lambda}^\vee$. On dualising back we have $Q = \Lambda_\bQ[\bar{\sigma}, \bar{\lambda}]$. This is a tensor product of monogenic Hopf algebras, so we immediately calculate that
$$\pi_{*,*}(\mathbf{q}) = \Cotor_Q^*(\bQ,\bQ)_* = \bQ[\sigma, \lambda].$$
It follows from the construction that the map $\mathbf{MCG}_\bQ \to \mathbf{q}$ sends the classes $\sigma$ and $ \lambda$ in $\pi_{*,*}(\mathbf{MCG}_\bQ)$ to the classes of the same name in $\pi_{*,*}(\mathbf{q})$. This proves the first part of the theorem. For the second part, we just consider the induced map $\mathbf{MCG}_\bQ/\sigma \to \mathbf{q}/\sigma$.
\end{proof}

\begin{rem}
Something like Theorem \ref{thm:MCGLambdaAndXHit} can be done with $\bF_p$-coefficients, using the approximation $\gA \to \mathbf{MCG}_{\bF_p}$ given in \cite[Section 5.3.2]{e2cellsII}. We briefly comment on how it goes. In this case one should form the quotient $\Delta_\mathbf{MCG} \to Q$ by repeatedly killing all primitives in gradings $\geq 4$, and a similar analysis will show that $Q \cong \Lambda_{\bF_p}[\bar{\sigma}, \bar{\rho}_3]$, and will yield a ring homomorphism
$$\pi_{*,*}(\mathbf{MCG}_{\bF_p}) \lra \pi_{*,*}(\mathbf{q}) = \bF_p[\sigma, \rho_3].$$
This can be checked to send $\lambda \in \pi_{3,2}(\mathbf{MCG}_{\bF_p})$ to $10 \rho_3$, so for $p \not\in\{2,5\}$ we again see that it is surjective, but for $p \in \{2,5\}$ it will not be surjective as it will not hit $\rho_3$. (It is interesting to ask what its image is: for example, is $\rho_3^2$ hit? Even after inverting $\sigma$ this is not immediately clear.) However the induced map
$$\pi_{*,*}(\mathbf{MCG}_{\bF_p}/\sigma) \lra \pi_{*,*}(\mathbf{q}/\sigma) = \bF_p[\rho_3]$$
sends $\mu \in \pi_{3,2}(\mathbf{MCG}_{\bF_p}/\sigma)$ to $\rho_3$, and the secondary stabilisation map $\varphi : S^{3,2} \otimes \mathbf{MCG}_{\bF_p}/\sigma \to \mathbf{MCG}_{\bF_p}/\sigma$ constructed in \cite[5.3.1]{e2cellsII} intertwines multiplication by $\rho_3$ on $\pi_{*,*}(\mathbf{q}/\sigma)$, so this module map \emph{is} surjective.
\end{rem}

\subsection{Steinberg homology}

For a ring $\mathcal{O}$ such that every projective module is free, and such that the Tits building is spherical, the relation between the bar construction and $E_1$-homology \cite[Theorem 13.7]{e2cellsI} together with Jansen's theorem \cite[Theorem 10.11]{Jansen2} shows that the homology groups $H_*(\GL(\mathcal{O}^n) ; \St(\mathcal{O}^n))$ with coefficients in the classical Steinberg module arise as the homotopy groups of
\begin{equation}\label{eq:DefnSt}
\mathbf{St} := \Bar(\mathbf{RBS}) \in \mathsf{Alg}_{E_\infty}(\mathsf{D}(\bk)^\bZ),
\end{equation}
by $\pi_{n,d}(\mathbf{St}) \cong H_{d-n}(\GL(\mathcal{O}^n) ; \St(\mathcal{O}^n))$. We wish to explain how a vanishing range for these groups can be approached using the principles of this paper applied to $\mathbf{St}$. It does not strictly speaking fit into the axiomatic framework used here, but can be approached by the same principles.

The following result, with essentially the same proof, has been discovered independently by Galatius and by Kupers--Rudenko--Sierra \cite[Theorem 4.8]{KRS}. It improves by a factor of 2 the range of Ash--Putman--Sam \cite{APS} in the case of $\GL$ for infinite fields (the case of finite fields was done in \cite[Corollary 7.2]{e2cellsIII}).

\begin{thm}\label{thm:VanishingSteinbergHomologyFields}
If $\mathcal{O}$ is an infinite field then $H_{d}(\GL(\mathcal{O}^n);\St(\mathcal{O}^n))=0$ for $d < n-1$.
\end{thm}
\begin{proof}
Such an $\mathcal{O}$ is a ring with many units, so the natural map $\mathbf{BGL} \to \mathbf{RBS}$ is an equivalence, and hence we may instead use $\mathbf{BGL}$ in the definition \eqref{eq:DefnSt} of $\mathbf{St}$.

Now $\Bar(\mathbf{St}) = \Bar(\Bar(\mathbf{BGL})) \simeq B^{E_2}(\mathbf{BGL})$ and in \cite{e2cellsIV} we have studied this in detail. There we showed that $\pi_{n,d}(\Bar(\mathbf{St}))=0$ for $d < 2n$ (\cite[Theorem 6.5]{e2cellsIV}), and also that the slope 2 diagonal
$$\Delta_\mathbf{St} := \bigoplus_{n \geq 0} \pi_{n,2n}(\Bar(\mathbf{St}))$$
is isomorphic as an algebra to the divided power algebra $\Gamma_\bk[\bar{\bar{\sigma}}]$ (\cite[Theorem 6.9]{e2cellsIV}). By \cite[Lemma 6.10]{e2cellsIV} this isomorphism is induced by applying $\Bar$ twice to the 0-truncation map $\mathbf{BGL} \to \bk[\sigma]$, i.e.\ the induced map
$$\Delta_\mathbf{St} \lra \Bar(\Bar(\bk[\sigma])) =  \Bar(\Lambda_\bk[\bar{\sigma}]) = \Gamma_\bk[\bar{\bar{\sigma}}]$$
is an isomorphism of Hopf algebras (and in particular the dual $\Delta^\vee_\mathbf{St}$ is a polynomial algebra on one generator).

Truncation to the slope 2 diagonal gives a bialgebra map $\Bar(\mathbf{St}) \to \Delta_\mathbf{St}$, and taking cobar constructions gives an $E_\infty$-map
$$\mathbf{St} \lra \Cobar(\Delta_\mathbf{St}) \simeq \Lambda_\bk[\bar{\sigma}].$$
This map satisfies hypothesis ($\dagger$) of Theorem \ref{thm:DetNilp} (\ref{it:DetNilp:1}) with $\theta=2$, by the argument of Theorem \ref{thm:IdealsSameForSmallAlgebra} (\ref{it:IdealsSameForSmallAlgebra:1}) and (\ref{it:IdealsSameForSmallAlgebra:2}) with slope 1 replaced by slope 2: this entails only notational changes to the argument. Now $\Lambda_\bk[\bar{\sigma}]$ has vanishing homotopy groups in bidegrees $(n,d)$ satisfying $d < 2n-1$, so by that theorem $\mathbf{St}$ does too.
\end{proof}

\begin{rem}
As a consequence of this theorem, it follows that the $E_\infty$-map $\mathbf{St} \to \Lambda_\bk[\bar{\sigma}]$ is split as a map of $E_1$-algebras. To see this, observe that as an $E_1$-algebra $\Lambda_\bk[\bar{\sigma}]$ has a cell structure of the form
$$\Lambda_\bk[\bar{\sigma}] \simeq \gE_1(S^{1,1} \bar{\sigma}) \cup^{E_1}_{\bar{\sigma}^2} D^{2,1} \rho_2 \cup^{E_1} D^{3,5} \rho_3 \cup^{E_1} D^{4,7} \rho_4 \cdots.$$
(This is because $\Bar(\Lambda_\bk[\bar{\sigma}]) \simeq \Gamma_\bk[\bar{\bar{\sigma}}]$ is 1-dimensional in each bidegree $(n,2n)$.) The attaching map for the relation $\rho_n$ is an element of $\pi_{n,2n-2}$, so the obstructions to extending the canonical $E_1$-map $\gE_1(S^{1,1} \bar{\sigma}) \to \mathbf{St}$ to a map out of $\Lambda_\bk[\bar{\sigma}]$ lie in $\pi_{n,2n-2}(\mathbf{St})$, but the proof of the theorem shows that these homotopy groups vanish.

A choice of $E_1$-map $\Lambda_\bk[\bar{\sigma}] \to \mathbf{St}$ allows us to form the left $\mathbf{St}$-module 
$$\gM := \mathbf{St} \otimes_{\Lambda_\bk[\bar{\sigma}]} \bk.$$
Then $\Lambda_\bk[\bar{\sigma}] \otimes_\mathbf{St} \gM \simeq \bk$ has a slope 2 vanishing line through the origin, so by Theorem \ref{thm:DetNilp} (\ref{it:DetNilp:1}) it follows that $\gM$ also has a slope 2 vanishing line through the origin.
\end{rem}

The same principle used in the proof of Theorem \ref{thm:VanishingSteinbergHomologyFields} also gives the following: to improve it to slope 1 as in that theorem one would need to solve \cite[Conjecture 7.8]{e2cellsIV}. It improves upon \cite[Theorem 10.2]{e2cellsIV}.

\begin{thm}
If $\mathcal{O}$ is a connected semi-local ring with all residue fields infinite, then $H_{d}(\GL(\mathcal{O}^n);\St(\mathcal{O}^n))=0$ for $d < \tfrac{3}{4}(n-1)$.
\end{thm}
\begin{proof}
Finitely-generated projective $\mathcal{O}$-modules are free, and the Tits buildings for $\mathcal{O}$ are spherical by \cite[Theorem 7.1]{e2cellsIV}. Using \cite[Corollary 4.23]{MPW} it follows that
$$\pi_{n,d}(\Bar(\mathbf{St})) \cong  H_{d-2n}(\GL(\mathcal{O}^n) ; \St^2(\mathcal{O}^n)),$$
for the 2-dimensional Steinberg module $\St^2(\mathcal{O}^n)$ defined in \cite[Definition 7.2]{e2cellsIV}.

The 0-truncation map $\mathbf{RBS} \to \bk[\sigma]$ again induces a map
$$\Delta_\mathbf{St} \lra \Bar(\Bar(\bk[\sigma])) = \Bar(\Lambda_\bk[\bar{\sigma}]) = \Gamma_\bk[\bar{\bar{\sigma}}]$$
but in this case using \cite[Theorem 7.7]{e2cellsIV} we only know that the map is surjective and is an isomorphism in gradings $\leq 3$. This gives an $E_\infty$-map $\mathbf{St} \to \Lambda_\bk[\bar{\sigma}]$ which satisfies hypothesis ($\dagger$) of Theorem \ref{thm:DetNilp} (\ref{it:DetNilp:1}) with $\theta=\tfrac{7}{4}$, by a slope 2 version of Proposition \ref{prop:ChangeOfHopfAlg} (\ref{it:ChangeOfHopfAlg:1}). That theorem implies that $\mathbf{St}$ has the same slope $\tfrac{7}{4}$ vanishing line as $\Lambda_\bk[\bar{\sigma}]$, i.e.\ $\pi_{n,d}(\mathbf{St})=0$ for $d-1 < \tfrac{7}{4}(n-1)$, which translates to the claimed vanishing range.
\end{proof}

\bibliographystyle{plain}
\bibliography{MainBib}  

\end{document}